\documentclass{thesis}

\usepackage[latin2]{inputenc}
\usepackage{euscript}
\usepackage{amsmath}
\usepackage{amsthm}
\usepackage{epsfig}
\usepackage{amssymb}
\usepackage{amscd}
\usepackage{epic}
\usepackage[matrix,arrow]{xy}
\usepackage{enumerate}
\usepackage{graphicx}
\usepackage{array}

\usepackage[usenames,dvipsnames]{pstricks}
\usepackage{pst-node} 
\usepackage{auto-pst-pdf}
\usepackage{pst-grad} 
\usepackage{pst-plot} 

\usepackage{hyperref}
\hypersetup{colorlinks,bookmarksnumbered,linkcolor=blue,citecolor=blue,menucolor=blue}

\makeatletter
\hypersetup{pdftitle=\@title,pdfauthor=\@author}
\makeatother
\theoremstyle{plain}
\newtheorem{theorem}[subsubsection]{Theorem}
\newtheorem{lemma}[subsubsection]{Lemma}
\newtheorem{proposition}[subsubsection]{Proposition}
\newtheorem{corollary}[subsubsection]{Corollary}

\theoremstyle{definition}
\newtheorem{definition}[subsubsection]{Definition}
\newtheorem{example}[subsubsection]{Example}
\newtheorem{examples}[subsubsection]{Examples}
\newtheorem{remark}[subsubsection]{Remark}
\newtheorem{eexample}[subsection]{Example}

\newtheorem*{defnonum}{Definition}
\newtheorem{bekezdes}[subsubsection]{}
\newtheorem{La}[subsubsection]{Laufer algorithm}
\newtheorem*{SWI}{SWI Conjecture}
\newtheorem*{GSWI}{GSWI Conjecture}
\newtheorem*{Nconj}{N\'emethi's Conjecture}

\DeclareMathOperator{\Hom}{{\rm Hom}}

\DeclareMathOperator{\rank}{{\rm rank}}

\newcommand{\et}{{\mathcal T}}
\newcommand{\bH}{{\mathbb H}}

\def\C{\mathbb C}
\def\Q{\mathbb Q}
\def\R{\mathbb R}
\def\Z{\mathbb Z}

\def\N{\mathbb N}
\def\bS{\mathbb S}

\newcommand{\cale}{{\mathcal E}}

\newcommand{\calv}{{\mathcal V}}

\newcommand{\calt}{{\mathcal T}}

\newcommand{\calj}{{\mathcal J}}
\newcommand{\ocalj}{\overline{{\mathcal J}}}
\newcommand{\scalj}{{\mathcal J}^*}
\newcommand{\calQ}{{\mathcal Q}}
\newcommand{\calF}{{\mathcal F}}
\newcommand{\calC}{{\mathcal C}}
\newcommand{\calS}{{\mathcal S}}
\newcommand{\ocalS}{\overline{{\mathcal S}}}

\newcommand{\calP}{\mathcal{P}}

\newcommand{\cals}{{\bf s}}
\newcommand{\calst}{{\bf st}}

\newcommand{\zc}{\Z^2({\bf c})}
\newcommand{\aalpha}{\ell}
\newcommand{\K}{k_{can}}

\newcommand{\m}{\mathfrak{m}}

\newcommand{\vasi}{\mathbf{i}}
\newcommand{\vast}{\mathbf{t}}

\newcommand{\lk}{l'_{[k]}}

\newcommand{\ii}{({\bf i},\overline{I})}

\newcommand{\iij}{({\bf i},\overline{J})}
\newcommand{\bt}{{\bf t}}

\newcommand{\frsw}{\mathfrak{sw}}

\newcommand{\cO}{\mathcal{O}}
\newcommand{\cH}{\mathcal{H}}

\newcommand{\cV}{\mathcal{V}}

\newcommand{\cF}{\mathcal{F}}
\newcommand{\cP}{\mathcal{P}}

\newcommand{\cS}{\mathcal{S}}
\newcommand{\cE}{\mathcal{E}}
\newcommand{\cG}{\mathcal{G}}
\newcommand{\calN}{\mathcal{N}}

\newcommand{\calT}{\mathcal{T}}

\newcommand{\calR}{\mathcal{R}}

\newcommand{\calW}{{\mathcal W}}
\newcommand{\calK}{\mathcal{K}}

\newcommand{\bz}{{\bf z}}

\newcommand{\setZ}{\mathbb{Z}}

\def\bH{\mathbb H}

\newcommand{\hh}{\mathfrak{h}}
\newcommand{\bl}{{\bf l}}\newcommand{\bll}{l}

\newcommand{\frd}{\mathfrak{d}}\newcommand{\frdd}{\mathfrak{d}}
\newcommand{\fro}{\mathfrak{o}}
\newcommand{\fra}{\mathfrak{a}}
\newcommand{\bc}{{\bf c}}
\newcommand{\tc}{\widetilde{c}}
\newcommand{\wtt}{\widetilde{t}}
\newcommand{\well}{\widetilde{\aalpha}}

\def\C{\mathbb C}
\def\Q{\mathbb Q}
\def\R{\mathbb R}
\def\Z{\mathbb Z}

\def\N{\mathbb N}
\def\bS{\mathbb S}

\newcommand{\calX}{{\mathcal X}}
\newcommand{\calL}{{\mathcal L}}

\newcommand{\frI}{{\frak I}}
\newcommand{\frv}{\mathfrak{v}}
\newcommand{\frl}{\mathfrak{l}}
\newcommand{\fH}{\mathfrak{H}}
\newcommand{\reds}{\dagger}

\newcommand{\bms}{\mbox{\boldmath$s$}}

\newcommand{\sol}{\mathfrak{Sol}}

\newcommand{\labelpar}{\label}

\author{Tam\'as L\'aszl\'o}

\adviser{Professor Andr\'as N\'emethi}

\title{Lattice cohomology and Seiberg--Witten invariants of normal surface singularities}

\abstract{

One of the main questions in the theory of normal surface singularities is to understand 
the relations between their geometry and topology. 

The lattice cohomology is an important tool in the study of topological properties of a 
plumbed 3--manifold $M$ associated with a connected negative definite plumbing graph $G$. It connects 
the topological properties with analytic ones when $M$ is realized as a 
singularity link, i.e. when $G$ is a good resolution graph of the singularity. 
Its computation is based on the (Riemann--Roch)
weights of the lattice points of
$\Z^s$, where $s$ is the number of vertices of $G$.

The first part of the thesis reduces the rank of this lattice to the number of `bad' vertices of the graph.
Usually, the geometry/topology of $M$ is encoded exactly by these `bad' vertices and 
their number measures how far the plumbing graph stays from a rational one.

In the second part, we identify the following three objects: the Seiberg--Witten invariant of a 
plumbed 3--manifold,
the periodic constant of its topological Poincar\'e series, 
and a coefficient of an equivariant multivariable Ehrhart polynomial. 
For this, we construct the corresponding polytope from the plumbing graph, together with an action of
$H_1(M,\Z)$, and we develop Ehrhart theory for them. Moreover, we generalize the concept of the 
periodic constant for multivariable series and establish 
its corresponding properties. 

The effect of the reduction appears also at the level of the multivariable 
topological Poincar\'e series, simplifying the corresponding polytope and the Ehrhart theory as well.
We end the thesis with detailed calculations and examples.}

\acknowledgements{

First of all I would like to express my most sincere gratitude and appreciation to my adviser, Professor 
N\'emethi Andr\'as, who introduced me into this nice subject and showed me how the life of a mathematician 
looks like. 

The writing of this thesis would not have been possible without his 
patience, support and encouragement, and without all the time we spent together working on the problems. 
I really enjoyed his guidance and I look forward to continuing to learn and work with him in the years to come.

I am very thankful to the community of the Department of Mathematics and Applications at the Central European 
University, and of the R\'enyi Institute of Mathematics. Special thanks go to Stipsicz Andr\'as for his 
support and useful discussions. It has been an incredible privilege for me to be part of these communities.
\vspace{0.5cm}

Szeretn\'ek k\"osz\"onetet mondani az eg\'esz csal\'adomnak, de legf\H ok\'eppen \'edes\-any\'am\-nak, L\'aszl\'o 
Juliann\'anak, \'es \'edesap\'amnak, L\'aszl\'o Zsigmondnak, a rengeteg szeretet\'ert \'es 
\"onzetlen t\'amogat\'as\'ert, amelyek n\'elk\"ul\"ozhetetlenek voltak tanulm\'anyaim sor\'an \'es 
n\'elk\"ul\"ozhetetlenek napjainkban is.

\vspace{0.5cm}
I am also grateful to my friends and colleagues who make our stay in Budapest pleasant and homelike.

Last, but not least, I dedicate this work to the most important person in my life, my wife Eszter, 
thanking for her unimaginable love and support.}
\dedication{To Eszter}

\begin{document}

\makefrontmatter


\chapter{Introduction}\
\indent The subject of this thesis can be placed in the {\em local singularity theory}, 
which is a meeting point 
in mathematics, where many areas come together, such as algebra, geometry, topology and combinatorics, 
just to mention some of them. 

Before we start to describe this subject, we would like to offer this chapter for non--specialists 
as well, as a survey of this extremely active area of current research, with challenging problems. 
Since a lot of results and directions were developed in the last decades and the presentation of all of them 
would 
be too long, we have to pick some pieces to present here, which make 
the overall clear, but they are also important from our point of view. 
We hope that this chapter will give the frame of the whole picture drawn by the thesis.
\bigskip

In algebraic geometry, the research on the smooth complex algebraic surfaces has a history of more than a hundred years. It 
started with the classification of Enriques and the Italian school. Then in the 60's, 
a `modern' classification was provided by Kodaira, which uses the new techniques of algebraic geometry 
and topology, e.g. sheaves, cohomologies and characteristic classes, with paying particular attention to 
the relationships of the analytic structures with topological invariants of the underlying 
smooth 4--manifolds. Typical examples were the topological characterization of rational surfaces 
or of the K3 surfaces. Later, the works of Donaldson and Witten (on 4--dimensional Seiberg--Witten 
theories) gave powerful tools for this comparison research.

In parallel with these theories, the study of singular surfaces started, giving birth to the local 
singularity theory too. This theory investigates the local behavior of the singularities and has to solve 
new problems in the shadow of the old question:
\begin{center}
{\bf what is the relation between the analytic and topological structures?} 
\end{center}
This is the guiding 
question of our research too, targeting {\em normal surface singularities}.

\begin{defnonum}
 Let $f_1,\dots,f_m:(\C^n,0)\rightarrow (\C,0)$ be germs of analytic functions. 
Then the germ of the common zero set 
  $$(X,0)=(\{f_1=\dots=f_m=0\},0)\subset (\C^n,0)$$ is called a 
{\em complex surface singularity}, if the 
  rank of the Jacobian matrix $J(x):=(\partial f_i/\partial z_j(x))_{i,j}$ is $n-2$ 
for any smooth point $x\in X$. Moreover, if $\rank J(0)<n-2$, but $\rank J(x)=n-2$ for any point 
$x\in X\setminus 0$ , we say that $(X,0)$ has an {\em isolated singularity} at the origin.
\end{defnonum}
In particular, if $m=1$ we talk about {\em 2--dimensional hypersurface singularities} 
and if $m=n-2$, then our object is called a {\em complete intersection surface singularity}. 
Notice that in general, $m$ can be higher then $n-2$ (cf. \cite[1.2]{Nfive}). 

The local ring $\cO_{(X,0)}$ of analytic functions on $(X,0)$ is defined as the quotient 
of the ring $\cO_{(\C^n,0)}$ of power 
series, convergent in a small 
neighbourhood of $0$, by the ideal $(f_1,\dots,f_m)$. Its unique maximal ideal is 
$\mathfrak{m}_{(X,0)}=(z_1,\dots,z_n)$. 
This ring determines the singularity up to a local analytic isomorphism. 
Let us provide the following example: assume that $\Z_p$ acts on $\C^2$ by 
$\xi\ast(z_1,z_2):=(\xi z_1,\xi^{-1} z_2)$. This induces an action on 
$\cO_{(\C^2,0)}=\C\{z_1,z_2\}$, for which the ring of invariants is $\langle z_1^p,z_1z_2,z_2^p\rangle$. 
This is isomorphic with $\C\{u,v,w\}/\langle uw=v^p\rangle$, hence the geometric quotient 
$(\C^2,0)/\Z_p\simeq \{(u,v,w)\in (\C^3,0) \ : \ uw=v^p\}$ defines a surface singularity.

In any dimension, the {\em normality} condition means that we require $\cO_{(X,0)}$ to be integrally 
closed in its quotient field, or equivalently, a 
bounded holomorphic function defined on $X\setminus 0$ can be extended to a holomorphic function 
defined on $X$. In the case of surface singularities, this condition implies that $(X,0)$ has 
at most an isolated singularity at the origin (see \cite[\S 3]{Lauferb}).  

One can define several invariants from the local ring in order to encode the type of the singularity. 
For example, we mention the 
{\em Hilbert--Samuel function}, or, in particular, the {\em embedding dimension} and the {\em multiplicity}
of $(X,0)$. For their definitions and 
properties we refer to $\cite{Nfive}$.
\bigskip

One may think of a normal surface singularity as an abstract geometric object $(X,0)$ with its local 
ring $\cO_{(X,0)}$ and maximal ideal $\mathfrak{m}_{(X,0)}$, which encode the local analytic type. 
Then the main approach to analyze $(X,0)$ is a `good' resolution 
$\pi:(\widetilde X,E)\rightarrow (X,0)$. Thus, $\widetilde X$ is a smooth surface, 
$\pi$ is proper and maps $\widetilde X \setminus E$ 
isomorphically onto $X\setminus 0$, where the exceptional divisor $E=\pi^{-1}(0)$ is a 
normal crossing divisor. This means that the irreducible 
components $E_i$ are smooth projective curves, intersect each other transversally and 
$E_i\cap E_j \cap E_k=\emptyset$ for 
distinct indices $i,j,k$. 

The numerical analytic invariants of this description might come from two different directions. They can be ranks 
of sheaf--cohomologies of analytic vector bundles on $\widetilde X$. 
The most important in this category, from our viewpoint, is  
the {\em geometric genus}, which can be defined by the following formula 
$$p_g:=\dim H^1(\widetilde X,\cO_{\widetilde X}),$$ 
cf. \ref{AL:1}.
Notice that $p_g$ can be 
expressed on the level of $X$ as well, using 
holomorphic 2--forms (\cite{Laufer72}, \cite[1.4]{Nfive}). 

The other direction is based on the {\em Hilbert--Poincar\'e series} of $\cO_{(X,0)}$ 
associated with $\pi^{-1}(0)$--divisorial multi--filtration. 
We will give the definition of this invariants in Section 
\ref{s:motSW}, since they serve a motivation for the topological counterpart, 
which will be one of the main objects in Chapter \ref{c:SW}.

The resolution makes a bridge with the topological investigation of the normal surface 
singularity, which, as we will see in \ref{ss:link}, is equivalent with the description of its {\em link}. 
This is a special 3--manifold which can be constructed using the {\em dual resolution graph} too, 
via the configuration of the exceptional irreducible curves $E_i$ 
of the resolution (see Section \ref{ss:link}). 
 
It raises the following natural questions: 
\begin{center}\bf 
Is it possible to recover some of the analytic invariants from the link, 
or equivalently, from the resolution graph? 
What kind of statements can be made about the analytic type of a singularity 
with a given topology?
\end{center}

Before we start to discuss the main questions, which motivated a huge amount of work in the last decades, 
we stop for a moment and motivate the reason why we choose the study 
of the surface singularities. 

If $(X,0)$ is a curve singularity, then its link consists of as many disjoint copies of the 
circle $S^1$, as the number of irreducible components of the curve at its singular point. 
Hence, it contains no other information about the analytic type of $(X,0)$. 

We have the same situation 
in higher dimension too: from the point of view of the main questions, the topological information 
encoded by the link is rather poor. To justify this sentence, consider the example of a 
Brieskorn singularity $(X,0)=\{x_1^2+x_2^2+x_3^2+x_4^3=0\}\subset (\C^4,0)$. 
Brieskorn proved in \cite{B66a}, 
that the link is diffeomorphic to $S^5$, but $X$ is far from being smooth. More 
examples can be found also in \cite{B66b}. 

It turned out that the case of surfaces is much more interesting and complicated: {\em one can have many 
analytic types with a given topology}. This suggests that we have to assume some conditions 
on the analytic side in order to investigate the connections. 

%
%

H. Laufer in \cite{Laufer73} (completing the work of Grauert, Brieskorn, Tjurina and Wagreich) gives 
the complete list of those resolution graphs which have a unique analytic structure. 
These are the so--called {\em taut} singularities. He classified even those 
resolution graphs, which support only a finitely many analytic structures, 
they are called {\em pseudo--taut} singularities. This class 
is very restrictive, in the sense that almost all of them are rational. 
Hence, in order to understand the case of more general resolution 
graphs, we don't need a complete topological description, 
just to characterize topologically some of the discrete invariants of $(X,0)$. 

Summerizing the above discussion, Artin and Laufer in 60's and 70's started to determine some of 
the analytic invariants from the graph. They characterized topologically the {\em rational} and 
{\em minimally} elliptic singularities. Laufer believed that this program, 
called the {\em Artin--Laufer program} (Section \ref{s:AL}), stops for more general cases. 
After twenty years, N\'emethi \cite{Nem99} 
clarified the elliptic case completely, and proposed the continuation of the program 
with some extra assumptions. This means that if we pose some analytical and topological 
conditions on the singularities, there is a hope to 
understand the connection between analytic and topological data further. 

We believe that for the continuation of the Artin--Laufer program, i.e. to make topological 
characterizations of some special analytic types, or their invariants, 
we have to find and understand first the topological counterparts of the numerical analytic invariants.
\bigskip

In 2002, the work of N\'emethi and Nicolaescu (\cite{SWI,SWII,SWIII}) suggested a new approach. 
They formulated the so--called {\em Seiberg--Witten invariant 
conjecture} (see Subsection \ref{ss:SWIC}), which relates the geometric genus of the normal surface 
singularity with the Seiberg--Witten invariant of its link. This generalizes the conjecture of 
Neumann and Wahl (\cite{NW}), formulated for complete intersection singularities with 
integral homology sphere links. 

They proved the relation for some `nice' analytic structures, but later Luengo-Velasco, Melle-Hernandez 
and N\'emethi showed that it fails in general (see \ref{ss:SWIC} for details). However, the corrected versions 
transfer us into the world 
of low dimensional topology, and create tools to understand the Seiberg--Witten 
invariants via some homology theories (\ref{ss:swgen}). 
For example, the Seiberg--Witten 
invariants appear as the normalized Euler characteristic of the Seiberg--Witten 
Floer homology of Kronheimer and Mrowka, and of the Heegaard--Floer homology introduced by Ozsv\'ath 
and Szab\'o. These theories had an extreme influence on modern mathematics of the last years, 
and solved a series of open problems and old conjectures related to the classification of the 
smooth 4--manifolds and the theory of knots. 

Motivated by the work of Ozsv\'ath and Szab\'o and the Seiberg--Witten invariant conjecture, N\'emethi 
opened a new channel towards the continuation  of the Artin--Laufer program. 

He constructed a new 
invariant, the {\em graded root}, which is a special tree--graph with vertices labeled by integers. The 
main idea is that the set of topological types, sharing the same graded root, form a family with 
uniform analytic behavior too. N\'emethi conjectures that each family, identified by a root, can be 
uniformly treated at least from the point of view of the analytic invariants as well. The graded roots 
describe and give a model for the Heegaard--Floer homology in the rational and almost rational cases, 
using the {\em computation sequences} of Laufer. This concept gave birth to the {\em theory of lattice 
cohomology}, which is our main research subject in this thesis (Section \ref{s:motforlatcoh}, 
Chapters \ref{ch:2} and \ref{ch:5}).

It is a cohomological theory attached to a lattice defined by the resolution graph of the singularity.
The lattice cohomology is a topological invariant of the singularity link, which has a strong relation 
with the geometry of the exceptional divisors in the resolution. This has even more structures than the 
Heegaard--Floer homology. Nevertheless, by disregarding these extra data, conjecturally they are 
isomorphic (see \ref{ss:HF} for the details about the conjecture).

Moreover, the normalized Euler characteristic of the lattice cohomology equals the 
Seiberg--Witten invariant, and the non--vanishing of higher cohomology modules explains the failure 
and corrects the Seiberg--Witten invariant conjecture in the pathological cases (\ref{ss:SWIrev}).
\bigskip 

The calculation of the lattice cohomology is rather hard, since it is based on the weights of all 
the lattice points. In general, the rank of the lattice is large, it equals the number of vertices 
in the resolution graph.
However, one of the main results of the thesis is the proof of a {\em reduction procedure}, which reduces 
the rank of the lattice to a smaller number. This is the number of {\em `bad' vertices}. It measures 
the topological complexity of the graph (how far is from a rational graph) and encodes the 
geometrical/topological structure of the singularity. This shows that the reduction procedure is not 
just a technical tool, but an optimal way to recover the main geometric structure of the 3--manifold 
and to relate it with the lattice cohomology.
\bigskip

There is another concept which is strongly related to the Seiberg--Witten invariant conjecture and 
connects the geometry with the topology. This is the {\em theory of Hilbert--Poincar\'e series 
associated with the singularities}. Campillo, Delgado and Gusein-Zade studied Hilbert--Poincar\'e series 
associated with a divisorial multi--index filtration on $\cO_{(X,0)}$ (\ref{FM}). Then, N\'emethi 
unified and generalized the formulae of this concept and defined the topological counterpart, the 
{\em multivariable topological Poincar\'e series}, showing their coincidence in some `nice' 
cases. 

It was proven that the {\em constant term} of a quadratic polynomial associated with the topological 
series equals the Seiberg--Witten 
invariant. This shows a strong analogy with the analytic side, where the geometric genus can be 
interpreted in this way. This analogy, together with the interactions between the analytical and 
topological series, puts the Seiberg--Witten invariant conjecture in a new framework. Then, it is 
natural to ask whether the topological Poincar\'e series has a common generalization with the lattice 
cohomology. A simpler question targets those cohomological information which are encoded by the 
Poincar\'e series.

The first step is to recover the Seiberg--Witten invariants from this series. This subject contributes the second main part of this thesis 
(Chapters \ref{c:SW} and \ref{ch:4}) and has very 
interesting final outputs. 

It turnes out that the Seiberg--Witten invariant is the 
{\em multivariable periodic constant} of the 
Poincar\'e series. Moreover, the Ehrhart theory identifies the Seiberg--Witten invariant with a certain 
coefficient of a multivariable equivariant Ehrhart polynomial. Furthermore, 
the reduction procedure applies to this series as well 
and reduces its variables to the variables associated with the bad vertices.

Using this approach, in the cases of series with at most two variables, one can give precise 
algorithm for the calculation of the periodic constants, or equivalently, for the 
Seiberg--Witten invariants. Moreover, the case of two variables has a surprising relation with the 
theory of modules over semigroups and affine monoids.

\chapter{Preliminaries}\

\indent This chapter is devoted to the presentation of the classical definitions and results regarding to the 
topology of normal surface singularities. It provides the first interactions of the 
geometrical and topological settings, presents a conjecture regarding these invariants, 
and last but not least, it 
motivates the subject of the next chapter and of the thesis too, the {\em lattice cohomology and its 
reduction}. Besides the references given in the body of this chapter, we recommend the following classical 
books and lecture notes as well \cite{Milnorbook,Sev02,Dimca,NP,Nfive,Ninv}.

\section{Topology of normal surface singularities}\ 
In this section we give an introduction to the topology of normal surface singularities with 
the definition of the main object, the {\em link} of the singularity, 
and using its key properties, 
we show how one can encode its 
topological data with combinatorial objects.
\subsection{The link of $(X,0)$} \labelpar{ss:link}\
All this topological research area was started with a breakthrough of Milnor \cite{Milnorbook}, regarding complex 
hypersurfaces. Nevertheless, his argument works not only 
in the hypersurface case, but also in general, when we consider 
an arbitrary complex analytic singularity, see \cite{LeDT}. 
In the case of surfaces, the idea is the following.

We consider a normal surface singularity $(X,0)$ embedded into $(\C^n,0)$. Then, if $\epsilon$ is small enough, the 
$(2n-1)$--dimensional sphere $S^{2n-1}_{\epsilon}$ intersects $(X,0)$ transversally and the intersection 
$$M:= X \cap S^{2n-1}_{\epsilon}$$ 
is a closed, oriented $3$--manifold, which does not depend on the embedding and on 
$\epsilon$. M is called the {\em link} of $(X,0)$. 
Moreover, if $B^{2n}_{\epsilon}$ is the $2n$--dimensional ball 
of radius $\epsilon$ around $0$, then one shows that $X\cap B^{2n}_{\epsilon}$ is homeomorphic 
to the cone over $M$, hence the {\em link characterizes completely the local topological type 
of the singularity}. 

An important discovery of Mumford \cite{Mum61} was that if $M$ is simply connected, 
then $X$ is smooth at $0$. Neumann 
\cite{NP} extended this fact as follows: the link of a normal surface singularity can be recovered from its 
fundamental group except two cases, which are completely understood. These exceptions are the Hirzebruch--Jung 
(or cyclic quotient) and the cusp singularities.

The first connection between the analytical and topological properties of $(X,0)$ is 
realized by the {\em resolution of the 
singular point}. The resolution of $(X,0)$ is a holomorphic map $\pi:(\widetilde X,E)\rightarrow (X,0)$ 
with the properties, that $\widetilde X$ is smooth, $\pi$ is proper and maps 
$\widetilde X \setminus E$ isomorphically onto $X\setminus 0$. $E:=\pi^{-1}(0)$ 
is called the {\em exceptional divisor} with irreducible components $\{E_j\}_j$. If, moreover, we 
assume that $E$ is a normal crossing divisor, namely the irreducible 
components $E_j$ are smooth projective curves, intersect each other transversally and 
$E_i\cap E_j \cap E_k=\emptyset$ for distinct indices $i,j,k$, then we talk about {\em good} resolution. One 
can define {\em minimal} (not necessarily good) resolutions as well, if there is no rational smooth 
irreducible component $E_j$ 
with self--intersection number $b_j:=(E_j,E_j)=-1$. But in almost all the cases in our investigation we use 
the good resolution.

To encode the combinatorial data of 
a good resolution, one can associate to it the {\em dual resolution graph} $G(\pi)$ (usually we omit $\pi$ from the notation). 
In this graph the vertices correspond to the irreducible components $E_j$ and the edges represent their 
intersection points. Moreover, we add two weights for every vertex of $G$: the self--intersection number 
$b_j$ and the genus $g_j$ of $E_j$. In this way we may also associate an intersection form $\frI$ whose 
matrix is $(E_i,E_j)_{i,j}$, where $(E_i,E_j)$ is the number of edges connecting 
the two corresponding vertices for $i\neq j$.

The first result, originated from DuVal and Mumford, says that $G$ is connected and $\frI$ is negative definite. 
Then a crucial work of Grauert \cite{G62} shows that every connected 
negative definite weighted dual graph does arise 
from resolving some normal surface singularity $(X,0)$.

$\pi$ identifies $\partial \widetilde X$, the boundary of $\widetilde X$, with $M$. Hence, 
the graph $G$ can be 
regarded as a plumbing graph which makes $M$ to be an $S^1$--plumbed 3--manifold. Using the plumbing 
construction (see e.g. \cite[1.1.9]{Sev02}), any resolution graph $G$ determines $M$ completely. 
Conversely, 
we have to consider the equivalence class of plumbing graphs defined by finite sequences of blow--ups and/or 
blow--downs along rational $(-1)$--curves, since the resolution $\pi$ and its graph are not unique. But 
different resolutions provide equivalent graphs in the aforementioned sense. Then a result of Neumann 
\cite{NP} shows that the oriented diffeomorphism type of $M$ determines completely the equivalence class 
of $G$.

Finally, we define two families of 3--manifolds, which we will be working with throughout the thesis. $M$ is 
called {\em rational homology sphere} ($\Q HS$) if $H_1(M,\Q)=0$. In particular, we say that it is an 
{\em integral homology sphere} ($\Z HS$) if $H_1(M,\Z)=0$. \
Notice that $H_1(M,\Q)$ vanishes if and only if 
the free part of $H_1(M,\Z)$ vanishes. The plumbing construction says that the first Betti number 
$b_1(M)$ is equal to $c(G)+2\sum_j g_j$, where $c(G)$ is the number 
of independent cycles of the graph $G$. Hence, the final conclusion is that
\begin{center}
 $M$ is $\Q HS$ if and only if $G$ is a tree and $g_j=0$ for all $j$.
\end{center}

\subsection{Combinatorics of the resolution/plumbing graphs}\labelpar{intro:comb}\
Let $G$ be a connected negative definite plumbing graph and denote the set of vertices by $\calj$. 
As described in the previous section, 
it can be realized as the resolution graph of some normal surface singularity $(X,0)$, 
and the link $M$ can be considered as the plumbed 3--manifold associated with $G$. 

In the sequel {\bf we assume that $M$ is a $\Q HS$}.

Let $\widetilde{X}$ be the smooth 4-manifold with boundary $M$ obtained either by 
resolution $\pi:\widetilde{X}\to X$ of $(X,0)$ with resolution graph $G$, or via 
plumbing disc bundles associated with the vertices of $G$ with Euler number $b_j$ 
(for more details on plumbings we refer to \cite[\textsection 8]{HNK71} or \cite[1.1.9]{Sev02}). 
Since $\widetilde X$ has 
a deformation retract to the bouquet of $|\calj|$ copies of 2--spheres 
$S^2 \vee \dots \vee S^2$, the only non--vanishing homologies are $H_0(\widetilde X,\Z)=\Z$ and 
$H_2(\widetilde X,\Z)=\Z^{|\calj|}$. Moreover, there is an intersection form $\frI$ on $H_2(\widetilde X,\Z)$.
Since we identify the homology classes of the zero sections with $\{E_j\}_{j\in {\mathcal J}}$, 
the matrix of $\frI$ with respect to the basis $\{E_j\}_{j\in {\mathcal J}}$ is given by 
$$\frI_{ij}=\left\{
\begin{array}{cl}
b_j & \mbox{if} \ i=j\\
1 & \mbox{if $i\neq j$ and the corresponding vertices are connected by an edge}\\
0 & \mbox{if $i\neq j$ otherwise}.
\end{array}\right.
$$
We know that in our case $\frI$ is non--degenerate, negative definite and makes 
$L:=H_2(\widetilde{X},\Z)$ to be a lattice generated by $\{E_j\}_{j\in {\mathcal J}}$. 
Let $L':=Hom(L,\Z)$ be the dual of $L$.
The fact that the homology of $\widetilde X$ has no torsion part and the Poincar\'e--Lefschetz duality 
imply that $L'\cong H^2(\widetilde{X},\Z)\cong H_2(\widetilde{X},M,\Z)$. Then the begining of the 
long exact relative homology sequence for the pair $(\widetilde X,M)$ splits into the short exact 
sequence
$$0\longrightarrow L\stackrel{\iota}{\longrightarrow}
L'{\longrightarrow} H \longrightarrow 0,$$
where $H:=L'/L=H_1(M,\Z)$. The morphism $\iota:L\to L'$ can be identified with 
$L\to Hom(L,\Z)$ given by $l\mapsto (l,\cdot)$. The intersection form has a natural 
extension to $L_{\Q}=L\otimes \Q$ and we can regard $L'$ as a sublattice of $L_{\Q}$ in a way that 
$L'=\{l'\in L\otimes \Q\,:\, (l',L)\subseteq  \Z\}$. For conventional reason, one may choose 
the generators of $L'$ to be the (anti)dual elements $E_j^*$ defined via
$(E_j^*,E_i)=-\delta_{ji}$ (the negative of the Kronecker symbol). Clearly, the
coefficients of $E^*_j$ are the columns of $-\frI^{-1}$, and the negative definiteness of $\frI$ 
guarantees that  
\begin{equation}\label{eq:POS}
\mbox{all the entries of $E^*_j$ are strict positive.}
\end{equation}
We will also set $\det(G):=\det(-\frI)$ to be the determinant associated with the graph $G$.

\bekezdes{\bf Cycles.}\labelpar{comb:cycles} 
The elements of $L_{\Q}$ are called {\em rational cycles}. 
There is a {\em natural ordering} of them: $l'_1\leq l'_2$ if
$l'_{1j}\leq l'_{2j}$ for all $j\in {\mathcal J}$. Moreover, we say that $l'$ is 
{\em effective} if $l'\geq 0$.
If $l'_i=\sum_jl'_{ij}E_j$ for $i\in \{1,2\}$, then we write
$\min\{l'_1,l'_2\}:= \sum_j\min\{l'_{1j},l'_{2j}\}E_j$. Furthermore,
if $l'=\sum_jl'_jE_j$ then we set $|l'|:=\{j\in {\mathcal J}\,:\, l'_j\not=0\}$ for the
{\em support} of $l'$.

\bekezdes{\bf Characteristic elements and $spin^c$--structures of $M$.}\labelpar{comb:chandspinc}

We define the set of characteristic elements in $L'$ by
$$Char:=\{k\in L': \, (k,x)+(x,x)\in 2\Z \ \mbox{for any $x\in L$}\}.$$
There is a unique rational cycle $\K\in L'$ which satisfies the system of
{\it adjunction relations} 
\begin{equation}\label{eq:adj}
(\K,E_j)=-b_j- 2 \ \mbox{for all} \ j\in \calj,
\end{equation}
and it is called the  {\em canonical cycle}. Then $Char=\K+2L'$ and there is
a natural action of $L$ on $Char$ by $l*k:=k+2l$, whose orbits are
of type $k+2L$. Then $H=L'/L$ acts freely and transitively on the
set of orbits by $[l']*(k+2L):=k+2l'+2L$.

Consider the tangent bundle $T\widetilde X$ of the oriented 4--manifold $\widetilde X$ 
(we can pick a Riemannian metric as well). Then $T\widetilde X$ determines an orthonormal 
frame bundle (principal $O(4)$--bundle) which we denote by $F_O(\widetilde X)$. It is 
well known that the orientability of $\widetilde X$ means that this bundle can be reduced to 
an $SO(4)$--bundle, making the fibers connected. Can be thought in a way, that any 
trivialization of the bundle over the disconnected $0$--skeleton of $\widetilde X$ can be 
extended to a trivialization over the connected $1$--skeleton. In this sense, $spin$ and 
$spin^c$--structures are generalizations of the orientation. 

A $spin$--structure on 
$\widetilde X$ (more precisely on $T\widetilde X$) means that the trivialization of the tangent 
bundle can be extended to the $2$--skeleton. Then the $spin^c$--structure is a `complexified' version 
of that: we say that $\widetilde X$ has a $spin^c$--structure if there exists a  
complex line bundle $\calL$ so that $T\widetilde X \oplus \calL$ 
has a $spin$--structure. This $\calL$ is called the {\em determinantal line bundle} 
of the $spin^c$--structure. If $\widetilde X$ admits a $spin$--structure, 
then using the fiber product one can construct a {\em canonical} $spin^c$--structure 
as well. This can be done also when an almost complex structure is given.
(More details regarding of these definitions and constructions can be found in 
\cite{GS,M95}.)

By \cite[Proposition 2.4.16]{GS}, the fact that in our case $L'=H^2(\widetilde X,\Z)$ has no 
$2$--torsion implies that $\calL$ determines the $spin^c$--structure, and the first Chern class 
(of $\calL$) realizes an identification
between the set of $spin^c$--structures $Spin^c(\widetilde{X})$ on
$\widetilde{X}$ and $Char\subseteq L'$. Moreover, 
$Spin^c(\widetilde{X})$ is an $L'$ torsor compatible with the
above action of $L'$ on $Char$.

If we look at the boundary, the image of the restriction $Spin^c(\widetilde{X})
\to Spin^c(M)$ consists of exactly those $spin^c$--structures on $M$, whose 
Chern classes are the restrictions $L'\to H^2(M,\Z)\cong H_1(M,\Z)$, 
i.e. are the torsion elements in $H$. \newline
Therefore, in our situation, 
all the  $spin^c$--structures on
$M$ are obtained by restriction,
$Spin^c(M)$ is an $H$ torsor, and the actions are compatible with the factorization
$L'\to H$. Hence, one has
an identification of $Spin^c(M)$  with the set of $L$--orbits of
$Char$, and this identification is compatible with the
action of $H$ on both sets. In this way, any
$spin^c$-structure of $M$ will be represented by 
$[k]:=k+2L\subseteq Char$. The canonical $spin^c$--structure corresponds to $[-\K]$, moreover 
$[k]$ has  the form $k_{can}+2(l'+L)$ for some $l'\in L'$.

\bekezdes{\bf The distinguished representatives of $[k]$.}\labelpar{ss:distin}
Notice that if we look at the (anti)canonical $spin^c$--structure $[\K]$, there is a special 
element in this orbit, namely the canonical cycle $\K$. In the following, we generalize this 
fact for all $[k]$: 
among all the characteristic elements in $[k]$ we will choose a very special one.

We define first the {\em Lipman (or anti--nef) cone} 
\begin{equation}\label{eq:Lipman}
\calS':=\{l'\in L' \ : \ (l',E_j)\leq0 \ \mbox{ for any }j\in\calj\}.
\end{equation}
Since $\frI$ is negative definite, if $l'\in \calS'$ then $l'\geq 0$.
Then we have the following lemma:
\begin{lemma}\emph{(\cite[5.4]{OSZINV})}\labelpar{def:KR}
 If we fix $[k]=\K+2(l'+L)$, there is a unique minimal element $l'_{[k]}$ of $(l'+L)\cap \calS'$.
\end{lemma}

\begin{definition}(\cite[5.5]{OSZINV})\labelpar{distrep}
For any class $[k]$ we define the {\em distinguished representative} 
$k_r:=k_{can}+2l'_{[k]}$.
\end{definition}
For example, since the minimal element  of  $L\cap \calS'$ is the  zero cycle,
we get  $l'_{[k_{can}]}=0$, and the distinguished representative in $[k_{can}]$
is  the canonical cycle $k_{can}$ itself. In general, $l'_{[k]}\geq 0$.
For their importance see Section \ref{red:section}, and further properties 
can be found in \cite{OSZINV,Nlat,Ng}. This motivates also to partition the elements of 
the Lipman cone into different classes $[k]$, therefore we define
\begin{equation}\label{eq:Lipmank}
 \calS_{[k]}:=\{l\in L \ : \ (l+\lk,E_j)\leq 0 \ \mbox{ for any }j\in\calj\}.
\end{equation}

\section{The Artin--Laufer program. Case of rational singularities}\labelpar{s:AL}

\subsection{Algebro--geometric definitions and preliminaries}\labelpar{AL:1}\
The aim of this section is to introduce some tools from the analytical (algebro)--geometric point of view 
for the study of the normal surface singularity $(X,0)$. Since our work restricts to the topology of $(X,0)$, this description will be 
rather sketchy: we need just those parts, which motivate the names and notations in \ref{intro:comb} 
and create the main tools connecting the geometry and topology of $(X,0)$. For more details regarding 
of this section, we recommend some general references such as \cite{Nfive} and \cite{BPVVbook}.

\bekezdes\labelpar{AL:11} 
We start with a resolution $\pi:\widetilde X\to X$. The group of divisors $Div(\widetilde X)$ 
of $\widetilde X$ consists of formal finite sums $D=\sum_i m_i D_i$, where $D_i$ is an irreducible curve 
on $\widetilde X$ and $m_i \in \Z$. For any divisor $D$, one can say that it is supported on $|D|=
\cup_{m_i\neq 0} D_i$. If we pick a meromorphic function $f$ defined on $\widetilde X$, then 
$(f)=\sum_i m_i D_i$ is a {\em principal divisor}, where $D_i$'s are irreducible components of the zeros 
and the poles of $f$, and $m_i$ is the multiplicity (order of zero, resp. pole) of $f$ along $D_i$. 
Divisors supported  on the exceptional divisor $E$ are called cycles, already defined in \ref{intro:comb}. 
We have seen, that one can define a natural ordering, the effectiveness and the intersection of cycles, which 
is determined by the resolution graph $G$. 

The pullback $f\circ \pi$ of a given analytic function $f:(X,0)\to (\C,0)$ determines 
an effective principal divisor $\mathrm{div}(f\circ\pi)$. 
Let $m_{E_j}(f)$ be the multiplicity of $\mathrm{div}(f\circ\pi)$ 
along $E_j$, then $\mathrm{div}(f\circ\pi)=\sum_{j\in \calj} m_{E_j}(f) E_j + St(f)$, where $St(f)$ is supported 
on the strict transform $\overline{\pi^{-1}(f^{-1}(0)\setminus 0)}$ 
(closure of $\pi^{-1}(f^{-1}(0)\setminus 0)$) of the set $f^{-1}(0)$. Then, 
for such an $f$ and a resolution $\pi$ (encoded by its resolution graph $G$) one can associate the 
cycle 
$$(f)_G=\sum_{j\in \calj} m_{E_j}(f) E_j.$$

In order to get some information on the local ring $\cO_{(X,0)}$ (i.e. about the structure of 
analytic functions $f:(X,0)\to (\C,0)$) from the resolution, we may define the set of 
cycles 
\begin{equation}
 \calS_{an}:=\{(f)_G \ : \ f\in \mathfrak{m}_{(X,0)} \}.
\end{equation}
Then $\calS_{an}$ is an ordered semigroup and if $l_1,l_2\in \calS_{an}$, then $l=\min\{l_1,l_2\}$ 
(defined in \ref{comb:cycles}) is an element of $\calS_{an}$ too. This fact guarantees the existence 
of a unique non--zero minimal element in $\calS_{an}$, which, according to S.S.-T. Yau, is called the 
{\em maximal ideal cycle} of the singularity and it is denoted by $Z_{max}$. One can show 
that $Z_{max}$ (or the whole $\calS_{an}$) depends on the analytic structure of the $(X,0)$. In 
general, it can not be recovered from the topology. However, there are some cases, when this situation 
may happen.

Can be proven, that for any $f\in \mathfrak{m}_{(X,0)}$ one has $(\mathrm{div}(f\circ\pi), E_j)=0$ 
for all $j\in \calj$. This, together with $(St(f), E_j)\geq 0$ imply that 
$((f)_G, E_j)\leq 0$ for every $j\in\calj$. This 
motivates the definition of the `topological candidate' for $\calS_{an}$, namely 
$$\calS_{top}:=\{l\in L \ : \ (l,E_j)\leq 0 \ \forall j\in \calj\}.$$ Notice that this is the 
same as the Lipman cone $\calS_{[\K]}$ (\ref{eq:Lipmank}) defined for $[\K]$.

$\calS_{top}$ shares the same properties as mentioned before for the analytic counterpart. Hence, 
it has a unique non--zero minimal element $Z_{min}$, which was introduced by Artin 
\cite{Artin62,Artin66} and we call it the {\em minimal cycle} or {\em Artin's (fundamental) cycle}. 
Notice that, since $\calS_{an}\subseteq \calS_{top}$, we have $Z_{min}\leq Z_{max}$, where in general 
strict inequality appears.

It turns out that $Z_{min}$ can be calculated easily by an algorithm on the graph $G$, 
established by Laufer \cite{Laufer72}. This is fundamental from the point of view of the generalized 
Laufer sequences which will be discussed in Section \ref{red:gLauferseq}. 

\begin{La}\labelpar{La}One constructs a sequence $\{z_n\}_{n=1}^t$ of cycles as follows. 
 \begin{enumerate}
  \item Start with a cycle $z_1=E_j$ for some $j\in \calj$.
  \item If $z_n$ is already constructed for some $n>0$ and there exists some $E_{j(n)}$ 
for which $(z_n, E_{j(n)})>0$, then set $z_{n+1}=z_n+E_{j(n)}$.
  \item If $(z_n, E_j)\leq 0$ for all $j$, then stop and $z_n$ gives $Z_{min}$.
 \end{enumerate}
\end{La}

\bekezdes\labelpar{CohSh} Some invariants of the geometry can be deduced from the cohomology 
of sheaves on $\widetilde X$. For example, consider $\cO_{\widetilde X}$, the sheaf of 
holomorphic functions on $\widetilde X$ and $\cO^*_{\widetilde X}$, 
the subsheaf of invertible functions. 
We may also consider the group $Pic(\widetilde X)$ of holomorphic line bundles on $\widetilde X$ 
(modulo isomorphism), which is naturally isomorphic to $H^1(\widetilde X, \cO^*_{\widetilde X})$. 
Notice that the groups $H^1(\widetilde X, \cO_{\widetilde X})$, or 
$H^1(\widetilde X, \cO^*_{\widetilde X})$ does not depend 
on the resolution $\pi:\widetilde X\to X$.

The analytic invariant 
$h^1(\widetilde X, \cO_{\widetilde X}):=\dim H^1(\widetilde X, \cO_{\widetilde X})$ 
is called the {\em geometric genus} of the singularity and will be denoted by $p_g$.

To any integral cycle $l=\sum_j m_j E_j$ we can associate the line bundle $\cO(-l)$, defined by the 
invertible sheaf of holomorphic functions on $\widetilde X$, which vanish of order $m_j$ on $E_j$. 
One can define $\cO_l:=\cO_{\widetilde X}/\cO(-l)$ as well.

According to \cite[\textsection 6]{BPVVbook}, the short exact exponential sequence
$$0\longrightarrow \Z_{\widetilde X}\longrightarrow
\cO_{\widetilde X}\stackrel{exp}{\longrightarrow} \cO^*_{\widetilde X}\longrightarrow 0$$
gives rise to the long exact exponential cohomology sequence, which in our case ($\widetilde X$ 
 is a smooth complex surface, $M$ is $\Q HS$) splits into the short exact sequence 
\begin{equation}\label{eq:ses}
0\longrightarrow \C^{p_g}\longrightarrow
Pic(\widetilde X)\stackrel{c_1}{\longrightarrow} L'\longrightarrow 0,
\end{equation}
where $c_1(\calL)$ is the first Chern class of $\calL \in Pic(\widetilde X)$. Notice that 
$c_1(\cO(-l))=l$, hence $c_1$ admits a group--section $L\to Pic(\widetilde X)$ above 
the subgroup $L$ of $L'$ which, in fact, can be extended naturally to $L'\to 
Pic(\widetilde X)$ (see \cite[3.6]{Nline}), defining the line bundles $\cO(-l')$ for any $l'\in L'$. 

As an example, we denote by $\varOmega^2_{\widetilde X}$ the sheaf of holomorphic $2$--forms on 
$\widetilde X$. It is an element of $Pic(\widetilde X)$, hence it corresponds to a class of divisors. 
Modulo the principal divisors, this class well--defines the {\em canonical divisor} $K_{\widetilde X}$. 
The adjunction formula showes that the intersections with the exceptional divisor can be calculated 
via the equations $(K_{\widetilde X},E_j)=-b_j-2$ for all $j$. $K_{\widetilde X}$ is analytic, but 
one can associate to it the canonical cycle $c_1(\varOmega^2_{\widetilde X})\in L'$, which is the same 
as $k_{can}$ in \ref{comb:chandspinc}. 
\begin{definition}
 We say that $(X,0)$ is {\em Gorenstein} if we can find a section $\widetilde\omega$ of 
 $\varOmega^2_{\widetilde X}$ whose divisor is supported on $E$.
 It is {\em numerically Gorenstein} if the coefficients of $\K$ are integers. \newline
 Notice that the first definition is equivalent with the fact that there is a global section of $\varOmega^2_{X\setminus 0}$ 
 which is nowhere vanishing on $X\setminus 0$, i.e. $\varOmega^2_{X\setminus 0}$ is holomorphically 
 trivial. On the other hand, numerical Gorenstein property means that $\varOmega^2_{X\setminus 0}$ 
is a topologically trivial line bundle. Therefore, if $(X,0)$ is Gorenstein, 
 then it is numerically Gorenstein as well.
The general theory says that numerical Gorenstein property is 
 equivalent with the fact that the first Chern class of $\varOmega^2_{\widetilde X}$ projected to 
$H^2(M,\Z)=H^2(X\setminus 0,\Z)$ is zero. 
In the sense of \ref{comb:chandspinc}, this means that the class 
 of $k_{can}$ in $H$ is zero, hence $k_{can}\in L$.
 
 As a generalization, one can define the {\em $\Q$--Gorenstein} property as well, which requires that some 
 power of $\varOmega^2_{X\setminus 0}$ should be holomorphically trivial.
\end{definition}

The formal neighbourhood theorem implies that $p_g=\dim_{\C}\varprojlim_{l>0} H^1(\widetilde X, \cO_l)$, 
hence if one wishes to compute $p_g$, one has to understand 
$\dim_{\C}H^1(\widetilde X, \cO_l)$ for $l\in L$ and $l>0$. Then, by 
Riemann--Roch theorem, it is known that although $\dim_{\C}H^0(\widetilde X, \cO_l)$ and 
$\dim_{\C}H^1(\widetilde X, \cO_l)$ are analytic, 
\begin{equation}\label{eq:chi}
\chi(l):=\chi(\cO_l)=\dim_{\C}H^0(\widetilde X, \cO_l)-\dim_{\C}H^1(\widetilde X, \cO_l)
\end{equation}
is topological and equal to $-(l,l+\K)/2$. One can consider also the `twisted' version, i.e. we fix 
an $\calL\in Pic(\widetilde X)$ and write $c_1(\calL)=l'\in L'$ for its Chern class. If we set 
$k:=\K-2l'$, then $\chi(\calL\otimes\cO_l)=-(l,l+k)/2$.

In this way, for any characteristic element $k\in Char$ one defines a function  
\begin{equation}\label{eq:RR}
 \chi_k:L\to \Z \ \ \ \mbox{by} \ \ \ \chi_k(l)=-\frac{1}{2}(l, l+k).
\end{equation}
This function will be the main ingredient defining the lattice cohomology in Chapter \ref{ch:2}, 
and somehow hides a deep connection with this analytic theory. 

In the sequel, we keep the notation $\chi$ associated 
with $\K$.

\subsection{Artin--Laufer program}\labelpar{AL:2}\
As we mentioned in the introductory part, it is interesting to investigate special families 
of normal surface singularities, where some of the analytic 
invariants (coming from $\cO_{(X,0)}$) are topological. Since one of the most important 
numerical analytic invariants of $(X,0)$ is the geometric genus $p_g$, we will localize our discussion 
around it. However, at some point we will mention what is happening with some other analytic invariants 
(defined in \ref{AL:1}) as well.

The {\em Artin--Laufer program} has a long history, started with the work of Artin in the 60's. 
In \cite{Artin62,Artin66} he showed that the {\em rational singularities} can 
be characterized completely from the graph (see also \ref{AL:3}).
 He computed even the multiplicity and the embedding dimension of 
these singularities from the topological data.

Then Laufer \cite{Lauferb,Laufer72} developed further the theory. 
Among others, he found an algorithm for finding the 
Artin's cycle $Z_{min}$, which is now called the {\em Laufer algorithm}, see \ref{La},  
and extended the topological characterization of rational singularities to 
{\em minimally elliptic singularities} (\cite{Laufer77}) as well. He also noticed that for more complicated 
singularities the program can not be continued. 

However, N\'emethi's work in \cite{Nem99} pointed out and conjectured that 
if we pose some analytical and topological conditions, e.g. the Gorenstein and $\Q HS$ conditions, 
then some numerical analytic invariants 
(including $p_g$) are topological. 
This was carried out explicitly for elliptic singularities.

In order to achieve results in the topological characterization of the aforementioned invariants, one has 
to find their `good' topological candidate. E.g., in the case of $p_g$ one has to find a topological 
upper bound for any normal surface singularities with $\Q HS$, which is optimal in the sense that for some 
`nice' singularities it yields exactly $p_g$. A good example for this phenomenon is 
the length of the elliptic sequence, the upper bound 
valid for elliptic singularities, introduced and intensively 
studied by Laufer \cite{Laufer77} and S.S.-T. Yau \cite{Yau80}. In Section \ref{SW:2}, we will expose 
another candidate for $p_g$ and give some details on the development of results of the last ten years. 
Another example can be found in Chapter \ref{c:SW}, where we study the topological counterpart of the 
Hilbert--Poincar\'e series associated with some filtrations on $(X,0)$.

But first, we recall the Artin--Laufer characterization of rational singularities, since this 
class is the origo of our research in the topology of normal surface singularities.

\subsection{Rational singularities}\labelpar{AL:3}\
In general (without any assumption on the link), a normal surface singularity $(X,0)$ is called 
{\em rational} if $p_g=0$. The formal neighbourhood theorem immediately implies that this is equivalent 
with $\dim_{\C} H^1(\widetilde X,\cO_l)=0$ for any $l>0$. 
In particular, this induces the vanishing of 
all the genera $g_j$ and that $G$ should be a tree. Hence the link of a rational singularity 
is automatically a $\Q HS$.

Notice that somehow the definition of the rational singularity is motivated by the short exact 
sequence \ref{eq:ses}, since if $p_g=0$, then $Pic(\widetilde X)$ is isomorphic to $L'$, hence 
it is completely topological. Artin \cite{Artin62,Artin66} proved that in this case 
$\calS_{an}=\calS_{top}$ 
and $Z_{max}=Z_{min}$. These equalities were enough to calculate some analytic invariants, such as 
the multiplicity, the embedding dimension and the Hilbert--Samuel function in terms of $Z_{min}$, 
which shows how this cycle controls most of the geometry of the rational singularities.

Moreover, Artin succeeded to replace the vanishing of $\dim_{\C} H^1(\widetilde X,\cO_l)$ by a 
criterion formulated in terms of $\chi(l)$, namely $\chi(l)\geq 1$ for all $l>0$. 
However, in general it is difficult 
to verify this criterion for all positive cycles. Therefore, another breakthrough was that, in fact, it 
is enough to consider only the Artin's cycle $Z_{min}$, since it controls the criterion for all the 
other positive cycles as well. This fact can be formulated also in terms of the Laufer algorithm. 

In the next theorem we summarize the results of Artin and Laufer, characterizing topologically 
the rational singularities.

\begin{theorem}[Topological characterization of rational singularities]\labelpar{thm:rat} 
Let $(X,0)$ be a normal surface singularity, then the following statements are equivalent:
\begin{enumerate}
 \item $p_g=0$;
 \item $\chi(l)\geq 1$ for any $l>0$;
 \item $\chi(Z_{min})=1$;
 \item In the Laufer algorithm \ref{La} one has $(z_n,E_{j(n)})=1$ for every $n\geq 1$.
\end{enumerate}
\end{theorem}

Starting from the topological point of view, we may set the following definition:

\begin{definition}\labelpar{def:rat}
 If a resolution graph $G$ satisfies one of the last three conditions in the last theorem, 
 we say that $G$ is a {\em rational graph}.
\end{definition}
The class of rational graphs is closed under taking subgraphs and decreasing the 
self--intersections. We observe that $Z_{min}\geq \sum_{j\in\calj}E_j$, a fact which follows from 
the Laufer algorithm and connectedness of $G$. If we have equality, we say that 
$Z_{min}$ is a {\em reduced Artin's cycle}: in this case $(X,0)$ is called {\em minimal 
rational} and $G$ is a {\em minimal rational graph}.

\begin{examples}\
 \begin{enumerate}
  \item Let $G$ be an arbitrary tree with all the genus decorations zero. For any vertex $j$, 
  we define the {\em valency} $\delta_j$ as the number of edges with endpoint $j$. Let 
   $$ b_j=\Bigg\{
  \begin{array}{l}
   -\delta_j \ \ \mbox{if} \ \ \delta_j\neq 1\\
   -2 \ \ \mbox{if} \ \ \delta_j=1
  \end{array} \ \ \mbox{for any} \ j\in \calj.
   $$
   Then the intersection matrix $\frI$ is automatically negative definite and with the 
   Laufer algorithm one can show that $Z_{min}=\sum_{j\in \calj}E_j$ and $\chi(Z_{min})=1$. 
   Hence, any $(X,0)$ with minimal resolution graph $G$ is a minimal rational singularity.
   
  \item Assume that $(X,0)$ is rational and numerically Gorenstein. We can show that 
  $\K=0$, hence the adjunction formulae \ref{eq:adj} implies $b_j=-2$ for all $j$. This 
  graphs are the minimal resolution graphs of rational double points (or ADE singularities). 
 \end{enumerate}
\end{examples}

As we will see in Section \ref{ss:bv}, from topological point of view, the rationality can be 
generalized and all the resolution graphs can be sorted into classes, where {\em lattice cohomology 
will serve as a measuring object for the topology of the corresponding singularities}.

\section{Seiberg--Witten invariants and a conjecture of N\'emethi and Nicolaescu}\labelpar{SW:2}\
Historically, the {\em Seiberg--Witten invariants} were defined for compact smooth $4$--manifolds. 
They were introduced by Witten \cite{Witten} during his investigation with Seiberg on the Seiberg--Witten 
gauge theory in theoretical physics. They are similar to the invariants defined by the 
Donaldson theory, and they provide a strong tool in proving key results of the smooth $4$--manifolds. 
Their advantage is that the main objects which define the numerical data, the moduli spaces of 
solutions of the Seiberg--Witten equations, are mostly compact, hence can be avoided the problems 
coming from the compactification of the moduli spaces in Donaldson theory. 

Besides the original work of Witten, detailed presentation of the theory can be found in the 
book of Nicolaescu \cite{Nic00_1}, see also the book of Morgan \cite{Morgan}.

\subsection{Seiberg--Witten invariants for closed 3--manifolds}\labelpar{ss:swgen}\
In our case, we analyze the Seiberg--Witten invariants for closed $3$--manifolds. 
Considering an additional geometric data $(g,\eta)$ on $M$, where $g$ is a Riemannian metric and 
$\eta$ is a closed $2$--form, one can define the {\em Seiberg--Witten equations} (we refer to 
\cite{Lim,Nic04} for precise definitions and details). Then for any $spin^c$--structure $\sigma$ 
on $M$, the space of solutions divided by the gauge group defines the moduli space of $(\sigma,
g,\eta)$--monopoles, and the {\em Seiberg--Witten invariant} $\widetilde\frsw_{\sigma}(M,g,\eta)$ 
is the signed count of them.

It turns out that, when $b_1(M)=0$ (i.e. $M$ is a $\Q HS$), the situation is the worst, since 
$\widetilde\frsw_{\sigma}(M,g,\eta)$ depends on the choice of the parameters $g$ and $\eta$, thus it 
is not an invariant. However, altering by a counter term $KS_M(\sigma,g,\eta)$, 
called the {\em Kreck--Stolz} invariant, 
solves the problem. Therefore we define the `modified' Seiberg--Witten 
invariant 
$$\frsw_{\sigma}(M):=\frac{1}{8}KS_M(\sigma,g,\eta)+\widetilde\frsw_{\sigma}(M,g,\eta).$$

\begin{theorem}[\cite{Lim}]  
 If $M$ is a connected $3$--manifold with $b_1(M)=0$, then 
\begin{equation*}
 \frsw: Spin^c(M) \longrightarrow \Q \ \ (\mbox{more precisely} \ \Z[1/8 |H|]) 
\end{equation*}
is an oriented diffeomorphism invariant of $M$.
\end{theorem}

In general, it is extremely difficult to compute $\frsw_{\sigma}(M)$ using its analytic definition. 
Therefore, there are some projects which aim to replace this definition with a different one, or, to 
provide a topological/combinatorial calculation for the invariants:
\begin{itemize}
 \item Answering a question of Turaev, Nicolaescu's result \cite{Nic04} shows that 
$\frsw_{\sigma}(M)$ is the Reidemeister--Turaev torsion normalized by the Casson--Walker invariant. 
This identification is based on the surgery formula for the monopole count given by Marcolli and 
Wang \cite{MW}, and for the Kreck--Stolz invariant contained in the paper of Ozsv\'ath and Szab\'o 
\cite{OSz00}. In terms of the graph $G$, combinatorial formula for the Casson--Walker invariant 
is given in the book of Lescop \cite{Lescop}, while the Reidemeister--Turaev torsion is determined 
by N\'emethi and Nicolaescu in \cite{SWI}. This formula for the torsion is based on a 
Dedekind--Fourier sum which, in most of the cases, is still hard to determine.  

On the other hand, Braun 
and N\'emethi \cite{BN} provides a cut--and--paste surgery formula for the Seiberg--Witten invariants 
in the case of negative definite plumbed 3--manifold, motivated by Okuma's formula \cite{Opg} 
targeting analytic invariants of splice--quotient singularities.

\item Another program is the {\em categorification} of the invariants. The aim is to construct homological 
theories whose `normalized Euler characteristic' gives the Seiberg--Witten 
invariant (with a suitable normalization). This interpretation also gives several alternative definitions 
for the $\frsw_{\sigma}(M)$. 

For examples, with a generalization of the Seiberg--Witten mo\-no\-poles, 
Kronheimer and Mrowka \cite{KM} constructed the {\em Seiberg--Witten Floer homology} which, in fact, is 
isomorphic to the {\em Heegaard--Floer homology}, developed by Ozsv\'ath and Szab\'o 
\cite{OSz,OSz7,OSzP}, and they categorify $\frsw_{\sigma}(M)$. Moreover, as a consequence of exact `triangles', 
one also gets further surgery formulae for the Seiberg--Witten invariants.

In \cite{NSW}, N\'emethi proved that the normalized Euler characteristic of the lattice cohomology 
is also the Seiberg--Witten invariant. This proof uses the surgery formula of \cite{BN}.
\end{itemize}

We will give more details regarding the Heegaard--Floer homology and its relation with the lattice 
cohomology in Sections \ref{s:motforlatcoh} and \ref{ss:HF}. Moreover, Chapter 
\ref{c:SW} provides an Ehrhart theoretical interpretation of the Seiberg--Witten invariants, which 
(at least in special cases) calculates them by using the {\em topological Poincar\'e series}, 
without knowing the Betti numbers of the lattice cohomology.

\subsection{The Seiberg--Witten invariant conjecture}\labelpar{ss:SWIC}\
In the spirit of the Artin--Laufer program, the article \cite{SWI} of N\'emethi and Nicolaescu 
formulates the following conjecture, giving a possible topological counterpart for the geometric 
genus. It is an extension of the Casson invariant conjecture of Neumann and Wahl \cite{NW}.

\begin{SWI}[\cite{SWI}]
Assume that $(X,0)$ is a normal surface singularity whose link $M$ is a $\Q HS$. Then the following 
facts hold:
\begin{enumerate}
 \item There is a topological upper bound for $p_g$, given by 
  $$p_g\leq \frsw_{\sigma_{can}}(M)-\frac{\K^2+|\calj|}{8}.$$
 \item If $(X,0)$ is $\Q$--Gorenstein, then in part 1 one has equality.
\end{enumerate}
\end{SWI}
This can be generalized in the following way:

\begin{GSWI}[\cite{Nline}] We consider $l'\in L'$ and define the characteristic element 
$k:=\K-2l'\in Char$. Then 
\begin{enumerate}
 \item For any line bundle $\calL \in Pic(\widetilde X)$ with $c_1(\calL)=l'$ one has 
  $$\dim_{\C}H^1(\widetilde X, \calL)\leq -\frsw_{[k]}(M)-\frac{k^2+|\calj|}{8}.$$
 \item If $\calL=\cO_{\widetilde X}(l')$ and $(X,0)$ is $\Q$--Gorenstein then in part 1 one has 
  equality.
\end{enumerate}
\end{GSWI}
In particular, if $\calL=\cO_{\widetilde X}$, then we get back SWI. The conjecture was verified 
first in \cite{SWI} for some families of rational, elliptic and hypersurface singularities. 
It was proved also for singularities with good $\C^*$--action \cite{SWII} and for suspension 
singularities (of type $\{f(x,y)+z^n=0\}$ with $f$ irreducible) \cite{SWIII}. 
Then \cite{NCL} proves the validity of the conjecture for {\em splice--quotient singularities}, 
a class which was defined by Neumann and Wahl \cite{nw-CIuac} and contains most of the other classes above.

Using the 
Heegaard--Floer homological interpretation of $\frsw$, N\'emethi verified GSWI for all 
{\em almost rational singularities} (see \ref{ss:bv} for their definition).

Unfortunately, the conjecture at this generality is {\bf not true}. A paper of Luengo-Velasco, 
Melle-Hern\'andez and N\'emethi on {\em superisolated singularities} \cite{LMN} 
gives counterexamples even for the SWI case (Example \ref{example:superisol}). 
However, one can use lattice cohomological methods to correct the upper bound, reinterpreting 
the topological candidate for the geometric genus (or for $\dim_{\C}H^1(\widetilde X, \calL)$). 
We will return to this discussion in Section \ref{ss:SWIrev}.

\section{Motivation of the Lattice cohomology} \labelpar{s:motforlatcoh}

\subsection{Historical remark}\
In the continuation of the work on {\it Heegaard--Floer theory} (\ref{ss:HF}), Ozsv\'ath and Szab\'o 
constructed in \cite{OSzP} a combinatorial $\Z[U]$--module $\bH^+(G,[k])$ for any $spin^c$--structure 
$[k]\in Spin^c(M)$. They considered a `special' class of graphs for which $\bH^+$ serves as a model 
for the original $\Z[U]$--module $HF^+(M,[k])$. 

N\'emethi \cite{OSZINV} extended this special class 
to the so--called {\em almost rational graphs} (\ref{ss:bv}), a class whose definition was strongly 
influenced by the Artin--Laufer program (\ref{AL:2}). They are characterized by the property 
that there exists 
a vertex such that decreasing its Euler decoration we get a rational graph. Moreover, he proved 
that for such graphs the isomorphism $HF^+(M,[k])\cong\bH^+(G,[k])$ is still valid, and provided a precise 
combinatorial algorithm for the calculation of this module. 

For this purpose, one has to define the 
notion of a {\em graded root} $(R_k,\chi_k)$ associated with any connected, negative definite plumbing 
graph $G$ and characteristic element $k$. Since its grading $\chi_k$ is given by the Riemann--Roch 
formula \ref{eq:RR}, this object, in fact, connects two different directions: the 
one coming from the Heegaard--Floer theory (and through this from the Seiberg--Witten theory) with the 
other one, coming from algebraic geometry. Conjecturally, $(R_k,\chi_k)$ guides the hierarchy of the 
topological types of links of normal surface singularities, containing all the information about 
the module $\bH^+(G,[k])$. On the other hand, examples show that the computation and results about 
$\bH^+(G,[k])$ can not be extended to a larger class than the almost rational graphs, hence 
this idea one had to be generalized. 

This observation gave birth to the idea of the {\it lattice cohomology}, which was introduced in 
\cite{Nlat} by N\'emethi, and its $0$--th degree cohomology module $\bH^0$ is given by $\bH^+$.

\subsection{Relation with other theories}\
Notice that, as can be seen in \ref{s:DP}, the lattice cohomology is purely combinatorial. 
Conjecturally, it contains all the information about the Heegaard--Floer homology of $M$ too. This 
would provide an alternative {\em combinatorial} definition for the theory of Ozsv\'ath and Szab\'o. 
We will present this conjecture and the active research around it in \ref{ss:HF}.

As we already pointed out in \ref{ss:swgen}, \cite{NSW} proves that the lattice cohomology (similarly 
as the Heegaard--Floer homology) categorifies the normalized Seiberg--Witten invariant of the link $M$, 
i.e. it realizes by its normalized Euler characteristic the Seiberg--Witten invariant. This provides
a new combinatorial formula for the Seiberg--Witten invariants as well.

From analytic point of view, the ranks of the lattice cohomology modules and their Euler characteristic 
have subtle connection with certain analytic invariants of analytic realizations of $M$ as singularity links 
(\ref{ss:SWIrev}). For example, the existence of the non--trivial higher cohomologies explain conceptually
the failure of the Seiberg--Witten invariant conjecture in the pathological cases,
see \cite{SWI,SWII,SWIII,NO} and \cite{LMN} for counterexamples. 

In this sense, {\em the lattice cohomology makes a bridge between the analytic and topological/combinatorial
invariants of the singularity}.

\subsection{The reduction procedure}\labelpar{ss:redmot}\
Usually, the explicit computation of the lattice cohomology is very hard.
A priori, it is based on the computation of the weights of all lattice points (of a certain $\Z^s$)
and on the description of those `regions', where the weights are less than a fixed integer.
The lattice, which appears in the construction, has a very `large' rank: it is the number of vertices of 
the corresponding plumbing/resolution graph $G$ of $M$.

The main result of the next chapter 
establishes a {\it reduction procedure} (Reduction Theorem \ref{red}), which reduces the rank of the 
lattice to $\nu$, the number of `{\em bad}' vertices (\ref{ss:bv}) on the plumbing graph $G$. 
A graph has no bad vertices if it is rational.
Otherwise, if one has to decrease the self--intersection number of (at most) $\nu$ vertices 
to get a rational graph, we say that these vertices are the `bad' ones. 
This number is definitely much smaller
(usually it is even smaller than the number of nodes) and 
provides some kind of `filtration' on negative definite
plumbing graphs/manifolds, which measures how far the graph stays from a rational graph. 

We wish to emphasize that the reduction to `bad' vertices is not just a technical procedure.
Usually, the geometry of the singularity link, or the key information about the structure
of the 3--manifold, is coded by these vertices. In other words, by a good choice of the bad vertices, 
we connect in a direct way the structure of the lattice cohomology with the essential geometrical/topological
structure of $M$.

For the illustration of this phenomenon, let us consider the following examples.
A minimal good star--shaped graph has at most one bad vertex, namely the central one.
In this case, the sequence $x(i)$ (see \ref{red:gLauferseq}) ($i\in \Z_{\geq 0}$) and the weights of 
its terms are closely related
with Dolgachev's and Pinkham's computation (\cite{P}) of the geometric genus and of the Poincar\'e series of weighted 
homogeneous singularities, see e.g. \cite{SWII,OSZINV}.

Or, let $K$ be the connected sum of $\nu$ irreducible algebraic knots $\{K_i\}_{i=1}^\nu$
of $ S^3$. Consider the surgery 3--manifold $M=S^3_{-d}(K)$ ($d\in \Z_{\geq 0}$).
Then the minimal number of bad vertices is exactly $\nu$, and they can be related with the knots, e.g., 
the lattice cohomology associated with these  vertices is guided by the
semigroups of the knot components $K_i$ (for details see \cite{NR},
where the Reduction Theorem already was applied).

Even the `naive  case of all nodes' can be interesting in the right situation.
If the graph is minimal good, then reducing the weight of the nodes we get a minimal rational
graph (with reduced fundamental cycle), hence the set of all nodes might serve as set of bad vertices.
This becomes especially meaningful  when we consider, e.g., the graph/link of a {\em Newton non--degenerate 
hypersurface singularity}. In this case the nodes correspond to the faces of the
Newton diagram (by toric resolution), cf. \cite{BN07}. Hence, this choice of the bad vertices establishes 
the connection with the combinatorics of the source object, the Newton diagram.

The methods used in the Reduction Theorem \ref{red} and in its proof have their origin
in \cite{OSZINV,Nlat}, although technically the general situation is more sophisticated.
The main ingredient is the generalization of the `special' cycles and Laufer sequences (\ref{red:gLauferseq}) 
defined by N\'emethi in \cite[7.6.]{OSZINV} for almost rational graphs (i.e. when $\nu$=1).

The effects of the reduction appear not only at the level of the cohomology modules.
The lattice cohomology has subtle connections with a certain  {\em multivariable topological Poincar\'e 
series}, where the 
number of variables of this series is the number of vertices of the plumbing graph. 
(This is defined combinatorially from the graph. It resonates and sometimes equals 
the multivariable Poincar\'e series, associated with the divisorial filtration indexed by all the divisors
in the resolution, provided by certain analytic realizations \cite{NPS,NSW,NCL}.)
For example, the Seiberg--Witten invariant appears as the {\em periodic constant} of this series
\cite{NSW,BN,NO} and can be interpreted via {\em Ehrhart theory}, as we will present in Chapter 
\ref{c:SW}.

One of the applications of the Reduction Theorem (and its proof) is that this series `reducts' 
by eliminating all the variables except those, corresponding to the `bad' vertices. The reduced 
series still contains all the lattice cohomological information.

The reduction recovers several known results as well: e.g.
the vanishing of the reduced lattice cohomology for rational graphs,
proved in \cite[\S 4]{Nlat}. More generally, it
implies the vanishing property
 $\bH^q(M)=0$ whenever $q\geq \nu$. The original 
proof of this fact can be found in \cite{Nexseq} and \ref{ss:exactseq}, where the proof
uses surgery exact sequences. Notice that this vanishing is sharp, e.g. consider the connected sum 
$K$ of $\nu$ copies of the $(2,3)$--torus knot,
and take the $(-d)$-surgery of the $3$--sphere $S^3$ along $K$, for some $d \in \Z_{>0}$. Then 
N\'emethi and Rom\'an \cite{NR} proved that 
the minimal number of bad vertices is $\nu$, and 
 $\bH^{\nu-1}(S^3_{-d}(K))=\Z$ (disregarding the $U$--action).

\chapter{Lattice cohomology and its reduction}\labelpar{ch:2}\

\indent In the beginning of this chapter we define the lattice cohomology and express its 
important properties, and interaction with questions related to the topology of normal surface 
singularities. Then in \ref{ss:bv}, we state the {\em new} characterization of rational singularities, 
which motivates the topological generalization of the rationality. It is worth to present a proof of the 
Vanishing Theorem \ref{vanishth}, which is using an exact sequence motivated by the work of Ozsv\'ath and Szab\'o 
\cite{OSz,OSz7,OSzP} on {\em Heegaard--Floer theory}. 

N\'emethi \cite{Nlat} formulated a conjecture 
which claims that lattice cohomology contains all the information about the Heegaard--Floer modules 
in the case of singularities. Therefore, we walk around this connection and list the current results. 
Moreover, in \ref{ss:SWIrev} we turn back to the GSWI conjecture and correct its inequality using a lattice 
cohomological invariant. General reference for this part is the long list of papers by N\'emethi, e.g. 
\cite{OSZINV,Nsur,Ng,Nlat,Nem10}.

The end of the chapter presents one of the main result of our research \cite{redthm}, 
the {\em Reduction Theorem} 
for lattice cohomology, which was motivated already in \ref{ss:redmot}. Note that the theorem 
implies immediately the aforementioned characterization and the 
Vanishing Theorem. Direct applications can be found in \ref{s:aplred} and \cite{NR}.

\section{Definitions and Properties}\labelpar{s:DP}

\subsection{General construction}

\bekezdes\labelpar{ss:3.1} {\bf Preliminaries, $\Z[U]$--modules.}
The lattice cohomology has a  graded $\Z[U]$--module structure.
For its building blocks we will use the following notations, cf. \cite{OSzP,OSZINV}.

Consider the graded $\Z[U]$--module $\Z[U,U^{-1}]$, and   denote by
 $\calt_0^+$ its quotient by the submodule  $U\cdot \Z[U]$.
This has a grading in such a way that $\deg(U^{-d})=2d$ ($d\geq
0$). Similarly, for any $n\geq 1$,  the quotient of $\Z\langle
U^{-(n-1)}, U^{-(n-2)},\ldots, 1,U,\ldots \rangle$ by $U\cdot
\Z[U]$ (with the same grading) defines the  graded module
$\calt_0(n)$. Hence, $\calt_0(n)$, as a $\Z$--module, is freely
generated by $1,U^{-1},\ldots,U^{-(n-1)}$, and has finite
$\Z$--rank $n$.
More generally, for any graded $\Z[U]$--module $P$ with
$d$--homogeneous elements $P_d$, and  for any  $r\in\Q$,   we
denote by $P[r]$ the same module graded (by $\Q$) in such a way
that $P[r]_{d+r}=P_{d}$. Then set $\calt^+_r:=\calt^+_0[r]$ and
$\calt_r(n):=\calt_0(n)[r]$. For example, the $\Z[U]$--module 
$\Z\langle U^{m}, U^{m-1},\ldots, U^{m-(n-1)} \rangle$ with this grading 
will be denoted by $\calt_{2m}(n)$.

\bekezdes{\bf Lattice cohomology associated with $\Z^s$ and a system
of weights.}\labelpar{ss:lw}

We fix  a free $\Z$--module, with a fixed basis
$\{E_j\}_{j=1}^s$, denoted by $\Z^s$. It is also convenient to fix
a total ordering of the index set ${\mathcal J}$, which in the
sequel will be denoted by $\{1,\ldots,s\}$.
Using  the pair $(\Z^s, \{E_j\}_j)$ and a system of weights, we determine a
cochain complex whose cohomology is our central object.

$\Z^s\otimes \R$ has a natural cellular decomposition into cubes. The
set of zero--dimensional cubes is provided  by the lattice points of 
$\Z^s$. Any $l\in \Z^s$ and subset $I\subseteq {\mathcal J}$ of
cardinality $q$  define a {\em $q$--dimensional cube}, denoted by
$(l,I)$ (or only by $\square_q$) which has its
vertices in the lattice points $(l+\sum_{j\in I'}E_j)_{I'}$, where
$I'$ runs over all subsets of $I$. On each such cube we fix an
orientation. For example, this can be determined by the order
$(E_{j_1},\ldots, E_{j_q})$, where $j_1<\cdots < j_q$, of the
involved base elements $\{E_j\}_{j\in I}$. The set of oriented
$q$--dimensional cubes defined in this way is denoted by $\calQ_q$
($0\leq q\leq s$).

Let $\calC_q$ be the free $\Z$--module generated by the oriented cubes
$\square_q\in\calQ_q$. Clearly, for each $\square_q\in \calQ_q$,
the oriented boundary $\partial \square_q$ has the form
$\sum_k\varepsilon_k \, \square_{q-1}^k$ for some
$\varepsilon_k\in \{-1,+1\}$, where the $(q-1)$--cubes $\{\square_{q-1}^k\}_k$
are the  {\em faces} of $\square_q$.
Then we have $\partial\circ\partial=0$, and  the
homology of the chain complex $(\calC_*,\partial)$ is just the homology of $\R^s$, so 
we don't get anything new. 

However, if we encode `some phenomena' on the cubes via 
a set of {\em weight functions}, more interesting (co)homology can be obtained.

\begin{definition}
A set of functions
$w_q:\calQ_q\to \Z$  ($0\leq q\leq s$) is called a {\em set of
compatible weight functions}  if the following hold:

(a) for any integer $k\in\Z$, the set $w_0^{-1}(\,(-\infty,k]\,)$
is finite;

(b) for any $\square_q\in \calQ_q$ and for any of its faces
$\square_{q-1}\in \calQ_{q-1}$ one has $w_q(\square_q)\geq
w_{q-1}(\square_{q-1})$.
\end{definition}

\begin{example}\labelpar{ex:ww}\
\begin{enumerate}
\item Assume that some $w_0:\calQ_0\to\Z$  satisfies (a) for all $k\in \Z$. For any $q>1$ set
\begin{equation*}
w_q(\square_q):=\max\{w_0(v) \,:\, \mbox{$v$ is a vertex of \,$\square_q$}\}.
\end{equation*}
Then $\{w_q\}_q$ is a set of compatible weight functions.
\item Consider the function $w_0:\Z\to \Z$ such that $w_0(n)=[|n|/2]$ or 
$w_0(n)=[|n|/2]+4\{|n|/2\}$ (where $[\ ]$ and $\{\ \}$ denote the integral and the rational part), 
and define $w_1$ as 
in the first example. Then $\{w_q\}_{q\in\{0,1\}}$ is compatible.
\end{enumerate}
\end{example}

In the presence  of a set of compatible weight functions
$\{w_q\}_q$,  one sets $\calF^q:=\Hom_{\Z}(\calC_q,\et^+_0)$.
Then  $\calF^q$ is a $\Z[U]$--module by the action 
$(p*\phi)(\square_q):=p(\phi(\square_q))$ where $p\in \Z[U]$ and 
$\phi\in \calF^q$. It has a $2\Z$--grading: $\phi\in \calF^q$ is
homogeneous of degree $2d$, if for each $\square_q\in\calQ_q$
with $\phi(\square_q)\not=0$, $\phi(\square_q)$ is a homogeneous
element in $\et^+_0$ of degree $2d-2\cdot w(\square_q)$. 
(In the sequel we will omit the index $q$ of $w_q$.)

Next, we define the (co)boundary operator $\delta_w:\calF^q\to \calF^{q+1}$. For
this, fix $\phi\in \calF^q$ and we show how $\delta_w\phi$ acts on
a cube $\square_{q+1}\in \calQ_{q+1}$. First write
$\partial\square_{q+1}=\sum_k\varepsilon_k \square ^k_q$, or a more precise 
form of $\partial$ can be determined via the orientation given by the order of 
the base elements: if
 $\square_{q+1}=(l,I)=(l,\{j_1,\ldots,j_{q+1}\})$, then
\begin{equation*}\label{eq:partial}
\partial (l,I)=\sum_{n=1}^{q+1} (-1)^n\big( \,
U^{w(l,I)-w(l,I\setminus j_n)} (l,I\setminus j_n)-U^{w(l,I)-w(l+E_{j_n},I\setminus j_n)}
(l+E_{j_n},I\setminus j_n)\, \big).
\end{equation*}
In any case, we set
$$(\delta_w\phi)(\square_{q+1}):=\sum_k\,\varepsilon_k\,
U^{w(\square_{q+1})-w(\square^k_q)}\, \phi(\square^k_q).$$
Then by an explicit calculation one has $\delta_w\circ\delta_w=0$, hence
$(\calF^*,\delta_w)$ is a cochain complex.
Moreover, $(\calF^*,\delta_w)$ has an
augmentation as well.  Indeed, set $\m_w:=\min_{l\in \Z^s}w_0(l)$ and
choose  $l_w\in \Z^s$ such that $w_0(l_w)=\m_w$. Then one defines
the $\Z[U]$--linear map
$\epsilon_w:\et^+_{2\m_w}\longrightarrow \calF^0$ such that
$\epsilon_w (U^{-\m_w-n})(l)$ is the class of $U^{-\m_w+w_0(l)-n}$
in $\et^+_0$ for any $n\in \Z_{\geq 0}$. \cite[3.1.7]{Nlat} shows that
 $\epsilon_w$ is injective, and $\delta_w\circ\epsilon_w=0$.

\bekezdes\labelpar{def12}{\bf Definitions of the Lattice cohomology.} 
The homology of the cochain complex
$(\calF^*,\delta_w)$ is called the {\em lattice cohomology} of the
pair $(\R^s,w)$, and it is denoted by $\bH^*(\R^s,w)$. The
homology of the augmented cochain complex
$$0\longrightarrow\et^+_{2\m_w}\stackrel{\epsilon_w}{\longrightarrow}
\calF^0\stackrel{\delta_w}{\longrightarrow}\calF^1
\stackrel{\delta_w}{\longrightarrow}\ldots$$ is called the {\em
reduced lattice cohomology} of the pair $(\R^s,w)$, and it is
denoted by $\bH_{red}^*(\R^s,w)$.
For any $q\geq 0$, both $\bH^q$ and $\bH_{red}^q$ admit an induced graded
$\Z[U]$--module structure, and one has  graded
$\Z[U]$--module isomorphisms 
\begin{center}
$\bH^0=\et^+_{2\m_w}\oplus\bH^0_{red}$ \ \ and \ \  
$\bH^q=\bH^q_{red}$ (for $q>0$).
\end{center}

In the case when each $\bH^q_{red}$ has finite $\Z$--rank, one can define
the {\em normalized Euler characteristic}
\begin{equation}\label{eq:eu}
eu(\bH^*(\R^s,w)):=-\m_w+
\textstyle{\sum_q}(-1)^q\rank_\Z \bH^q_{red}(\R^s,w).
\end{equation}

\bekezdes\labelpar{mod} {\bf Modification.} Instead of
all the cubes of $\R^s$ we can consider an arbitrary subset $T$ of cubes in $\R^s$ 
(e.g. $[0,\infty)^s$, or the `rectangle'
$R:=[0,T_1]\times\cdots\times [0,T_s]$ for some $T_i\in \Z_{\geq
0}$). In such a case, we write $\bH^*(T,w)$ for the corresponding lattice cohomologies, since 
the restriction map induces a natural graded $\Z[U]$--module homomorphism 
$r^*:\bH^*(\R^s,w)\to \bH^*(T,w)$.
\begin{example}\labelpar{ex:pathcoh}
Consider a sequence 
$\gamma=\{x_n\}_{n=0}^t$ ($t$ can be $\infty$) such that 
$x_n\neq x_m$ 
for $n\neq m$ and $x_{n+1}=x_n\pm E_{j(n)}$ for $0\leq n< t$. Let $T$ be the union of 
$0$--cubes marked by the points $\{x_n\}$ and $1$--cubes (segments) of type $[x_n,x_{n+1}]$. Repeating 
the above construction, 
we get a graded $\Z[U]$--module $\bH^*(T,w)$. 
It is called the {\em path cohomology} associated 
with the `path' $\gamma$ and the compatible weights $\{w_0,w_1\}$. It will be denoted by 
$\bH^*(\gamma,w)$. \newline
The construction implies that $\bH^q(\gamma,w)=0$ for $q\geq 1$. Hence, in 
`finite' ($\bH^0_{red}$ has finite $\Z$--rank, or in particular, the length of $\gamma$ is finite) 
cases one can define the Euler characteristic 
$$eu(\gamma,w):=eu(\bH^*(\gamma,w))=-m_w+\rank_\Z \bH^0_{red}(\gamma,w).$$
Then \cite[3.5.2]{Nlat} gives the formula 
$$eu(\gamma,w)=-w_0(0)+\sum_{n=0}^{t-1}w_1([x_n,x_{n+1}])-w_0(x_{n+1}).$$
\end{example}

\bekezdes{\bf The geometric $\bS^*$--realization.}\labelpar{geomdef} 

A more geometric realization of the
modules $\bH^*$ can be given in the following way.
For each $N\in\Z$, define
$S_N=S_N(w)\subseteq\R^s$ as the union of all the cubes $\square_q$
(of any dimension) with $w(\square_q)\leq N$. Clearly, $S_N=\emptyset$,
whenever $N<\m_w$. For any $q\geq0$, set
\begin{equation*}
\bS^q(\R^s,w):=\oplus_{N\geq \m_w}H^q(S_N,\Z).
\end{equation*}
Then $\bS^q$ is  $2\Z$--graded, the $d=2N$--homogeneous elements
 $\bS^q_d$ consists of $H^q(S_N,\Z)$. Also, $\bS^q$ is a $\Z[U]$--module.
 The $U$--action is given by the restriction map
 $r_{N+1}:H^q(S_{N+1},\Z)\longrightarrow H^q(S_N,\Z)$,
 namely, $U*(\alpha_N)_N:=(r_{N+1}\alpha_{N+1})_N$. Moreover, for $q=0$, a fixed
 basepoint $l_w\in S_{\m_w}$ provides an augmentation $H^0(S_N,\Z)=\Z\oplus\widetilde{H}^0(S_N,\Z)$,
 hence an augmentation of the graded $\Z[U]$--modules

\begin{equation*}
\bS^0=(\oplus_{N\geq \m_w}\Z)\oplus(\oplus_{N\geq \m_w}\widetilde{H}^0(S_N,\Z))=\et^+_{2\m_w}\oplus\bS^0_{red}.
\end{equation*}
The point is that this $\Z[U]$--module $\bS^*$ coincides with the lattice cohomology $\bH^*$.
More precisely, we have the following theorem.

\begin{theorem}\emph{(\cite[3.1.12]{Nlat})}\labelpar{th:HS}
There exists a graded $\Z[U]$--module isomorphism, compatible with
the augmentations, between $\bH^*(\R^s,w)$ and $\bS^*(\R^s,w)$.
Similar statement is valid for $\bH^*(T,w)$ for any $T\subseteq \R^s$ as in \ref{mod}.
\end{theorem}
From now on we  denote both  realizations with the same symbol
$\bH^*$, no matter which one we use. In the next examples we illustrate how to use this realization 
for the calculation of the lattice cohomology.
\begin{example}\labelpar{ex:latcoh1}
\begin{itemize}
\item[(a)] Consider the first case from Example \ref{ex:ww}(2), when we have the lattice 
$\Z\subset\R$ and $w_0(n)=[|n|/2]$ for all $n\in\Z$. Obviously $\m_w=0$ and $S_N$ is 
the segment $[-2N-1,2N+1]$ which is contractible for all $N\geq 0$. Hence $\bH^0(\R,w)=\calT_0^+$.
\item[(b)] Let $w_0(n)=[|n|/2]+4\{|n|/2\}$. Then $\m_w=0$ and one can show that if $N\geq 1$, 
$S_N$ has three components belonging to the `central' component of $S_{N+1}$ as it is shown in Figure 3.1.
Therefore, taking into account the $\Z[U]$--action, the lattice cohomology can be written as 
$\bH^0(R,w)=\calT_0^+\bigoplus \oplus_{N\geq 1}\calT_{2N}(1)$.
\begin{figure}[h!]\label{fig:1}
\vspace{1cm}
\centering
\begin{pspicture}(0,-0.97296876)(14.007,0.93296874)
\psdots[dotsize=0.1](1.0,0.13296875)
\psdots[dotsize=0.1](2.0,0.13296875)
\psdots[dotsize=0.1](3.0,0.13296875)
\psdots[dotsize=0.1](4.0,0.13296875)
\psdots[dotsize=0.1](5.0,0.13296875)
\psdots[dotsize=0.1](6.0,0.13296875)
\psdots[dotsize=0.1](7.0,0.13296875)
\psdots[dotsize=0.1](8.0,0.13296875)
\psdots[dotsize=0.1](9.0,0.13296875)
\psdots[dotsize=0.1](10.0,0.13296875)
\psdots[dotsize=0.1](11.0,0.13296875)
\psdots[dotsize=0.1](12.0,0.13296875)
\psdots[dotsize=0.1](13.0,0.13296875)
\psline[linewidth=0.0139999995cm](0.0,0.13296875)(14.0,0.13296875)
\psellipse[linewidth=0.0139999995,dimen=outer](6.993,0.13296875)(0.2,0.2)
\psellipse[linewidth=0.0139999995,dimen=outer](8.993,0.13296875)(0.2,0.2)
\psellipse[linewidth=0.0139999995,dimen=outer](4.993,0.13296875)(0.2,0.2)
\psellipse[linewidth=0.0139999995,dimen=outer](10.993,0.13296875)(0.2,0.2)
\psellipse[linewidth=0.0139999995,dimen=outer](2.993,0.13296875)(0.2,0.2)
\psellipse[linewidth=0.0139999995,dimen=outer](12.993,0.13296875)(0.2,0.2)
\psellipse[linewidth=0.0139999995,dimen=outer](0.993,0.13296875)(0.2,0.2)
\psellipse[linewidth=0.0139999995,dimen=outer](6.993,0.13296875)(2.4,0.6)
\psellipse[linewidth=0.0139999995,dimen=outer](6.993,0.13296875)(4.6,0.8)
\rput(7.0533123,-0.3){$S_0$}
\rput(7.0533123,0.5){$S_1$}
\rput(5.4533124,0.3){$S_1$}
\rput(8.653313,-0.05){$S_1$}
\rput(9.453313,-0.18203124){$S_2$}
\rput(10.653313,0.41796875){$S_2$}
\rput(3.4533124,0.41796875){$S_2$}
\rput(10.053312,-0.78203124){$S_3$}
\rput(13.453313,-0.18203124){$S_3$}
\rput(1.2533125,-0.18203124){$S_3$}
\end{pspicture}
\caption{The $w_0(n)=[|n|/2]+4\{|n|/2\}$ case.}
\end{figure}
\end{itemize}
\end{example}

\subsection{Case of the singularities}\label{ss:pl}
Let $G$ be a negative definite
plumbing graph as in \ref{intro:comb}. Let $|\calj|=s$ be the number of vertices.
Then we can associate with $L=\Z^s$ the free $\Z$--module $\calC_q$ generated by oriented cubes
$\square_q\in\calQ_q$, as in \ref{ss:lw}.

To any $k\in Char$ we associate  weight functions
$\{w_q\}_q$ as follows.
One can use the function $\chi_k:L\to \Z$ we have given in (\ref{eq:RR}) by
$$\chi_k(l)=-(l,l+k)/2,$$
and set $\m_k:=\min\, \{\, \chi_k(l)\, :\, l\in L\}$.
Then the weight functions are defined as in \ref{ex:ww}(1) via 
$$w_0:=\chi_k \ \ \ \mbox{and} \ \ \ w_q(\square_q)=\max\{\chi_k(v)\,:\, v\ \mbox{is a vertex of} \  
\square_q\}.$$
\begin{definition}
The associated lattice cohomologies with this weight functions are called the 
{\em lattice cohomology associated with the pair $(G,k)$} and are denoted by 
$\bH^*(G,k)$ and $\bH^*_{red}(G,k)$. We write $\m_k:=\m_{w}=\min_{l\in L}\chi_k(l)$.
\end{definition}

\begin{theorem}\emph{(\cite[3.2.4]{Nlat})}
 The $\Z[U]$--modules $\bH^*_{red}(G,k)$ are finitely generated over $\Z$, 
hence $eu(\bH^*(G,k)):=eu(\bH^*(\R^s,w))$ is well--defined, cf. (\ref{eq:eu}). 
In particular, this implies that $S_N$ is contractible for $N$ sufficiently large.
\end{theorem}
The proof (cf. \cite[p.7]{Nlat}) of this theorem uses the techniques of \ref{ss:2}, therefore 
we omit here. We remark that Example \ref{ex:latcoh1}(b) can not be the lattice cohomology 
associated with some surface singularity, since $\bH^0_{red}$ is not finitely generated over $\Z$.

Although, each $k\in Char$ provides a different cohomology module, 
there are only $|H|$ essentially different ones. Indeed, assume that $[k]=[k']$, 
hence $k'=k+2l$ for some $l\in L$. Then one has the identity 
$$\chi_{k'}(x-l)=\chi_k(x)-\chi_k(l) \ \ \ \mbox{ for any $x\in L$},$$
which tells that 
the transformation $x\mapsto x':=x-l$ realizes the following
identification:
$$\bH^*(G,k')=\bH^*(G,k)[-2\chi_k(l)].$$
Therefore, up to this {\em shift}, we have well--defined modules $\bH^*(G,[k])$ for any 
$spin^c$--structure $[k]$, and we may highlight uniformly a specific one, which represents 
$\bH^*(G,[k])$. One way to do this is to choose the distinguished representative $k_r$ 
(\ref{distrep}) for the class $[k]$, then $\bH^*(G,k_r)$ will represent the modules 
associated with $[k]$.

Notice that the $3$--manifold $M$ can be given by many different negative definite plumbing graphs 
$G$, but all these graphs can be connected by a finite sequence of blow ups and blow downs of 
$(-1)$--vertices. In order to see the invariance of the lattice cohomology, one has to check 
that the representative module $\bH^*(G,k_r)$ does not change under this calculus.

The next proposition emphasizes the advantage of the choice of $k_r$ for any $spin^c$ structure 
$[k]$, together with the invariance of $\bH^*(G,[k])$ under changing the negative definite 
plumbing representation of $M$.
\begin{proposition}\labelpar{propF2} \
\begin{itemize}
 \item[(a)]  \ $\bH^*(G,k_r)\cong\bH^*([0,\infty)^s,k_r)$ for any $k_r$.
 \item[(b)] \ The set $\{\bH^*(G,k_r)\}_{[k_r]}$ is independent on
the plumbing representation $G$ of the 3--manifold $M$, hence it associates a $\Z[U]$--module to
any pair $(M,[k_r])$, where $[k_r]\in\, Spin^c(M)$.
\end{itemize}
\end{proposition}
The property is proved in \cite[3.3.4 \& 3.3.5 \& 3.4]{Nlat}. Another 
interpretation of the construction and the invariance can be found in \cite{OSSz2}.

One can consider also the sum 
$$\bH^*(M):=\oplus_{[k]\in Spin^c(M)}\bH^*(M,[k]).$$

\begin{example}
 Consider the most basic example, when the normal surface singularity is $(\C^2,0)$. It is smooth at 
the origin and its link is just an $S^3$. We may pick one of its negative definite plumbing 
representation given by: 
\begin{picture}(110,20)(80,15)
\put(100,18){\circle*{3}}
\put(100,28){\makebox(0,0){$-1$}}
\put(110,15){.}
\end{picture}
 If $E$ represents the vertex, then the lattice $L=\Z\langle E \rangle\cong \Z$, $E^2=-1$ and 
the adjunction formula immediately gives $\K=E$. The only $spin^c$ structure is $[\K]$, and 
$\chi(n)=-\frac{((n+1)E,nE)}{2}=\frac{n(n+1)}{2}$ for any $n\in \Z$ (i.e. for $nE\in L$). By 
\ref{propF2}(a), its enough to look at $\Z_{\geq 0}$ on which $\chi$ is increasing. Hence, it follows that 
$S_N$ is contractible to the point $n=0$ for $N\geq 0$, therefore the lattice cohomology 
is the {\em trivial} one, $\bH^0(S^3,\K)=\calT_0^+$ and $\bH^q(S^3,\K)=0$ for $q>0$. 
\end{example}

\subsection{`Bad' vertices and the rationality of graphs}\labelpar{ss:bv}\
We continue the discussion held in \ref{AL:3} from the point of view of lattice cohomology. 
Recall that a normal surface singularity is rational if its geometric genus is zero.
This vanishing property was characterized combinatorially by Artin's criterion (see 
Theorem \ref{thm:rat}):
\begin{equation}\label{eq:artin}
\mbox{rationality}\ \Longleftrightarrow \ \ \chi(l)> 0 \ \ \mbox{for any $l>0$, $l\in L$,}
\end{equation}
where the notaion $\chi$ is associated with $\K$. 
Subsection \ref{def:rat} defines the set of rational graphs (resolution graphs of rational singularities), which is 
closed under taking subgraphs and decreasing the weights of vertices. 

The next theorem points out 
that the lattice cohomology of rational graphs is trivial, and in this way it gives a {\em new 
topological characterization} of rational normal surface singularities. The idea behind it is that 
one can produce a deformation retract of the space $\R^s$ to the origin along which $\chi_k$ is decreasing 
(using the methods of \ref{ss:2}), hence $S_N$ is contractible whenever is 
non--empty.
\begin{theorem}\emph{(N\'emethi \cite[6.3]{OSZINV} and \cite[4.1]{Nlat})}\labelpar{thm:ratlatcoh}\newline
 $G$ is a rational graph if and only if $\bH^0(G,\K)=\calT_0^+$. 
Moreover, in this case, $\bH^q(G,\K)=0$ for $q\geq 1$, and one has the same result for any 
distinguished representative $k_r \in Char$ as well.
\end{theorem}

\begin{remark}
One can say that a graph $G$ is {\em lattice cohomologically `weak'} if the module $\bH^*(G,\K)$ 
associated with the $spin^c$ structure $[\K]$ dominates all the others. E.g., by the previous theorem 
this is the case for rational graphs. We refer to \cite[4.2]{Nlat}, which shows the same phenomenon 
for the elliptic graphs too. Therefore, one can ask the question whether there is any other graph with similar 
properties, or in 
other words, what can we say about the maximal set of weak graphs in this sense? See 
\cite[5.2.6]{Nlat} for further details in this direction.
\end{remark}

Any non--rational graph can be transformed into a rational one by decreasing some of the 
decorations along some of its vertices. Indeed, if all the decorations of a graph $G$ 
are sufficiently negative (e.g. $(E_j,E)\leq 0$ for any $j$), then $G$ is rational.
In order to measure how far the given graph is from the `rationality', one can give the following 
definitions (see also \cite{OSzP,OSZINV,Nlat,Nexseq,NR,redthm}).
\begin{definition}[\bf Family of bad vertices]
We say that a graph has a {\it family of $\nu$ bad vertices}, if one can find a subset 
of vertices  $\{j_k\}_{k=1}^\nu$, called {\it bad vertices},  such that replacing their decorations 
$b_j=(E_j,E_j)$ by some more negative integers $b'_j\leq b_j$ we get a rational graph.

There is a (usually non--unique) family of bad vertices with smallest cardinality. 
If this cardinality is less than or equal to $\nu$, then the graph is called {\em $\nu$--rational} 
(or {\em type--$\nu$} as in \cite{OSSz3}). 
The case $\nu=1$ appeared earlier in \cite{OSZINV} and it was called {\it almost rational}.
\end{definition}

The main result of this chapter will show that the geometry encoded by the lattice cohomology is 
concentrated to these vertices.

\subsection{Exact sequence and vanishing}\labelpar{ss:exactseq}\
We are going to present an exact sequence, called the {\em surgery exact sequence} (or surgery exact triangle), 
which was firstly proved by Ozsv\'ath and Szab\'o in the context of Heegaard--Floer homology. 
In order to understand the deep connection between these 
theories (\ref{ss:HF}), there was a desire to prove it for lattice cohomology as well. 

This was done over $\Z_2$--coefficients by Greene \cite{Greene}, then N\'emethi \cite{Nexseq} extended 
over $\Z$. Since the current subsection is a summary of parts of \cite{Nexseq}, we will omit 
the proofs.

For any graph $G$ and a fixed vertex $j_0$, we may consider the graphs $G\setminus j_0$ and $G_{j_0}^+$. 
The first one is obtained by deleting the vertex $j_0$ and its adjacent edges, while the second is 
defined by replacing the decoration $b_{j_0}$ of $j_0$ by $b_{j_0}+1$. 
The negative definiteness of $G$ implies that $G\setminus j_0$ is negative definite too, but 
this is not true for $G_{j_0}^+$. However, $G_{j_0}^+$ is negative definite if and only if 
$\det(G)>\det(G\setminus j_0)$ (see also \cite[Lemma 6.1.1]{Nexseq}). Indeed, we have 
$$\det(G)=\det(G_{j_0}^+)+\det(G\setminus j_0),$$ 
where $\det(G)$ and $\det(G\setminus j_0)$ are positive. 
Conversely, if $G_{j_0}^+$ is negative definite then $G$, hence 
$G\setminus j_0$ is so. However, $G\setminus j_0$ fails to be connected in a generic situation.

In order to speak about lattice cohomologies of these graphs, we have to extend the definition. 
Notice that formally \ref{ss:pl} allows to drop the connectedness and the negative definiteness 
conditions, and assume only that the graphs are non--degenerate (i.e. $\det(G)\neq 0$).
However, the effect of leaving the negative definite assumption is more serious: 
we loose the geometric interpretation 
since $S_N$ may not necessarily be compact. Moreover, the lattice cohomology may not be stable under 
the blow ups and blow downs connecting the plumbing representation (for example and further discussion 
see \cite[2.4]{Nexseq}). Nevertheless, it is convenient to extend the definition in order to have 
a larger flexibility for computations using the following surgery exact sequence. 

\begin{theorem}\emph{(\cite{Nexseq})}\labelpar{thm:exseq} 
\begin{enumerate}
 \item 
There exists a long exact sequence of $\Z[U]$--modules
$$\ldots\longrightarrow \bH^q(G_{j_0}^+)\longrightarrow \bH^q(G)\longrightarrow 
\bH^q(G\setminus j_0) \longrightarrow \bH^{q+1}(G_{j_0}^+) \longrightarrow\ldots \ .$$ 
\item If $G_{j_0}^+$ is negative definite, then at the begining of the exact sequence 
$$0\longrightarrow \bH^0(G_{j_0}^+)\longrightarrow \bH^0(G)\longrightarrow 
\bH^0(G\setminus j_0) \longrightarrow \bH^{1}(G_{j_0}^+) \longrightarrow\ldots \ ,$$
the canonical submodule $\calT^+(G\setminus j_0)$ of $\bH^0(G\setminus j_0)$ is mapped to zero.
\end{enumerate}
\end{theorem}
A disadvantage of this sequence is that the operators mix the classes $[k]$, hence it is hard to 
calculate the modules $\bH^*(G,[k])$ separately. Still, one can provide an exact sequence which connects 
the lattice cohomologies of $G$ and $G\setminus j_0$ with fixed classes, making the concept of 
{\em relative lattice cohomology}. We omit the details here and refer to \cite[4]{Nexseq}.
Nevertheless, using this surgery exact sequence we can prove the following vanishing result of lattice 
cohomology.

\begin{theorem}[Vanishing Theorem]\labelpar{vanishth}
Assume that $G$ has a family of $\nu$ bad vertices, then $\bH^q(G)=0$ for $q\geq \nu$. 
(In particular, $\bH^q(G,[k])=0$ for any $[k]$ too.)
\end{theorem}
\begin{proof}
 It goes using induction over $\nu$. When $\nu=0$, then all the components of $G$ are rational. Hence 
by \ref{thm:ratlatcoh}, their reduced lattice cohomology is vanishing. 
Assume that the statement is true for 
$\nu-1$ and let $G$ be a graph with $\nu$ bad vertices. Choose a bad vertex $j$ and form the graph 
$G_j(-m)$ by replacing the decoration $b_j$ by $b_j-m$ for $m\geq 0$. The long exact sequence 
associated with $G_j(-m-1)$, $G_j^+(-m-1)\equiv G_j(-m)$ and $G\setminus j$ and the inductive 
argument say that $\bH^q(G_j(-m))\cong \bH^q(G_j(-m-1))$ for $q\geq \nu$. Hence, induction over $m$ 
shows that $\bH^q(G)\cong\bH^q(G_j(-m))$ for all $m$ and $q\geq \nu$. Since for large enough $m$, 
$G_j(-m)$ has only $\nu-1$ bad vertices, the result follows.
\end{proof}

\section{Relation with other theories revisited}

\subsection{Heegaard--Floer homology and N\'emethi's conjecture}\labelpar{ss:HF}\
First of all, we review some basic facts from the theory of Heegaard--Floer homology (the $HF^+$ version) 
introduced by Ozsv\'ath and Szab\'o in \cite{OSz,OSz7}. Besides the long list of original papers of 
Ozsv\'ath and Szab\'o, for more details on the definitions and properties, we recommend 
the lecture notes \cite{OSz06a,OSz06b}.

Consider an oriented 3--manifold $M$, which we assume 
to be a rational homology sphere. Then $HF^+(M)$ is an abelian group with a $\Z_2$--grading, 
and it splits as a direct sum according to the $spin^c$--structures on $M$. We may write  
$$HF^+(M)=\oplus_{[k]\in Spin^c(M)}\ HF^+(M,[k]),$$
and denote by $HF^+_{even}(M,[k])$, respectively $HF^+_{odd}(M,[k])$, the parts of $HF^+(M,[k])$ 
with the corresponding parity. For any $spin^c$--structure $[k]$, $HF^+(M,[k])$ admits a 
$\Z[U]$--action which preserves the $\Z_2$--grading 
and gives the Heegaard--Floer homology a $\Z[U]$--module structure. 
Since all the $spin^c$--structures are torsion, 
the corresponding components admit a $\Q$--grading compatible with the $\Z[U]$-action, where $\deg(U)=-2$. 

One has a graded $\Z[U]$-module isomorphism
$$HF^+(M,[k])=\calt^+_{d(M,[k])}\oplus HF^+_{red}(M,[k]),$$
where $d(M,[k])$ is the smallest degree of non--trivial
elements of $HF^+(M,[k])$. The reduced part 
$HF^+_{red}(M,[k])$ has a
finite $\Z$-rank and an induced $\Z_2$--grading as well. Therefore, one 
also considers the Euler characteristic
$$\chi(HF^+(M,[k])):=\rank_\Z HF^+_{red,even}(M,[k])
-\rank_\Z HF^+_{red,odd}(M,[k]),$$
which, following \cite{Rus}, recovers the Seiberg--Witten invariant of $(M,[k])$, normalized 
by $d(M,[k])/2$, i.e.
$$\chi(HF^+(M,[k]))=\frsw_{[k]}(M)-d(M,[k])/2.$$

With respect to the change of orientation the above invariants
behave as follows:
the $spin^c$-structures $Spin^c(M)$ and
$Spin^c(-M)$ are canonically identified (where $-M$ denotes $M$
with opposite orientation). Moreover,
$d(M,[k])=-d(-M,[k])$ and
$\chi(HF^+(M,[k]))=-\chi(HF^+(-M,[k]))$.

In the case when $M$ is a negative definite plumbed 3--manifold, its plumbing graph $G$ gives a 
cobordism from $-M$ to $S^3$ which induces a map 
\begin{equation}\label{eq:TG}
T_G:HF_{even}^+(-M)\longrightarrow \bH^0(M),
\end{equation}
defined in \cite{OSzP}. By results of \cite{OSzP,OSZINV}, this creates an identification between the Heegaard--Floer and 
lattice cohomology theories in the case when $\nu=1$, i.e. $G$ is almost rational. 
More precisely, for any $[k]\in Spin^c(M)$
$$HF^+_{odd}(-M,[k_r])=0 \ \ \mbox{and} \ \ 
HF^+_{even}(-M,[k_r])=\bH^0(G,k_r)\Big[-\frac{k_r^2+s}{4}\Big],$$
in particular
$d(M,[k_r])=\max_{k'\in[k_r]}\frac{(k')^2+s}{4}=\frac{k_r^2+s}{4}-2\m_{k_r}$.

In the spirit of this connection, one can predict the following identification as well (\cite[5.2.4]{Nlat}).
\begin{Nconj} Let $G$ be the negative definite plumbing representation of $M$ 
as before. Then for any distinguished representative $k_r$ one has 
\begin{eqnarray*}
HF^+_{red,even}(-M,[k_r])=\bigoplus_{q \ even} \bH^q_{red}(G,k_r)\Big[-\frac{k_r^2+s}{4}\Big] 
\ \ \mbox{and} \\ 
HF^+_{red,odd}(-M,[k_r])=\bigoplus_{q \ odd} \bH^q_{red}(G,k_r)\Big[-\frac{k_r^2+s}{4}\Big].
\end{eqnarray*}
\end{Nconj}

In \cite{NR}, N\'emethi and Rom\'an prove the conjecture for $2$--rational graphs associated 
with the manifold $S^3_{-d}(K)$, obtained by $(-d)$--surgery of $S^3$ along the connected sum $K$ 
of a collection of algebraic knots determined by irreducible plane curve singularities. They use 
the Reduction Theorem \ref{red} in order to split the exact sequence \ref{thm:exseq}(1) with 
the vanishing of $\bH^2(G)$. 
Their argument does not really require the specialty of the $S^3_{-d}(K)$ graph. Therefore, one can  
mimic the proof with the 
assumption that $G$ has to be the simplest $2$--rational graph, in the sense that 
there exists a bad vertex so that if we decrease its decoration by $-1$, we get an almost rational 
graph. The failure of this argument in arbitrary case is based on the fact that at this moment 
there is no natural morphisms connecting the modules of the two theories, except the level $q=0$. 
In this special case, the isomorphism between $HF^+_{odd}(-M)$ and 
$\bH^1(G)$ can be induced by the $0$--level morphisms. 

There is an another approach, done by 
Ozsv\'ath, Stipsicz and Szab\'o \cite{OSSz1}, which constructs a spectral sequence converging
to the Heegaard--Floer homology, and its $E_2$--term agrees with the lattice cohomology theory. As 
an application, they finished the identification in this $2$--rational case.
Moreover, they considered the relative version (for knots in $M$) of lattice cohomology too 
\cite{OSSz2}, and with its help they proved in \cite{OSSz3} the case, when $G$ has a 
vertex $v$ with the property that if we delete $v$ and its adjacent edges, we get a rational graph.

A different version of the relative lattice cohomology was defined by Gorsky and N\'emethi \cite{GN}, 
which is associated with local plane curve singularities and it is identified
with the {\em motivic Poincar\'e series} of such germs.

\subsection{Seiberg--Witten invariant conjecture revisited}\labelpar{ss:SWIrev}\
We finished Section \ref{ss:SWIC} with the promise that we return and correct the upper bound given in 
GSWI conjecture. This can be done using path cohomological methods \ref{ex:pathcoh}. 

Consider the notations of 
\ref{CohSh} and pick a line bundle $\calL \in Pic(\widetilde X)$ with $c_1(\calL)=l'$. 
For simlicity, we use the 
notation $h^*(\calL)$ for $\dim_{\C}H^*(\widetilde X, \calL)$. Then we need a theorem which is 
a generalization of the {\em Kodaira type Vanishing Theorem} \cite[pg. 301]{Nfive}.
\begin{theorem}\emph{(Laufer--Grauert--Riemenschneider, \cite[6.1.2]{Nlat})} 
If $c_1(\calL)\in \K-\calS'$, then for any $l\in L$, $l>0$ we have 
$h^1(\calL\otimes \cO_l)=0$, hence $h^1(\calL)=0$ as well. 
\end{theorem}

If we choose a path $\gamma=\{x_i\}_{i=0}^t$ so that $x_0=0$, $x_{i+1}=x_i+ E_{j(i)}$ and 
$x_t\in -l'-\K+\calS'$, then the exact sequence 
$0\longrightarrow \calL\otimes\cO(-x_t)\longrightarrow \calL \longrightarrow \calL\otimes \cO_{x_t}
\longrightarrow 0$ and the theorem above imply that 
$$h^1(\calL)=h^1(\calL\otimes \cO_{x_t}),$$ i.e. $h^1(\calL)$ can be achieved restricting 
$\calL$ to a cycle in the `special' zone. Moreover, one can prove the following property:
\begin{proposition}\emph{(\cite[6.2.2]{Nlat})}
 For any $0\leq i<t$ one has
 $$h^1(\calL\otimes \cO_{x_{i+1}})-
h^1(\calL\otimes \cO_{x_i})\leq \max\{0, \chi_k(x_i)- \chi_k(x_{i+1})\}.$$
\end{proposition}

Then by summing up the inequalities we get 
$h^1(\calL)\leq \sum_{i=0}^{t-1} \max\{0, \chi_k(x_i)- \chi_k(x_{i+1})\}$. Notice that even if we expand  
the sequence arbitrarily long inside the special zone $-l'-\K+\calS'$, nothing will be changed. Therefore, 
if $\calP$ is the set of paths with connecting $x_0=0$ with some elements in the special zone, then 
Example \ref{ex:pathcoh}, together with the above discussion deduce the following inequality
$$h^1(\calL)\leq \min_{\gamma\in\calP} eu(\gamma,k),$$ 
where $eu(\gamma,k)$ denotes the normalized Euler characteristic of the path cohomology associated with 
$\gamma$ and $\chi_k$.
Be aware that in general 
$\min_{\gamma\in\calP} eu(\gamma,k)<eu(\bH^0(G,k))$, see Example \ref{ex:boundlesscan}.

\section{Reduction Theorem}\labelpar{red:section}\
The goal of the present section is to show that
the lattice cohomology of the lattice $L$
(or any rectangle of it) can be reduced to a considerably `smaller rank object'. The main tool
in this reduction is the {\em theory of computation sequences}, initiated by Laufer (\cite{Laufer72}). 
In the first subsection we introduce the needed generalization, in the second we state the main theorem and the 
third subsection presents the proof. 
Notice that the idea of the Reduction Theorem is present already in \cite{OSZINV}.

The new lattice of rank $\nu$ will be associated with a {\em family of bad vertices},
the new  lattice points are associated with some important cycles of $L$
as distinguished members of Laufer--type computational sequences of $L$.
We start with their definition.

\subsection{Special cycles and generalized Laufer sequences}\labelpar{red:gLauferseq}\
Suppose we have a
 family of {\it distinguished}  vertices $\ocalj:=\{j_k\}_{k=1}^\nu\subseteq \calj$
 (usually they are defined by some geometric property).
Then split the set of vertices $\calj$ into the disjoint union $\overline{\calj}\sqcup\calj^*$.  Furthermore,
let $\{m_j(x)\}_j$ denote  the coefficients of a  rational cycle  $x$, that is $x=\sum_{j\in\calj}m_j(x)E_j$.

In order to simplify the notation we set  ${\bf i}:=(i_1,\ldots,i_j,\ldots,i_\nu)\in\Z^\nu$;
for any $j\in \ocalj$ we write $1_j\in \Z^\nu$ for the vector with all entries zero except at place
$j$ where it is 1,  and for     any
 $I\subseteq \ocalj$ we define $1_I=\sum _{j\in I}1_j$. Similarly, for any
 $I\subseteq \calj$ set $E_I=\sum_{j\in I}E_j$.

Then the cycles $x({\bf i})=x(i_1,\ldots,i_\nu)$ are defined via the next Proposition.

\begin{proposition}\labelpar{lemF1} Fix $[k]$ and $\ocalj\subseteq \calj$ as above.
For any ${\bf i}\in(\Z_{\geq 0})^\nu$ there exists a unique cycle
$x({\bf i})\in L$ satisfying the following properties:

\begin{itemize}
\item[(a)] $m_{j}(x({\bf i}))=i_j$ for any distinguished  vertex $j\in\ocalj$;

\item[(b)] $(x({\bf i})+l'_{[k]},E_j)\leq0$ for every `non--distinguished vertex' $j\in\calj^*$;

\item[(c)] $x({\bf i})$ is minimal with the two previous properties.

\end{itemize}
Moreover,
\ (i)  $x(0,\ldots ,0)=0$;
\ (ii)   $x({\bf i})\geq 0$; and
\ (iii)  $x({\bf i})+E_{\overline{I}}
\leq x({\bf i}+1_{\overline{I}})$ for any $\overline{I}\subseteq \ocalj$.
\end{proposition}
\begin{proof}
The proof is similar to the proof of  \cite[Lemma 7.6]{OSZINV}, valid
for  $\nu=1$ (or to the existence of the Artin's cycle which corresponds to
  $\nu=0$ and the canonical class).

  First we verify the existence of an element $x\in L$ with (a)--(b). By (the proof of)
\cite[7.3]{OSZINV} there exists
$\widetilde{x}\geq \sum_{j\in\ocalj}E_j$ such that $(\widetilde{x}+l'_{[k]},E_j)\leq 0$
for any $j\in \calj$.  Take some $a\in\Z_{>0}$ sufficiently large so that $(a-1)l'_{[k]}\in L$, and
 $h_j:=m_{j}(a\widetilde{x}+(a-1)l'_{[k]})-i_j\geq 0$ for any $j\in \ocalj$. Since $l'_{[k]}\geq 0$, this is
 possible. Then set $x:=a\widetilde{x}+(a-1)l'_{[k]}-\sum_{j\in \ocalj} h_jE_j$. Clearly $m_{j}(x)=i_j$
 for any $j\in\ocalj$ and $(x+l'_{[k]}, E_i)=a(\widetilde{x}+l'_{[k]},E_i)-\sum_{j\in\ocalj}h_j(E_j,E_i)\leq 0$ for any
 $i\in \calj^*$.

 Next, we verify that there is a unique minimal element with (a)--(b). This follows from the fact that
 if $x_1$ and $x_2$ satisfy (a)--(b), then $x:=\min\{x_1,x_2\}$ does too. Indeed, for any $j\in\calj^*$,
 at least for one index $n\in\{1,2\}$ one has $E_j\not\in|x_n-x|$. Then $(x+l'_{[k]},E_j)=(x_n+l'_{[k]},E_j)-
 (x_n-x,E_j)\leq 0$.

 Finally, we verify (i)--(ii)--(iii).  For (ii)  write $x({\bf i})$ as
 $x_1-x_2$ with $x_1\geq 0$, $x_2\geq 0$, $|x_1|\cap |x_2|=\emptyset$. Fix an index $j\in \calj^*$.
 If $j\not\in |x_1|$ then $(l'_{[k]}-x_2,E_j)\leq (l'_{[k]}-x_2+x_1,E_j)\leq 0$. If
 $j\in |x_1|$ then $(l'_{[k]}-x_2,E_j)\leq (l'_{[k]},E_j)\leq 0$, cf. \ref{def:KR}.
 Moreover, $|x_2|\subset  \calj^*$ implies $(-x_2,E_j)\leq 0$ for any $j\in \ocalj$.
 Hence $l'_{[k]}-x_2\in (l'_{[k]}+L)\cap \calS'$,
 which implies $x_2=0$ by the minimality of $l'_{[k]}$.
 This ends (ii) and shows (i) too.
 For (iii) notice that
 $(x({\bf i}+1_{\overline{I}})+l'_{[k]},E_j)  -(E_{\overline{I}},E_j)\leq 0$ for any $j\in \calj^*$,
 hence the result follows from the minimality property (c) applied for $x({\bf i})$.
 \end{proof}

\begin{remark}\labelpar{rem:xi}
For the system of inequalities determining the cycles $x({\bf i})\in L$, we consider the $2\times 2$ 
block structure of the intersection matrix $\frI$ associated 
with the decomposition $L'=\overline{L'}\oplus (L')^*$. 
More precisely, we consider the intersection matrix in the 
form
\begin{equation*}
\Bigg(
\begin{array}{cc}
A & B\\
B^t & C\\
\end{array}
\Bigg),
\end{equation*}
where $A:=(E_v,E_w)_{v,w\in \overline{\mathcal{J}}}$, $B:=(E_v,E_w)_{v\in \overline{\mathcal{J}},w\in 
\mathcal{J}^*}$ and $C:=(E_v,E_w)_{v,w\in \mathcal{J}^*}$. 
Let $(\overline{l},l^*)$ to be the components 
(in $E_j$--basis) of some $l'\in L'$ according to the decomposition. In particular, we write $x(\bf i)$ 
as $({\bf i},x({\bf i})^*)$ and set $ \lk=(\overline{c},c^*)$. Then the property \ref{lemF1}(b) 
reformulates as $ B^t(\vasi +\overline{c})+ C (x(\vasi)^*+c^*)\leq 0$. 
\end{remark}

These cycles  satisfy the following universal property as well.

\begin{lemma}\label{lem:lauf} Fix some ${\bf i}\in(\Z_{\geq 0})^\nu$.
Assume that   $x\in L$ satisfies $m_{j}(x)= m_{j}(x({\bf i}))$ for all $j\in\ocalj$.

If $x\leq x({\bf i})$, then there is a `generalized Laufer computation sequence'
connecting $x$ with $x({\bf i})$. More precisely, one constructs a sequence $\{x_n\}_{n=0}^t$ as follows.
Set $x_0=x$. Assume that $x_n$ is already constructed. If for some $j\in\calj^*$
one has $(x_n+l'_{[k]},E_j)>0$ then take $x_{n+1}=x_n+E_{j(n)}$, where $j(n)$ is such an index. If $x_n$
satisfies \ref{lemF1}(b), then stop and set $t=n$. Then this procedure stops after finite steps and
$x_t$ is exactly $x({\bf i})$.

Moreover, along the
computation sequence  $\chi_{k_r}(x_{n+1})\leq  \chi_{k_r}(x_{n})$ for any $0\leq n<t$.
\end{lemma}
\begin{proof}
We show by induction that $x_n\leq x({\bf i})$ for any $0\leq n\leq t$; then
 the minimality property $(c)$ of $x({\bf i})$ will finish the argument. For $n=0$ this is clear.
Assume it is true for $x_n$. Then we have to verify
 that  $m_{j(n)}(x_{n})<m_{j(n)}(x({\bf i}))$. Suppose that this is not true, that is
 $m_{j(n)}(x({\bf i})-x_n)=0$. Then $(x_n+l'_{[k]},E_{j(n)})=(x({\bf i})+
 l'_{[k]},E_{j(n)})-(x({\bf i})-x_n,E_{j(n)})\leq0$, a contradiction.

Finally, notice
that $(x_n+l'_{[k]},E_{j(n)})>0$ implies   $\chi_{k_r}(x_{n+1})\leq\chi_{k_r}(x_n)$.
\end{proof}

Note that the generalized computation sequence usually is not unique, one can make several choices for $j(n)$
at each step $n$.

If the choice of the  {\it distinguished vertices} $\ocalj$ is guided by some specific geometric feature,
then  the cycles $x({\bf i})$ will inherit  further properties.

\bekezdes{}Therefore, in the sequel {\bf we fix a (non--necessarily minimal) set $\ocalj$ of bad vertices} 
(that is, by modification of their decorations one gets a rational graph as in \ref{ss:bv}). 
Next, we start to list some additional properties
satisfied by the cycles  $x({\bf i})$ associated with $\ocalj$.
The first is an addendum of Lemma \ref{lem:lauf}.

\begin{lemma}\label{lem:laufb} Fix some ${\bf i}\in(\Z_{\geq 0})^\nu$.
Assume that   $x\in L$ satisfies $m_{j}(x)= m_{j}(x({\bf i}))$ for all $j\in\ocalj$. Then
 $\chi_{k_r}(x)\geq \chi_{k_r}(x({\bf i}))$.
\end{lemma}

\begin{proof}
Write $x=x({\bf i})-y_1+y_2$ with $y_1\geq 0$, $y_2\geq 0$, both $y_i$ supported on
$\calj^*$, and $|y_1|\cap|y_2|=\emptyset$. Then $\chi_{k_r}(x)=\chi_{k_r}(x({\bf i})-y_1)
+\chi(y_2)+(y_1,y_2)-(x({\bf i})+l'_{[k]},y_2)$.
Via this identity $\chi_{k_r}(x)
\geq \chi_{k_r}(x({\bf i})-y_1)$.
Indeed, $(y_1,y_2)\geq 0$ by support--argument, $-(x({\bf i})+l'_{[k]},y_2)\geq 0$ by definition of
$x({\bf i})$, and $\chi(y_2)\geq 0$ since $y_2$ is supported on a rational subgraph
(cf. \cite[(6.3)]{OSZINV}). On the other hand, by \ref{lem:lauf},
$\chi_{k_r}(x({\bf i})-y_1)\geq \chi_{k_r}(x({\bf i}))$. \end{proof}

The computation sequence of Lemma \ref{lem:lauf}
is a generalization of Laufer's computation sequence \ref{La} targeting
Artin's fundamental cycle $Z_{min}$ (see \ref{AL:11}).
In fact, for rational graphs, the algorithm is more precise. For further references we repeat it here:

\begin{bekezdes}\labelpar{LC}{\bf Laufer algorithm and criterion} (\cite{Laufer72} or \ref{La}). \
{\it Let $\{z_n\}_{n=0}^T$ be the  computation sequence
(similar as above with $[k]=[k_{can}]$)
connecting $z_0=E_{j}$ (for some $j\in\calj$)
 and the Artin's fundamental cycle $z_T=Z_{min}$.
 (This means that $z_{n+1}=z_n+E_{j(n)}$ for some $j(n)$, where
 $(z_n,E_{j(n)})>0$.)
 Then the graph is  {\it rational} if and only if
at every step $0\leq n<T$ one has $(E_{j(n)},z_n)=1$.

 The same statement is true for a sequence connecting $z_0=E_{I}$  with $z_{min}$ for any connected $E_I$.

 (Both statement can be reinterpreted by the identity $\chi(E_{I})=\chi(z_{min})=1$.)}
\end{bekezdes}

In some of the applications regarding the cycles $x({\bf i})$
 we do not really need their precise forms,
rather  the values $\chi_{k_r}(x({\bf i}))$. These can be
computed inductively thanks to the following.

\begin{proposition}\labelpar{propF1}
For any  $k_r\in Char$,  ${\bf i}\in(\Z_{\geq 0})^\nu$ and
$j\in \overline{\calj}$ one has
\begin{equation*}
\chi_{k_r}(x({\bf i}+1_j))=
\chi_{k_r}(x({\bf i}))+1-(x({\bf i})+l'_{[k]},E_{j}).
\end{equation*}
Moreover,  $\chi_{k_r}(x(0,\ldots,0))=0$.
\end{proposition}

\begin{proof}
We consider  the  computation sequence $\{x_n\}_{n=0}^t$
connecting $x({\bf i})+E_{j}$ and $x({\bf i}+1_j)$ and
we  prove  that $(x_n+l'_{[k]},E_{j(n)})$ is exactly $1$ for any $0\leq n<t$.
Indeed, we take $z_n:=x_n-x({\bf i})$ for $0\leq n\leq t$ and
 one verifies that $\{z_n\}_{n=0}^t$ is the beginning of a
 Laufer sequence $\{z_n\}_{n=0}^T$ (with $t\leq T$)
connecting $E_{j}$ with $z_{min}$ (as in \ref{LC}).
This follows from $(x_n+l'_{[k]},E_{j(n)})>0$ and $(x({\bf i})+l'_{[k]},E_{j(n)})\leq 0$.
Moreover,
the values $(z_n,E_{j(n)})$ will stay unmodified for every $n$ if we replace our graph $G$
with the rational graph $\widetilde{G}$ by decreasing the decorations of the bad vertices.
 Therefore, by Laufer's Criterion \ref{LC}, $(z_n,E_{j(n)})=1$ in
 $\widetilde{G}$, hence consequently in $G$ too. This shows that
\begin{equation*}
1=(x_n-x({\bf i}),E_{j(n)})=(x_n+l'_{[k]},E_{j(n)})-
(x({\bf i})+l'_{[k]},E_{j(n)})
\geq (x_n+l'_{[k]},E_{j(n)}).
\end{equation*} Since $(x_n+l'_{[k]},E_{j(n)})>0$, this number  must equal 1.

This shows   $\chi_{k_r}(x_{n+1})=\chi_{k_r}(x_n)$, or
$\chi_{k_r}(x({\bf i}+1_j))=\chi_{k_r}(x({\bf i})+E_{j})$.
\end{proof}

The next technical result about computation sequences is crucial in the proof of the main result.

\begin{proposition}\label{prop:G}
Fix ${\bf i}\in(\Z_{\geq 0})^\nu$ and a subset $\overline{J}\subseteq \ocalj$.
Let $\cals\iij\subseteq \calj^*$
be the support of $x({\bf i}+1_{\overline{J}})-x({\bf i})-E_{\overline{J}}$.

(I) \ For any subset $\cals'\subseteq \cals\iij$  one can find  a generalized
Laufer computation sequence $\{x_n\}_{n=0}^t$ as in Lemma \ref{lem:lauf} connecting
$x_0=x({\bf i})+E_{\overline{J}} +E_{ \cals'}$ with $x_t=x({\bf i}+1_{\overline{J}})$ with the property that there exists a certain
$t_s$ ($0\leq t_s\leq t$)  such that

(a) $x_{t_s}=x({\bf i})+E_{\overline{J}}+E_{\cals\iij}$, and

 (b)  $\chi_{k_r}(x_n)=\chi_{k_r}(x({\bf i}+1_{\overline{J}}))$  for any $t_s\leq n\leq t$,
or,   $(x_n+l'_{[k]},E_{j(n)})=1$ for $t_s\leq n<t$.

 (II) \ Let $\widetilde{\cals}$ be a subset of $\calj^*$  such that
 \begin{equation}\label{eq:sss}
 \chi_{k_r}(x({\bf i})+E_{\overline{J}\cup\widetilde{\cals}})=
 \chi_{k_r}(x({\bf i}+1_{\overline{J}})).
 \end{equation}
 Then $\widetilde{\cals}\subseteq \cals\iij$. Moreover, there exists a computation
 sequence $\{x_n\}_{n=0}^t$ as in Lemma \ref{lem:lauf} connecting
$x_0=x({\bf i})+E_{\overline{J}}$ with $x_t=x({\bf i})+E_{\overline{J}\cup\widetilde{\cals}}$
such that  $\chi_{k_r}(x_{n+1})\leq \chi_{k_r}(x_n)$ for any $0\leq n<t$.

 (III) \ For any  cycle $l^*>0$ with support $|l^*|\subseteq \calj^*\setminus \cals\iij$, there exists a
 computation sequence  $\{x_n\}_{n=0}^t$  of type  $x_{n+1}=x_n+E_{j(n)}$ (for  $n<t$),
 $x_0=x({\bf i})+E_{\overline{J}\cup\cals\iij}$ and
 $x_t=x({\bf i})+E_{\overline{J}\cup\cals\iij}+l^*$ such that $\chi_{k_r}(x_{n+1})\geq \chi_{k_r}(x_{n})$
 for any $0\leq n< t$ (that is, with $(x_n+l'_{[k]},E_{j(n)})\leq 1$).
\end{proposition}
\begin{proof}
(I) We will use the following notation: for any $x\geq x({\bf i})+E_{\overline{J}}$ we write
 $\|x\|$ for the support $|x- x({\bf i})-E_{\overline{J}}|$. Note that Lemma \ref{lem:lauf}
guarantees the existence of a  computation sequence connecting $x({\bf i})+E_{\overline{J}\cup \cals'}$ with
$x({\bf i}+1_{\overline{J}})$.  We consider such a sequence $\{x_n\}_{n=0}^t$ constructed in such a way
that in the procedure of choices of $j(n)$'s  at the first steps
we try to  increase  $\|x_n\|$ as much as possible.
 More precisely, for any $0\leq n<t_1$, the index $j(n)\in\calj^*$ is chosen as follows:
 \begin{equation}\label{eq:jn}
 \left\{\begin{array}{l} (x_n+l'_{[k]},E_{j(n)})>0 \\
 E_{j(n)}\not\in \|x_n\|.\end{array}\right.
 \end{equation}
 Assume that this  stops for $n=t_1$, that is,
for $n=t_1$ there is no index $j(n)\in\calj^*$ which would  satisfy (\ref{eq:jn}).
We claim that
$\|x_{t_1}\|=\|x({\bf i}+1_{\overline{J}})\|=\cals\iij$,
hence $t_s=t_1$ satisfies part (a) of the proposition.

Indeed, assume that this is not the case. Then we continue the construction of the sequence, and let
$t_2+1$ be the first index when $\|x\|$ increases again, that is $\|x_n\|=\|x_{t_1}\|$ for
$t_1\leq n\leq t_2$ and $\|x_{t_2+1}\|=\|x_{t_1}\|\cup \{j^*\}\not=\|x_{t_1}\|$ for some $j^*\in\calj^*$.
Hence $j^*=j(t_2)$.

Since $(x_{t_2}+l'_{[k]},E_{j^*})>0$ and $(x_{t_1}+l'_{[k]},E_{j^*})\leq 0$,
we get $(x_{t_2}-x({\bf i}),E_{j^*})>-(x({\bf i})+l'_{[k]},E_{j^*})\geq (x_{t_1}-x({\bf i}),E_{j^*})$.
Since $x_{t_2}-x({\bf i})$ and $x_{t_1}-x({\bf i})$ have the same support, which does not contain $j^*$,
this strict inequality can happen only if  $(x_{t_1}-x({\bf i}),E_{j^*})>0$.
By the same argument, in fact, there exists a connected component $C$ of the reduced cycle $x_{t_1}-x({\bf i})$
such that \begin{equation}\label{eq:GG}
((x_{t_2}-x({\bf i}))|_C,E_{j^*})> (C,E_{j^*})>0.\end{equation}
Next, we analyze the restriction of the sequence $z_n:=x_n-x({\bf i})$ to $C$ for $t_1\leq n\leq t_2$.
First note that $(z_n,E_{j(n)})=(x_n+l'_{[k]},E_{j(n)})-(x({\bf i})+l'_{[k]},E_{j(n)})>0$.
If $E_{j(n)}$ is supported by $C$ then it does not intersect any other components of $x_{t_1}-x({\bf i})$,
hence $(z_n|_C,E_{j(n)})>0$ too.
Let us consider that subsequence $\tilde{z}_*$
of $z_n|_C$ which is obtained from $z_n|_C$ by eliminating those steps from the computation sequence
of $\{x_n\}_{n=t_1}^{t_2}$  which correspond to elements $j(n)$ not supported by $C$.
Then the sequence starts with $C$, ends with $(x_{t_2}-x({\bf i}))|_C$, it is the beginning of a Laufer
sequence connecting the {\it connected}
 $C$ with the fundamental cycle of $C$, but at the step $t_2$ one has
$(z_{t_2}|_C,E_{j(t_2)})\geq 2$, cf. (\ref{eq:GG}).

Note also that the sequence $z_n|_C$ is reduced along $\overline{J}$, hence along the procedure we do not add any base element from $\overline{J}$, hence if we decrease the self--intersections of these vertices we will not
modify the Laufer data along the sequence. Hence, we can assume that $C$ is  supported by a
rational graph.
But this contradicts the existence of  $\tilde{z}_*$, cf. \ref{LC}.

Part (b) uses the same argument.
We fix a connected component  of $x_{t_s}-x({\bf i})$. Since in the Laufer steps the components do not
interact, we can even assume that the support of $x_{t_s}-x({\bf i})$ is connected.
 Then  $x_n-x({\bf i})$ for  $n\geq t_s$ is part of the computations sequence connecting
 the reduced connected $x_{t_s}-x({\bf i})$ to its fundamental cycle. Since we may assume that
  $C$ is rational (since the steps do not involve $\overline{J}$),
  along the sequence we must have
 $(x_n-x({\bf i}),E_{j(n)})=1$  by \ref{LC}. This happens only if
  $(x_n+l'_{[k]},E_{j(n)})=1$ and $(x({\bf i})+l'_{[k]},E_{j(n)})=0$.

(II) Assume that $\widetilde{\cals}\not\subseteq \cals\iij$, and set
$\cals':=  \widetilde{\cals}\cap\cals\iij$
and  $\Delta\cals:=  \widetilde{\cals}\setminus \cals\iij$.
Take  a computation sequence $\{x_n\}_{n=0}^t$ as in (I) connecting
$x({\bf i})+E_{\overline{J}\cup \cals'}$ with $x({\bf i}+1_{\overline{J}})$.
 Since $\chi_{k_r}(x_n)$ is non--increasing, cf.  \ref{lem:lauf},
$1-(E_{j(n)}, x_n+l'_{[k]})\leq 0$.
Therefore,
$1-(E_{j(n)}, x_n+E_{\Delta\cals}+l'_{[k]})\leq 0$
too, since $j(n)\not\in \Delta\cals$. Since $\{x_n+E_{\Delta\cals}\}_n$ connects
$x({\bf i})+E_{\overline{J}\cup \widetilde{\cals}}$ with $x({\bf i}+1_{\overline{J}})+E_{\Delta\cals}$, we get
$$\chi_{k_r}(x({\bf i})+E_{\overline{J}\cup \widetilde{\cals}})\geq \chi_{k_r}
(x({\bf i}+1_{\overline{J}})+E_{\Delta\cals}).$$
This together  with assumption (\ref{eq:sss}) and Lemma \ref{lem:laufb} guarantee that, in fact,
\begin{equation}\label{eq:EQ}
\chi_{k_r}(x({\bf i}+1_{\overline{J}})+E_{\Delta\cals})=
\chi_{k_r}(x({\bf i}+1_{\overline{J}})).\end{equation}
On the other hand,
\begin{equation*}
\chi_{k_r}(x({\bf i}+1_{\overline{J}})+E_{\Delta\cals})-
\chi_{k_r}(x({\bf i}+1_{\overline{J}}))
=\chi(E_{\Delta\cals})-(E_{\Delta\cals},
x({\bf i}+1_{\overline{J}})+l'_{[k]})\geq \chi(E_{\Delta\cals}),
\end{equation*}
where the last inequality follows from the definition of $x({\bf i}+1_{\overline{J}})$. Since
$\chi(E_{\Delta\cals})$ is the number of connected components of $E_{\Delta\cals}$, it is strictly
positive, a fact which contradicts (\ref{eq:EQ}).

For the second part we construct a computation sequence as in (I), applied for
$\cals'=0$, in such a way that first we choose only the $j(n)$'s from $\widetilde{\cals}$. We claim that in this
way we fill in all $\widetilde{\cals}$. Indeed, assume that this procedure stops at the level of $x_m$; that is,
$x({\bf i})+E_{\overline{J}}\leq x_m <  x({\bf i})+E_{\overline{J}\cup \widetilde{\cals}}$
and
\begin{equation}\label{eq:ineqtilde}
(E_j,x_m+l'_{[k]})\leq 0 \ \ \mbox{for all $j\in \Delta\widetilde{\cals}:=
\widetilde{\cals}\setminus ||x_m||$}.
\end{equation}
Then
$$\chi_{k_r}( x({\bf i})+E_{\overline{J}\cup \widetilde{\cals}})
-\chi_{k_r}(x_m)=\chi(E_{\Delta\widetilde{\cals}})-
(E_{\Delta\widetilde{\cals}}, x_m+l'_{[k]})\geq
\chi(E_{\Delta\widetilde{\cals}}),$$
where the last inequality follows from (\ref{eq:ineqtilde}). Since
$\chi(E_{\Delta\widetilde{\cals}})>0$,   the
assumption (\ref{eq:sss}) imply
 $\chi_{k_r}(x_m)<\chi_{k_r}( x({\bf i}
+1_{\overline{J}}))$, a fact which contradicts Lemma \ref{lem:laufb}.

(III) The statement follows by induction from the following fact: if $l^*>0$,
$|l^*|\subseteq\calj^*\setminus \cals\iij$, then there exists $j\in|l^*|$ so that
$$\chi_{k_r}(x({\bf i})+E_{\overline{J}\cup\cals\iij}+l^*-E_j)\leq
\chi_{k_r}(x({\bf i})+E_{\overline{J}\cup\cals\iij}+l^*).$$
Indeed, if not, then $(E_j,   x({\bf i})+E_{\overline{J}\cup\cals\iij}+l'_{[k]}+l^*-E_j)\geq 2$
for any $j\in |l^*|$. On the other hand,
$(E_j,  x({\bf i})+E_{\overline{J}\cup\cals\iij}  +l'_{[k]})\leq 0$, by the proof of part (I) (namely, the choice of
$t_s=t_1$), or by the definition of $\cals\iij$. Therefore,
$(E_j,l^*-E_j)\geq 2$, or, $(E_j,l^*+k_{can})\geq 0$ for all $j$. Summing up over the coefficients of $l^*$,
we get $(l^*,l^*+k_{can})\geq 0$, which contradicts (\ref{eq:artin}) since the subgraph generated by $|l^*|$
is rational.
\end{proof}

\subsection{The statement}\labelpar{ss:reduction}\
Now we are ready to formulate the main result of this section: in the definition of
the lattice cohomology we wish to replace the (cubes of the) lattice $L$ with cubes of a smaller rank
free $\Z$--module associated with the bad vertices.

\bekezdes\label{bek:331}
 {\bf Definition of the (quadrant of the) new free $\Z$--module.} Let us fix $[k]$ and
assume that the graph $G$ admits a family of $\nu$
bad vertices as above. Then define $\overline{L}=(\Z_{\geq 0})^\nu$, and the function
$\overline{w}_0:(\Z_{\geq 0})^\nu\to \Z$ by
\begin{equation}
\overline{w}_0(i_1,\ldots,i_\nu):=\chi_{k_r}(x(i_1,\ldots,i_\nu)).
\end{equation}
 Then $\overline{w}_0$ defines a set $\{\overline{w}_q\}_{q=0}^\nu$ of compatible weight functions
 depending on $[k]$, defined similarly as in \ref{ex:ww}, denoted by $\overline{w}[k]$.

\begin{theorem}[\bf Reduction Theorem]\labelpar{red}
Let $G$ be a negative definite connected graph and let $k_r$ be the distinguished representative
of a characteristic class. Suppose $\overline{\calj}=\{j_k\}_{k=1}^\nu$ is a
family of bad vertices and $(\overline{L},\overline{w}[k])$ is the
first quadrant of the new weighted free $\Z$--module
associated with $\overline{\calj}$ and $k_r$. Then there is a graded $\Z[U]$--module isomorphism
\begin{equation}\label{eq:reda}
\bH^*(G,k_r)\cong\bH^*(\overline{L},\overline{w}[k]).
\end{equation}
\end{theorem}

Note that via  \ref{propF2}, (\ref{eq:reda}) is equivalent to the isomorphism:
\begin{equation}\label{eq:redb}
\bH^*([0,\infty)^s,k_r)\cong\bH^*([0,\infty)^\nu,\overline{w}[k]).
\end{equation}

\begin{remark}
 The reduction theorem immediately implies the Vanishing Theorem \ref{vanishth} for lattice cohomology, 
in particular, the new classification of rational surface singularities 
\ref{thm:ratlatcoh}.
\end{remark}

\subsection{Proof of the Reduction Theorem}\
In this section we abbreviate $k_r$ into $k$, $\overline{w}[k]$ into $\overline {w}$.
Assume that there exists a pair  $j, j'\in\ocalj$, $j\not=j'$, such that $(E_j,E_{j'})=1$. Then we can blow up
the intersection point $E_j\cap E_{j'}$. We have to observe two facts. First,
the lattice cohomology $\bH^*(G,k)$ is stable with respect to this blow up \cite{Nlat,Nexseq}. Second,
the `strict transform' of the set $\ocalj$ can serve for a new set of bad vertices
and the right hand side of (\ref{eq:reda}) stays stable as well. Therefore, by additional blow ups,
{\it  we can assume that}
\begin{equation}\label{ss:assumption}
(E_j, E_{j'})=0 \ \ \mbox{for every pair  $j, j'\in\ocalj$, $j\not=j'$.}
\end{equation}

\bekezdes{\bf The first step. Comparing $S_N$ and $\overline{S}_N$.}\label{Leray}

We consider the projections  $\phi:(\Z_{\geq 0})^s\to (\Z_{\geq 0})^\nu$
and  $\phi:[0,\infty)^s\to [0,\infty)^\nu$
given by $(m_j)_{j\in\calj}\mapsto (m_{j})_{j\in\ocalj}$. This induces a projection
of the cubes too. If $(l,I)\in \calQ(L)$ is a cube of $L$, then write $I$ as $\overline{I}\cup I^*$ where
$\overline{I}=I\cap \ocalj$ and $I^*=I\cap \calj^*$. Then the vertices of $(l,I)$ are projected via $\phi$ into the
vertices of the cube $(\phi(l),\overline{I})\in \calQ(\overline{L})$ of $\overline{L}$.
It is convenient to write $\overline {I}:=\phi(I)$ and $\phi(l,I):=(\phi(l),\overline{I})$.

By \ref{lem:laufb}, we get that  for any $l\in (\Z_{\geq 0})^s$ we have $w(l)\geq \overline{w}(\phi(l))$, hence
\begin{equation}\label{eq:ineq}
w((l,I))\geq \overline{w}(\phi(l,I)) \ \  \ \mbox{for any cube $(l,I)\in \calQ(L)$}.
\end{equation}
Recall that for any $N$ we define $S_N\subseteq [0,\infty)^s$ as the union of cubes of
$[0,\infty)^s$ of weight $\leq N$.
Similarly, let $\overline{S}_N\subseteq [0,\infty)^\nu$ be the
union of cubes $({\bf i},\overline{I})$
 with $\overline{w}({\bf i},\overline{I})\leq N$.
 Then, the statement of Theorem \ref{red}, via Theorem
 \ref{th:HS},  is equivalent to the fact that
 \begin{equation}\label{eq:leray}
\mbox{ {\it
$S_N$ and $\overline{S}_N$ have the same cohomology groups for any integer $N$.}}
\end{equation}

Note that by (\ref{eq:ineq}) $\phi(S_N)\subseteq \overline{S}_N$, and by construction
 $\phi|_{S_N}:S_N\to \overline{S}_N$ is a cubical map.
For any $\ii\subseteq \overline{S}_N$
we consider $\phi^*_N\ii\subseteq S_N$ defined as
the union of all cubes $(l,I)\subseteq S_N$ with  $\phi(l,I)=\ii$.
[We warm the reader that this is not the inverse image
$(\phi|_{S_N})^{-1}\ii$, rather it is the closure of the inverse image of the interiour
of the cube $\ii$; see also below.] If $\psi:[0,\infty)^s\to [0,\infty)^{s-\nu}$ is the second projection
on the $\calj^*$--coordinate direction, then $\phi^*_N\ii$ is the product of $\psi(\phi^*_N\ii)$ with the
cube $\ii$; in particular, it has the homotopy type of $\psi(\phi^*_N\ii)$.

A Mayer--Vietoris inductive (or Leray type spectral sequence)
argument shows that (\ref{eq:leray}) follows from
\begin{equation}\label{eq:leray3}
\mbox{ {\it
$\phi^{*}_N\ii$ is non--empty and contractible  for any $\ii\in\overline{S}_N$.}}
\end{equation}

\bekezdes{\bf Generalities about contractions.}\labelpar{ss:2}
 {\it In the sequel we fix a cube $\ii$ from $\overline{S}_N$}  and we start
to prove (\ref{eq:leray3}).  For any such cube $\ii$ we also consider the inverse image
$\phi^{-1}\ii$ consisting of the union of
all cubes $(l,I)$ of $[0,\infty)^s$ with $\phi(l,I)\subseteq \ii$ (not
necessarily from  $S_N$). We can also consider $(\phi|_{S_N})^{-1}\ii$,
the union of cubes $(l,I)$ from $S_N$ with $\phi(l,I)\subseteq \ii$.
Clearly, $$\phi^*_N\ii\subseteq (\phi|_{S_N})^{-1}\ii\subseteq \phi^{-1}\ii.$$
Note that $\phi^{-1}\ii$ is the product of the cube   $\ii$ with $[0,\infty)^{s-\nu}$.
Our goal is to contract this `fiber direction  space' $[0,\infty)^{s-\nu}$
in such a way that  along the contraction $\chi_{k}$ does not increase, and the
contraction preserves the subspaces $\phi^*_N\ii$ and $(\phi|_{S_N})^{-1}\ii$ as well.

The cycles supported on $\calj^*$ (`fiber direction') will be denoted by  $l^*=\sum_{j\in\calj^*}m_jE_j$.
For any pair $l_1^*$ and $l_2^*$ with
$l_1^*\leq l_2^*$ we consider the real $s$--dimensional rectangle
 $R_{\ii}(l_1^*,l_2^*)$, the product of a rectangle in the $(s-\nu)$--dimensional space
with the cube $\ii$:  it is the convex closure of the lattice points, which have the form
$$x({\bf i})+E_{\overline{J}}+l^* \ \ \mbox{with} \ \ \ \overline{J}\subseteq \overline{I} \ \ \
\mbox{and} \ \ l^*\in  L, \ \ l_1^*\leq l^*\leq l_2^*.$$
We extend this notation allowing $l_2^*$ to have all its entries  $\infty$.

Note that the lattice points $x({\bf i})+E_{\overline{J}}+l^*$, being in $[0,\infty)^s$, are effective,
hence the relevant $l^*$ satisfies $l^*\geq l^*_{1,min}:=-x({\bf i})+\sum_{j\in\ocalj}i_jE_j$
(the projection of $-x({\bf i})$ on the $\calj^*$-components). In particular,
 $R_{\ii}(l^*_{1,min},\infty)=\phi^{-1}\ii\subseteq [0,\infty)^s$, and  we can assume
 that $l_1^*$ and $l_2^*$ satisfy $l^*_{1,min}\leq l_1^*\leq l_2^*\leq \infty$. Note also that
 $l_{1,min}^*\leq 0$.

\medskip

We start to discuss the existence of  a contraction
$c:R_{\ii}(l_1^*,l_2^*+E_j)\to R_{\ii}(l_1^*,l_2^*)$ for some $j\in\calj^*$, acting in the direction of the
$\calj^*$--coordinates and having the property that
 $\chi_k$ will  not increase along it. The map
$c$ is defined  as follows. If a lattice point $l$ is in $R_{\ii}(l_1^*,l_2^*)$, then
 $c(l)=l$. Otherwise $l$ has the form $l=x({\bf i})+E_{\overline{J}}+l^*+E_j$ for some
 $l^*$ with $l_1^*\leq l^*\leq l_2^*$ and $m_j(l^*)=m_j(l^*_2)$. Then set $c(l)=l-E_j$.
The next criterion guarantees that $\chi_k$ does not increase along this contraction.

\begin{lemma}\label{lem:con1}
Assume that for some $l_2^*$ and $j\in\calj^*$ one has
$$\chi_k(x({\bf i})+E_{\overline{I}}+l^*_2+E_j)\geq \chi_k(x({\bf i})+E_{\overline{I}}+l^*_2).
$$
Then, for any $l^*$ with $l_1^*\leq l^*\leq l_2^*$ and $m_j(l^*)=m_j(l^*_2)$,
and for every $\overline{J}\subseteq \overline{I}$, one also has
$$\chi_k(x({\bf i})+E_{\overline{J}}+l^*+E_j)\geq \chi_k(x({\bf i})+E_{\overline{J}}+l^*).
$$
Therefore, $\chi_k(c(l))\leq \chi_k(l)$ for any $l\in  R_{\ii}(l_1^*,l_2^*+E_j)$.
\end{lemma}
\begin{proof} Use $\chi_k(z+E_j)=\chi_k(z)+1-(E_j,z+l'_{[k]})$ and
$(E_j, E_{\overline{I}}-E_{\overline{J}}+l^*_2-l^*)\geq  0$.\end{proof}
The following lemma generalizes  results of \cite[\S\,3.2]{Nlat}, where the case $\nu=1$ is treated.
\begin{lemma}\label{lem:con2}
Assume that for some fixed $l_2^*$ there exists an infinite sequence of cycles
$\{x^*_n\}_{n\geq 0}$, $x^*_n=\sum_{j\in\calj^*}m_{j,n}E_j$,  with $x^*_0=l_2^*$ such that

\begin{itemize}
\item[(a)] $x^*_{n+1}=x^*_n+E_{j(n)}$ for some $j(n)\in\calj^*$, $n\geq 0$;
\item[(b)] $\chi_k(x({\bf i})+E_{\overline{I}}+x^*_{n+1})\geq \chi_k(x({\bf i})+E_{\overline{I}}+x^*_n)$
for any $n\geq 0$.
\item[(c)] for any fixed $j$ the sequence $m_{j,n}$ tends to infinity as $n$ tends to infinity;
\end{itemize}
Then there exists a contraction of  $R_{\ii}(l_1^*,\infty)$ to  $R_{\ii}(l_1^*, l_2^*)$
along which $\chi_k$ is non--increasing.
\end{lemma}
\begin{proof} Use Lemma \ref{lem:con1} and induction over $n$. 
\end{proof}
Symmetrically, by similar proof, one has the following statements too.

\begin{lemma}\label{lem:con3} \

(I) \ For any  fixed $l_1^*$ and $j\in\calj^*$ with $l_1^*-E_j\geq l_{1,min}^*$ if
$$\chi_k(x({\bf i})+l^*_1-E_j)\geq \chi_k(x({\bf i})+l^*_1),
$$
then for any $l^*$ with $l_1^*\leq l^*\leq l_2^*$ and $m_j(l^*)=m_j(l^*_1)$,
and for every $\overline{J}\subseteq \overline{I}$, one also has
$$\chi_k(x({\bf i})+E_{\overline{J}}+l^*-E_j)\geq \chi_k(x({\bf i})+E_{\overline{J}}+l^*).
$$
Therefore, $ R_{\ii}(l_1^*-E_j,l_2^*)$ contracts onto $ R_{\ii}(l_1^*,l_2^*)$ such that
$\chi_k$ does non increase  along the contraction.

(II) \ Assume that there exists a sequence of cycles
$\{x^*_n\}_{n=0}^t$ with $x^*_0=l_{1,min}^*$  and  $x^*_t=l_{1}^*$ such that for any $0\leq n<t$
one has
\begin{itemize}
\item[(a)] $x^*_{n+1}=x^*_n+E_{j(n)}$ for some $j(n)\in\calj^*$,
\item[(b)] $\chi_k(x({\bf i})+x^*_{n})\geq \chi_k(x({\bf i})+x^*_{n+1})$.
\end{itemize}
Then there exists a contraction of  $R_{\ii}(l_{1,min}^*,l_2^*)$ to  $R_{\ii}(l_1^*, l_2^*)$
along which $\chi_k$ is non--increasing.
\end{lemma}

\bekezdes{\bf Contractions.}\label{ss:3}

First we show the existence of a sequence of cycles
$\{x^*_n\}_{n=0}^t$ with $x^*_0=l_{1,min}^*$ and $x^*_t=0$  which satisfies the assumptions
of Lemma \ref{lem:con3}(II). This follows inductively from the following lemma.
\begin{lemma}\label{lem:con4} For any $x^*$  with $l^*_{1,min}\leq x^*<0$ and supported on $\calj^*$
there exists at least one index
$j\in |x^*|$ such that
\begin{equation}\label{eq:xx}
\chi_k(x({\bf i})+x^*)\geq \chi_k(x({\bf i})+x^*+E_j).\end{equation}
\end{lemma}
\begin{proof}
(\ref{eq:xx}) is equivalent to $(E_j,x({\bf i})+l'_{[k]}+x^*)\geq 1$ for some $j\in |x^*|$.
Assume the opposite, that is,  $(E_j,x({\bf i})+l'_{[k]}+x^*)\leq 0$ for every $j\in|x^*|$.
 On the other hand, for  $j\in\calj^*\setminus |x^*|$ one has $(E_j,x^*)\leq 0$ and
 $(E_j,x({\bf i})+l'_{[k]})\leq 0$ by \ref{lemF1}(b). Hence  $(E_j,x({\bf i})+l'_{[k]}+x^*)\leq 0$
 for every $j\in\calj^*$.
 This contradicts the minimality of $x({\bf i})$ in \ref{lemF1}(c).
 \end{proof}

 In particular, Lemma  \ref{lem:con3}(II) applies for $l_1^*=0$ and any $l_2^*\geq 0$ (including $\infty$).

Next, we search for a convenient small cycle $l_2^*$ for which Lemma \ref{lem:con2} applies as well.
First we show that $l^*_2=\infty$ can be replaced by
$x({\bf i}+1_{\overline{I}})-x({\bf i})-E_{\overline{I}}$.
\begin{lemma}\label{lem:i+I}
There exists a sequence as in Lemma \ref{lem:con2} with $x^*_0=
x({\bf i}+1_{\overline{I}})-x({\bf i})-E_{\overline{I}}$.
\end{lemma}
\begin{proof} First we show the existence of some $l^*_2$, with all its coefficient very large,
which  can be connected by a computation sequence to $\infty$ with properties
(a)-(b)-(c) of \ref{lem:con2}. For this,
consider the full subgraph supported by $\calj^*$. Since it is negative definite, it supports an
effective cycle $Z^*$ such that $(Z^*,E_j)<0$ for any $j\in \calj^*$. Consider any sequence $\{x^*_n\}_{n=0}^t$,
$x^*_{n+1}=x^*_n+E_{j(n)}$,  such that $x^*_0=0$ and $x^*_t=Z^*$. Then, there exists
$\ell_0\geq 1$ sufficiently large such that   for any $\ell\geq \ell_0$ and $n$ one has
$$\chi_k(x({\bf i})+E_{\overline{I}}+\ell Z^*+x^*_{n+1})
\geq \chi_k(x({\bf i})+E_{\overline{I}}+\ell Z^*+x^*_{n}).$$
Hence the sequence $\{\ell Z+x_n\}_{\ell \geq \ell_0,\, 0\leq n\leq t}$ connects $l_2^*=\ell_0Z^*$ with $\infty$ with the required properties.

Next, we connect $x({\bf i}+1_{\overline{I}})-x({\bf i})-E_{\overline{I}}$ with this $l^*_2$
via a sequence which satisfies (a)-(b)-(c) of Lemma \ref{lem:con2}. Its existence follows from the following statement:

For any $l^*>0$ supported by $\calj^*$ there exists at least one index $j\in|l^*|$ such that
$$\chi_k(x({\bf i}+1_{\overline{I}})+l^*-E_j)\leq \chi_k(x({\bf i}+1_{\overline{I}})+l^*).$$
Indeed, assume the opposite. Then $(E_j,l^*)\geq E_j^2+2$ for any $j\in|l^*|$.
Hence $(E_j,l^*+k_{can})\geq 0$, or $\chi(l^*)\leq 0$, which contradict the rationality of the subgraph
 supported by $\calj^*$.
\end{proof}

Finally,  by Proposition \ref{prop:G}(I) (applied for $\overline{I}=\overline{J}$ and
${\bf s'}={\bf s}({\bf i}, \overline{J})$),
the newly determined `upper' bound $l^*_2=x({\bf i}+1_{\overline{I}})-x({\bf i})-E_{\overline{I}}$
can be pushed down further to its support $\cals\ii$.
Hence \ref{prop:G}(I), \ref{lem:i+I} and \ref{lem:con4} imply the following.

\begin{corollary}\label{cor:con}
There exists a deformation contraction of $\phi^{-1}\ii$ to $R_{\ii}(0, E_{\cals\ii})$ along which $\chi_k$ is
non--increasing.  Moreover, its restriction induces a deformation
 retract from  $(\phi|_{S_N})^{-1}\ii$ to $S_N\cap R_{\ii}(0, E_{\cals\ii}  )$. Restricting further,
 it gives a deformation
 retract from  $\phi^*_N\ii$ to  $\Phi^*_N\ii$, where $\Phi^*_N\ii$ is the product of the cube $\ii$ with
  $$\psi(\phi^*_N\ii)\cap \{l^*\,:\, 0\leq l^*-\psi(x({\bf i}))\leq  E_{\cals\ii}\}.$$
\end{corollary}
Note that this last space  $\Phi^*_N\ii$ is now rather `small': it is contained in the cube
$(x({\bf i}),\overline{I}\cup \cals\ii)$.
Nevertheless, the $N$--filtration of this cube can be rather complicated!

The statement of the above corollary
means that if  $\Phi^*_N\ii$ is empty if and only if  $\phi^*_N\ii$ is empty,
and when  they are not empty then they
have the same homotopy type. Therefore, via (\ref{eq:leray3}), we need to show that
$$\Phi^*_N\ii \ \ \mbox{ {\it is non--empty and contractible}.}$$

\bekezdes{\bf The non--emptiness of  $\Phi^*_N\ii$.}\labelpar{ss:nonempty}
Recall that we fixed an integer $N$ and a   cube
$({\bf i},\overline{I})$ which belongs to $\overline{S}_N$. By Definition \ref{bek:331} and
Proposition~\ref{prop:G}(I)(b) this reads as
\begin{equation}\label{eq:NN}
 \chi_k(x({\bf i}+1_{\overline{J}}))=
  \chi_k(x({\bf i})+E_{\overline{J}\cup \cals({\bf i},\overline{J})})
 \leq N \ \
 \mbox{for every} \ \  \overline{J}\subseteq\overline{I}.
\end{equation}
The non-emptiness follows from the following statement.
 \begin{proposition}\label{th:nonempty} For any fixed cube $({\bf i},\overline{I})\in \overline{S}_N$
 there exists  a cycle in $L$ of the form $x({\bf i})+E_{\widetilde{\cals}\ii}$ such that
 $\widetilde{\cals}\ii\subseteq \cals\ii$ and
 $(x({\bf i})+E_{\widetilde{\cals}\ii},\overline{I}) \subseteq \Phi_N^{*}\ii$; that is
 \begin{equation}\label{eq:NN2}
 \chi_k(x({\bf i})+E_{\overline{J}\cup \widetilde{\cals}\ii})
 \leq N \ \
 \mbox{for every} \ \  \overline{J}\subseteq\overline{I}.
\end{equation}
 \end{proposition}

\begin{proof}
The proof is long, it fills all this Subsection \ref{ss:nonempty}.
 It is an induction over the cardinality of $\calj$, respectively of $\overline{I}$.
At start we reformulate it by keeping only the necessary combinatorial data, and we
also perform three  reductions to simplify the involved combinatorial complexity.
We will also write $\widetilde{\cals}:=\widetilde{\cals}\ii$ for the wished cycle.

\bekezdes\label{bek:egy}  {\bf Starting the reformulation.} Define (cf. Proposition \ref{prop:G}(I))
\begin{equation}\label{eq:NN3}
N(G):= \max_{\overline{J}\subseteq \overline{I}}\chi_k(x({\bf i}+1_{\overline{J}}))=
 \max_{\overline{J}\subseteq \overline{I}} \chi_k(x({\bf i})+E_{\overline{J}\cup \cals({\bf i},\overline{J})}).
\end{equation}
$N(G)$ is the smallest integer $N$ for which (\ref{eq:NN}) is valid; hence it is enough to prove
Theorem \ref{th:nonempty}  only for $N=N(G)$. Note that $N(G)$ depends on $\ii$, though in its notation this is
not emphasized.

In fact, even the weight $\chi_{k}(x({\bf i}))$ ---
and partly the cycle $x({\bf i})$, cf. \ref{bek:sigma}, ---
are irrelevant in the sense that it is enough to treat a relative version of the statement.
Indeed, we can consider only
the value $\Delta N(G):=N(G)-\chi_{k}(x({\bf i}))$, which equals (use the last term of (\ref{eq:NN3})):
\begin{equation}\label{eq:red1}
\Delta N(G)=\max_{\overline{J}\subseteq \overline{I}}\, \Big(
 \chi(E_{\overline{J}\cup \cals({\bf i},\overline{J})})-
(E_{\overline{J}\cup \cals({\bf i},\overline{J})}\,,\, x({\bf i})+l'_{[k]})\Big).\end{equation}
 Then, cf. (\ref{eq:NN2}),  we have to find
$\widetilde{\cals}\subseteq \cals\ii$, such that  for any $\overline{J}\subseteq \overline{I}$ one has
  \begin{equation}\label{eq:NN4}
  \chi(E_{\overline{J}\cup \widetilde{\cals}})-
(E_{\overline{J}\cup \widetilde{\cals}}\,,\, x({\bf i})+l'_{[k]})\leq \Delta N(G).
\end{equation}
 Note also that for a reduced cycle $Z$ of $G$ (as $E_{\overline{J}\cup \cals({\bf i},\overline{J})}$
 or $E_{\overline{J}\cup \widetilde{\cals}}$),  $\chi(Z)$ is the number of components of $Z$,
 which sometimes will also be denoted by $\#(Z)$.

\medskip

It is convenient to set the following notation.
For any vertex $j$ and $J\subseteq \calj$ set
$$\sigma_j:=1-(E_j,x({\bf i})+l'_{[k]}) \ \ \mbox{and} \ \
\sigma_j(J):=\sigma_j -(E_j,E_J).$$
By definition of $x({\bf i})$,  one has  $\sigma_j>0$ for any $j\in\scalj$.
Note also  that the information needed in (\ref{eq:red1}) and (\ref{eq:NN4}) about $x({\bf i})+l'_{[k]}$ can
be totally encoded  by the integers $\sigma_j$. This permits to reformulate the statement of the paragraph
\ref{bek:egy} into the following version:

\bekezdes\label{bek:sigma}{\bf Final Reformulation.}
Let $G$ be a connected graph
(e.g. a plumbing graph whose Euler decorations are deleted),
with $\calj=\ocalj\sqcup\scalj$, such that any two vertices of $\ocalj$ are not adjacent, and
with additional decorations
$\{\sigma_j\}_{j\in\calj}$ where $\sigma_j>0$ for $j\in\scalj$.
Fix $\overline{I}\subseteq \ocalj$.  For each $\overline{J}\subseteq \overline{I}$ we
define $\cals(\overline{J})$ as the minimal support in $\scalj$ such that for any $j\in \scalj\setminus \cals(\overline{J})$ one has $\sigma_j(\overline{J}
\cup\cals(\overline{J}))>0$. [Clearly, $\cals(\overline{J})$ corresponds to $\cals({\bf i}, \overline{J})$
in the original version, see also  \ref{prop:G}.]

The `modified' Laufer algorithm to find $\cals(\overline{J})$ (transcribed in the language
of $\sigma_j$'s) is the following. We construct the sequence of supports
$\{s_n\}_{n=0}^t$ by the next principle: $s_0=\emptyset$, and  if $s_n$ is already constructed and there exists some
$j(n)\in \scalj\setminus s_n$ such that
\begin{equation}\label{eq:degger}
\sigma_{j(n)}(\overline{J}\cup s_n)
=\sigma_{j(n)}-(E_{j(n)},E_{\overline{J}\cup s_n})
\leq 0\end{equation}
then take $s_{n+1}:=s_n\cup j(n)$; otherwise stop,
and set $t=n$. [This again follows from the fact that $(E_j, x({\bf i})+E_{\overline{J}\cup s_n}+l'_{[k]})>0$
if and only if  $\sigma_j(\overline{J}\cup s_n)\leq 0$.]
Note that $\cals(\emptyset)=\emptyset$.

Then  the statements  form  \ref{bek:egy} (hence what we need to show) read as follows.


For any $\overline{J}\subseteq \overline{I}$ set
\begin{equation}\label{eq:red2}
\Delta (\overline{J};G) :=
 \#(E_{\overline{J}\cup \cals(\overline{J})})+ \sum_{j\in
\overline{J}\cup \cals(\overline{J})}( \sigma_j-1), \ \  \mbox{and} \ \ \
\Delta N(G)=\max_{\overline{J}\subseteq \overline{I}}\,
\Delta (\overline{J};G).\end{equation}
Then there exists  $\widetilde{\cals}\subseteq \cals(\overline{I})$ which for any $\overline{J}\subseteq \overline{I}$
satisfies
\begin{equation}\label{eq:red3}
 \#(E_{\overline{J}\cup \widetilde{\cals}})+ \sum_{j\in
\overline{J}\cup \widetilde{\cals}}( \sigma_j-1)\leq \Delta N(G).\end{equation}

\medskip

Before we formulate the reductions, we list some additional {\it properties of this setup}.

\bekezdes\label{bek:s_n}
\ {\bf (P1)} \
We analyze how the numerical invariants are modified along the computation sequence
$\{s_n\}_{n=0}^t$ of \ref{bek:sigma}.
Note that if (\ref{eq:degger}) occurs, since $\sigma_{j(n)}>0$, $j(n)$ should be adjacent to $\overline{J}\cup s_n$.
If it is adjacent to only one vertex of $\overline{J}\cup s_n$, then necessarily $\sigma_{j(n)}=1$.
Furthermore,  in any situation,  $\#(E_{\overline{J}\cup s_n})$ is decreasing by
$(E_{j(n)}, E_{\overline{J}\cup s_n})-1$. Therefore, the sequence
$a_n(\overline{J}):=\#(E_{\overline{J}\cup s_n})+
\sum_{j\in \overline{J}\cup s_n}(\sigma_j-1)$ is modified during this step by
$$a_{n+1}(\overline{J})-a_{n}(\overline{J})=
\sigma_{j(n)}-(E_{j(n)}, E_{\overline{J}\cup s_n})\leq 0.$$

\medskip

\noindent {\bf (P2)} \
For any $\overline{J}\subseteq \overline{I}$ and vertex  $j\in \overline{I}\setminus \overline{J}$ one has
\begin{equation*}\label{eq:j_0}
\Delta(\overline{J}\cup j;G)=\Delta(\overline{J};G)+
\sigma_{j}- (E_j,E_{\cals(\overline{J})}).
\end{equation*}
The proof runs as follows. Let $\{s_n\}_{n=0}^t$ be the computation sequence for
$\cals(\overline{J})$. It can be considered as the first part of a sequence for $\cals(\overline{J}\cup j)$
too; let $\{s_n\}_{n=t+1}^{t'}$ be its continuation for $\cals(\overline{J}\cup j)$.
The coefficients $a_n(\overline{J})$ and $a_n(\overline{J}\cup j)$ for $n\leq t$ can be compared. Indeed,
$a_0(\overline{J}\cup j)=a_0(\overline{J})+\sigma_j$, and, similarly as in (P1),
$a_{t}(\overline{J}\cup j)=a_t(\overline{J})+\sigma_j-(E_j,E_{\cals(\overline{J})})$,
which is
the right hand side of the above identity (since $a_t(\overline{J})=\Delta (\overline{J};G)$).

   Next, we show that $a_n(\overline{J}\cup j)$ is constant for any further value $n\geq t$.
First take  $n=t$. Then $\sigma_{j(t)}-(E_{j(t)}, E_{\overline{J}\cup \cals(\overline{J})})>0$
(since $\cals(\overline{J})$ is completed), but
$\sigma_{j(t)}-(E_{j(t)}, E_{\overline{J}\cup \cals(\overline{J})\cup j})\leq 0$
(since $\cals(\overline{J}\cup j)$ is not completed). Hence $(E_j,E_{j(t)})=1$ and (using (P1) too)
$a_{t+1}(\overline{J}\cup j)-a_t(\overline{J}\cup j)=
\sigma_{j(t)}-(E_{j(t)}, E_{\overline{J}\cup \cals(\overline{J})\cup j})= 0$.

In general, set  $s^j_n:=s_n\setminus \cals(\overline{J})$, e.g. $s_t^j=\emptyset$.
 At every step,
by induction, $E_{j\cup s^j_n}$is connected, hence $(E_{j(n)},  E_{j\cup s^j_n})$ can
be at most one (since the graph contains no loops). Hence,
$\sigma_{j(n)}-(E_{j(n)}, E_{\overline{J}\cup \cals(\overline{J})})>0$,
and $\sigma_{j(n)}-(E_{j(n)}, E_{\overline{J}\cup \cals(\overline{J})\cup s^j_n\cup j})\leq 0$ imply
 $(E_{j(n)},  E_{j\cup s^j_n})=1$ and $a_{n+1}(\overline{J}\cup j)=a_n(\overline{J}\cup j)$.

\medskip

\noindent {\bf (P3)}
Fix a vertex $\overline{j}\in\overline{I}$ with $\sigma_{\overline{j}}\geq 1$,  and
assume that  {\it for all  realizations} of  $\Delta N(G) $  as $\Delta(\overline{J},G)$
(as in (\ref{eq:red2})) one has  $\overline{J}\ni \overline{j}$.
Let $G_{-1}$ be the graph obtained from $G$ by replacing the decoration $\sigma_{\overline{j}}$ by
$\sigma_{\overline{j}}-1$. We claim that
\begin{equation}\label{eq:-1}
\Delta N(G_{-1})=\Delta N(G)-1.
\end{equation}
Indeed,  since $\{\sigma_j\}_{j\in\calj^*}$ is unmodified,
the support $\cals(\overline{J})$
for any $\overline{J}$  is the same determined in $G_{-1}$ or in $G$.
If $\overline{J}\not \ni \overline{j}$ then $\Delta(\overline{J},G_{-1})=\Delta (\overline{J},G)$ by (\ref{eq:red2}),
hence $\Delta (\overline{J},G_{-1}) <\Delta N(G)$.
If $\overline{J}\ni \overline{j}$ then
$\Delta (\overline{J},G_{-1})=\Delta (\overline{J},G)-1$ by the same
(\ref{eq:red2}). Since one such $\overline{J}$ realizes $\Delta N(G)$, the claim follows.

 \bekezdes\label{bek:ibar} {\bf First Reduction: $\overline{I}=\ocalj$.} \
Consider $\ocalj\setminus \overline{I}=\overline{I}^c$ and the graph $G\setminus \overline{I}^c$ obtained from the
original graph $G$ by deleting the vertices $\overline{I}^c$ and adjacent edges.
The connected components of  $G\setminus \overline{I}^c$
do not interact  from the point of view of the statement of the above theorem. Indeed, the Laufer algorithm does not
propagate along the bad vertices  $\overline{I}^c$, and it is also enough to find supports  $\widetilde{\cals}$ for each component independently. Hence, {\it we may assume that $\overline{I}=\ocalj$}.

\bekezdes
{\bf Second Reduction: $\sigma_j>0$ for any $j$.} \
Consider the situation from \ref{bek:sigma} with $\overline{I}=\ocalj$, cf. \ref{bek:ibar}.
Assume that $\sigma_j\leq 0$ for some $j\in \overline{I}=\ocalj$, and consider the graph $G\setminus j$ obtained from
$G$ by deleting the vertex $j$ and its adjacent edges. Note the following facts:

$\bullet$ \ The maximum $\Delta N(G)$ in (\ref{eq:red2})
can be realized by a subset $\overline{J}$ which does not contain $j$.
In fact, for any $\overline{J}$ with $j\not\in\overline{J}$ one has
$\Delta(\overline{J}\cup j;G)\leq \Delta(\overline{J};G)$. Indeed, using the notations from \ref{bek:s_n},
$a_0(\overline{J}\cup j)\leq a_0(\overline{J})$; the sequence $s_n$ associated with
 $\overline {J}$ is good as the beginning of the sequence of $\overline{J}\cup j$, and during this inductive steps
 $a_n(\overline{J}\cup j)$ drops more than $a_n(\overline{J})$; and finally, if the sequence of
 $\overline{J}\cup j$ is longer, then its $a_{n}$--values decrease
 even more (cf. \ref{bek:s_n}).

$\bullet$ \  All the supports of type
 $\cals(\overline{J})$ definitely are included  in $G\setminus j$ (since are subsets of
${\mathcal J}^*$).

$\bullet$ \ If we find for each component of $G\setminus j$ some $\widetilde{\cals}$ satisfying the statements of the theorem for that component, then their union solves the problem for $G$ as well.

Therefore, having $G$ with some $\sigma_j\leq 0$, we can delete $j$ and continue to search for $\widetilde{\cals}$
for $G\setminus j$: that support will work for $G$ as well.

If we delete all vertices with
$\sigma_j\leq 0$ ($j\in \overline{I}$) then
 we arrive to a situation when $\sigma_j>0$ for any $j\in\overline{I}$, hence,
a posteriori, $\sigma_j>0$ for any $j\in\calj$.

\medskip

{\it Note that the wished reformulated statement from \ref{bek:sigma}, even for all $\sigma_j=1$, when the problem depends purely on the shape of the graph, is far to be trivial.}

\bekezdes\label{bek:G^-} {\bf Third Reduction: $G=G^-$}. \ Assume $\overline{I}=\ocalj$,
 cf. \ref{bek:ibar}.
 Let $G^-$ be the minimal connected subgraph of $G$ generated  by the vertices $\overline{I}$.
Here the vertices $\calj(G^-)$ have an induced  disjoint decomposition into $\ocalj(G^-)=\ocalj$ and
$\scalj(G^-)=\calj(G^-)\cap \scalj$.
Moreover, each connected component of $G\setminus G^-$ is glued to $G^-$
via a unique $j\in \overline{I}$.

We claim that a solution $\widetilde{\cals}$ for $G^-$ provides a solution for $G$ too.
Indeed, for any $\overline{J}\subseteq \overline{I}$, the supports  $\cals_G(\overline{J})$ and
$\cals_{G^-}(\overline{J})$ generated in $G$, respectively in $G^-$ satisfy the following.

 $\overline{J}\cup\cals_G(\overline{J})$
can be obtained from $\overline{J}\cup \cals_{G^-}(\overline{J})$ by gluing some subtrees of
$G\setminus G^-$ along some elements of $\overline{J}$. These subtrees are maximal among
those connected subgraphs of $G$ (supported in $\calj^*\setminus \calj(G^-)$)
with all $\sigma_j=1$ and adjacent to $G^-$. In particular,
 $\overline{J}\cup \cals_{G^-}(\overline{J})\subseteq \overline{J}\cup\cals_G(\overline{J})$,
and their topological realizations are homotopy equivalent;
$\sigma_j=1$ for any $j\in  \cals_{G}(\overline{J})\setminus  \cals_{G^-}(\overline{J})$;
and the integers $\#E_{\overline{J}\cup \cals(\overline{J})}$
computed for $G$ and $G^-$ are the same.

Therefore,  $\Delta N(G)=\Delta N(G^-)$, and
a solution $\widetilde{\cals}$ for $G^-$ is a solution for $G$ too.

Hence, {\it we can assume that $G=G^-$}.
\bigskip

This ends the possible reductions/preparations  and we start the inductive argument.

\bekezdes {\bf The induction.}
The proof is based on inductive argument over $\sigma_j$-- decorated  graphs
(with  $\overline{I}=\ocalj$, $\sigma_j>0$ and $G=G^-$), where we will consider subgraphs (with
induced decorations $\sigma_j$), and eventually we will decrease the decorations $\{\sigma_j\}_{j\in\ocalj}$.

If $\overline{I}$ is empty then $\Delta N(G)=0$;  if
$\overline{I}$  contains exactly one element $j_0$,  then by (\ref{bek:G^-}) $G=\{j_0\}$ and
by (\ref{eq:red2}) $\Delta N(G)=\sigma_{j_0}$. In both cases
$\widetilde{\cals}=\emptyset$ answers the problem.
\bekezdes
The inductive step is based on the following picture.
Recall that $G$ agrees with the smallest connected subgraph generated by
$\ocalj$. Let $j_0\in\ocalj$ be one of its end--vertices (that is, a vertex which has only one adjacent vertex
in $G$). Denote that connected component of $G\setminus \ocalj$ which is adjacent to $j_0$ by $G^*_0$.

If $G\setminus \ocalj=G^*_0$ then all the vertices from $\ocalj$ are adjacent to $G^*_0$ and
$\ocalj=\overline{I}$ is {\it exactly}  the set of end--vertices of $G$. Then one verifies
 (use \ref{bek:s_n}(P2)) that

$\bullet$ \   $\Delta(\overline{J};G) $ is increasing function in $\overline{J}$, hence
  $\Delta N(G)=\Delta(\overline{I},G)$,  and

$\bullet$ \   $\#(E_{\overline{J}\cup \cals(\overline{I})})=
\#(E_{\overline{I}\cup \cals(\overline{I})})$, hence (\ref{eq:red3}) holds for $\widetilde{\cals}=
\cals(\overline{I})$.

\medskip

Next, assume that $G\setminus \ocalj\not=G^*_0$. We may also assume (by a good choice of
$j_0$) that there is {\it only one} vertex $\overline{j}$ of $\ocalj$ which is simultaneously adjacent to
$G^*_0$ and to some other component of  $G\setminus \ocalj$. Let $\{j_0,j_1,\ldots, j_k,\overline{j}\}$
be the elements of $\ocalj$ which are adjacent to $G^*_0$. Then $j_0,j_1,\ldots, j_k$ are end--vertices of
$G$.
%
Let $G'$ be  obtained from $G$ be deleting $G^*_0$, $\{j_0,j_1,\ldots, j_k\}$
and all their adjacent edges.
Figure 3.2 is the schematic picture of $G$, where the vertices from $\scalj$ are not emphasized.

\begin{figure}[h!]
\begin{picture}(390,100)(10,-10)
\put(30,10){\framebox(65,30)}
\put(80,35){\line(2,-1){20}}
\put(80,15){\line(2,1){20}}
\put(83,28){\makebox(0,0){$\vdots$}}\put(190,45){\makebox(0,0){$\ldots$}}
\put(100,25){\circle*{4}}
\put(100,25){\line(1,0){200}}
 \put(300,25){\circle*{4}}
\put(150,25){\line(0,1){20}}
\put(150,45){\line(1,1){20}}
\put(150,45){\line(-1,1){20}}
\put(170,65){\circle*{4}} \put(130,65){\circle*{4}}
\put(250,25){\line(0,1){40}} \put(250,25){\line(-1,1){40}}
 \put(250,45){\line(-1,1){20}}
 \put(250,65){\circle*{4}} \put(230,65){\circle*{4}} \put(210,65){\circle*{4}}
\put(100,33){\makebox(0,0){\tiny{$\overline{j}$}}}
\put(300,33){\makebox(0,0){\tiny{$j_0$}}}
\put(130,72){\makebox(0,0){\tiny{$j_k$}}}
\put(250,72){\makebox(0,0){\tiny{$j_1$}}}
\put(230,72){\makebox(0,0){\tiny{$j_2$}}}
\put(190,72){\makebox(0,0){$\ldots$}}
\put(285,50){\makebox(0,0){\small{$G^*_0$}}}
\put(35,50){\makebox(0,0){\small{${G}'$}}}
\put(340,25){\circle*{4}}
\put(385,27){\makebox(0,0){\small{\mbox{= elements of $\ocalj$}}}}
\dashline[3]{3}(108,5)(295,5)\dashline[3]{3}(108,5)(108,60)
\dashline[3]{3}(108,60)(295,60)\dashline[3]{3}(295,5)(295,60)
\dashline[3]{3}(25,5)(104,5)\dashline[3]{3}(104,5)(104,60)
\dashline[3]{3}(25,60)(104,60)\dashline[3]{3}(25,5)(25,60)
\end{picture}
\caption{The inductive step}
\end{figure}
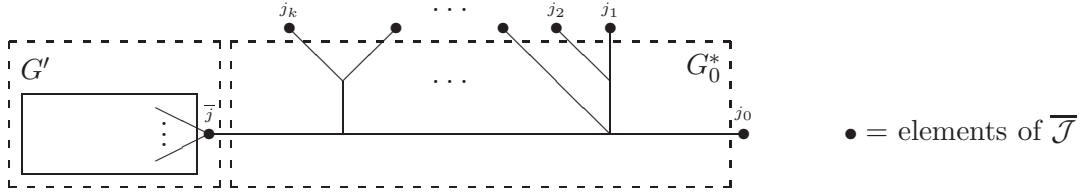

%

The inductive step splits in several cases ({\bf A} and {\bf B}, {\bf A} splits
into {\bf I} and {\bf II},
while {\bf I} has two subcases {\bf I.a} and {\bf I.b}).

\bekezdes\label{bek:j_02} {\bf A.} {\it \underline{ Assume that $\Delta N(G)$ in (\ref{eq:red2}) can be
realized by some $\overline{J}$
with  $j_0\not\in \overline{J}$.} }

\medskip
Fix such a $\overline{J}$.
Since $\sigma_{j_0}\geq 1$ and 
$\Delta(\overline{J},G)\geq\Delta(\overline{J}\cup j_0,G)$, from \ref{bek:s_n}(P2) one gets
\begin{equation}\label{eq:delta2}
\mbox{$\sigma_{j_0}=1$ and  $j_0$ is adjacent to a vertex  of $\cals(\overline{J})$}.
\end{equation}
Assume that some $j_\ell$ ($1\leq \ell\leq k$) is not in $\overline{J}$. Then again by
$\Delta(\overline{J}\cup j_\ell,G)\leq \Delta(\overline{J},G)$ and \ref{bek:s_n}(P2) we get
that $\sigma_{j_\ell}=1$ and $j_\ell$ is adjacent to $\cals(\overline{J})$.
In particular, $\Delta(\overline{J}\cup j_\ell,G)=\Delta(\overline{J},G)$, and we can replace
$\overline{J}$ by $\overline{J}\cup j_\ell$. Hence, for uniform treatment, in such a situation
 we can always assume that
\begin{equation}\label{eq:jell}
\{j_1,\ldots, j_k\}\subseteq \overline{J}.
\end{equation}
Let $\cals^*_0$ be the support generated by $\{j_1,\ldots,j_k\}$ via the (reformulated)
Laufer algorithm   \ref{bek:sigma};
then  $\cals^*_0\subseteq G^*_0$.

We will need another fact  too. Let $\overline{J}'$ be a subset of $\ocalj(G')$. Then
\begin{equation}\label{eq:''}
\Delta(\overline{J}',G')=\Delta(\overline{J}',G),
\end{equation}
that is, the $\Delta$--invariants of $\overline{J}'$
 in $G'$ and in $G$ are the same. Indeed, if $\overline{j}\not\in \overline{J}'$,
then the identity  is clear since $\overline{J}'$  generates the same supports
 $\cals(\overline{J}',G')=\cals(\overline{J}',G)$ in $G'$ and $G$. Otherwise,
  $\cals(\overline{J}',G)$ is the union of $\cals(\overline{J}',G')$
with the  maximal element of those
connected subgraph of $G^*_0$ which are  adjacent to $\overline{j}$ and  $\sigma_j=1$ for all
 their vertices $j$.

Now, our discussion bifurcates into two cases: {\it
whether $\overline{j}$ \, is adjacent to $\cals^*_0$ or not.}

\medskip

\noindent
{\it {\bf I.} \underline{The case when  $\overline{j}$ is not adjacent to $\cals^*_0$.} }

\medskip

We start with the following general statement, valid for any  $\overline{J}\subseteq \overline{I}$,
which does not contain $j_0$ but it contains $\{j_1,\ldots, j_k\}$. For such $\overline{J}$,
whenever $\overline{j}$ is not adjacent to $\cals^*_0$  one has:
\begin{equation}\label{eq:Delta}
\Delta(\overline{J},G)=\Delta(\{j_1,\ldots, j_k\},G)+\Delta(\overline{J}\cap G',G),
\end{equation}
where $\overline{J}\cap G'$ stands for $\overline{J}\cap \calj(G')$.
For its proof run first the Laufer algorithm for the vertices $ \{j_1,\ldots, j_k\}$ getting
$\cals^*_0$, then add the remaining vertices from $\overline{J}\cap G'$ and continue the
algorithm.

Therefore, for any $\overline{J}$ as in the assumption \ref{bek:j_02} (and with (\ref{eq:jell}))
we get that $\overline{J}\cap G'$ realizes $\Delta N(G')$. (Otherwise, we would be able to replace the subset
$\overline{J}\cap G'$ of $\overline{J}$ by another subset of $\ocalj\cap G'$ which would give larger
$\Delta(\overline{J}\cap G',G')=\Delta(\overline{J}\cap G',G)$, cf. also with (\ref{eq:''}), which would contradict
(\ref{eq:Delta}).)
Hence, (\ref{eq:Delta})  combined with (\ref{eq:''})  give:
\begin{equation*}\label{eq:DeltaN}
\Delta N(G)=\Delta N(G')+\Delta(\{j_1,\ldots, j_k\},G).
\end{equation*}

\noindent
{\bf I.a.} {\it \underline{Assume that $\Delta N(G') $ can be  realized by some $\overline{J}'$ in $G'$
which does not contain $\overline{j}$.}}\medskip

Then, we can apply the above statements for $\overline{J}=\overline{J}'\cup\{j_1,\ldots, j_k\}$.
Note that the Laufer algorithm runs in two independent regions cut by $\overline{j}$,
namely in $G^*_0$ and in  $G'\setminus \overline{j}$. Hence (\ref{eq:delta2}) guarantees that
$j_0$ is adjacent to $\cals^*_0$.

Furthermore, if $\widetilde{\cals}(G')$ is a support answering the problem
for $G'$, then $\widetilde{\cals}=\widetilde{\cals}(G') \cup \cals^*_0$ is a solution for $G$.
Note also that in this case $\cals^*_0$ coincides with the collection of components of
$\cals(\overline{I})$ sitting in $G^*_0$.

\medskip

\noindent
{\bf I.b.} {\it \underline{Assume that all realizations of  $\Delta N(G') $  by some $\overline{J}'$ in $G'$
 contain $\overline{j}$.}}\medskip

Let $G'_{-1}$ be the graph obtained from $G'$ by replacing the decoration $\sigma_{\overline{j}}$ by
$\sigma_{\overline{j}}-1$. Then,  by \ref{bek:s_n}(P3),  we get
\begin{equation*}
\Delta N(G'_{-1})=\Delta N(G')-1.
\end{equation*}
By induction, one can find a support  $\widetilde{\cals}(G'_{-1})$ which solves the problem for  $G'_{-1}$.
Let $\calst$ be the connected (minimal) string in $G^*_0$ adjacent to both  $\overline{j}$ and $j_0$ (connecting them).

If $j_0$ is adjacent to $\cals^*_0$ then $\widetilde{\cals}= \widetilde{\cals}(G'_{-1})\cup \cals^*_0$
is a solution for $G$.

Otherwise  $\widetilde{\cals}= \widetilde{\cals}(G'_{-1})\cup \cals^*_0\cup \calst$
is a solution for $G$.

\medskip

\noindent
{\it {\bf II.} \underline{ The case when  $\overline{j}$ is adjacent to $\cals^*_0$. }}

\medskip

Note  that in this case by the combinatorics of the choice of $j_0$ and by (\ref{eq:delta2})
we get that $j_0$ is adjacent to $\cals^*_0$ too.

We claim that for $G'_{-1}$ associated with the graph $G'$ and its vertex $\overline{j}$ one gets
\begin{equation*}\label{eq:DeltaN2}
\Delta N(G)=\Delta N(G'_{-1})+\Delta(\{j_1,\ldots, j_k\},G).
\end{equation*}
Moreover, $\widetilde{\cals}= \widetilde{\cals}(G'_{-1})\cup \cals^*_0$
is a solution for $G$.

\bekezdes\label{bek:j_03} {\bf B.} {\it \underline{Assume that for all  realizations of
$\Delta N(G)$ as $\Delta(\overline{J},G)$ one has
  $j_0\in \overline{J}$. }}

\medskip

Replace in $G$ the decoration $\sigma_{j_0}$ by  $\sigma_{j_0}-1$, find a solution for
$G_{-1}$, then that solution works for $G$ too.
\end{proof}

This ends the proof of Proposition \ref{th:nonempty}. We continue with the contraction part.

\bekezdes{\bf Additional properties of $\widetilde{\cals}$.}\label{ss:tilde}

 Fix an integer $N$ and $\ii\in \overline{S}_N$ as in subsection \ref{ss:nonempty}.
 The cube $\ii$ determines the integer
  $N(G)=\max_{\overline{J}\subseteq \overline{I}}\chi_k(x({\bf i}+1_{\overline{J}}))$, cf.
 (\ref{eq:NN3}).
  Choose $\widetilde{J}\subseteq \overline{I}$ which realizes this maximum:
$N(G)=\chi_k(x({\bf i}+1_{\widetilde{J}}))$.   $N(G)$  is the smallest integer $N$
 for which $\ii\in \overline{S}_N$.

 Theorem \ref{th:nonempty} applied for $\ii$ and $N=N(G)$
 provides a cycle $x({\bf i})+E_{\widetilde{\cals}}$ with
 $\widetilde{\cals}\subseteq \cals\ii $  and
 \begin{equation}\label{eq:ssng}
 \chi_k(x({\bf i})+E_{\widetilde{\cals}}+E_{\overline{J}})\leq N(G) \ \ \mbox{for any $\overline{J}
 \subseteq \overline{I}$}.
 \end{equation}
In the next paragraphs we will list some additional properties of  $\widetilde{\cals}$ and $\widetilde{J}$.
\begin{lemma}\label{lem:tildeS} (a) \ $\chi_k(x({\bf i}) +E_{\widetilde{\cals}}+E_{\widetilde{J}})=N(G)$.
In particular, the weight of the cube $(x({\bf i})+E_{\widetilde{\cals}},\overline{I})$ is $N(G)$.

(b) (i) There exists a computation sequence $\{x_n\}_{n=0}^t$ with
$x_0=x({\bf i})+E_{\widetilde{J}}$ and $x_t=x({\bf i})+E_{\widetilde{J}}+E_{\widetilde{\cals}}$ such that
$\chi_k(x_{n+1})\leq \chi_k(x_n)$ for any $n$.

(ii) There exists a computation sequence $\{y_n\}_{n=0}^{t'}$ with
$y_0=x({\bf i})+E_{\widetilde{J}}+E_{\widetilde{\cals}}$ and $y_{t'}=
x({\bf i})+E_{\widetilde{J}}+E_{\cals\ii}$ such that
$\chi_k(x_{n+1})\geq \chi_k(x_n)$ for any $n$.

(c) Using the notation  $\sigma_j(J)$ from \ref{bek:sigma}, one has:
$$\begin{array}{lll} (i) & \sigma_j(\widetilde{\cals})\geq 0 & \ \mbox{if \ $j\in \widetilde{J}$}\\
(ii) & \sigma_j(\widetilde{\cals}
)\leq 0 & \ \mbox{if \ $j\not\in \widetilde{J}$}.\end{array}$$

\end{lemma}
\begin{proof} Note that
$$N(G)\stackrel{(1)}{=}
\chi_k(x({\bf i} +1_{\widetilde{J}}))\stackrel{(2)}{\leq}
\chi_k(x({\bf i}) +E_{\widetilde{\cals}}+E_{\widetilde{J}})
\stackrel{(3)}{\leq }N(G).$$
(1) follows from the definition of $N(G)$ and the choice of $\widetilde{J}$,
(2) from Lemma \ref{lem:laufb}, and (3) from Theorem \ref{th:nonempty} applied for $N=N(G)$. This proves  (a).
Identity (a) together with Proposition \ref{prop:G}(II) imply that $\widetilde{\cals}\subseteq
\cals({\bf i},\widetilde{J})$. Then there exists a  computation sequence connecting $x({\bf i})+E_{\widetilde{J}}$
with $x({\bf i})+E_{\widetilde{J}}+E_{\widetilde{\cals}}$  by  \ref{prop:G}(II), a sequence
connecting $x({\bf i})+E_{\widetilde{J}}+E_{\widetilde{\cals}}$  with
$x({\bf i})+E_{\widetilde{J}}+E_{\cals({\bf i},\widetilde{J})}$ by \ref{prop:G}(I), and finally, from
$x({\bf i})+E_{\widetilde{J}}+E_{\cals({\bf i},\widetilde{J})}$ to
$x({\bf i})+E_{\widetilde{J}}+E_{\cals({\bf i},\overline{I})}$ by \ref{prop:G}(III). This ends part (b).

Part (c) follows from (a) and equation (\ref{eq:ssng}) applied for $\widetilde{J}\setminus \{j\}$
(case $j\in \widetilde{J}$), respectively $\widetilde{J}\cup \{j\}$
(case $j\not\in \widetilde{J}$), and from the assumption \ref{ss:assumption}, which guarantees   $(E_j,E_{\widetilde{J}\setminus \{j\}})=0$.
\end{proof}

\bekezdes Let us recall what we already proved. For any fixed $\ii\in\overline{S}_N$ the space
 $\phi^{*}_N\ii$ is non--empty, cf. \ref{th:nonempty}, and it has the homotopy type of
 the product (cf. \ref{cor:con}):
 $$\Phi^{*}_N\ii=
 \psi(\phi^*_N\ii)\cap \{l^*\,:\, 0\leq l^*-\psi(x({\bf i}))\leq  E_{\cals\ii}\}
 \times \ii.$$
If $x\in \Phi^{*}_N\ii$ then  $x-x({\bf i})$ is reduced.  Moreover, $\Phi^{*}_N\ii$ has in it a distinguished $|I|$--dimensional
cube  $\{\psi(x({\bf i}))+E_{\widetilde{\cals}}\}\times \ii=(x({\bf i})+E_{\widetilde{\cals}},\overline{I})$.
Our goal is to construct a deformation retract from  $\Phi^{*}_N\ii$ to this cube (acting in the fiber direction).
This will be more complicated than the `standard' retractions
\ref{lem:con1}--\ref{lem:con2}--\ref{lem:con3}.
(Note that the point $x({\bf i})+E_{\widetilde{\cals}}+E_{\widetilde{J}}$
is not a $\chi_k$--minimal point of $\Phi^{*}_N\ii$,
it is maximal point in the direction $\ocalj$ and a minimal point in the direction $\calj^*$.)

To start with, we consider the connected components $\{G_\alpha\}_{\alpha\in A}$ of $\widetilde{\cals}$, and
the connected components $\{C_\beta\}_{\beta\in B}$ of $\cals\ii\setminus \widetilde{\cals}$.
During the contraction the supports $G_\alpha$ should be `added' and the supports $C_\beta$ should be `deleted'.
According to this, it is performed in several steps, during one step either we add one $G_\alpha$--type component, or we
delete one $C_\beta$--type component. At each step the fact that
which type is performed, or which $G_\alpha/C_\beta$ is manipulated is
decided by a technical `selection procedure'. This is the subject of the next Proposition, which will be applied at any
situation when the components $\{G_\alpha\}_{\alpha\in A'}$ still should be added and the
components $\{C_\beta\}_{\beta\in B'}$ still should be deleted: it chooses an element of $A'\cup B'$.
 The technical properties
associated with  the corresponding cases  will guarantee that the contraction stays below level $N$  of $\chi_k$.

Below, for any subset ${\calj'}\subseteq \ocalj$ and $i\in\calj^*$ we write $\calj'_i:=\{j\in \calj'\,:\, (E_i,E_j)=1\}$.
\begin{proposition}[Selection Procedure]\label{prop:sel}
Fix subsets $A'\subseteq A$ and $B'\subseteq B$ such that $A'\cup B'\not=\emptyset$.
Then either there exists $\alpha\in A'$ such that
$$(i) \ \ \ \mbox{for every  $i\in |G_\alpha|$ and every $j\in \widetilde{J}_i$ one has
$\sigma_j(({\widetilde{\cals}}\setminus i )\cup \cup_{\beta\in B'}C_\beta ) >0$}\hspace{8mm}$$
or, there exists $\beta\in B'$ such that
$$(ii)  \ \ \mbox{for every  $i\in |C_\beta|$ and every
$j\in \overline{I}_i\setminus  \widetilde{J}$ one has
$\sigma_j(({\widetilde{\cals}}\cup i)\setminus \cup_{\alpha\in A'}G_\alpha)< 0$}.$$
\end{proposition}
\begin{proof}
Fix some $\alpha\in A'$ and assume that it does not satisfy (i). Then there exists $i_\alpha \in |G_\alpha|$
and $j_\alpha\in \widetilde{J}_{i_\alpha}$ such that
$\sigma_{j_\alpha}(({\widetilde{\cals}}\setminus i )\cup \cup_{\beta\in B'}C_\beta ) \leq 0$.
Note that $\sigma_{j_\alpha}({\widetilde{\cals}}\setminus i )=
\sigma_{j_\alpha}({\widetilde{\cals}} )+(E_{j_\alpha},E_{i_\alpha})>0$ by \ref{lem:tildeS}(c). These two combined
prove the existence of some $\beta\in B'$ and $i_\beta\in|C_\beta|$ with $(E_{j_\alpha}, E_{i_\beta})=1$.

Symmetrically, if for some $\beta\in B'$ (ii) is not true, then there exists $i_\beta\in |C_\beta|$ and
$j_\beta\in \overline{I}_{i_\beta}\setminus \widetilde{J}$ with
$\sigma_{j_\beta}(({\widetilde{\cals}}\cup i_\beta)\setminus \cup_{\alpha\in A'}G_\alpha)\geq  0$.
Since by \ref{lem:tildeS}(c) we have $\sigma_{j_\beta}({\widetilde{\cals}}\cup i_\beta)=
\sigma_{j_\beta}({\widetilde{\cals}})-(E_{j_\beta},E_{i_\beta})<0$, we get the existence of some
$\alpha\in A'$ and $i_\alpha\in |G_\alpha|$ with $(E_{j_\beta}, E_{i_\alpha})=1$.

Now the proof runs as follows. Start with any $\alpha\in A'$. If it satisfy (i) we are done. Otherwise,
as in the first paragraph,
we get  a $\beta$, such that $G_\alpha$ and $C_\beta$ are connected by a length two  path
having the middle vertex in $\widetilde{J}$.
If this $\beta $ satisfy (ii) we stop, otherwise we get by the second paragraph an $\alpha'$
such that $C_\beta$ and  $G_{\alpha'}$ are connected by a length two path whose middle vertex is not in
$\widetilde{J}$. Since the graph $G$  has no cycles,
$\alpha'\not=\alpha$. Then we continue the procedure with $\alpha'$. Either it satisfies (i) or
$G_{\alpha'}$ is connected with some $C_{\beta'}$ with $\beta'\not=\beta$. Continuing in this way, all the involved
 $\alpha$ indices, respectively  all the  $\beta$ indices are pairwise distinct because of the non--existence of
 a cycle in the graph. Since $A'\cup B'$ is finite, the procedure must stop.
\end{proof}

\bekezdes{\bf Contraction of  $\Phi^{*}_N\ii$.}\labelpar{ss:cont} 

We will drop the symbol $\ii$ from the notation
$\Phi^*_N\ii$: we  write simply $\Phi_N^*$. On the other hand, for any pair
$\emptyset\subseteq \cals_1\subseteq \cals_2\subseteq \cals\ii$, we define
$$\Phi^*_N(\cals_1,\cals_2):=
[\psi(\phi^*_N\ii)\cap
 \{l^*\,:\, E_{\cals_1}\leq l^*-\psi(x({\bf i}))\leq  E_{\cals_2}\}]
 \times \ii.$$
For example, $\Phi^*_N(\emptyset,\cals\ii)=\Phi^*_N$, while
$\Phi^*_N(\widetilde{\cals},\widetilde{\cals})=
\{(\psi(x({\bf i}))+E_{\widetilde{\cals}})\}\times \ii$, the cube on which we wish to contract $\Phi^*_N$.

If the Selection Procedure chooses some $\alpha'\in A'$ then we have to construct a deformation retract
$$c_{\alpha'}:\Phi^*_N( \bigcup_{\alpha\not\in A'}|G_\alpha|\,,\,
 \widetilde{\cals}\cup \bigcup_{\beta\in B'}|C_\beta|) \longrightarrow
 \Phi^*_N( \bigcup_{\alpha\not\in A'\setminus \alpha'}|G_\alpha|\,,\,
 \widetilde{\cals}\cup \bigcup_{\beta\in B'}|C_\beta|).$$
Otherwise, if some   $\beta'\in B'$ is chosen then we have to construct a deformation retract
$$c_{\beta'}:\Phi^*_N(\bigcup_{\alpha\not\in A'}|G_\alpha|\,,\,
 \widetilde{\cals}\cup \bigcup_{\beta\in B'}|C_\beta|) \longrightarrow
 \Phi^*_N( \bigcup_{\alpha\not\in A'}|G_\alpha|\,,\,
 \widetilde{\cals}\cup \bigcup_{\beta\in B'\setminus \beta'}|C_\beta|).$$
Their composition (in the selected order)  provides the wished deformation retract
$\Phi^*_N\to \Phi^*_N(\widetilde{\cals},\widetilde{\cals})$.
The two types of contractions have some asymmetries,
hence we will provide the details for both of them.

\bekezdes {\bf The construction of $c_{\alpha'}$. } Let $|G_{\alpha'}|=\{j_1,\ldots , j_t\}$.
By the properties of $\widetilde{J}$, cf. \ref{lem:tildeS}(b),
we have a computation sequence with $\chi_k$ non--increasing from
$x({\bf i})+E_{\widetilde{J}}$ to $x({\bf i})+E_{\widetilde{J}\cup\widetilde{\cals}}$.
Since the components $\{G_\alpha\}_\alpha$ do not interact, we can permute elements
belonging to different components $G_\alpha$, hence we may assume that the first part
completed the components $\cup_{\alpha\not\in A'}G_{\alpha}$, then we complete $G_{\alpha'}$
and the order $\{j_1,\ldots , j_t\}$ is imposed by the computation sequence. Therefore, for any $1\leq n\leq t$,
\begin{equation}\label{eq:CS}
\sigma_{j_n}(\widetilde{J}\cup \cup_{\alpha\not \in A'}|G_\alpha|\cup\{j_1,\ldots,j_{n-1}\})\leq 0.
\end{equation}

The contraction $c_{\alpha'}$ will be a composition $c_{\alpha',t}\circ \cdots \circ c_{\alpha',1}$, where
$c_{\alpha',n}$ corresponds to the completion of the cycles with $E_{j_n}$ ($1\leq n\leq t$):
\begin{align*}c_{\alpha',n}:\Phi^*_N( \bigcup_{\alpha\not \in A'}|G_\alpha|\cup\{j_1,\ldots,j_{n-1}\}\,,\,
& \widetilde{\cals}\cup \bigcup_{\beta\in B'}|C_\beta|) \longrightarrow \\
& \Phi^*_N(\bigcup_{\alpha\not \in A'}|G_\alpha|\cup\{j_1,\ldots,j_n\}
 \,,\,
 \widetilde{\cals}\cup \bigcup_{\beta\in B'}|C_\beta|)
 \end{align*}
defined as follows. Write $x=x({\bf i})+E_{\overline{J}}+l^*$ ($l^*$ is reduced) with
\begin{equation}\label{eq:SUPL}
\cup_{\alpha\not \in A'}|G_\alpha|\cup\{j_1,\ldots,j_{n-1}\}\subseteq
|l^*|\subseteq  \widetilde{\cals}\cup \bigcup_{\beta\in B'}|C_\beta|.\end{equation}
Then
$$c_{\alpha',n}(x)=\left\{\begin{array}{ll}
x & \ \mbox{if $ j_n\in |l^*|$},\\
x+E_{j_n} & \ \mbox{if $ j_n\not\in |l^*|$}.\end{array}
\right.$$
Note that for any $l^*$ as above with $|l^*|\not\ni j_n$, the inequality (\ref{eq:CS}) implies
\begin{equation}\label{eq:CS2}
\sigma_{j_n}(\widetilde{J}\cup |l^*|)\leq 0.
\end{equation}
Fix such an $l^*$ with $|l^*|\not \ni j_n$. Then,
for {\it  any} $\overline{J}\subseteq \overline {I}$, we have to prove
\begin{equation}\label{eq:NNN}
\chi_k(x({\bf i})+E_{\overline{J}}+l^*+E_{j_n})\leq N.
\end{equation}
Set $\overline{J}(l^*):=\{j\in\overline{I}\,:\, \sigma_j(|l^*|)>0\}$.
We claim that if (\ref{eq:NNN}) is valid for $\overline{J}(l^*)$ then it is valid for every 
$\overline{J}\subseteq
\overline{I}$. This follows from the next identity whose second term is
$\leq 0$ by the definition of $\overline{J}(l^*)$.
\begin{equation}\label{eq:MMM}\begin{split}
&\chi_k(x({\bf i})+E_{\overline{J}}+l^*+E_{j_n})-
\chi_k(x({\bf i})+E_{\overline{J}(l^*)}+l^*+E_{j_n})\\
&=\sum _{j\in \overline{J}\setminus \overline{J}(l^*)}\big[\sigma_j(|l^*|)-(E_j,E_{j_n})\big]
-
\sum _{j\in \overline{J}(l^*)\setminus \overline{J}}\big[\sigma_j(|l^*|)-(E_j,E_{j_n})\big].
\end{split}\end{equation}
On the other hand, using  Selection Procedure (and its notations) we get
$\widetilde{J}_{j_n}\subseteq \overline{J}(l^*)$. Indeed, by the choice of $\alpha'$ in \ref{prop:sel}(i),
for $j_n\in|G_{\alpha'}|$ and for any $j\in \widetilde{J}_{j_n}$ one has
$\sigma_j(\widetilde{\cals}\setminus j_n\cup\cup_{\beta\in B'}C_\beta)>0$.
Then
$\sigma_j(|l^*|)>0$ by the support condition (\ref{eq:SUPL}).
Then $\widetilde{J}_{j_n}\subseteq \overline{J}(l^*)$ implies:
\begin{equation}\label{eq:SIGIN}
\sigma_{j_n}(\overline{J}(l^*)\cup |l^*|)\stackrel{(1)}{\leq} \sigma_{j_n}(\widetilde{J}_{j_n}\cup |l^*|)
\stackrel{(2)}{=}
\sigma_{j_n}(\widetilde{J}\cup |l^*|)\stackrel{(3)}{\leq} 0.
\end{equation}
(1) follows from $\widetilde{J}_{j_n}\subseteq \overline{J}(l^*)$, (2) from
$(E_{j_n},E_{\widetilde{J}_{j_n}})=(E_{j_n},E_{\widetilde{J}})$, and (3)  from (\ref{eq:CS2}).
Therefore,
$$\chi_k(x({\bf i})+E_{\overline{J}(l^*)}+l^*+E_{j_n})-
\chi_k(x({\bf i})+E_{\overline{J}(l^*)}+l^*)=\sigma_{j_n}(\overline{J}(l^*)\cup |l^*|)\leq 0.$$
Since $\chi_k(x({\bf i})+E_{\overline{J}(l^*)}+l^*)\leq N$ (by induction),
(\ref{eq:NNN}) is valid for $\overline{J}(l^*)$.

\bekezdes {\bf The construction of $c_{\beta'}$ }. Let $|C_{\beta'}|=V_1 \cup V_2$, where $V_1:=|C_{\beta'}|
\cap (\cals(\vasi,\widetilde{J})\setminus \widetilde\cals)$ and $V_2:=|C_{\beta'}|\cap (\cals\ii\setminus 
\cals(\vasi,\widetilde{J}))$. The Laufer computation sequence given by \ref{prop:G}(I) connecting 
$x(\vasi)+E_{\widetilde{J}}+E_{\widetilde\cals}$ with $x(\vasi)+E_{\widetilde{J}}+
E_{\cals(\vasi,\widetilde{J})}$ gives an ordering on $V_1=\{j_1,\ldots,j_{t_s}\}$ with the property
\begin{equation}\label{eq:CS12}
\sigma_{j_n}(\widetilde{J}\cup \{j_1,\ldots,j_{n-1}\})=0
\end{equation}
for every $1\leq n \leq t_s$. Similarly, applying \ref{prop:G}(III) for $E_{\cals\ii \setminus \cals(\vasi,
\widetilde{J})}$ we have an ordering on $V_2=\{j_{t_s+1},\ldots,j_t\}$ such that
\begin{equation}\label{eq:CS13}
\sigma_{j_n}(\widetilde{J}\cup \{j_1,\ldots,j_{n-1}\})\geq 0
\end{equation}
for every $t_s+1\leq n \leq t$.

The contraction $c_{\beta'}$ will be $c_{\beta',1}\circ \ldots \circ c_{\beta',t}$, where $c_{\beta',n}$ 
corresponds to the deletion of the cycles with $E_{j_n}$ ($1\leq n\leq t$), i.e.
\begin{align*}
c_{\beta',n}:\Phi^*_N( \bigcup_{\alpha\not \in A'}|G_\alpha|\,,\,
& \widetilde{\cals}\cup \bigcup_{\beta\in B'}|C_\beta|\setminus \{j_{n+1},\ldots,j_t\}) \longrightarrow
\\ & \Phi^*_N(\bigcup_{\alpha\not \in A'}|G_\alpha|
 \,,\,
 \widetilde{\cals}\cup \bigcup_{\beta\in B'}|C_\beta| \setminus \{j_{n},\ldots,j_t\})
 \end{align*}
defined in the following way. Write $x=x(\vasi)+E_{\widetilde{J}}+l^*$ with
\begin{equation}\label{eq:SUPL2}
\cup_{\alpha\not \in A'}|G_\alpha|\subseteq
|l^*|\subseteq  \widetilde{\cals}\cup \bigcup_{\beta\in B'}|C_\beta|\setminus \{j_{n+1},\ldots,j_t\},
\end{equation}
then
$$c_{\beta',n}(x)=\left\{\begin{array}{ll}
x & \ \mbox{if $ j_n\not\in |l^*|$},\\
x-E_{j_n} & \ \mbox{if $ j_n\in |l^*|$}.\end{array}
\right.$$
Fix such an $l^*$ with $j_n\in |l^*|$, then we have to prove
\begin{equation}\label{eq:NNN2}
\chi_k(x({\bf i})+E_{\overline{J}}+l^*-E_{j_n})\leq N
\end{equation}
for any $\overline{J}\subseteq \overline{I}$. In this case the inequalities (\ref{eq:CS12}) and (\ref{eq:CS13}) 
implies
\begin{equation}\label{eq:CS22}
\sigma_{j_n}(\widetilde{J}\cup |l^*|\setminus j_n)\geq 0.
\end{equation}

Here we set $\overline{J}(l^*):=\{j\in \overline{I}\,:\, \sigma_j(|l^*|)\geq 0 \}$. Then if (\ref{eq:NNN2}) 
is valid for $\overline{J}(l^*)$ then it is so for any $\overline{J}\subseteq \overline{I}$. Indeed,
\begin{equation}\label{eq:MMM2}\begin{split}
&\chi_k(x({\bf i})+E_{\overline{J}}+l^*-E_{j_n})-
\chi_k(x({\bf i})+E_{\overline{J}(l^*)}+l^*-E_{j_n})\\
&=\sum _{j\in \overline{J}\setminus \overline{J}(l^*)}\big[\sigma_j(|l^*|)+(E_j,E_{j_n})\big]
-
\sum _{j\in \overline{J}(l^*)\setminus \overline{J}}\big[\sigma_j(|l^*|)+(E_j,E_{j_n})\big]\leq 0,
\end{split}\end{equation}
by the definition of $\overline{J}(l^*)$.
By the selection of $\beta'$ via \ref{prop:sel}(ii), for $j_n\in |C_\beta'|$ and for any 
$j\in \overline{I}_{j_n}\setminus \widetilde{J}$ one has $\sigma_j(({\widetilde{\cals}}\cup j_n)\setminus 
\cup_{\alpha\in A'}G_\alpha)< 0$, hence $\sigma_j(|l^*|)<0$, in other words $\overline{J}(l^*)\subseteq 
\widetilde{J}$.
Finally, from (\ref{eq:CS22}) we can deduce the inequality
\begin{align*}
\chi_k(x({\bf i})+E_{\overline{J}(l^*)}+l^*-E_{j_n})- &
\chi_k(x({\bf i})+E_{\overline{J}(l^*)}+l^*)=\\
& -\sigma_{j_n}(\overline{J}(l^*)\cup |l^*|\setminus j_n)\leq -\sigma_{j_n}(\widetilde{J}\cup |l^*|\setminus 
j_n)\leq 0.\end{align*}

\chapter{Seiberg--Witten invariants, \ periodic constants and \newline Ehrhart coefficients}\labelpar{c:SW}\

\indent This chapter is devoted to the study of the Seiberg--Witten invariants. 
We introduced the terminology in \ref{ss:swgen}, 
where we mentioned that in the last years several combinatorial expressions were established 
regarding these invariants. Recall that \cite{BN} provides a surgery formula, 
which is not induced by a surgery exact sequence, but --- more in the spirit of the present chapter --- 
involves the {\em periodic constant of a series with one variable}.

The breakthrough, which is the starting point of the theory presented in this chapter, is given in 
\cite{NSW}. It says that the Seiberg--Witten invariant appears as the 
{\em constant term} of a multivariable quadratic polynomial given by some special 
truncation of a series. It is important to emphasize that the origin and main motivation 
of this identity was an analytic identity. 
Several of the combinatorial objects have their analytic counterparts, 
for example, the analogue of the topological series $Z(\bt)$ (defined in \ref{SW}) is the
Hilbert--Poincar\'e series associated with the multivariable equivariant divisorial filtration of the
local ring of the singular germ, and its equivariant periodic constants are the
 equivariant geometric genera. This will be described also in Section \ref{s:motSW}, where we motivate 
the results from the analytical and topological point of view as well. 

In the case of one--variable series, the afformentioned constant term is realized by the concept of the  
{\em periodic constant} (of the corresponding function or its series), which appeared first in 
\cite{Opg,NO1}. This original definition will be presented in Subsection 
\ref{PC}. 

Our aim is to extend this concept to the multivariable case (see \ref{ss:mpc}) 
in order to get a combinatorial 
computation of the Seiberg--Witten invariants. It turns out that the right understanding of the 
multivariable periodic constant goes through {\em multivariable Ehrhart theory}, which is described 
in Section \ref{ss:PPET}. It helps to understand how {\em the multivariable Poincar\'e series} encodes  
this generalized periodic constant, explaining the difficulties in the cases with 
`higher complexity level'. In fact, the complexity level of the (non--convex) polytopes, associated 
by the Ehrhart theory, is `measured' by the number of vertices of the corresponding graph. However, 
we will prove in \ref{s:REDZt} that this 
can be considerably reduced and measured with the number of nodes, 
or even more, with the number of bad vertices. In this way, the Reduction Theorem \ref{red} extends 
to the level of these invariants and their connections.

This gives the final output which is a nice identification of the Seiberg--Witten invariants 
with certain coefficients of a multivariable Ehrhart polynomial (cf. \ref{s:Last}).     

The chapter is based on \cite{Ehrhart}. 
The terminology and results in Ehrhart theory, relevant to the present discussion, can be found 
in \cite{Bar,BP,Beck_c,Beck_m,Beck_p,BR1,BR2,BDR,CL,DR}, while for the connections with 
partition functions, see 
\cite{BV,SZV,Str}.

\section{Analytic and topological motivation}\labelpar{s:motSW}\
In this section we start with some useful notations and facts which will be used throughout the chapter. 
Then we present definitions and results regarding the {\em analytic Hilbert--Poincar\'e series} of normal 
surface singularities, which serve as a motivation for the topological side. 
After this part, we continue with the definition and immediate properties of the topological Poincar\'e series. 
A discussion regarding the statement of Theorem \ref{th:JEMS} will serve as a motivation and it 
provides a short summary for the connections between the three numerical datas: 
the Seiberg--Witten invariant, the periodic 
constant and the Ehrhart coefficient.

\subsection{Notations and facts (addendum to section \ref{intro:comb})}\labelpar{ss:11} \
Let $(X,0)$ be a complex normal surface singularity
whose {\em link $M$ is a rational homology sphere}.  Let
$\pi:\widetilde{X}\to X$ be a good resolution with dual graph
$G$ whose vertices will be denoted by $\cV$. Hence $G$ 
is a tree and all the irreducible exceptional divisors have genus
$0$. 

Let  $\delta_v$ be the valency of the vertex $v$. We distinguish the following
subsets of vertices: the set of {\it nodes}
$\calN=\{v\in \cV:\delta_v\geq 3\}$, and the set of {\it ends}
$\cE=\{v\in \cV:\delta_v= 1\}$. If we delete from $G$ the nodes and their adjacent
edges we get the collection of (maximal) {\it chains}
 of the graph. A {\it leg} is a chain which is connected by only one node.
$|\cV|$ or  $s$  stay for the number of vertices, while $|\calN|$ and $|\cE|$ for the
 number of nodes and ends, and $H:=H_1(M,\Z)$. 

We look at the combinatorics of the graph $G$ according to Section \ref{intro:comb}. Recall that 
the module $L'$ over $\setZ$ is freely generated by the (anti)duals $\{E_v^*\}_v$, where we
prefer the convention $ ( E_v^*, E_w) =  -1 $ for $v = w$, and
$0$ otherwise.
It will be useful to write $\det(G):=\det(-\frI)$, where $\frI$ is the negative definite intersection 
matrix. The inverse of $\frI$ has entries
 $(\frI^{-1})_{vw}=(E_v^*,E^*_w)$, all of them are negative. Furthermore, by a result 
 of \cite[page 83 and \S 20]{EN},
 \begin{equation}\label{eq:DETsgr}
  \begin{split}
 \mbox{ $-|H|\cdot (E_v^*,E^*_w)$
equals the determinant of the subgraph obtained}\\
\mbox{from $G$ by eliminating the
shortest path connecting $v$ and $w$.}\end{split}\end{equation}
The canonical class $\K\in L'$ was defined by the
adjunction formulae $(\K+E_v,E_v)+2=0$ for all $v\in\cV$.
The expression $\K^2+|\cV|$ will appear as the normalization term in several formulae.  
Therefore, we quote its combinatorial
expression in terms of the graph, cf. \cite{SWI}:
\begin{equation}\label{eq:K2}
\K^2+|\cV|=\sum_{v\in\cV}(E_v,E_v)+3|\cV|+2+\sum_{v,w\in\cV}\, (2-\delta_v)(2-\delta_w)\frI^{-1}_{vw},
\end{equation}
where $\delta_v$ is the valency of the vertex $v$.

Recall, that the Lipman cone is defined as $\calS'=\{l'\in L'\,:\, (l',E_v)\leq 0 \ \mbox{for all
$v$}\}$. It is generated over $\setZ_{\geq 0}$ by the
elements $E_v^*$. Since  all the entries of $E_v^*$
are strict positive, cf. (\ref{eq:POS}), for any fixed $a\in L'$ one has:
\begin{equation}\label{eq:finite}
\{l'\in \cS'\,:\, l'\ngeq a\} \ \ \mbox{is finite}.
\end{equation}

For any class $h\in H$ there exists a unique minimal element of
$\{l'\in L'\,:\,[l']=h\}\cap \calS'$, cf. \cite[5.4]{OSZINV} or Lemma \ref{def:KR}, which 
will be denoted by $s_h$ in this chapter. Nevertheless, if we look at it for a 
fixed class $[k]$, we use the notation $\lk$ as before. 

Furthermore, we set \,$\square=\{\sum_v
l'_vE_v\in L'\,:\, 0\leq l'_v <1\}$ for the `semi--open cube', and
for any $h\in H$ we consider the unique representative
$r_h\in \square$ with $[r_h]=h$. One has  $s_h\geq r_h$, and
usually $s_h\not=r_h$ (see e.g. \cite[4.5.3]{Ng}). Moreover,
using the generalized Laufer computation sequence of 
\cite[4.3.3]{Ng} connecting $-r_h$ with $-s_h$ one gets
\begin{equation}\label{chiineq}
\chi(s_h)\leq \chi(r_h).\end{equation}

One considers also the Pontrjagin dual $\widehat{H}$ of $H$ and denote by 
$\theta:H\to\widehat{H}$ the isomorphism $[l']\mapsto e^{2\pi i (l',\cdot )}$ between them.

\subsection{Equivariant multivariable Hilbert series \newline of divisorial
filtrations}\labelpar{FM} \
We fix a resolution $\pi$ of $(X,0)$ with resolution graph $G$. The lattice $L$ defines a 
{\em divisorial multi--index filtration} on $\cO_{(X,0)}$ in the following way:
for any $l=\sum_j l_j E_j\in L$ one can associate an ideal 
$$\cF(l):=\{ f\in \cO_{(X,0)} \ : \ (f)_G\geq l\}.$$
The ususal way to describe this multi--index filtration is taking the {\em Hilbert function} 
$\hh(l):=\dim \cO_{(X,0)}/\cF(l)$ and its corresponding generating series, called the 
{\em multivariable Hilbert series}
\begin{equation}
\cH(\bt)=\sum _{l=\sum l_jE_j\in L}
\hh(l)\bt^{l}\in
\Z[[L]], 
\end{equation}
where $\bt^l=t_1^{l_1}\cdots t_s^{l_s}$ and $\Z[[L]]$ stands for the $\Z[L]$--submodule 
of formal power series $\Z[[t_1^{\pm 1/\det(\frI)},\dots,t_s^{\pm 1/\det(\frI)}]]$, generated 
by the monomials $\bt^l$. More details and informations can be read from \cite{CHR,CDG}. 

We may also 
define the multivariable Poincar\'e series, which is more close to the topology of $(X,0)$. But first, 
let us present a more general interpretation defined in \cite{CDGEq,NPS}, which gives the equivariant 
version of this concept.

Let $c:(Y,0)\to (X,0)$ be the {\em universal abelian cover} of $(X,0)$ with Galois group $H=H_1(M,\Z)$, 
$\pi_Y :\widetilde{Y}\to Y$ the normalized pullback of $\pi$ by $c$, and 
$\widetilde{c}:\widetilde{Y}\to \widetilde{X}$ the morphism which covers $c$, i.e. the induced 
finite map which makes the diagram commutative. If we denote the pullback 
of the cycle $l'\in L'$ by $\widetilde c$ with $\widetilde{c}^*(l')$, then \cite[3.3]{Nline} proves 
that $\widetilde{c}^*(l')$ is an integral cycle (an element of the lattice $L_Y$ associated with 
$\widetilde Y$ which is, in fact, a partial resolution of $(Y,0)$ with Hirzebruch--Jung singularities, 
cf. \cite[3.2]{Nline}).\newline
Then $\cO_{(Y,0)}$ inherits the divisorial multi--index filtration:
\begin{equation*}\label{eq:03}
\cF(l'):=\{ f\in \cO_{Y,o} \ : \ {\rm div}(f\circ\pi_Y)\geq \widetilde{c}^*(l')\}.
\end{equation*}
The natural action of $H$ on $Y$ induces an action on $\cO_{(Y,0)}$ which keeps $\cF(l')$ 
invariant. Hence, $H$ acts on $\cO_{(Y,0)}/\cF(l')$ and we can define 
$\hh(l') $ to be the dimension of the $\theta([l'])$--eigenspace
of $\cO_{Y,o}/\cF(l')$, where  $\theta([l'])=e^{2\pi i (l',\cdot )}$ is a multiplicative 
character in $\widehat H$ (cf. \ref{ss:11}). Then  the {\em equivariant
multivariable Hilbert series} is
\begin{equation*}\label{eq:04}
\cH(\bt)=\sum_{l'\in L'}\hh(l')\bt^{l'}\in
\setZ[[L']].
\end{equation*}
In $\cH(\bt)$ the exponents $l'$ of the terms $\bt^{l'}$ reflect the $H$ eigenspace
decomposition too. E.g., $\sum_{l\in L}\hh(l)\bt^{l}$ corresponds
to the $H$--invariants, hence it is the Hilbert series defined at the beginning of this 
subsection.

If $l'$ is in the special `vanishing zone' $-\K+\calS'$,
then by vanishing (of a certain  first cohomology), and by the 
Riemann--Roch formula, one obtains (see \cite{NCL}) that the expression
 \begin{equation}\label{eq:KV}
 \hh(l')+\frac{(K+2l')^2+|\cV|}{8}
\end{equation}
depends only on the class $[l']\in L'/L$ of $l'$. 

The key bridge connecting  $\cH(\bt)$ with the
topology of the link  and with  $G$ is realized by defining the 
{\em equivariant multivariable Poincar\'e series} from $\cH(\bt)$ (cf. \cite{CDG,CDGEq,NPS,NCL}):
\begin{equation*}\label{eq:06}
\cP(\bt)=-\cH(\bt) \cdot \prod_v(1-t_v^{-1})\in \setZ[[L']].
\end{equation*}
Notice that apparently $\cP$ loses some analytic information of $\cH$. However, 
\cite[(3.2.6)]{NCL} shows explicitly that the identity can be `inverted'. Namely, 
if we write $\cP(\bt)=\sum_{l'}\bar{p}_{l'}\bt^{l'}$, then 
\begin{equation*} \label{eq:inv}
\hh(l')=\sum_{l\in L,\, l\not\geq 0} \bar{p}_{l'+l}.
\end{equation*}
This is well--defined, since by \cite[(3.2.2)]{NCL} one has that $\cP$ is supported on
$\cS'$, therefore the sum in the formula is finite via \ref{eq:finite}.
In particular, cf. (\ref{eq:KV}),
\begin{equation}\label{eq:KV2}
\sum_{l\in L,\, l\not\geq 0} \bar{p}_{l'+l}=-\mathrm{const}_{[-l']} -
\frac{(\K+2l')^2+|\cV|}{8}
\end{equation}
for any $l'\in-\K+\cS'$, where $\mathrm{const}_{[-l']}$ depends only on the class $[-l']$ of $-l'$.
The right hand side can be thought as a `multivariable Hilbert polynomial' of degree 2 associated with 
the series $\cH(\bt)$ ( or with $\cP(\bt)$). Its constant term is
the {\em normalized equivariant geometric genus} of the universal abelian cover $Y$ 
(see details in \cite{NCL}), that is 
\begin{equation}\label{eq:KV2b}
-\mathrm{const}_{[-r_h]}=
\dim(H^1(\widetilde{Y},\cO_{\widetilde{Y}}\,)_{\theta(h)})+
\frac{(\K+2r_h)^2+|\cV|}{8}.
\end{equation}
The main point is that $\cP(\bt)$ has a {\it topological candidate}, which is defined purely from the graph $G$ 
and will be the subject of the next subsection. The two series agree
for several singularities, see for example \cite{CDGEq,NPS,NCL}. \cite{NCL} proves that 
it is valid for splice--quotient singularities as well. 

It turns out that identification of their
constant terms (for `nice' analytic structures) is the subject of the
Seiberg--Witten Invariant Conjecture \ref{ss:SWIC}, since the constant term of the topological 
candidate will realize the Seiberg--Witten invariant (cf. \ref{th:JEMS}). 
Hence, if the identification holds, then $\mathrm{const}_{[-l']}=
\frsw_{[-l']*\sigma_{can}}(M)$ too, and (\ref{eq:KV2b}) creates the bridge between the combinatorial/
topological Seiberg--Witten theory and the analytic counterpart.

\subsection{The topological Poincar\'e series and $\frsw_\sigma(M)$}\labelpar{SW}
\begin{definition}
Consider the following rational function 
\begin{equation}\label{eq:INTR}
\prod_{v\in \cV} (1-\bt^{E^*_v})^{\delta_v-2}.
\end{equation}
Then its multivariable Taylor expansion $Z(\bt)=\sum p_{l'}\bt^{l'}$ at the origin is called 
the {\em topological (combinatorial) Poincar\'e series} associated with the plumbing graph $G$.
\end{definition}
Since the Lipman cone $\calS'$ is generated by the elements $E_v^*$ over $\Z_{\geq 0}$, $Z(\bt)$ 
is supported on $\calS'$ (i.e. $p_{l'}=0$ for every $l'\notin \calS'$). Therefore, if we apply the 
same {\em special truncation} as in the analytic case (\ref{eq:KV2}), then we get a finite sum 
\begin{equation}\label{eq:sump}
\sum_{l\in L,\, l\not\geq 0} p_{l'+l} \ .
\end{equation}
One has a natural decomposition $Z(\bt)=\sum_{h\in H}Z_h(\bt)$, where 
$Z_h(\bt)=\sum_{[l']=h}p_{l'}\bt^{l'}$ ($[l']$ is the class of $l'$). Then the sum (\ref{eq:sump}) involves
only the part $Z_{[l']}$ (sometimes we also write $Z_{l'}$ for $Z_{[l']}$). 

As we already mentioned at the end of \ref{FM}, $Z(\bt)$ is the topological candidate 
for $\cP(\bt)$, since they agree for `nice' analytic structures. This is the reason why (\ref{eq:KV2}) 
motivated the birth of the next theorem, which proves that $Z(\bt)$ encodes the Seiberg--Witten invariants 
of the link $M$. Moreover, it is the starting point of the research of the present 
chapter.
\begin{theorem}[\cite{NSW}]\labelpar{th:JEMS} Fix some $l'\in L'$.
Assume that for any  $v\in\cV$ the  $E^*_v$--coordinate  of $l'$
is larger than or equal to $-(E_v^2+1)$ for all $v$. Then
\begin{equation}\label{eq:SUM}\sum_{l\in L,\, l\not\geq 0}p_{l'+l }=
-\frsw_{[-l']*\sigma_{can}}(M)-\frac{(\K+2l')^2+|\cV|}{8},
\end{equation}
where
$*$ denotes the torsor action of $H$ on $\mathrm{Spin}^c(M)$.
\end{theorem}

The finite sum on the left hand side appears as a {\em counting function}
of the coefficients of $Z_{[l']}$ associated with the special truncation, while the right hand side is a
{\em multivariable quadratic Hilbert polynomial} whose {\em constant term} is the 
normalized Seiberg--Witten invariant
$$-\frsw_{[-l']*\sigma_{can}}(M)-\frac{\K^2+|\cV|}{8}.$$

In order to guarantee the validity of the formula, the vector $l'$ should sit in a special
{\it chamber} described by the inequalities of the assumption. This, after we
establish the necessary  bridges,
will read as follows: \newline
{\em `the third degree' coefficient of a multivariable
Ehrhart quasipolynomial associated with a certain polytope and specific chamber can be
identified with the Seiberg--Witten invariant}.

In the followings, we will motivate and summarize the results of this chapter, which explains 
the above highlighted sentence. The way how one recovers the needed information from 
the series $Z(\bt)$ can be done at several
levels: 
\begin{itemize}
\item The first one is entirely at the level of series.
We develop a theory which associates with any series the counting function of its coefficients
(given by the truncation of the monomials) --- like the right hand side of (\ref{eq:SUM}).
This is usually a {\em piecewise quasipolynomial}. Once we fix a chamber, the free term of the counting
function is the so--called {\em periodic constant} (denoted by ${\rm pc}$). In this terminology, 
the Seiberg--Witten invariant can be interpreted
as the {\it multivariable periodic constant} ${\rm pc}(Z)$ (cf. \ref{ss:mpc})
of the series $Z(\bt)$, where the chosen chamber is described by the inequalities of the
assumption (a part of the Lipman cone $\calS'$). The `periodicity'
is related with the quasipolynomial behavior of the counting function.

The periodic constant of one--variable series
was introduced by N\'emethi and Okuma. Its idea, cf. \cite{NO1,Opg}, will be detailed in \ref{PC}. 
(For applications see e.g. \cite{NO,NO1,NSW,BN}.)

We create the general theory, which carries necessarily several
difficult technical ingredients. For example, one has to choose the `right'
truncation and summation procedure of the coefficients, which, in the context of general series, 
is not automatically motivated, and also it depends on the chamber decomposition of the space of exponents.
The theory has some similarities with the theory of vector partition functions as well.

\item On the other hand, there is a more sophisticated way to generalize the identity (\ref{eq:SUM}) too.
From any Taylor expansion of a multivariable rational function 
with denominator of type $\prod_i(1-\bt^{a_i})$ we construct a {\em polytope}
situated in a lattice which carries also a representation of a finite abelian group $H$.
Associated with these data, 
we consider the {\em equivariant multivariable Ehrhart piecewise quasipolynomials},
whose existence, main properties (like the {\em Ehrhart--MacDonald--Stanley type reciprocity law} 
or {\em chamber decompositions}) will also be established in \ref{ss:PPET}. 
This applied to the series $Z(\bt)$ above,
and to the quasipolynomial of those chambers which belong to the Lipman cone shows that 
the first three top--degree {\em Ehrhart coefficients} (at least)
will carry geometrical/topological meaning, including the Seiberg--Witten invariants of the link $M$.
\end{itemize}
Figure 4.1 (cf. \cite{Ehrhart}) is helping to summarize these two points with a 
schematic picture of 
these connections and areas we target.

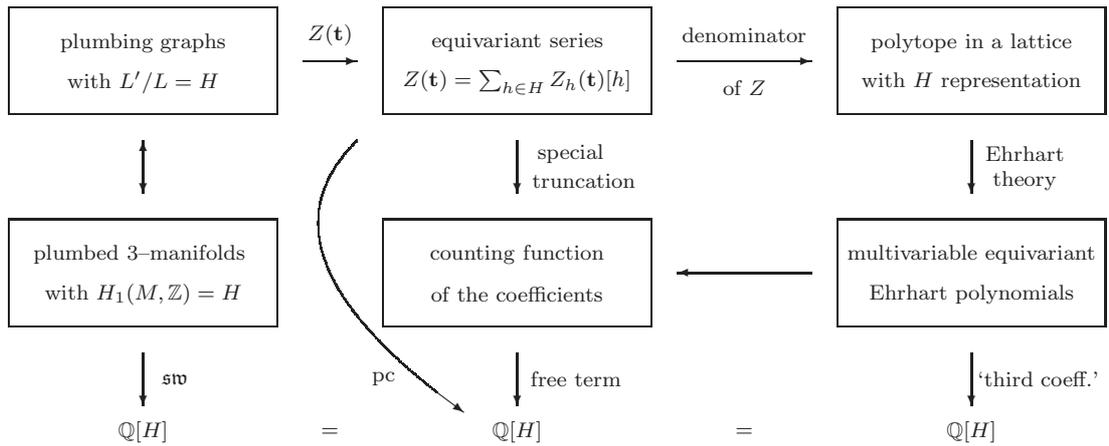
\begin{figure}[h!]\label{fig:2}
\begin{center}
\begin{picture}(400,180)(0,60)

\put(10,180){\framebox(100,40){}}
\put(60,207){\makebox(0,0){\scriptsize{plumbing graphs}}}
\put(60,192){\makebox(0,0){\scriptsize{with $L'/L=H$}}}

\put(10,100){\framebox(100,40){}}
\put(60,127){\makebox(0,0){\scriptsize{plumbed 3--manifolds }}}
\put(60,112){\makebox(0,0){\scriptsize{with $H_1(M,\Z)=H$}}}

\put(60,150){\vector(0,1){20}}\put(60,170){\vector(0,-1){20}}
\put(60,90){\vector(0,-1){20}}
\put(60,60){\makebox(0,0){\scriptsize{$\Q[H]$}}}
\put(72,80){\makebox(0,0){\scriptsize{$\frsw$}}}
\put(120,200){\vector(1,0){20}}
\put(130,210){\makebox(0,0){\scriptsize{$Z(\bt)$}}}
\put(130,60){\makebox(0,0){\scriptsize{$=$}}}
\qbezier(140,170)(100,130)(170,75)
\put(170,75){\vector(4,-3){10}}
\put(150,80){\makebox(0,0){\scriptsize{${\rm  pc}$}}}
\put(150,180){\framebox(100,40){}}
\put(200,207){\makebox(0,0){\scriptsize{equivariant series}}}
\put(200,192){\makebox(0,0){\scriptsize{$Z(\bt)=\sum_{h\in H}Z_h(\bt)[h]$}}}

\put(150,100){\framebox(100,40){}}
\put(200,127){\makebox(0,0){\scriptsize{counting function}}}
\put(200,112){\makebox(0,0){\scriptsize{of the coefficients}}}

\put(220,165){\makebox(0,0){\scriptsize{special}}}
\put(225,155){\makebox(0,0){\scriptsize{truncation}}}
\put(222,80){\makebox(0,0){\scriptsize{free term}}}
\put(200,170){\vector(0,-1){20}}
\put(200,90){\vector(0,-1){20}}
\put(200,60){\makebox(0,0){\scriptsize{$\Q[H]$}}}

\put(260,200){\vector(1,0){50}}
\put(285,210){\makebox(0,0){\scriptsize{denominator}}}
\put(285,190){\makebox(0,0){\scriptsize{of $Z$}}}
\put(285,60){\makebox(0,0){\scriptsize{$=$}}}
\put(310,120){\vector(-1,0){50}}

\put(320,180){\framebox(100,40){}}
\put(370,207){\makebox(0,0){\scriptsize{polytope in a lattice}}}
\put(370,192){\makebox(0,0){\scriptsize{with $H$ representation}}}

\put(320,100){\framebox(100,40){}}

\put(370,127){\makebox(0,0){\scriptsize{multivariable equivariant}}}
\put(370,112){\makebox(0,0){\scriptsize{Ehrhart polynomials}}}

\put(390,165){\makebox(0,0){\scriptsize{Ehrhart}}}
\put(390,155){\makebox(0,0){\scriptsize{theory}}}
\put(395,80){\makebox(0,0){\scriptsize{`third coeff.'}}}
\put(370,170){\vector(0,-1){20}}
\put(370,90){\vector(0,-1){20}}
\put(370,60){\makebox(0,0){\scriptsize{$\Q[H]$}}}
\end{picture}
\end{center}
\caption{The theories associated with $G$.}
\end{figure}

\subsection{A `classical' connection between  polytopes and
gauge invariants (and its limits).}\labelpar{ss:CLASSI} \
The coefficient identification (\ref{s:Last}), and in fact
(\ref{eq:SUM}) too, supply an additional addendum to the intimate relationship between
lattice point counting  and the Riemann--Roch formula, exploited in global algebraic geometry
by toric geometry.

In the literature of normal surface singularities there
is a sequence of results which connect the topology of the link
with the number of lattice points in a certain polytope. Here we list some 
historical details on this subject.

The first is based on the theory of {\em Newton non--degenerate hypersurface
singularities}, see e.g. the second volume of the monograph of 
Arnold, Gussein--Zade and  Varchenko \cite{AGV}. According to this, for such a germ one defines the
{\em Newton polytope} $\Gamma^-_N$ using the non--trivial monomials of the
defining equation of the germ. Then one can prove that several
invariants of the germ can be recovered from $\Gamma^-_N$. 
For example, by a result of Merle and Teissier
\cite{MT}, the geometric genus $p_g$ equals the {\em number of
lattice points} in $((\Z_{>0})^3\cap \Gamma^-_N)$, see also the work of 
Braun and N\'emethi \cite{BN07} into this direction. 

The second is provided by the Laufer--Durfee formula, which determines the
signature of the Milnor fiber $\sigma$ as $-8p_g-K^2-|\cV|$
(\cite{D78}).  Finally, there is a conjecture of Neumann and Wahl \cite{NW},
formulated for hypersurfaces with integral homology sphere links, 
and proved for Brieskorn, suspension \cite{NW}
and splice--quotient \cite{NO1} singularities, according to which
$\sigma/8=\lambda(M)$, the Casson invariant of the link.
Therefore, if all these steps run, for example as in the Brieskorn case,
then the Casson invariant of the link, normalized by $K^2+|\cV|$,
 can be expressed as the number of lattice points
of a polytope associated with the equation of the germ.

 This correspondence has several deficiencies. First, even in simple cases,
we do not know how to extend the correspondence to the equivariant
case, more precisely, how to express the equivariant geometric genus from
$\Gamma^-_N$. Second, the expected generalization, the
Seiberg--Witten invariant conjecture (see \ref{ss:SWIC}), which aims to
identify the Seiberg--Witten invariant of the link with $p_g$
(or $\sigma$), is still open in this case. And, finally, this family of germs is rather
 restrictive.

 The present chapter defines another polytope, which carries an action of the group $H$, and
its {\em Ehrhart invariants determine the Seiberg--Witten invariant in any case}. 
It is not described from
the equations of the germ, but from its multivariable `zeta--function' $Z(\bt)$.

\section{Equivariant multivariable Ehrhart theory}\labelpar{ss:PPET}\ 
In this section we generalize the classical Ehrhart theory to the equivariant multivariable version,
involving non--convex polytopes,
which will fit with our comparison with the equivariant multivariable series provided by plumbing graphs.

Let us start with a $d$--dimensional {\it rational lattice}
 $\calX\subset \Q^d$ and a group homomorphism $\rho:\calX\to
\fH$
to a finite abelian group $\fH$. We consider a {\it rational  vector--dilated
polytope} with parameter  $\bl=(\bl_1,\ldots, \bl_r)$, $\bl_v\in \Z^{m_v}$,
\begin{equation}\label{eq:POL}P^{(\bl)}=\bigcup_{v=1}^rP^{({\bl}_v)}_v, \ \ \mbox{where} \ \
P^{({\bl}_v)}_v=\{{\bf x}\in\R^d\,:\, {\bf A}_v{\bf x}\leq
\bl_v\},\end{equation} where ${\bf A}_v$
is an integral
$m_v\times d$ matrix. If $\{A_{v,\lambda i}\}_{\lambda i}$ and
$\{\bll_{v,\lambda}\}_\lambda$  are the entries of ${\bf A}_v$ and
$\bl_v$, then the inequality ${\bf A}_v{\bf x}\leq \bl_v$ in
(\ref{eq:POL}) reads as $\sum_{i=1}^dx_iA_{v,\lambda i}\leq
\bll_{v,\lambda}$ for any $\lambda=1,\ldots,m_v$.

We will vary the parameter $\bl$ in some `chambers' (described
below for the needed cases) such that the polytopes $P^{(\bl)} $
remain  {\it combinatorially stable} (or preserve their {\it combinatorial type})
 when $\bl$ runs in the same chamber.
This means that their face lattices are isomorphic.
(This implies that they are connected by homeomorphisms, which preserve
the stratification of the faces.) We also suppose that $P^{(\bl)}$
is homeomorphic to a $d$--dimensional manifold. Denote the set of
all closed facets of $P^{(\bl)}$ by $\calF$ and let $\calT$ be a
subset of $\calF$, such that $\cup_{F^{(\bl)}\in \calT}F^{(\bl)}$
is homeomorphic to a $(d-1)$--manifold.

Then  we have the
following generalization to the {\it equivariant version} of
results of Stanley \cite{S74}, McMullen \cite{M78} and Beck
\cite{Beck_c,Beck_m}.
\begin{theorem}\labelpar{th:REC}
For any $h\in \fH$ and $\calT\subset \calF$ let
\begin{equation}\label{eq:REC}
\calL_h({\bf A}, \calT,\bl)
:=\mbox{cardinality of}\ \big(\big(P^{(\bl)}\setminus \cup_{F^{(\bl)}\in \calT}
F^{(\bl)}\big)\cap \rho^{-1}(h)\big).\end{equation}

(a) If $\bl$ moves in some region in such a way that
$P^{(\bl)} $ stays  combinatorially stable then the expression
$\calL_h({\bf A},\calT,\bl)$ is a quasipolynomial in $\bl\in \Z^{\sum m_v}$.

(b) For a fixed combinatorial type of $P^{(\bl)} $ and for a fixed
$\calT$, the quasipolynomials $\calL_h({\bf A},\calT,\bl)$ and
$\calL_{-h}({\bf A},\calF\setminus \calT,\bl)$ satisfy the
Ehrhart--MacDonald--Stanley reciprocity law
\begin{equation}\label{eq:EMDS}
\calL_h({\bf A},\calT,\bl) =(-1)^d\cdot \calL_{-h}({\bf
A},\calF\setminus \calT,\bl)|_{\mbox{\tiny{\rm{replace  $\bl$ by
$-\bl$}}}}.\end{equation}
\end{theorem}

To avoid any confusion regarding the expression of (\ref{eq:EMDS}) we note:
the  two quasipolynomials in (\ref{eq:EMDS}) are associated with that domain of definition
(chamber) which corresponds to the fixed combinatorial type.
Usually for $-\bl$ the combinatorial type  of $P^{(\bl)} $ is
different, hence the right hand side of (\ref{eq:EMDS}) {\it need
not equal} $(-1)^d\cdot \calL_{-h}({\bf A},\calF\setminus
\calT,-\bl)$. This last expression is the value at $-\bl$ of the
quasipolynomial associated with the chamber which contains $-\bl$.

For a  reformulation of the identity (\ref{eq:EMDS}) 
in terms of the fixed chamber see Theorem \ref{th:PQP}(c).

\begin{proof} The statements for $\fH=0$
 are identical with those of Beck from \cite{Beck_m}.
Part (a) above for arbitrary $\fH$ can be proved identically as in \cite{Beck_m}
applied for the situation when the parameters $\bl$ run in
an overlattice of $\Z^{\sum m_v}$, instead of  $\Z^{\sum m_v}$.
Equivalently, one can apply  \cite{CL},
which considers the non--equivariant case, but the integral  parameters $\bl$ of Beck
are replaced by  {\it rational affine  parameters}.

For the convenience of the reader we provide the proof.  First we
notice that via standard additivity formulae, cf.
\cite[\S\,2]{Beck_m}, it is enough to prove the statement for each
convex $P_v^{(\bl_v)}$. But, considering  $P_v^{(\bl_v)}$ and $K:=
\ker(\rho)$, for any ${\bf r}\in\calX$ one has the isomorphism
$$\{{\bf x}\in K+{\bf r}\,:\, {\bf A}_v{\bf x}\leq \bl_v\}\simeq
 \{{\bf y}\in K\,:\, {\bf A}_v{\bf y}\leq \bl_v-{\bf A}_v{\bf r}\}.$$
Hence \cite[Theorem 2]{CL} (or \cite{Beck_m} for an overlattice of $\Z^{\sum m_v}$)
 can  be applied, which shows (a). Next, part (b) can also be reduced
to \cite{Beck_m}.
Indeed, we can reduce the discussion again to $P_v^{(\bl_v)}$. We drop
the index $v$,  we  choose ${\bf r}_h\in \calX$ with
$\rho({\bf r}_h)=h$, and  we fix some $\bl_0$.
 Then for ${\bf x}\in K\pm {\bf r}_h $ with ${\bf A}{\bf x}\leq \bl_0$ we take
${\bf y}:={\bf x}\mp {\bf r}_y$ and ${\bf k}:=\bl_0\mp {\bf A}{\bf r}_h$, which satisfy
${\bf y}\in K$ and ${\bf A}{\bf y}\leq {\bf k}$. Therefore, using \cite{Beck_m}
for this polytope, we obtain 
$$\calL_h({\bf A},\calT,\bl_0)=
\calL_0({\bf A},\calT,{\bf k})=
 (-1)^d\cdot \calL_0({\bf A},\calF\setminus\calT,{\bf k})|_{\mbox{\tiny{\rm{replace  ${\bf k}$ 
 by $-{\bf k}$}}}}$$
 $$\hspace{5.4cm} =
 (-1)^d\cdot \calL_{-h}({\bf A},\calF\setminus \calT,\bl_0)|_{\mbox{\tiny{\rm{replace  ${\bf l_0}$
 by $-{\bf l_0}$}}}},
$$
where the second and the third term is associated with the lattice $K$. 
\end{proof}
\begin{definition}
The quasipolynomial $\calL_h({\bf A}, \calT,\bl)$  considered in Theorem \ref{th:REC}, associated with a fixed
combinatorial type of $P^{({\bf l})}$, is called the {\it equivariant multivariable quasipolynomial }
associated with the corresponding data.

If we vary ${\bf l}$ in $\Z^{\sum m_v}$ (hence we allow the variation of the combinatorial type)
we obtain the  {\it equivariant multivariable piecewise quasipolynomial } $\calL_h({\bf A}, \calT,\bl)$
(see also Theorem \ref{th:PQP} and Corollary \ref{cor:Taylor} below).
\end{definition}

\begin{remark}
Parallel to the collection $\{\calL_h\}_h$ defined in
(\ref{eq:REC}) one can consider their Fourier transforms as well:
for any character $\xi\in \widehat{\fH}=\Hom (\fH,S^1)$, one defines
\begin{equation}\label{eq:REC2}
\calL_\xi({\bf A}, \calT,\bl) :=\sum_{{\bf x}\in P^{(\bl)}\setminus
\cup_{F^{(\bl)}\in \calT} F^{(\bl)}}
\xi^{-1}(\rho({\bf x})),
\end{equation} which
satisfies $\calL_\xi=\sum_h\calL_h\cdot \xi^{-1}(h), \ \ \mbox{and
} \ \ |\fH|\cdot \calL_h=\sum_\xi\calL_\xi\cdot \xi(h)$. Hence, the
above properties of $\calL_h$ can be obtained from
similar properties of $\calL_\xi$ as well. Hence, Theorem \ref{th:REC} can be
deduced from \cite[\S\,4.3]{BV} too.
\end{remark}

\begin{remark} In the sequel we will not consider polytopes with this
high generality: our polytopes will be special ones associated with the
denominators of type $\prod_i (1-\bt^{a_i})$ of multivariable rational functions, or their
Taylor series.
In order to avoid unnecessary technical details,
the stability of the combinatorial type of $P^{({\bf l})}$, and the corresponding chamber decomposition of
$\R^{\sum m_v}$  will also be treated for this special polytopes, see \ref{bek:combtypes}.
\end{remark}

\section{Multivariable rational functions and their \newline periodic constants}\labelpar{s:PC}

\subsection{Historical remark: the one--variable case}\labelpar{PC} \
The concept of the periodic constant for one--variable series was introduced by N\'emethi and Okuma. 
One can find the details in \cite[3.9]{NO1} and \cite[4.8(1)]{Opg}, however, we present it in the sequel. 
 
Let  $S(t) = \sum_{l\geq 0} c_l t^l\in \Z[[t]]$  be a
formal power series. Suppose that for some positive integer $p$,
the expression $\sum_{l=0}^{pn-1} c_l$ is a polynomial $P_p(n)$ in
the variable $n$.  Then the constant term $P_p(0)$ of $P_p(n)$ is
independent of the `period' $p$. We call $P_p(0)$ the
\emph{periodic constant} of $S$ and denote it by $\mathrm{pc}(S)$.
For example, if $l\mapsto Q(l)$ is a quasipolynomial and
$S(t):=\sum_{l\geq 0}Q(l)t^l$, then one can take for $p$ the period
of $Q$,  and one shows that $\mathrm{pc}(\sum_{l\geq
0}Q(l)t^l)=0$.

Assume that $S(t)$ is the Hilbert series associated with a graded
algebra/vector space $A=\oplus_{l\geq 0}A_l$ (i.e. $c_l=\dim
A_l$), and the series $S$ admits a Hilbert quasipolynomial $Q(l)$
(that is, $c_l=Q(l)$ for $l\gg 0$). Since the periodic constant of
$\sum_lQ(l)t^l$ is zero, the periodic constant of $S(t)$ measures
exactly the difference between $S(t)$ and its `regularized series'
$S_{reg}(t):=\sum_{l\geq 0}Q(l)t^l$. That is:
$\mathrm{pc}(S)=(S-S_{reg})(1)$ collecting all the
anomalies of the starting elements of $S$.

Note that $S_{reg}(t)$ can be represented by a rational function
of negative degree with denominator of type
$A(t)=\prod_i(1-t^{a_i})$, and $(S-S_{reg})(t)$ is a polynomial.
Conversely,  one has the following reinterpretation of the
periodic constant \cite[7.0.2]{BN}. If $\sum_lc_lt^l$ is a
rational function $B(t)/A(t)$ with $A(t)= \prod_i(1-t^{a_i})$, and
one rewrites it as $C(t)+D(t)/A(t)$ with $C$ and $D$  polynomials
and $D(t)/A(t)$ of negative degree, then $\mathrm{pc}(S)=C(1)$.
From this fact one also gets that
$\mathrm{pc}(S(t))=\mathrm{pc}(S(t^N))$ for any $N\in \Z_{>0}$. We
will refer to $C(t)$ as the {\it polynomial part} of $S$.

As an  example, consider a subset $\calS\subset\Z_{\geq 0}$ with finite complement.  Then
 $S(t)=\sum_{s\in \calS}t^s$ rewritten is $1/(1-t)-\sum_{s\not\in \calS}t^s$, hence
 $\mathrm{pc}(S)=-\#(\Z_{\geq 0}\setminus \calS)$. In particular, if $\calS$ is the
 semigroup of a local irreducible plane curve singularity, then $-\mathrm{pc}(S)$ is the
 delta--invariant of that germ. Our study  below   includes the generalization of  this fact to
  surface singularities.
  
\subsection{Multivariable  generalization}\labelpar{ss:GENRF}
\bekezdes\labelpar{bek:LL'}
We wish  to extend the definition of the periodic constant to the case
  of Taylor expansions at the origin  of multivariable rational functions
of type
\begin{equation}\label{eq:func}
f(\bt)=\frac{\sum_{k=1}^r\iota_k\bt^{b_k}}{\prod_{i=1}^d
(1-\bt^{a_i})} \ \ \ \ (\iota_k\in\setZ).
\end{equation}
Let us explain the notation.
Let $L$ be a lattice of rank $s$ with fixed bases  $\{E_v\}_{v=1}^s$.
Let $L'$ be an overlattice of it with same rank, $L\subset L'\subset L\otimes \Q$ with $|L'/L|=\frdd$.
Then, in  (\ref{eq:func}),
$\{b_k\}_{k=1}^r, \ \{a_i\}_{i=1}^d\in L'$ and for any
$l'=\sum_vl'_vE_v\in L'$ we write $\bt^{l'}=t_1^{l'_1}\dots
t_{s}^{l'_{s}}$.  We also assume that
{\it all the coordinates $a_{i,v}$ of $a_i$  are strict positive},
Hence, in general, the coefficients $l'_v$ are not integral, and the
Laurent   expansion $Tf({\bt})$ of $f({\bt})$ at the origin is
$$Tf(\bt)=\sum_{l'}p_{l'}\bt^{l'}\in \Z[[t_1^{1/\frdd},\ldots, t_s^{1/\frdd}]]
[t_1^{-1/\frdd},\ldots, t_s^{-1/\frdd}]:=\Z[[\bt^{1/\frdd}]][\bt^{-1/\frdd}].$$
We also consider the natural partial ordering of $L\otimes \Q$ (defined as in  \ref{ss:11}).
If all vectors $b_k\geq 0$ then $Tf(\bt)$ is in
$\sum_{l'}p_{l'}\bt^{l'}\in \Z[[\bt^{1/\frdd}]]$.
Sometimes we will not make difference between $f$ and $Tf$.

\bekezdes\labelpar{bek:LL'2} 
This will be extended to the following equivariant case. We fix a
finite abelian group $\cG$, and for each $g\in \cG$ a series (or rational function)
 $Tf_g\in \Z[[\bt^{1/\frdd}]][\bt^{-1/\frdd}]$ as in
\ref{bek:LL'}, and we set
$$Tf^e(\bt):=\sum_{g\in G}\, Tf_g(\bt)\cdot [g]\in \Z[[\bt^{1/\frdd}]][\bt^{-1/\frdd}][G].$$
Sometimes this equivariant extension is given automatically in the context of \ref{bek:LL'}.
Indeed, if  in \ref{bek:LL'} we set  $H:=L'/L$, and for
\begin{equation}\label{eq:fh}
Tf=\sum_{l'}p_{l'}\bt^{l'} \ \ \ \mbox{we define} \ \ \
Tf_h:=\sum_{[l']=h}p_{l'}\bt^{l'},
\end{equation}
we obtain a decomposition of $Tf$ as a sum $\sum_hTf_h\in \Z[[\bt^{1/\frdd}]][\bt^{-1/\frdd}][H]$
(with $\frdd=|H|$).

In our cases we always start  with this  group $L'/L=H$ (hence
$f$ determines its decomposition $\sum_hf_h$). Nevertheless, some alterations will appear.
First, we might consider the non--equivariant case, hence we can forget the decomposition over $H$.
Another case appears as follows.  In order to simplify the rational function we will eliminate
some of its  variables (e.g., we substitute $t_i=1$ for certain indices $i$),
or we restrict $f$ to a linear subspace $V$. Then,
after this substitution,  the restricted function
$f|_{t_i=1}$ will not determine anymore the restrictions $(f_h)|_{t_i=1}$  of the `old' components $f_h$.
That is, the new pair of lattices $(L_V,L'_V)=(L\cap V,L'\cap V)$
and the `old group' $H=L'/L$ become rather independent.  In such cases we will
 keep the old group $H=L'/L$ (and the `old' decomposition  $f_h$)
without asking any compatibility with $L'_V/L_V$.

\bekezdes\label{bek:LL'3} 
Since all the coordinates $a_{i,v}$ of $a_i$ are strict positive,
for any $Tf(\bt)=\sum_{l'}p_{l'}\bt^{l'}$  we get a well--defined  counting function of the coefficients,
$$l'\mapsto Q(l'):=\sum_{l''\not\geq l'} \, p_{l''}.$$
If $Tf=\sum_hTf_h$, then each $Tf_h$ determines a counting function  $Q_h$ defined in the same way.

If $H=L'/L$ and $Tf$ decomposes into $\sum_hTf_h$ under the law from (\ref{eq:fh}), then
\begin{equation}\label{eq:PCDEFa}
\sum_{l''\not\geq l'} \, p_{l''}\cdot [l'']=\sum _{h\in H}\,Q_h(l')[h].
\end{equation}
The definitions are motivated by formulae (\ref{eq:SUM}) and (\ref{eq:KV2}).
The functions $Q_h(l')$ will be studied in the next subsections via Ehrhart theory.

\subsection{Ehrhart quasipolynomials associated with \newline denominators of rational functions}\labelpar{ss:EP} \
First we consider the
case $d>0$, the special  case  $d=0$ will  be treated  in
\ref{ss:d0}.

\bekezdes\labelpar{bek:pol} {\bf The polytope associated with $\{a_i\}_{i=1}^d$.}\labelpar{constr:pol}
In order to run the Ehrhart theory we have first to fix the lattice $\calX$ and the representation
$\rho:\calX\to \fH$, cf. section \ref{ss:PPET}.
First, we set  $\calX=\Z^d$ and  $\alpha:\calX\to L'$ given by
$\alpha({\bf x})=\sum_{i=1}^dx_ia_i\in L'$.
In the sequel we  consider two possibilities for  $(\fH,\rho)$
which basically will cover all the cases we wish to study (equivariant/non--equivariant cases combined with situations  before  or after the reduction of  variables,
 see the comment in \ref{bek:LL'2}):

(a) $\fH=H=L'/L$ and $\rho$ is the composition
 $\calX\stackrel{\alpha}{\longrightarrow} L'\to L'/L$.

(b) $\fH=0$ and $\rho=0$.

This choice has an effect on the equivariant decomposition $f^e=\sum_gf_g[g]$ of $f$ too. In case (a)
usually we have  $\cG=H$ and the decomposition is given by \ref{eq:fh}. In case (b) we can take either
$\cG=0$ (this can happen e.g. when we forget the decomposition in case (a), and
we sum up all the components),
or we can take any $\cG$ (by specifying each $f_g$). In this latter case each fixed $f_g$
behaves like a function in the non--equivariant case $\cG=0$, hence can be treated in the same way.

Since the case (b) follows from case (a) (by forgetting the extra information from $\fH$), in the sequel
we provide the details for case (a). Hence let us assume $\fH=\cG=L'/L$.

 Consider the
matrix ${\bf A}$ with column vectors $|H|a_i$ and write ${\bf A}_v$ for its rows. Then the construction of (\ref{eq:POL})
can be repeated (eventually completing each ${\bf A}_v$ to assure the inequalities $x_i\geq 0$ as well).
For  $l\in \sum_vl_vE_v\in L$ consider
\begin{equation}\label{eq:Pv}
P_v^{\triangleleft} :=\{{\bf x}\in (\R_{\geq 0})^d\,:\, |H|\cdot
\sum_ix_ia_{i,v} < l_v\} \ \ \mbox{and} \ \
P^\triangleleft :=\bigcup_{v=1}^sP_v^\triangleleft.\end{equation}  The closure $P_v$ of $P_v^\triangleleft$
 is a dilated
convex (simplicial)  polytope depending on the one-dimensional parameter $l_v$.
Moreover,  $P^\triangleleft$
is described  via the partial ordering  of $L\otimes\R$ as the set
$\{l\,:\, \sum_ix_ia_i\not\geq l/|H|\}$. Since $L'\subset L/|H|$, we can
restrict ourself to the lattice $L'$ (preserving all the general
results of section \ref{ss:PPET}). Hence for any $l'\in L'$ we
set
\begin{equation}\label{eq:PL}
P^{(l'),\triangleleft }:=\{{\bf x}\in (\R_{\geq 0})^d\,:\, \sum_ix_ia_i\not\geq l'\}, \ \ \
P^{(l')}=\mbox{closure of }(
P^{(l'),\triangleleft }).\end{equation}
The combinatorial type of $P^{(l')}$ might vary with $l'$. Nevertheless, by definition, the facets will be
grouped for all different combinatorial types  by the same principle: we consider the
coordinate facets $F_i:=P^{(l')}\cap \{x_i=0\}$, $1\leq i\leq d$, and
we denote by $\calT$ the collection of all other facets. Hence $P^{(l'),\triangleleft }=
P^{(l')}\setminus  \cup_{F^{(l')}\in \calT}F^{(l')}$. The construction is
motivated by the summation from (\ref{eq:SUM})
(although in the general statements the choice of $\calT$ is irrelevant).

 Then \ref{th:JEMS} and \ref{PC} lead to the next counting function defined in the
 group ring $\Z[H]$ of $H$:
\begin{equation}\label{eq:calP}
\calL^e({\bf A},\calT,l'):=\sum_{h\in H} \calL_h({\bf A},\calT,l')\cdot [h]
:=\sum 1\cdot [l'']\in \Z[H],
\end{equation}
where the last sum runs over $l''\in \big(P^{(l')}\setminus \cup_{F^{(l')}\in \calT}
F^{(l')}\big)\cap L'=P^{(l'),\triangleleft}\cap L'$.

The corresponding  non--equivariant counting function, corresponding to $\cG=0$ is denoted by
$$ \calL_{ne}({\bf A},\calT,l'):=\sum_{h\in H} \calL_h({\bf A},\calT,l')\in \Z.$$

Similarly, we set $\calL^e({\bf A},\calF\setminus \calT,l')$ too. For both of them
Theorem \ref{th:REC} applies.

By the very construction, we have the following identity. Consider the
 equivariant Taylor expansion at the origin of the function determined by the {\it denominator of $f$},  namely
\begin{equation}\label{eq:Taylorden}
\widetilde{f}^e(\bt)=\frac{1}{\prod_{i=1}^d (1-[a_i]\bt^{a_i})}=
\sum_{l''}\widetilde{p}_{l''}\bt^{l''}\cdot [l'']\in \Z[[\bt^{1/|H|}]]\, [H].
\end{equation}
Note that since all the $\{E_v\}$--coefficients of each $a_i$ are strict positive, for any
$l'\in L'$ the   set $\{l''\,:\, \widetilde{p}_{l''}\not=0, l''\not\geq l'\}$ is finite. Then, by the above construction,
\begin{equation}\label{eq:PCDEFaden}
\sum_{l''\not\geq l'} \, \widetilde{p}_{l''}\cdot [l'']=\calL^e({\bf A},\calT,l').
\end{equation}

\bekezdes\labelpar{bek:combtypes}{\bf Combinatorial types,
chambers.}  Next, we wish to make precise the {\it
combinatorial stability} condition. The result of Sturmfels
\cite{Str}, Brion--Vergne \cite{BV}, Clauss--Loechner \cite{CL} and
Szenes--Vergne \cite{SZV} implies that $\calL^e$ from (\ref{eq:PCDEFaden})
(that is, each $\calL_h$)
is a {\it piecewise
quasipolynomial} on $L'$: the parameter space $L\otimes \R$
decomposes into several chambers, the restriction of $\calL^e$ on
each chamber is a quasipolynomial, and $\calL^e$ is continuous. The chambers
are described as follows.

Notice that the combinatorial type of $P^{(l')}$ in (\ref{eq:PL})
vary in the same way as the closure of its {\it convex} complement in $\R_{\geq
0}^d$, namely
\begin{equation}\label{eq:PCOMP}
\{ {\bf x}\in (\R_{\geq 0})^d\,:\, \sum_ix_ia_i\geq l'\},
\end{equation}
since both are  determined by their common boundary $\calT$. The
inequalities  of (\ref{eq:PCOMP}) can be viewed as a {\it vector
partition}  $\sum_ix_ia_i+ \sum_v y_v(-E_v)=l'$, with $x_i\geq 0$
and $y_v\geq 0$.  Hence, according to the above references, we
have the following chamber decomposition of $L\otimes \R$.

Let ${\bf M}$  be the matrix with column vectors $\{a_i\}_{i=1}^d$
and $\{-E_v\}_{v=1}^s$. A subset $\sigma$ of indices of columns is
called {\it basis} if the corresponding columns form a basis of
$L\otimes \R$; in this case we write   $Cone({\bf M}_{\sigma})$
for the positive closed cone generated by them.  Then the chamber
decomposition is the polyhedral subdivision of $L\otimes \R$
provided by the common refinement of the cones $Cone({\bf
M}_{\sigma})$, where $\sigma$ runs all over the basis. {\it A chamber
is a closed cone of the subdivision whose interior is non--empty.}
Usually we denote them by $\calC$, let their
index set (collection) be $\mathfrak{C}$.

We will need the associated {\it disjoint} decomposition of $L\otimes
\R$ with relative open cones as well. A typical element  of this disjoint decomposition
is the {\it
relative interior of an intersection of type $\cap_{\calC\in
\mathfrak{C}'}\calC$}, where $\mathfrak{C}'$ runs over the subsets
of $\mathfrak{C}$. For these cones we use the notation
$\calC_{op}$.

Each chamber $\calC$ determines an open cone, namely its interior.
And, conversely, each top dimensional open cone determines a
chamber $\calC$, namely its closure.

 The next theorem is the
direct consequence of \cite[4.4]{BV}, \cite[0.2]{SZV} and
(\ref{th:REC}) using the additivity of the Ehrhart quasipolynomial
on the suitable convex parts of $P^{(l')}$. (We state it  for our
specific facet--collection $\calT$, the case which will be used
later, but it is true for any other facet--decomposition of the
boundary whenever  $\cup_{F^{(l')}\in \calT}F^{(l')}$ is
homeomorphic to a $(d-1)$--manifold.)
\begin{theorem}\labelpar{th:PQP}  (a) For each relative open cone  $\calC_{op}$ of $L\otimes \R$,
 $P^{(l')}$ is combinatorially stable, that is, the polytopes
  $\{P^{(l')}\}_{l'\in \calC_{op}}$ have the same combinatorial type.
 Therefore, for any fixed $h\in H$, the restrictions
 $\calL^{\calC_{op}}_h({\bf A},\calT)$ and  $\calL^{\calC_{op}}_h({\bf A},\calF\setminus \calT)$
to $\calC_{op}$  of
  $\calL_h({\bf A},\calT)$ and  $\calL_h({\bf A},\calF\setminus \calT)$
 respectively  are quasipolynomials.

(b) These quasipolynomials have a continuous extension to the
closure of $\calC_{op}$. Namely, if $\calC_{op}'$ is in the
closure of $\calC_{op}$, then
 $\calL^{\calC_{op}'}_h({\bf A},\calT)$ is the restriction to
 $\calC_{op}'$ of the (abstract) quasipolynomial  $\calL^{\calC_{op}}_h({\bf
 A},\calT)$. (Similarly for $\calL^{\calC_{op}}_h({\bf A},\calF\setminus
 \calT)$.)

In particular, for any chamber $\calC$ one has a well--defined
quasipolynomial  $\calL^{\calC}_h({\bf A},\calT)$, defined as
$\calL^{\calC_{op}}_h({\bf A},\calT)$, where $\calC_{op}$ is the
interior of $\calC$, which equals $\calL_h({\bf A},\calT)$ for
all points of $\calC$.

This also shows that for any two chambers $\calC_1$ and $\calC_2$
one has the continuity property
\begin{equation}\label{eq:CONTINa}
\calL_h^{{\calC_1}}({\bf A},\calT)|_{\calC_1\cap \calC_2}=
\calL_h^{{\calC_2}}({\bf A},\calT)|_{\calC_1\cap
\calC_2}.\end{equation}

(c) $\calL^{\calC}_h({\bf A},\calT)$ and $\calL^{\calC}_{-h}({\bf A},\calF\setminus \calT)$,
as abstract quasipolynomials associated with a fixed chamber $\calC$,   satisfy the reciprocity
 $$\calL^{\calC}_h({\bf A},\calT,l')=(-1)^d\cdot\calL^{\calC}_{-h}({\bf A},\calF\setminus \calT,-l').$$

\end{theorem}

\noindent
We have the following consequences regarding the counting function $l'\mapsto Q_h(l')$ of $f^e(\bt)$
defined in (\ref{eq:PCDEFa}):

\begin{corollary}\labelpar{cor:Taylor}
(a)  $Q_h$ is a piecewise quasipolynomial.
Indeed, for any $h\in H$ and $l'\in L'$
\begin{equation}\label{eq:IND}
Q_h(l')=\sum_k\iota_k\cdot\calL_{h-[b_k]}({\bf A},\calT,l'-b_k).
\end{equation}   
In particular, the right hand side of (\ref{eq:IND}) is independent
of the representation of $f$ as in (\ref{eq:func}) (that is, of the
choice of $\{b_k,\,a_i\}_{k,i}$), it depends only on the rational function $f$.

(b) Fix a chamber $\calC$  of $L\otimes \R$, cf. \ref{th:PQP}, and for any $h\in H$ define
the quasipolynomial
\begin{equation}\label{eq:qbar}
\overline{Q}_h^\calC(l'):=
\sum_k\iota_k\cdot\calL^\calC_{h-[b_k]}({\bf A},\calT,l'-b_k).
\end{equation}
Then  the restriction of
$Q_h(l')$ to $\cap_k(b_k+\calC)$ is a quasipolynomial, namely
\begin{equation}\label{ex:qbar2}
Q_h(l')=\overline{Q}_h^\calC(l') \ \mbox{ on \
$\bigcap_k(b_k+\calC)$.}\end{equation} Moreover,
there exists $l'_*\in\calC$ such that $l'_*+\calC\subset \cap_k(b_k+\calC)$.

 (Warning:
$\calL^\calC_{h-[b_k]}({\bf A},\calT,l'-b_k)\not=\calL_{h-[b_k]}({\bf A},\calT,l'-b_k)$
unless $l'-[b_k]\in \calC$.)

(c) For any fixed $h\in H$,
the quasipolynomial  $\overline{Q}_h^\calC(l')$ satisfies  the following property:
for any $l'\in L'$ with $[l']=h$, and any $q\in\square$ (the semi-open unit cube), one has
\begin{equation}\label{eq:cccc}\overline{Q}_h^\calC(l')=\overline{Q}_h^\calC(l'-q).\end{equation}
In particular, by taking $l'=q=r_h$:
\begin{equation}\label{eq:cccc2}\overline{Q}_h^\calC(r_h)=\overline{Q}_h^\calC(0).\end{equation}
\end{corollary}
\begin{proof}
For (a) use (\ref{eq:PL}) and the fact that $b_k+\sum x_ia_i\not\geq l'$ if and only
if $\sum x_ia_i\not\geq l'-b_k$. Since the coefficients of the
Taylor expansion depend only on $f$, the second sentence  follows too.

For (b) use part (a)  and
the fact that $\calC\cap\bigcap_k(b_k+\calC)$ contains a set of type $l_*'+\calC$.

(c) Consider those values $l'$ in some $l'_*+\calC$ for which  all elements of type
$l'-b_k$ and $l'-q-b_k$ are in $\calC$. For these values $l'$, (\ref{eq:cccc})
follows from the identity
 $P^{(l'),\triangleleft}\cap\rho^{-1}(h)= P^{(l'-q),\triangleleft}\cap\rho^{-1}(h)$
whenever $[l']=h$.
This is true since  for any $l''$ with
$[l'']=[l']$, $l''\geq l'$ is equivalent with $l''\geq l'-q$. Indeed, taking $y=l''-l'$,
this reads as follows: for any $y\in L$, $y\geq 0$ if and only if $y\geq -q$.

Now, if two quasipolynomials agree on $l'_0+\calC$ then they are equal.
\end{proof}

\begin{remark}\labelpar{re:Szenes} Thanks to  \cite[Theorem 0.2]{SZV},
the continuity property \ref{eq:CONTINa} has the following
extension (coincidence of the quasipolynomials on neighboring strips).
Set $\Box({\bf A}):=\sum_i [0,1)a_i$. Then for any two chambers $\calC_1$ and $\calC_2$,
and  $S:=(-\Box({\bf A})+\calC_1) \cap (-\Box({\bf
A})+\calC_2)$
\begin{equation}\label{eq:CONTINb}
\calL_h^{{\calC_1}}({\bf A},\calT)|_{S}= \calL_h^{{\calC_2}}({\bf
A},\calT)|_{S}.\end{equation}
\end{remark}

\bekezdes\labelpar{ss:d0} {\bf The $d=0$ case.} All the above
properties can be extended for $d=0$ as well. Although the
polytope constructed in \ref{eq:PL} does not exist,  we can
look at the polynomial  $f(\bt)=\sum_k \iota_k
\bt^{b_k}$ itself. Then using notation of (\ref{eq:PCDEFa}) we set
$$\sum _{h\in H}\,Q_h(l')[h]=
\sum_{l''\not\geq l'} \, p_{l''}\cdot [l'']=\sum_{\{k \,:\, b_k\not\geq l'\}}\iota_k [b_k].$$
Moreover, we have  the chamber decomposition of $L\otimes\R$ defined by
$\{-E_v\}_{v=1}^s$ via the same principle as above. This means two chambers:
 $\calC_0:=\R_{\geq 0}\langle -E_v
\rangle$ and $\calC_1$,  the closure of the  complement of $\calC_0$ in $\R^s$.
Then $Q_h(l')=\sum_{\{k \,:\, [b_k]=h\}}\iota_k$ on
$\cap_k(b_k+\calC_1)$ and $0$ on $\cap_k(b_k+\calC_0)$.

\subsection{Multivariable equivariant periodic constant}\labelpar{ss:mpc}

We consider the situation of \ref{bek:LL'} and \ref{bek:pol}(a).
For each $h\in H$ define  $r_h\in L'$ as in \ref{ss:11}.
\begin{definition}\labelpar{def:pc}
Let  $\calK\subset L'\otimes \R$ be a
closed real cone whose affine closure aff$(\calK)$ has positive dimension.
For any $h\in H$ we assume
that there exist

$\bullet$ $l'_*\in \calK$

$\bullet$ a sublattice $\widetilde{L}\subset L$ of finite index, and

$\bullet$ a quasipolynomial $l'\mapsto \widetilde{Q}_h(l')$,
defined on  $\widetilde{L}\cap{\mathrm{aff}}(\calK)$
such that
\begin{equation}\label{eq:PCDEFx}
Q_h(l')=\widetilde{Q}_h(l') \ \ \mbox{
for any $\widetilde{L}\cap(l'_*+\calK)$.}
\end{equation}
Then we define the {\it equivariant periodic constant} of $f$ associated with   $\calK$  by
\begin{equation}\label{eq:PCDEF}\mathrm{pc}^{e,\calK}(f)=\sum_{h\in H}\,
\mathrm{pc}^{\calK}_h(f)\cdot [h]:= \sum_{h\in H}\,
\widetilde{Q}_h(0)\cdot [h]\in \Z[H],\end{equation} and  we say
that {\it $f$ admits a periodic constant in $\calK$}. (Sometimes
we will use the same notation for the real  cone $\calK$ and for
its lattice points  $\calK\cap L'$ in $L'$.)
\end{definition}
\begin{remark}
The above definition is independent of the choice of the
sublattice $\widetilde{L}$: it can be replaced by any
sublattice  of finite index. The advantage of such sublattices is
that convenient restrictions of $Q_h$ might  have nicer forms which are easier to
compute. The choice of  $\widetilde{L}$ corresponds to
the choice of $p$ in \ref{PC}, and it is responsible for the name
`periodic' in the name of $\mathrm{pc}^{e,\calK}(f)$.
\end{remark}
\begin{proposition}\labelpar{prop:pc}
(a)  Consider the chamber decomposition of $L\otimes \R$
given by the denominator $\prod_i(1-\bt^{a_i})$ of $f$
as in Theorem~\ref{th:PQP}.
Then $f$ admits  a periodic constant in each chamber $\calC$ and
\begin{equation}\label{eq:PCC}
{\mathrm{pc}}_h^\calC(f)\, = \,
\overline{Q}_h^\calC(r_h)=\overline{Q}_h^\calC(0).
\end{equation}
(b) If two functions $f_1$ and $f_2$ admit periodic constant in  some cone  $\calK$,
then the same is true for  $\alpha_1f_1+\alpha_2f_2$  and
$$\mathrm{pc}^\calK (\alpha_1f_1+\alpha_2f_2)=\alpha_1
\mathrm{pc}^\calK (f_1)+\alpha_2\mathrm{pc}^\calK (f_2) \ \ \ \ \
(\alpha_1,\,\alpha_2\in \C).$$ (c) If $f$ admits periodic
constants in two (top dimensional) cones $\calK_1$ and $\calK_2$,
and the interior $int(\calK_1\cap\calK_2)$ of the intersection
$\calK_1\cap\calK_2$ is non--empty, then $\mathrm{pc}^{\calK_1}(f)=
\mathrm{pc}^{\calK_2}(f)$.

In particular, if $\{\calC_i\}_{i=1,2}$ are two chambers as in (a), and $f$ admits a periodic constant
in $\calK$,  and $int(\calC_i\cap\calK)\not=\emptyset$ ($i=1,2$), then
$\mathrm{pc}^{\calC_1}(f)= \mathrm{pc}^{\calC_2}(f)$.
\end{proposition}
\begin{proof}
For (a) use Corollary \ref{cor:Taylor};   (b) is clear. For (c) we
can assume that $\calK_2\subset \calK_1$ (by considering $\calK_i$
and $\calK_1\cap\calK_2$). Then if $Q_h$ is quasipolynomial on
$l'_1+\calK_1$  (with $l'_1\in\calK_1$), then $(l'_1+\calK_2)\cap
\calK_2$ contains a set of type $l_2'+\calK_2$ with
$l'_2\in\calK_2$, on which one can take the restriction of the
previous quasipolynomial.
\end{proof}

\begin{remark}\labelpar{rem:fh}
  Note that in the rational presentation of $f$ we
might assume that  $a_i \in L$ for all $i$. Indeed, take  $o_i\in
\Z_{>0}$ such that $o_ia_i\in L$, and amplify the fraction by
 $\prod_i(1-\bt^{o_ia_i})/(1-\bt^{a_i})$. Therefore, for each $h$ we can
write $f_h(\bt)$ in the form
$$f_h(\bt)=\bt^{r_h}\sum_k \iota_k \cdot\frac{\bt^{\overline{b}_k}}{\prod_i(1-\bt^{a_i})},$$
where $a_i,\ \overline{b}_k \in L$, hence $f_h(\bt)/\bt^{r_h}\in \Z[[\bt]][\bt^{-1}]$.
Then if we consider the
non--equivariant periodic constant ${\rm pc}^\calC$ of $f_h(\bt)/\bt^{r_h}$,
\ref{eq:PCDEFa}, \ref{ex:qbar2} and \ref{eq:PCC} imply that
${\rm pc}_h^\calC(f(\bt))={\rm pc}^\calC(f_h(\bt)/\bt^{r_h})$ for all chambers
$\calC$ associated with $\{a_i\}_i$.
\end{remark}

 \begin{example}\labelpar{ex:1} \
Assume that $L=L'=\Z$ and $\calK=\R_{\geq 0}$, and consider
$S(t)$ as in \ref{PC}. If $S(t)$ admits  a periodic constant in
$\calK$, then $\mathrm{pc}^{\calK}(S)=\mathrm{pc}(S)$, where $\mathrm{pc}(S)$ is
the periodic constant
defined in \ref{PC}.
\end{example}

\begin{example}\labelpar{ex:2}
(a)  (The $d=0$ case) \ Assume that
$f(\bt)=\sum_{k=1}^r\iota_k\bt^{b_k}$. Then, using \ref{ss:d0}
(and its notation), $\mathrm{pc}^{e,\calC_0}(f)=0$ and
$\mathrm{pc}^{e,\calC_1}(f)=\sum_{k=1}^r\iota_k[b_k]\in\Z[H]$.

(b) Assume that the rank is $s=2$ and $f(\bt)=\bt^b/(1-\bt^a)$,
with both entries $(a_1,a_2)$ of $a$ positive.
We assume that $a\in L$ while $b\in L'$. Again, for $h\not=[b]$ the counting function,
hence its periodic constant too, is zero. Assume $h=[b]$, and write
$b=(b_1,b_2)$. Then the denominator provides three chambers:
$\calC_0:=\Z_{\geq  0}\langle -E_1,-E_2\rangle$,
$\calC_1:=\Z_{\geq  0}\langle a,-E_2\rangle$, $\calC_2:=\Z_{\geq
0}\langle a,-E_1\rangle$. Then the three quasipolynomials for
$1/(1-\bt^a)$ are $\calL^{\calC_0}_h=0$, \
$\calL^{\calC_i}_h(n_1,n_2)=\lceil n_i/ a_i\rceil$; hence
 ${\mathrm{pc}}^{\calC_0}_h(f)=0$, \
${\mathrm{pc}}^{\calC_i}_h(f)=\lceil -b_i/a_i\rceil$  ($i=1,2$).
In particular, ${\mathrm{pc}}^{\calC}_h(f)$, in general, depends on
the choice of $\calC$.

(c) Assume that $L=L'$ and
$f(t)=\frac{t_1^{b_1}t_2^{b_2}}{(1-t_1t_2)(1-t_1^2t_2)}$.  Then
the  chambers associated with the denominator are:
$\calC_0:=\R_{\geq 0}\langle -E_1,-E_2\rangle$, $\calC_2:=\R_{\geq
0}\langle -E_1,(1,1)\rangle$, $\calC:=\R_{\geq 0}\langle
(1,1),(2,1)\rangle$ and $\calC_1:=\R_{\geq 0}\langle
(2,1),-E_2\rangle$. Then, by a computation,
\begin{equation}
 \begin{array}[l]{ll}
    \calL^{\calC_0}=0; & \calL^{\calC_2}(l_1,l_2)=\frac{l_2^2}{2}+\frac{l_2}{2};\\
    \calL^{\calC}(l_1,l_2)=\frac{l_1^2}{2}+l_2^2+\frac{l_1}{2}-l_1
    l_2; &
    \calL^{\calC_1}(l_1,l_2)=\frac{l_1^2}{4}+\frac{l_1}{2}+\frac{1+(-1)^{l_1+1}}{8}.\\
 \end{array}
\end{equation}
Hence, by Proposition \ref{prop:pc} and (\ref{eq:qbar}), one has
$\mathrm{pc}^{\calC_*}(f)= \calL^{\calC_*}(-b_1,-b_2)$.
\end{example}

\begin{example}\labelpar{ex:NAM} {\bf Normal affine monoids.} \
Consider the following objects (cf.  \ref{bek:LL'}): a lattice $L$
with fixed bases $\{E_v\}_{v=1}^d$ (hence $s=d$) and with induced
partial ordering $\leq$, $L'\subset L\otimes \Q$ an overlattice
with finite abelian quotient $H:=L'/L$ and projection $\rho:L'\to
H$. Furthermore,  let  $\{a_i\}_{i=1}^d$ be linearly independent
vectors in $L'$ with all their $\{E_v\}$--coordinates positive.
Let $\calK$ be the positive real cone generated by the vectors
$\{a_i\}_i$, and consider the Hilbert series of $\calK$
$$f(\bt ):=\sum_{l'\in \calK\cap L'}\bt^{l'}.$$
Since $\calK$ depends only on the rays generated by the vectors
$a_i$, we can assume that $a_i\in L$ for all $i$.

Set  $\Box({\bf A})= \sum_{i=1}^d [0,1) a_i$ as above, and
consider the monoid $M:=\Z_{\geq 0}\langle a_i \rangle$ (cf.  e.g.
\cite[2.C]{BG}). Then the normal affine monoid
$\calK\cap L'$ is a module over  $M$ and if we set $B:=\Box({\bf
A})\cap L'$, \cite[Prop. 2.43]{BG} implies that
$$\calK\cap L'=\bigsqcup_{b\in B}b+M.$$
In particular, $f(\bt)$ equals $\sum_{b\in B}\bt^b/\prod_{i=1}^d (1-\bt^{a_i})$ and has the
form considered in \ref{ss:GENRF}. If the rank $d$ is $\geq 3$ then $\calK$ usually is cut
in more chambers. Indeed, take e.g. $d=3$, $a_i=(1,1,1)+E_i$ for $i=1,2,3$. Then $\calK$ is cut in its
 barycentric subdivision.  Nevertheless,  if $d= 2$ then $\calK$ consists of  a unique chamber and $f$ admits a periodic
 constant in $\calK$. Indeed, one has:
\begin{lemma}\labelpar{lem:d2}
 If $d=2$ then ${\mathrm pc}^{\calK}_h(f)=0$ \,for all $h \in H$.
\end{lemma}
\begin{proof}
It is elementary to see that $\calK$ is one of the chambers (use the
construction from  \ref{bek:combtypes}). Take
$B=\{b_k\}_k$, and write $f=\sum_kf_k$, where
$f_k=\bt^{b_k}/(1-\bt^{a_1})(1-\bt^{a_2})$. The only relevant
classes $h\in H$ are given by $\{[b_i]: b_i\in B\}$, otherwise
already the Ehrhart quasipolynomials are zero (since $a_i\in L$). Fix such a class
$h=[b_i]$. Let $\calL^{\calK}_{h}(\calT)$ be the quasipolynomial
associated with the chamber $\calK$ and the denominator of $f$.
Then, by (\ref{eq:PCC}) and (\ref{eq:qbar}),
$\mathrm{pc}_h^{\calK}(f_k)=\calL^{\calK}_{[b_i-b_k]}(\calT)(-b_k)$.
This, by Reciprocity Law \ref{th:PQP}(c) equals
$\calL^{\calK}_{[b_k-b_i]}(\calF\setminus \calT)(b_k)$. Again,
since the denominator is a series in $L$, for $[b_k-b_i]\not=0$
the series is zero; so we may assume $[b_k-b_i]=0$.  But, since
$b_k\in \calK$, the value $\calL^{\calK}_0(\calF\setminus
\calT)(b_k)$  of the quasipolynomial carries its geometric
meaning, it is the cardinality of the set  $\{m=n_1a_1+n_2a_2\,:\,
n_1>0,\, n_2>0,\, m\not>b_k\}$. But since for any such $m$ one has
$m\geq a_1+a_2>b_k$, contradicting $m\not>b_k$, this set is empty.
\end{proof}
\end{example}

\begin{example}\labelpar{ex:AM} {\bf General  affine monoids of rank $d=2$.} \
 Consider the situation of Example \ref{ex:NAM} with $d=2$, and let $N$ be a submonoid of
$\widehat{N}=\calK\cap L'$ of rank 2, and we also assume that
$\widehat{N}$ is the normalization of $N$. Set $$f(\bt
):=\sum_{l'\in N}\bt^{l'}.$$ Then $f(\bt)$ is again of type
(\ref{eq:func}). Indeed, by \cite[Prop. 2.35]{BG},
$\widehat{N}\setminus N$ is a union of a finite family of sets of
type (I) $b\in \widehat{N}$, or (II) $b+\Z ka_i$, where $b\in
\widehat{N}$, $k\in\Z_{\geq 0}$, $i=1$ or 2.  Obviously, two sets
of type (II) with different $i$-values might have an intersection
point of type (I). In particular,
$$f(\bt)=\sum_{l'\in \widehat{N}}\bt^{l'}-\sum_i\frac{\bt^{b_{i,1}}}{1-\bt^{k_{i,1}a_1}}-
\sum_j\frac{\bt^{b_{j,2}}}{1-\bt^{k_{j,2}a_2}}+\sum_k(\pm
\bt^{b_k}).$$ Note that the periodic constant of the first sum is
zero by Lemma \ref{lem:d2}, and the others can easily be computed
(even with closed formulae) via Example \ref{ex:2}, parts (a) and
(b).

The computation shows that the periodic constant carries
information about the failure of normality of $N$ (compare with
the delta-invariant computation from the end of \ref{PC}).

The situation is similar when we consider a {\it semigroup} of $\widehat{N}$, that is, when we eliminate the
neutral element of the above $N$ as well.
\end{example}

\begin{example} {\bf Reduction of variables.}
The next statement is an example when the number of variables of the function $f$ can be reduced
in the procedure of the periodic constant computation. (For another reduction result, see Theorem \ref{th:REST}.)
 For simplicity we assume  $L'=L$.\end{example}
\begin{proposition}\labelpar{prop:pc2}
 Let $f(\bt)=\frac{\bt^{b}}{\prod_{i=1}^d (1-\bt^{a_i})}$ and
 assume that $b=\sum_{v=1}^s b_v E_v \in \calC$, where
 $\calC$ is  a chamber associated with the denominator.

 We consider  the
subset $Pos:=\{v \, : \, b_v>0\}$ with
cardinality $p$, and the projection $\R^s\to \R^p$,
defined by $(r_v)_{v=1}^s\mapsto (r_v)_{v\in Pos}$ and denoted by
$v\mapsto v^\reds $. Accordingly,  we set a new function
$f^\reds(\bz):=\frac{\bz^{b^\reds }} {\prod_{i=1}^d (1-\bz^{a_i^*\reds })}$ in
$p$ variables, and a new chamber $\calC^{\reds} :=\R_{\geq 0}\langle
\{w_j^\reds \}_j\rangle$, where $w_j$ are the generators of
\,$\calC=\R_{\geq 0}\langle \{w_j\}_j\rangle$. Then
$\mathrm{pc}^{\calC}(f)=\mathrm{pc}^{\calC^\reds }(f^\reds )$.
\end{proposition}
\begin{proof}
This is a direct application of  Theorem \ref{th:REC}(b).
 Indeed, by the Ehrhart--MacDonald--Stanley
reciprocity law,
 we get $\mathrm{pc}^{\calC}(f)=\calL^{\calC}({\bf A},
\calT,-b)=(-1)^d\cdot \calL^{\calC}({\bf A},\calF\setminus
\calT,b)$. Since $b \in \calC$, by the very definition of
$\calL^{\calC}({\bf A},\calF\setminus \calT)$,
this (modulo the sign) equals
the number of integral points of $P^{(b)}\setminus
\cup_{F^{(b)}\in \calF\setminus \calT} F^{(b)} \subset \R^d$.
But, if  $v\notin Pos$, i.e., $b_v\leq 0$, then in (\ref{eq:POL}) $P_v^{(b_v)}$  has only
non--positive integral points. Therefore we can
omit these polytopes without affecting the periodic constant. Then, this fact and
$b^\reds \in \calC^\reds$ imply that  $\mathrm{pc}^{\calC}$ can be computed as
$(-1)^d\calL^{\calC^\reds}({\bf A^\reds},\calF^\reds\setminus \calT^\reds,b^\reds)$.
\end{proof}

\begin{remark}\label{rem:SZenes}
Under the conditions of Proposition \ref{prop:pc2} we have the
following application of the statement from Remark \ref{re:Szenes}
(based on \cite{SZV}): {\it Assume that  $b \in \Box(\bf A)-\calC$ and
$b\geq 0$.  Then $\mathrm{pc}^{\calC}(f)=0$.} Indeed,
$\mathrm{pc}^{\calC}(f)=\calL^{C}({\bf
A},\calT,-b)=\calL^{C(-b)}({\bf A},\calT,-b),$ where $C(-b)$ is a
chamber containing $-b$. But  since $-b\leq 0$ one gets
$\calL^{C(-b)}({\bf A},\calT,-b)=0$ by \ref{prop:pc2}.

\bigskip

One of the key messages of the above examples (starting from
 \ref{ex:2}) is the following: `if $b$ is small compared with the $a_i$'s, then the periodic
constant is zero' (compare with \ref{PC} too).
\end{remark}

\subsection{The polynomial part in the $d=s=2$ case} \labelpar{ss:POLPART}\
In this case $\mathrm{rank}(L)=2$, and we have two vectors in the
denominator of $f$, namely $a_i=(a_{i,1},a_{i,2})$, $i=1,2$. We
will order them in such a way that $a_2$ sits in the cone of $a_1$
and $E_1$, that is, $\det {a_{1,1}\ a_{1,2}\choose a_{2,1}\
a_{2,2}}<0$. The chamber decomposition will be the following:
$\calC_0:=\R_{\geq 0}\langle -E_1,-E_2\rangle$, $\calC_2:=\R_{\geq
0}\langle -E_1, a_1\rangle$, $\calC:=\R_{\geq 0}\langle
a_1,a_2\rangle$ and $\calC_1:=\R_{\geq 0}\langle a_2,-E_2\rangle$
(the index choice is motivated by the formulae from
\ref{ex:2}(b)).

Our goal is to write any rational function (with denominator
$(1-\bt^{a_1})(1-\bt^{a_2})$) as a sum of $f^+(\bt)$ and
$f^-(\bt)$, such that $f^+\in\Z[L']$ (the  `polynomial part of
$f$'), and $\mathrm {pc}^{e,\calC}(f^-)=0$. This is a generalization
of the decomposition in the one--variable case discussed in
\ref{PC}, and will be a major tool in the computation of the
periodic constant in Section \ref{s:TN} for graphs with two nodes. 
The specific  form of the decomposition  is
motivated  by Examples \ref{ex:2}(b) and \ref{ex:NAM}.

As above, we set $\Box({\bf A})=[0,1)a_1+[0,1)a_2$ and for $i=1,2$
we also consider the strips $$\Xi_i:= \{b=(b_1,b_2)\in L\otimes
\R\ |\ 0\leq b_i<a_{i,i}\}.$$

\begin{theorem}\labelpar{th:POLPART}
(1) \ Any function
$f(\bt)=\big(\sum_{k=1}^r\iota_k\bt^{b_k}\big)/\prod_{i=1}^2
(1-\bt^{a_i})$ (with $\iota_k\in\Z$)
can be written as a sum $f(\bt)=f^+(\bt)+f^-(\bt)$,
where

\medskip

(a) \ $f^+(\bt)$ is a finite sum $\sum_{\ell} \kappa_\ell
\bt^{\beta_\ell}$, with $\kappa_\ell\in\Z$ and $\beta_\ell\in L'$;

(b) \ $f^-(\bt)$ has the form
\begin{equation}\label{eq:POLPART}
f^-(\bt)=\frac{\sum_{k=1}^r\iota_k\bt^{b'_k}}{\prod_{i=1}^2
(1-\bt^{a_i})}+
\frac{\sum_{i=1}^{n_1}\iota_{i,1}\bt^{b_{i,1}}}{1-\bt^{a_1}}+
\frac{\sum_{i=1}^{n_2}\iota_{i,2}\bt^{b_{i,2}}}{1-\bt^{a_2}},
\end{equation}
with $b'_k\in L'\cap \Box({\bf A})$ for all $k$, and   $b_{i,j}\in
L'\cap \Xi_j$ for any $i$ and $j=1,2$, and $\iota_k\, \iota_{i,1},\, \iota_{i,2}\in\Z$.

\medskip

(2) Consider a sum
\begin{equation}\label{eq:SUMPOL} \Sigma(\bt):=
\frac{Q(\bt)}{\prod_{i=1}^2 (1-\bt^{a_i})}+
\frac{Q_1(\bt)}{1-\bt^{a_1}}+
\frac{Q_2(\bt)}{1-\bt^{a_2}}+f^+(\bt),
\end{equation}
where $ Q(\bt):=\sum_{k=1}^r\iota_k\bt^{b'_k}$ with $b'_k\in
L'\cap \Box({\bf A})$ for all $k$;
$Q_j(\bt)=\sum_{i=1}^{n_1}\iota_{i,j}\bt^{b_{i,j}}$  with
 $b_{i,j}\in
L'\cap \Xi_j$ for any $i$ and $j=1,2$;  and finally $f^+\in\Z[L']$
is a polynomial as in part (a) above.

Then, if $\Sigma(\bt)=0$, then
$Q(\bt)=Q_1(\bt)=Q_2(\bt)=f^+(\bt)=0$.

In particular, the decomposition in part (1) is unique.

\medskip

 (3) The periodic constant of $f^-(\bt)$ associated
with the chamber $\calC$ is zero. Hence, in the decomposition (1)
one also has
$\mathrm{pc}^{e,\calC}(f)=\mathrm{pc}^{e,\calC}(f^+)=\sum_{\ell}
\kappa_\ell [\beta_\ell]\in\Z[H]$.
\end{theorem}

\begin{proof}
(1) For every $b_k\in L'$ we have a (unique) $b_k'\in L'\cap
\Box({\bf A})$ such that $b_k-b_k'\in \Z\langle a_1,a_2\rangle$.
Set $Q(\bt):=\sum_{k=1}^r\iota_k\bt^{b'_k}$. Then $f(\bt)-Q(\bt)/
\prod_{i=1}^2 (1-\bt^{a_i})$ is a sum of terms of type
$\bt^{b'}(\bt^{k_1a_1+k_2a_2}-1)/\prod_{i=1}^2 (1-\bt^{a_i})$.
This decomposes as a sum with terms of type $\bt^c/(1-\bt^{a_i})$.
Then for every such  expression,
there exists $c_i\in \Xi_i$ such that $(\bt^c-\bt^{c_i})/(1-\bt^{a_i})$
is as in (a).

Part (2) is again elementary.
First we show that $Q(\bt)=0$. For any $b'\in L'\cap\Box({\bf A})$
consider $\Xi_{b'}:=b'+\Z\langle a_1,a_2\rangle$. For any $P(\bt)=
\sum \iota_k\bt^{c'_k}$ write $P_{b'}(\bt)=
\sum_{c_k\in \Xi_{b'}} \iota_k\bt^{c'_k}$ for its part supported on $\Xi_{b'}$.
This decomposition can be done for $Q$, $Q_1$, $Q_2$ and $f^+$, hence for $\Sigma(\bt)$.
Note that it  is enough to prove (2) for such $\Sigma_{b'}(\bt)$, for a fixed $b'$.
Hence, we can assume that $\Sigma(\bt)$ is supported on some $\Xi_{b'}$, $b'\in L'\cap\Box({\bf A})$.
Since $\Xi_{b'}\cap \Box({\bf A})=\{b'\}$, in this case $Q(\bt)=\iota \,\bt^{b'}$.
Multiplying $\Sigma(\bt)$ by $\prod_{i=1}^2 (1-\bt^{a_i})$ and substituting $t_1=t_2=1$ we get $\iota=0$.
Hence $Q(\bt)=0$.

Next, consider the identity $(1-\bt^{a_2})Q_1(\bt)+(1-\bt^{a_1})Q_2(\bt)+
\prod_{i=1}^2 (1-\bt^{a_i})\cdot f^+(\bt)=0$. Since $\Z[t_1,t_2]$ is UFD and the polynomials
 $1-\bt^{a_1}$ and $1-\bt^{a_2}$ are relative primes, we get that $1-\bt^{a_i}$ divides
 $Q_i(\bt)$. This together with the support assumption of  $Q_i$ implies $Q_i=0$.

(3) The vanishing of the periodic constant of the first fraction of
$f^-$ follows from the proof of Lemma \ref{lem:d2}. The vanishing
of $\mathrm{pc}^{e,\calC}$ of the other two fractions follows from
Example  \ref{ex:2}(b). For the last expression see Example
\ref{ex:2}(a).
\end{proof}
\begin{remark}\label{rem:POL} (a)
Part (a) of the proof provides an algorithm how one finds the  decomposition.

(b) Since $\mathrm{pc}^{e,\calC}(f^-)=0$ by (3), the above decomposition $f=f^++f^-$ is
well--suited for computing the periodic constant of $f$ associated with chamber $\calC$ via
$f^+$.
\end{remark}

\section{The case associated with plumbing graphs}\labelpar{s:PLRF}

\subsection{The new construction. Applications of Section \ref{s:PC}.}\labelpar{ss:cor}\
Consider the topological setup of a surface singularity, as in
subsection \ref{ss:11}. The lattice $L$ has  a canonical basis
$\{E_v\}_{v\in \calv}$ corresponding to the vertices of the graph $G$.
We investigate  the periodic constant of the rational function $Z(\bt)$, defined  in
\ref{SW} from $G$. Since  $Z(\bt)$ has the form
(\ref{eq:func}), all the results of section \ref{s:PC} can be applied.
In particular, if $\cale=\{v\in\calv\,:\, \delta_v=1\}$ denotes the set of {\em ends} of the
graph, then ${\bf A}$ has column vectors $a_v=E_v^*$ for $v\in\cale$. Hence,
the rank of the lattice/space where the polytopes $P^{(l')}=\cup_vP_v$ sit is $d=|\cale|$,
and the convex polytopes $\{P_v\}$ are indexed by $\calv$. Furthermore, the dilation parameter $l'$
of the polytopes runs in a $|\cV|$--dimensional space.   In the sequel we will drop
the symbol ${\bf A}$ from $\calL^\calC_h({\bf A},\calT,l')$.

(The construction has some analogies with the construction of the splice--quotient singularities
\cite{nw-CIuac}: in that case  the equations of the universal abelian cover of the
singularity are written in $\C^d$, together with an action of $H$. Nevertheless, in the present situation,
we are not obstructed with the semigroup and congruence relations present in that theory.)

In this new construction, a crucial additional ingredient comes
from singularity theory, namely Theorem \ref{th:JEMS} (in fact,
this is the main starting point and motivation of the whole
approach). This combined with facts from Section \ref{s:PC} give:
 \begin{corollary}\labelpar{cor:4.1}
 Let $\calS=\calS_{\R}$ be the (real) Lipman cone $\{x\in \R^{|\cV|}: (x,E_v)\leq 0 \ \mbox{for all $v$}\}$.

  (a) The rational function
  $Z(\bt)$ admits a periodic constant in the cone $\calS$, which  equals the
   normalized Seiberg--Witten invariant
 \begin{equation}\label{eq:4.2}
 \mathrm{pc}^{\calS}_h(Z)=
-\frac{(K+2r_h)^2+|\cV|}{8}-\frsw_{-h*\sigma_{can}}(M).\end{equation}

(b)  Consider the chamber decomposition associated with the
 denominator  of $Z(\bt)$ as in Theorem \ref{th:PQP}, and let $\calC$ be a chamber
such that $int(\calC\cap \calS)\not=\emptyset$.
Then $Z(\bt)$ admits a periodic constant in $\calC$,
which  equals both  $\mathrm{pc}^{\calS}_h(Z)$ (satisfying  (\ref{eq:4.2})) and also
\begin{equation}\label{eq:PCC2}
{\mathrm{pc}}_h^\calC(Z)\, = \,
\sum_k\iota_k\cdot\calL^\calC_{h-[b_k]}(\calT,-b_k)=
\sum_k\iota_k\cdot\calL^\calC_{[b_k]-h}(\calF\setminus \calT,b_k).
\end{equation}
In particular, ${\mathrm{pc}}_h^\calC(Z)$ does not depend on the choice of \,$\calC$
(under the above assumption).
\end{corollary}
\begin{proof}
 Write  $l'=\tilde{l}+r_h$ with $\tilde{l}\in L$ in  (\ref{eq:SUM}).
Since
$\sum_{l\in L,\, l\not\geq 0}p_{l'+l}=
\sum_{l''\not\geq \tilde{l},\, [l'']=r_h}p_{l''}$,
(a)  follows from Theorem \ref{th:JEMS}. For  (b) use Corollary \ref{cor:Taylor} and
 Proposition \ref{prop:pc}.
\end{proof}

We note that the Lipman cone $\calS$ can indeed be cut in several
chambers (of the denominator of $Z$). This can happen even in the
simple case of Brieskorn germs.  Below we provide such an example
together with several exemplifying details of the construction.

\begin{example}\labelpar{ex:tref} {\bf Lipman cone cut in several chambers.} \
Consider the 3-manifold $S^3_{-1}(T_{2,3})$ (where $T_{2,3}$ is the right-handed trefoil knot), or, equivalently,
 the link of the
hypersurface singularity $z_1^2+z_2^3+z_3^7=0$. Its
plumbing graph $G$ and matrix $-\frI^{-1}$ are:

\begin{picture}(200,75)(60,10)
\put(150,55){\circle*{3}}
\put(180,55){\circle*{3}}
\put(120,55){\circle*{3}}
\put(150,25){\circle*{3}}
\put(150,55){\line(-1,0){30}}
\put(150,55){\line(1,0){30}}
\put(150,55){\line(0,-1){30}}
\put(150,70){\makebox(0,0){$E_0$}}
\put(180,70){\makebox(0,0){$E_3$}}
\put(120,70){\makebox(0,0){$E_1$}}
\put(135,25){\makebox(0,0){$E_2$}}
\put(160,45){\makebox(0,0){$-1$}}
\put(105,55){\makebox(0,0){$-2$}}
\put(195,55){\makebox(0,0){$-7$}}
\put(165,25){\makebox(0,0){$-3$}}
\put(345,45){\makebox(0,0){$-\frI^{-1}=\begin{pmatrix}
42&21&14&6\\
21&11&7&3\\
14&7&5&2\\
6&3&2&1
\end{pmatrix}$}}
\end{picture}

\noindent
 where the row/column vectors of $-\frI^{-1}$ are $E_0^*$,
$E_1^*$, $E_2^*$ and $E_3^*$ in the $\{E_v\}$ basis. The polytopes defined in (\ref{eq:POL}), or in (\ref{eq:Pv}),
with parameter $l=(l_0,l_1,l_2,l_3)\subset \Z^4$, sit in $\R^3$. Let $u_1,u_2,u_3$ be the basis of
$\R^3$. Then the polytopes  are the following convex closures:
\begin{eqnarray*}
 P_0^{(l)}& = & conv\left(0,\left(l_0/21\right)u_1,\left(l_0/14\right)u_2,
\left(l_0/6\right)u_3\right)\\
 P_1^{(l)}& = & conv\left(0,\left(l_1/11\right)u_1,\left(l_1/7\right)u_2,
\left(l_1/3\right)u_3\right)\\
 P_2^{(l)}& = & conv\left(0,\left(l_2/7\right)u_1,\left(l_2/5\right)u_2,
\left(l_2/2\right)u_3\right)\\
 P_3^{(l)}& = & conv\left(0,\left(l_3/3\right)u_1,\left(l_3/2\right)u_2,
\left(l_3/1\right)u_3\right).
\end{eqnarray*}
Since $E_0^*+\varepsilon(-E_0)$ is in the interior of the (real) Lipman cone
for $0<\varepsilon \ll 1$, we get that the Lipman cone is cut in several chambers.
The periodic constant can be computed with any of them.
In fact, by the continuity of the quasipolynomials associated with the chambers, any quasipolynomial
associated with any ray in the Lipman cone, even if it is situated at its boundary, provides the
periodic constant.
One such degenerated polytope provided by a ray on the boundary of $\calS$ is of special interest. Namely,
if we take $l=\lambda E_0^* \in \calS$ for $\lambda>0$, then $P^{(l)}=\bigcup_{v=0}^3 P_v^{(l)}$ is
the same as $P_0^{(l)}=conv(0,2\lambda u_1,3\lambda u_2,7\lambda u_3)$.
%
Moreover, if  $\calC$ is any chamber which contains the ray
$\R_{\geq 0} E_0^*$ at its boundary, then for any $l=\lambda
E^*_0$ one has
$\calL^{\calC}({\bf
A},\calT,l)=\calL(\widetilde{P}_0,\calT,\lambda)$, 
 where the last
is the classical Ehrhart polynomial of the tetrahedron
$\widetilde{P}_0:=conv(0,2 u_1, 3u_2,7u_3)$. Here we witness an
additional  coincidence of $\widetilde{P}_0$ with the Newton
polytope $G^-_N$ of the equation $z_1^2+z_2^3+z_3^7$.

We compute $ \calL(\widetilde{P}_0,\calT,\lambda)$  as follows.
From (\ref{eq:KV2})--(\ref{eq:KV2b}) and Corollary
\ref{cor:Taylor}, we get that
\begin{equation}\label{eq:NEW}
\chi(\lambda E_0^*)+\mbox{geometric genus of }
\{z_1^2+z_2^3+z_3^7=0\}=\calL(\widetilde{P}_0,\calT,\lambda)-\calL(\widetilde{P}_0,\calT,\lambda-1).
\end{equation}
Since this geometric genus is 1, and the free term of
$\calL(\widetilde{P}_0,\calT,\lambda)$ is zero (since for $\lambda
=0$ the zero polytope with boundary conditions contains no lattice
point), and $-K=2E_0+E_1+E_2+E_3$, we get that
$\calL(\widetilde{P}_0,\calT,
\lambda)=7\lambda^3+10\lambda^2+4\lambda$.
We emphasize that a formula as in (\ref{eq:NEW}), realizing a bridge between
the Riemann--Roch expression $\chi$ (supplemented with the geometric genus)
and the Ehrhart polynomial of the Newton diagram, was not known
for Newton non--degenerate germs.
\medskip

 In the sequel we will provide several examples, when
the Newton polytope is not even defined.
\end{example}

\subsection{Example. The case of lens spaces}\labelpar{ex:lens}\
As we will see in  Theorem \ref{th:REST},
the complexity of the problem depends basically  on the number of nodes  of  $G$. 
 In this subsection we treat the case when there are no nodes at all, that is $M$ is a
lens space. In this case the numerator of the rational function $f(\bt)$ is 1, hence
everything is described by the 2--dimensional polytopes determined by the denominator.
In the literature there are several results about lens spaces fitting in the present program,
here we collect the relevant ones completing with the new interpretations.
This subsection also serves as a preparatory part, or model,  for the study of chains of arbitrary graphs.

Assume that the plumbing graph is
\begin{picture}(170,20)(80,15)
\put(100,20){\circle*{3}}
\put(130,20){\circle*{3}}
\put(200,20){\circle*{3}}
\put(230,20){\circle*{3}}
\put(100,20){\line(1,0){50}}
\put(230,20){\line(-1,0){50}}
\put(165,20){\makebox(0,0){$\cdots$}}
\put(100,30){\makebox(0,0){$-k_1$}}
\put(130,30){\makebox(0,0){$-k_2$}}
\put(230,30){\makebox(0,0){$-k_s$}}
\put(200,30){\makebox(0,0){$-k_{s-1}$}}
\end{picture}
with all $k_v\geq 2$, and
$p/q$ is expressed via the (Hirzebruch, or negative) continued fraction
\begin{equation}\label{eq:HCF}
[k_1,\ldots, k_s]=k_1-1/(k_2-1/(\cdots -1/k_s)\cdots ).
\end{equation}
Then $M$ is the lens space $L(p,q)$. We also define $q'$ by
$q'q\equiv 1$ mod $p$, and $0\leq q'<p$. Furthermore, we set $g_v=[E^*_v]\in H$.
Then $g_s$ generates $H=\Z_p$, and any element
of $H$ can be written as $ag_s$ for some $0\leq a<p$. Recall the definitions of
$r_h$ and $s_h$ from \ref{ss:11} as well. 

From the analytic point of view $(X,0) $  is a cyclic quotient singularity
$(\C^2,o)/\Z_p$, where the action is $\xi*(x,y)=(\xi x,\xi^q y)$ (here $\xi$ runs over $p$--roots of unity).

\bekezdes {\bf The Seiberg--Witten invariant.} Since $(X,0)$ is rational,
in this case $Z(\bt)=P(\bt)$ (cf. subsection \ref{FM}). Moreover, in (\ref{eq:KV2b})
 $H^1(\cO_{\widetilde{Y}})=0$, hence
\begin{equation}\label{eq:lens2}
 \frsw_{-h*\sigma_{can}}(M)=
-\frac{(K+2r_h)^2+|\cV|}{8}=-\frac{K^2+|\cV|}{8}+\chi(r_h).
\end{equation}
On the other hand, in  \cite{OSZINV,Nlat} a similar formula is proved for the Seiberg--Witten invariant:
one only has to replace in (\ref{eq:lens2}) $\chi(r_h)$ by $\chi(s_h)$.
In particular, for lens spaces, and for any $h\in H$ one has
\begin{equation}\label{eq:lens3}
\chi(r_h)=\chi(s_h).
\end{equation}
(Note that, in general, for other links, $\chi(r_h)>\chi(s_h)$ might happen, see Example~\ref{ex:sh}.
Here, (\ref{eq:lens3})
follows from the vanishing of the geometric genus of the universal abelian cover of $(X,0)$.)

In general,  the coefficients of
the representatives $s_{ag_s}$ and $r_{ag_s}$ ($0\leq a<p$)
are rather complicated arithmetical expressions;
for $s_{ag_s}$ see \cite[10.3]{OSZINV} (where $g_s$ is defined with opposite sign).
The value  $\chi(s_{ag_s})$ is computed in
\cite[10.5.1]{OSZINV} as
\begin{equation}\label{eq:lens1}
\chi(s_{ag_s}) =\frac{a(1-p)}{2p} +\sum_{j=1}^a \Big\{\frac{jq'}{p}\Big\}.
\end{equation}

For completeness of the discussion we also recall that
$K=E_1^*+E_s^*-\sum_vE_v$ and
\begin{equation}\label{eq:lens4}
(K^2+|\cV|)/4=(p-1)/(2p)-3\cdot \bms(q,p),
\end{equation}
cf. \cite[10.5]{OSZINV}, where $\bms(q,p)$ denotes the Dedekind sum
\[
\bms(q,p)=\sum_{l=0}^{p-1}\Big(\Big( \frac{l}{p}\Big)\Big)
\Big(\Big( \frac{ ql }{p} \Big)\Big), \ \mbox{where} \ \
((x))=\left\{
\begin{array}{ccl}
\{x\} -1/2 & {\rm if} & x\in {\R}\setminus {\Z}\\
0 & {\rm if} & x\in {\Z}.
\end{array}
\right.
\]
In particular,  $\frsw_{-h*\sigma_{can}}(M)$ is determined via the formulae
(\ref{eq:lens2}) -- (\ref{eq:lens4}).

The non--equivariant picture looks as follows:
 $\sum_h \frsw_{-h*\sigma_{can}}=\lambda$,
the Casson--Walker invariant of $M$, hence (\ref{eq:lens2}) gives
$$\lambda=-p(K^2+|\cV|)/8+\textstyle{\sum_h}\chi(r_h).$$
This is compatible with (\ref{eq:lens4}) and formulae
$\lambda(L(p,q))=p\cdot\bms(q,p)/2$ and $\sum_h\chi(r_h)=(p-1)/4-p\cdot \bms(q,p)$, cf.
\cite[10.8]{OSZINV}.

\bekezdes {\bf The polytope and its quasipolynomial.}
We compare the above data with Ehrhart theory.
In this case $Z(\bt)=(1-\bt^{E^*_1})^{-1}(1-\bt^{E^*_s})^{-1}$.
The vectors $a_1=E^*_1$ and $a_s=E^*_s$ determine the polytopes
$P^{(l')}$ and a chamber decomposition.

For $1\leq v\leq w\leq s$ let $n_{vw}$ denote the numerator of the continued fraction
$[k_v,\ldots,k_w]$ (or, the determinant of the corresponding bamboo subgraph).
For example, $n_{1s}=p$, $n_{2s}=q$ and $n_{1,s-1}=q'$. We also set
$n_{v+1,v}:=1$. Then $pE^*_1=\sum_{v=1}^sn_{v+1,s}E_v$ and $pE^*_s=\sum_{v=1}^sn_{1,v-1}E_v$.

In particular, for any $l'=\sum_vl'_vE_v\in\calS'$, the (non--convex) polytopes are
\begin{equation}\label{eq:lens5}
P^{(l')}=\bigcup_{v=1}^s\Big\{\,(x_1,x_s)\in\R_{\geq 0}^2\,:\, x_1n_{v+1,s}+x_sn_{1,v-1}\leq pl'_v\Big\}
\subset \R_{\geq 0}^2.
\end{equation}
The representation $\Z^2\stackrel{\rho}{\longrightarrow}\Z_p$ is $(x_1,x_s)\mapsto (qx_1+x_s)g_s$.

Though $P^{(l')}$ is a plane polytope, the direct computation of its equivariant Ehrhart multivariable polynomial
(associated with a chamber, or just with the Lipman cone)
 is still highly non--trivial. Here  we will rely again on  Theorem \ref{th:JEMS}.
On a subset of type $l'_0+\calS'$ the identity (\ref{eq:SUM}) provides the counting function. The right hand side
of (\ref{eq:SUM})  depends on all the
coordinates of $l'$, hence all the  triangles $P_v$ contribute in $P^{(l')}$. Since this can happen only in
a unique combinatorial way, we get that there is a chamber $\calC$ which contains the Lipman cone.
Let $\calL^{e,\calC}$ be its quasipolynomial, and  $\calL^{e,\calS}$ its restriction to $\calS$.
Since the numerator of $Z(\bt)$ is 1, $\overline{Q}^{\calC}_h=\calL^{\calC}_h$.
Since this agrees with the right hand side of (\ref{eq:SUM}) on a cone of type
$l'_0+\calS'$, and the Lipman cone is in $\calC$,  we get that
\begin{equation}\label{eq:lens6}
Q_h(l')=\overline{Q}^{\calC}_h(l')=\calL^{\calS}_h(l')=-\frsw_{-h*\sigma_{can}}(M)-\frac{(K+2l')^2+|\cV|}{8}
\end{equation}
for any $l'\in (r_h+L)\cap \calS'$ and  $h\in H$.
Using the identity (\ref{eq:lens2}), this reads as
\begin{equation}\label{eq:lens7}
\calL^{\calS}_h(\calT,l')=\chi(l')-\chi(r_h),
\ \ \ l'\in (r_h+L)\cap \calS'.
\end{equation}
Note that for any fixed $h$ and any $l'$ there exists a unique $q=q_{l',h}\in \square$ such that
$l'+q:=l''\in r_h+L$.
Indeed, take for $q$ the representative of $r_h-l'$ in $\square$. Then (\ref{eq:cccc})
and (\ref{eq:lens7}) imply
\begin{equation}\label{eq:Lg}
\calL^{\calS}_h(\calT,l')
=\calL^{\calS}_h(\calT,l'')=\chi(l'+q_{l',h})-\chi(r_h).\end{equation}
This formula emphasizes the quasi--periodic behavior of  $\calL^{\calS}_h(\calT,l')$ as well.

If $l'$ is an element of  $L$ 
then $q_{l',h}=r_h$, hence (\ref{eq:Lg}) gives in this case
\begin{equation}\label{eq:Lg2}
\calL^{\calS}_h(\calT,l)
=\chi(l+r_h)-\chi(r_h)=\chi(l)-(l,r_h) \ \ \ \mbox{for $l\in L\cap \calS$}.\end{equation}
In particular,
$\mathrm{pc}(\calL_h^\calS(\calT))=\chi(r_h)-\chi(r_h)=0$ (a fact compatible with  $H^1(\cO_{\widetilde{Y}})=0$).



\medskip

Even the non--equivariant case looks rather interesting.
Let $\calL^{\calS}_{ne}(\calT)=\sum_{h\in H}\calL^{\calS}_h(\calT)$ be the Ehrhart polynomial
of $P^{(l')}$ (with boundary condition $\calT$), where we count all the lattice points
independently of their class in $H$.
Then,  (\ref{eq:Lg2}) gives
for $ l\in L\cap \calS$
\begin{equation}\label{eq:lens8}
\calL^{\calS}_{ne}(\calT,l)=p\cdot \chi(l)-(l,\textstyle{\sum_h}r_h)=
-p\cdot
(l,l)/2-p\cdot (l,K)/2-(l,\textstyle{\sum_h}r_h).
\end{equation}
In fact, $\sum_hr_h$ can  explicitly be computed. Indeed, set
 $d_v={\mathrm{gcd}}(p,n_{1,v-1})$ and  $p_v=p/d_v$.  Then one
checks that $aE^*_s=\sum_vn_{1,v-1}\frac{a}{p}E_v$,
$r_h=\sum_v\big\{n_{1,v-1}\frac{a}{p}\big\}E_v$ and $\sum_hr_h=\sum_vd_v\frac{p_v-1}{2}E_v$.

The coefficients of the polynomial $\calL^{\calS}_{ne}(\calT,l)$
can be compared with the coefficients given by general theory
of Ehrhart polynomials applied for $P^{(l)}$. E.g., the leading coefficient gives
 \begin{equation*}\label{eq:lens9}
-p \cdot(l,l)/2= \ \mbox{Euclidian area of }\ P^{(l)}.
\end{equation*}
Knowing that in $P^{(l)}$ all the $P_v$'s contribute, and it depends on $s$ parameters, and the
intersection of their boundary is messy, the simplicity and conceptual form of
(\ref{eq:lens8}) is rather remarkable.

\section{Reduction theorems for $Z(\bt)$}\labelpar{s:REDZt}\
The number of terms in the denominator $\prod_i(1-\bt^{a_i})$
of the series equals the number of variables of the corresponding partition function
(associated with vectors $a_i$), and it is also
the rank of the lattice where the corresponding polytope sit.
In the case of the series $Z(\bt)$ associated with plumbing graph, this is the number of
{\it end vertices} of $G$.
On the other hand, the number of variables of $Z(\bt)$ is the number $|\calv|$
of vertices of $G$.
Furthermore, in the Ehrhart theoretical part, the associated (non--convex) polytope will be a
union of $|\calv|$ simplicial polytopes. Hence, the number of facets and the
complexity of the polytope increases considerably with the number of vertices as well.

Nevertheless, the Theorem \ref{th:REST} eliminates a part of this abundance of
parameters: it says that from the periodic constant point of view, the number of variables of the series,
and also the number of simplicial polytopes in the union, can be reduced 
to the number of {\it nodes} of the graph.
Hence, in fact, the complexity level can be measured by the number of nodes. 

We can do even more: if we apply the machinery of the Reduction Theorem \ref{red} from the previous 
chapter, one can reduce the number of variables of $Z(\bt)$ to the number of the chosen bad vertices 
of the graph $G$ (in the sense of \ref{ss:bv}). 

The first approach is purely combinatorial, using the specialty of $G$. However, the second 
uses the Reduction Theorem \ref{red}, i.e. the hidden geometry which measures the rationality of the 
graph.

\subsection{Reduction to the node variables}\labelpar{ss:REST}\
Let $\calN$ denote the set of nodes
as above. Let
$\calS_\calN$ be the positive cone
$\R_{\geq 0}\langle E^*_n\rangle_{n\in \calN}$  generated by the dual base elements indexed by $\calN$, and  $V_\calN:= \R\langle E^*_n\rangle_{n\in \calN}$
be its supporting linear subspace in $L\otimes \R$. Clearly  $\calS_\calN\subset \calS$.
Furthermore, consider  $L_\calN:=\Z\langle E_n\rangle_{n\in \calN}$ generated by the node
base elements, and
the projection $pr_\calN:L\otimes \R\to L_\calN\otimes \R$ on the node coordinates.

\begin{lemma}\label{lem:RES}
The restriction of $pr_\calN$ to $V_\calN$, namely $pr_\calN:V_\calN\to L_\calN\otimes \R$, is an isomorphism.
\end{lemma}
\begin{proof}
Follows from the negative definiteness of the intersection form of the plumbing, which guarantees that
any minor situated centrally on the diagonal is non--degenerate.
\end{proof}
Our goal is to prove that restricting the counting function to the subspace $V_\calN$, the non--node variables of $Z(\bt)$ and $Q(l')$ became non--visible, hence they can be eliminated. This fact will
provide a remarkable simplification in the periodic constant computation.
But, {\it before} any elimination--substitution,
we have first to decompose our series $Z(\bt)$ 
into $\sum_{h\in H}Z_h(\bt)[h]$
if we wish to preserve the information about all the $H$ invariants,
cf. the comment at the end of \ref{bek:LL'2}. 

\begin{theorem}\labelpar{th:REST}
(a) The restriction of $\calL^e_h({\bf A}, \calT,l')$ to $V_\calN$
 depends only on those coordinates which are indexed by the nodes (that is, it
depends only on $pr_\calN(l')$ whenever $l'\in V_\calN$).

(b) The same is true for the counting function $Q_h$ associated with $Z_h(\bt)$ as well.
In other words, if we consider the restriction
$$Z_h(\bt_\calN):=Z_h(\bt)|_{t_v=1\ \mbox{\tiny{for all $v\not\in\calN$}}}$$
then for any $l'\in L_\calN$, the counting functions $\sum_{l''\not\geq l'} p_{l''}[l'']$
of  $Z_h(\bt)$ and $Z_h(\bt_\calN)$ are the same.

(c) Consider the chamber decomposition of $\calS_\calN$ by intersections of type
$\calC_\calN:=\calC\cap\calS_\calN$, where $\calC$ denotes a chamber (of $Z$)
such that $int(\calC\cap \calS)\not=\emptyset$, and
the intersection of $\calC$ with the relative interior of $\calS_\calN$ is also
non--empty.
Then  \begin{equation}\label{eq:REDPC}
{\mathrm{pc}}^\calC(Z_h(\bt))=
 {\mathrm{pc}}^{\calC_\calN}(Z_h(\bt_\calN)).
 \end{equation}
\end{theorem}
The theorem applies as follows. Assume that we are interested in the computation of
$\mathrm{pc}^\calC_h(Z(\bt))$ for some chamber $\calC$
(e.g. when $\calC\subset \calS$, cf. Corollary \ref{cor:4.1}). Assume that
$\calC$ intersects the relative interior of $\calS_\calN$. Then,
the restriction to  $\calC\cap \calS_\calN$ of the quasipolynomial associated with
$\calC$ has two properties: it still preserves sufficient information to
determine $\mathrm{pc}^\calC_h(Z(\bt))$
(via the periodic constant of the restriction, see (\ref{eq:REDPC})),
but it also has the advantage that
for these dilation parameters $l'$ the  union $\cup_{v\in\cV} P^{(l'),\triangleleft}_v$
equals the union of significantly less polytopes, namely $\cup_{n\in\calN} P^{(l'),\triangleleft}_v$.

For example, when we have only one node, one has to handle one simplex instead of $|\cV|$ many.

\begin{proof} (a) We show that for any $l'\in  V_\calN$ one has the inclusions
\begin{equation}\label{eq:4.13}
P_v^{(l'),\triangleleft}\,\subset\,
\bigcup_{n\in \calN}P_n^{(l'),\triangleleft} \ \mbox{for any $v\not\in \calN$}.
\end{equation}
We  consider two cases. First we assume that $v$ is on  a leg (chain) connecting an
end $e(v)\in\cale$ with a node $n(v)$ (where $e(v)=v$ is also possible).
Then, clearly, (\ref{eq:4.13}) follows from
\begin{equation}\label{eq:4.13b}
P_v^{(l'),\triangleleft}\,\subset\,
P_{n(v)}^{(l'),\triangleleft} \ \ \mbox{\ \ \ for any $l'\in \calS_\calN$}.
\end{equation}
Let  $E^*_{uv}=(E^*_u)_v=-(E^*_u,E^*_v)$ be the $v$--cordinate of $E^*_u$. Note that
$E^*_{uv}=E^*_{vu}$.  Using the definition of the polytopes,
(\ref{eq:4.13b}) is equivalent with the implication (cf. \ref{bek:pol})
\begin{equation}\label{eq:4.14}
\big(\ \sum_{e\in\cale}x_eE^*_{ve} < l'_{v}\ \big)
\Longrightarrow
\big(\ \sum_{e\in\cale}x_eE^*_{n(v)e} < l'_{n(v)}\ \big) \ \
\mbox{\ \ for any $l'\in \calS_\calN$ and $x_e\geq 0$}.
\end{equation}
Let $\calW$ be the set of vertices
on this leg (including $e(v)$ but not $n(v)$). Then, one verifies that there exist
positive rational numbers $\alpha$ and $\{\alpha_w\}_{w\in \calW}$, such that
\begin{equation}\label{eq:4.15}
E^*_v=\alpha\, E^*_{n(v)}+\sum_{w\in\calW}\alpha_wE_w.
\end{equation}
The numbers $\alpha$ and $\{\alpha_w\}_{w\in \calW}$ can be determined from the linear
system obtained by  intersecting the identity (\ref{eq:4.15}) by $\{E_w\}_w$ and $E_{n(v)}$.
Intersecting (\ref{eq:4.15}) by $E^*_e$ ($e\in\cale$), we get that $E^*_{ve}=\alpha E^*_{n(v)e}$ for any
$e\not=e(v)$,  and $E^*_{v,e(v)}=\alpha E^*_{n(v),e(v)}+\alpha_{e(v)}$. Hence
\begin{equation}\label{eq:4.16}
 \sum_{e\in\cale}x_eE^*_{ve} =\alpha
 \sum_{e\in\cale}x_eE^*_{n(v)e}+ x_{e(v)}\alpha_{e(v)}.
\end{equation}
On the other hand, intersecting (\ref{eq:4.15}) with $E^*_n$, for $n\in\calN$, we get
$E^*_{vn}=\alpha E^*_{n(v)n}$. Since $l'$ is a linear combination of $E^*_{n}$'s,
we get that
\begin{equation}\label{eq:4.17}
-l'_v=(l',E^*_v)=\alpha (l',E^*_{n(v)})=-\alpha l'_{n(v)}.
\end{equation}
Since $x_{e(v)}\alpha_{e(v)}\geq 0$, (\ref{eq:4.16}) and (\ref{eq:4.17}) imply (\ref{eq:4.14}).
This ends the proof of this case.

Next, we assume that $v$ is on a chain  connecting two nodes $n(v)$ and $m(v)$.
Let $\calW$ be the set of vertices
on this bamboo (not including $n(v)$ and  $m(v)$).
Then we will show that
\begin{equation}\label{eq:4.13c}
P_v^{(l'),\triangleleft}\,\subset\,
P_{n(v)}^{(l'),\triangleleft}\cup P_{m(v)}^{(l'),\triangleleft} \ \ \mbox{for any $l'\in \calS_\calN$}.
\end{equation}
This follows  as above from the
existence of positive rational numbers $\alpha$, $\beta$ and $\{\alpha_w\}_{w\in \calW}$ with
\begin{equation}\label{eq:4.18}
E^*_v=\alpha\, E^*_{n(v)}+\beta\, E^*_{m(v)}+\sum_{w\in\calW}\alpha_wE_w.
\end{equation}

(b) follows from (a) and from the fact that
 all $b_k$  entries in the numerator of $Z(\bt)$  belong to $\calS_\calN$.

 (c) If ${\mathrm{pc}}^{\calC}Z_h(\bt)$ is computed as $\widetilde{Q}_h(0)$ for some
 quasipolynomial $\widetilde{Q}_h$ defined on $\widetilde{L}\subset L$, then part (b) guarantees that
 ${\mathrm{pc}}^{\calC_\calN}Z_h(\bt_\calN)$ can be  computed as
 $(\widetilde{Q}_h|_{\widetilde{L}\cap \calS_\calN})(0)$, which equals $\widetilde{Q}_h(0)$.
\end{proof}

\begin{example}
Consider the following graph $G$:
\begin{center}
\begin{picture}(150,80)(80,15)
\put(150,55){\circle*{3}}
\put(180,55){\circle*{3}}
\put(210,55){\circle*{3}}
\put(240,55){\circle*{3}}
\put(120,55){\circle*{3}}
\put(90,55){\circle*{3}}
\put(60,55){\circle*{3}}
\put(150,25){\circle*{3}}
\put(90,25){\circle*{3}}
\put(210,25){\circle*{3}}
\put(150,55){\line(-1,0){90}}
\put(150,55){\line(1,0){90}}
\put(150,55){\line(0,-1){30}}
\put(210,55){\line(0,-1){30}}
\put(90,55){\line(0,-1){30}}
\put(150,80){\makebox(0,0){$E_0$}}
\put(180,80){\makebox(0,0){$E_{02}$}}
\put(210,80){\makebox(0,0){$E_2$}}
\put(240,80){\makebox(0,0){$E_{21}$}}
\put(120,80){\makebox(0,0){$E_{01}$}}
\put(90,80){\makebox(0,0){$E_1$}}
\put(60,80){\makebox(0,0){$E_{11}$}}
\put(150,10){\makebox(0,0){$E_{03}$}}
\put(90,10){\makebox(0,0){$E_{12}$}}
\put(210,10){\makebox(0,0){$E_{22}$}}
\put(150,65){\makebox(0,0){$-1$}}
\put(180,45){\makebox(0,0){$-13$}}
\put(210,65){\makebox(0,0){$-1$}}
\put(255,55){\makebox(0,0){$-2$}}
\put(220,20){\makebox(0,0){$-3$}}
\put(160,20){\makebox(0,0){$-2$}}
\put(100,20){\makebox(0,0){$-3$}}
\put(60,45){\makebox(0,0){$-2$}}
\put(90,65){\makebox(0,0){$-1$}}
\put(120,45){\makebox(0,0){$-9$}}
\end{picture}
\end{center}
By Theorem \ref{th:REST} we are interested only in those polytopes $P_v\subset \R^5$
which are associated with the nodes $E_1$, $E_2$ and $E_0$. Let $l\in \calS_{\calN}$, i.e.
$l=\lambda_1 E_1^*+\lambda_2 E_2^*+\lambda_0 E_0^*$.
Then one can verify that  $\calS_{\calN}$ is divided by  the plane
$\lambda_1= (13/9)\lambda_2$. Hence, in general   $\calS_{\calN}$
can also be divided into several chambers.
(On the other hand, for graphs with at most two nodes this does not happen.)
\end{example}

\subsection{An application of the Reduction Theorem \ref{red}}\labelpar{s:aplred}\
As we already discussed earlier, N\'emethi \cite{NSW} proved that the normalized Euler characteristics 
of the lattice cohomology also agrees with the Seiberg--Witten invariant.
This result together with Theorem \ref{th:JEMS} emphasize that the Seiberg--Witten 
invariant can be recovered from the topological Poincar\'e series as well. 
The fact that (by Reduction Theorem \ref{red}) the Euler characteristic can be replaced 
by the Euler characteristic of the reduced lattice, suggests the existence of a reduction for the 
series as well.

First of all, let us recall the theorem from \cite{NSW} (same as \ref{th:JEMS}) in a different form 
which is more convenient for this subsection.
\begin{theorem}[\cite{NSW}]\labelpar{NSW} Fix one of the elements $l'_{[k]}$.
Then the following facts hold.

(1)\begin{equation*}\label{eq:Z}
Z_{\lk}(\bt)=
\sum_{l\in L}\,\Big(\sum_{I\subseteq \mathcal{J}} (-1)^{|I|+1}w(l,I)\Big)\, \vast^{l+l'_{[k]}}.
\end{equation*}

(2) Fix some   $l\in L$ with  $l+\lk \in -k_{can}+{\rm interior}(\calS')$. Then
\begin{equation*}
\sum_{\overline{l}\in L,\, \overline l \not\geq l}p_{\overline{l}+\lk }=
\chi_{k_r}(l)+eu( \bH^*(G,k_r).\end{equation*}
\end{theorem}

\noindent (In \cite{NSW} $w(k)$ is defined as $-(k^2+|\calj|)/8$ for  $k \in Char$.
If $k=k_r+2l$ then $w(k)=\chi_{k_r}(l)-(k_r^2+|\calj|)/8$. The last constant can be neglected in the
 sum of (1) since $\sum_{I\subseteq \mathcal{J}}(-1)^{|I|}=0$.
The sum in (2) is finite since $Z$ is supported in $\Z_{\geq 0}\langle E_j^*\rangle_j$ and
 all the entries of $E^*_j$ are strictly positive, cf. (\ref{eq:POS}).)

Recall that $\calj=\overline{\calj}\sqcup \calj^*$, where $\overline{\calj}$ is an
index set containing all the bad vertices. Let $\phi:L\to \overline{L}$ be the projection to the
$\overline{\calj}$--coordinates.
We also write $\overline{\bt}=\{t_j\}_{j\in\overline{\calj}}$ for the monomial variables
associated with $\overline{L}$, and
$\overline{\bt}^{\vasi}=\prod_{j\in\overline{\calj}} t_{j}^{i_j}$ for $\vasi=(i_1,\dots,i_{\nu})\in \overline{L}$.
Therefore,  $\bt^{l'}|_{t_j=1, \,\forall \,j\in \calj^*}=\overline{\bt}^{\phi(l')}$.

\begin{definition}{\bf The reduced series.}
 For any $h\in H$ define
\begin{equation*}
\overline Z_{h}(\overline{\bt}):=Z_{h}(\vast)\vert_{t_j=1, \, \forall j\in \mathcal{J}^*}.
\end{equation*}
\end{definition}
(We warn the reader that the reduced `non--decomposed' series $Z(\vast)\vert_{t_j=1, \,
\forall  j\in \mathcal{J}^*}$
usually does not contain sufficient information to reobtain each term
$\overline Z_{h}(\overline{\bt})$ ($h\in H$) from it.)

Fix one $\lk$, and write  $\overline Z_{\lk}(\overline{\bt})=
\sum_{\vasi\in \overline{L}} \overline p_{\vasi+\phi(l'_{[k]})}\overline{\bt}^{\vasi+\phi(l'_{[k]})}$.
Moreover, let $\overline \calS_{k}'$ be the projection of $\calS'\cap (l'_{[k]}+L)$.
Then $\overline Z_{k}(\overline{\bt})$ is supported on
$\overline \calS_{k}'$, and
for any $\vasi$ the sum  $ \sum_{\vasi'\ngeq \vasi}\overline p_{\vasi'+\phi(l'_{[k]})}$ is  finite
(properties inherited from $Z$).  Note that $\overline{\calS}:=\phi(\calS'\cap L)$ is a semigroup, and
$\overline \calS_{k}'$ is an $\overline{\calS}$--module.

Our next goal is to show that the series introduced above with reduced
variables preserves all these properties from Theorem \ref{NSW}: it can be recovered from the
reduced weighted cubes and has all the information about the Seiberg-Witten invariant.

\begin{theorem}\labelpar{thm:redZ}
Let $(\overline{L},\overline{w}[k])$ be as in \ref{bek:331}.
Then

(1) \begin{equation*}
  \overline Z_{\lk}(\overline{\bt})=
  \sum_{\vasi \in \overline L}\Big(\sum_{\overline{I}
  \subseteq \overline{\mathcal J}}(-1)^{|\overline{I}|+1}
  \overline{w}(\vasi,\overline{I})\Big)\overline{\bt}^{\vasi+\phi(l'_{[k]})}.
 \end{equation*}

(2)
 There exists $\vasi_0 \in \ocalS $ (characterized in the next Lemma
\ref{ilem})  such that for any $\vasi \in \vasi_0+\ocalS$
 $$ \sum_{\vasi'\ngeq \vasi}\overline p_{\vasi'+\phi(\lk)}=\overline{w}(\vasi)+
 eu(\bH^*(\overline{L},\overline{w}[k])).$$
Here $\overline{w}(\vasi)$ is a quasipolynomial
and   $eu(\bH^*(\overline{L},\overline{w}[k]))$ equals\,
 $eu( \bH^*(G,k_r))$.
\end{theorem}

\begin{proof} (1)
We abbreviate $k_r$ by $k$ and $\overline{w}[k]$ by $\overline{w}$.
By \ref{eq:Z}(1) we get
$$\overline{Z}_{\lk}(\overline{\vast})=
\sum_{\vasi\in\overline{L}}\sum_{\overline{I}\subseteq \overline{\mathcal J}}(-1)^{|\overline{I}|+1}
\Big( \sum_{l^*\in L^*} \sum_{I^*\subseteq \mathcal{J}^*} (-1)^{|I^*|}w(x(\vasi)+l^*,\overline{I}\cup I^*)
 \Big)\,\overline{\vast}^{\vasi+\phi(l'_{[k]})},$$
 where $L^*\subset L$ is the sublattice of $\calj^*$--coordinates.
For a fixed $\vasi$ and $\overline{I}\subseteq \overline{\mathcal{J}}$, denote the coefficient in the last bracket by $T=T(\vasi,\overline{I})$.  Then we have to show that $T=\overline{w}(\vasi,\overline{I})$.

We define a weighted lattice $(L^*,w^*)$ as follows: the weight
 of a cube $(l^*,I^*)$ in $L^*$ is  $w^*(l^*,I^*):=w(x(\vasi)+l^*,\overline{I} \cup I^*)$
(hence it  depends on $(\vasi,\overline{I})$).
This is a compatible weight function on $L^*$ since $w$ is so, moreover
$T= \sum_{l^*\in L^*}\sum_{I^*\subseteq \mathcal{J}^*}(-1)^{|I^*|}w^*(l^*,I^*)$.

Note also that for any fixed ${\bf i}$ there are only finitely many
 $l^*\in L^*$ for which $({\bf i},l^*)\in\calS'$
(use (\ref{eq:POS})).  Hence,  the sum in $T$ is finite. Therefore,
(cf. \ref{ss:2} and \ref{ss:3}),  we can find a `large' rectangle
$R^*=R^*(l_1^*,l_2^*)=\{l^*\in L^* \,:\, l_1^*\leq l^*\leq l_2^*\}$ with certain $l_1^*$ and $l_2^*$
such that
$$T=\sum_{l^*\in R^*}\sum_{I^*\subseteq \mathcal{J^*}}(-1)^{|I^*|}w^*(l^*,I^*) \ \ \ \mbox{and} \ \ \
\bH^*(L^*,w^*)=\bH^*(R^*,w^*).$$
Using the result and methods of \cite[Theorem 2.3.7]{NSW}, for the counting function
$\mathcal{M}(t):=\sum_{l^*\in R^*}\sum_{I^*\subseteq \mathcal{J^*}}(-1)^{|I^*|}t^{w^*(l^{*},I^*)}$ we have
$$\lim_{t\rightarrow 1}\frac{\mathcal{M}(t)-t^{\min( w^*\vert_{R^*})}}{1-t}=\sum_{q\geq 0}(-1)^q
\rank_\Z(\bH^q_{red}(R^*,w^*)).$$
The Reduction Theorem \ref{red} and its proof says that $(L^*,w^*)$ has vanishing reduced cohomology, in particular
$\bH^q_{red}(R^*,w^*)=0$ for any $q\geq 0$. Hence
 $$T=\frac{d\mathcal{M}(t)}{dt}\big|_{t=1}=\min( w^*\vert_{R^*})=\min_{l^*\in L^*} \{w(x({\bf i})+l^*,\overline{I})\}=
 \min_{l^*\in L^*}\max_{\overline{J}\subseteq \overline{I}} \{\chi_k(x({\bf i})+l^*+E_{\overline{J}})\}.$$
By Lemma \ref{lem:laufb} $\chi_k(x({\bf i})+l^*+E_{\overline{J}})\geq \chi_k(x({\bf i}+1_{\overline{J}}))$, hence
\begin{equation}\label{eq:minmax}
\max_{\overline{J}\subseteq \overline{I}}
\chi_k(x({\bf i})+l^*+E_{\overline{J}})\geq
\max_{\overline{J}\subseteq \overline{I}}
\chi_k(x({\bf i}+1_{\overline{J}}))=\overline{w}({\bf i},\overline{I}).\end{equation}
But, by Lemma \ref{lem:tildeS}(a) (for notations see also \ref{ss:tilde}), the minimum
over $l^*$ of the left hand side is realized for $l^*=E_{\widetilde{\cals}}$ with equality
in (\ref{eq:minmax}),  hence
$T=\overline{w}({\bf i},\overline{I})$.

\medskip

We start the proof of part (2) by the following lemma,
 the analogue for $(\overline{L},\overline{w})$
  of Lemmas \ref{lem:con2} and \ref{lem:i+I},
which identifies $\vasi_0$.

\begin{lemma}\label{ilem}
(a) Fix $l+\lk\in \calS'$ and take the projection
$\vasi:=\phi(l)$. Then $x(\vasi)+\lk \in \calS'$, hence
$\overline{w}(\vasi+1_j)>\overline{w}(\vasi)$
for every $j\in \overline{\mathcal{J}}$.

(b) There exists $\vasi_0 \in  \overline \calS$ such that for any $\vasi\in \vasi_0+
\overline \calS$ one has a
sequence $\{\vasi_n\}_{n\geq 0}\in\overline \calS$ with
\begin{itemize}
\item[(i)] $\vasi_0=\vasi$, $\vasi_{n+1}=\vasi_n+1_{j(n)}$ for certain
$j(n)\in \overline{\mathcal{J}}$, and all entries of $\vasi_n$ tend to infinity as
$n\rightarrow\infty$;

\item[(ii)] for any  $n$ and  $0\leq \vasi_n'\leq \vasi_n$ with the same  $j(n)$-th coefficients, one has
$$\overline{w}(\vasi_n'+1_{j(n)})>\overline{w}(\vasi_n').$$
\end{itemize}

\end{lemma}
\begin{proof} (a) Since $l+\lk$ satisfies conditions (a)-(b) of \ref{lemF1}
in the definition of $x(\vasi)$, by the minimality of $x(\vasi)$ we get
that $l-x(\vasi)$ is effective and is supported on $\calj^*$.
Hence, $(x(\vasi)+\lk,E_j)\leq (l+\lk,E_j)\leq 0$ for any $j\in \ocalj$.
The last inequality follows from  \ref{propF1}.

(b) The negative definiteness of the intersection from guarantees the existence of
$\vasi_0$ with (i). For (ii) note that if $\vasi=\phi(l)$ as in (a), and $0\leq \vasi'\leq \vasi$, 
such that their $j$-entries agree, then  
automatically $\overline{w}(\vasi'+1_{j})>\overline{w}(\vasi')$. Indeed, 
$x(\vasi)-x(\vasi') $ is effective and supported on $\calj\setminus j$, hence 
 $(x(\vasi')+\lk,E_j)\leq (x(\vasi)+\lk,E_j)\leq 0$ and \ref{propF1} applies again.
\end{proof}

We fix an $\vasi$ as in Lemma \ref{ilem}(b). Then 
similarly as in subsection \ref{ss:3}, one obtains
\begin{equation}
\bH^*(\overline{L},\overline{w})\cong \bH^*(R(0,\vasi),\overline{w}),
\end{equation}
where $R(0,\vasi)=\{\vasi'\in \overline{L} \, : \, 0\leq \vasi'\leq \vasi \}$.
In particular, if we set
$$\cale(R(0,\vasi)):=\sum_{(\vasi',\overline{I})\subseteq R(0,\vasi)}(-1)^{|\overline{I}|+1}\overline{w}(\vasi',\overline{I})$$
(sum over all the cubes of $R(0,\vasi)$),
then \cite[Theorem 2.3.7]{NSW} ensures that
\begin{equation}\label{cale}
\cale(R(0,\vasi))=eu(\bH^*(R(0,\vasi),\overline{w})).
\end{equation}
In the sequel we follow closely the 
proof of Theorem \ref{NSW}(2) from  \cite[Theorem 3.1.1]{NSW}).

We choose a computation sequence $\{\vasi_n\}_{n\geq 0}$ as in \ref{ilem}
and set $R':=\{ \vasi' \in \overline L \ : \vasi'\geq 0 \ \mbox{and} \ \exists \,
j\in \overline{\calj} \ \mbox{with} \
(\vasi'-\vasi)_j\leq 0\}$.  $R'$ is
not finite, but $R'\cap  \ocalS_{k}'$ is a finite set.  Fix $\widetilde n$ so that $R'\cap
\ocalS_{k}'\subseteq R(0,\vasi_{\widetilde n})$, and define
$R'(\widetilde n):=R'\cap R(0,\vasi_{\widetilde n})$, \
$\partial_1 R'(\widetilde n):=R'\cap R(\vasi,\vasi_{\widetilde n})$,
and
$$\partial_2 R'(\widetilde n):=\{\vasi'\in  R'(\widetilde n)\,:\,
\exists j\in\overline{\calj} \ \mbox{with } \ (\vasi'-\vasi_{\widetilde n})_j=0\}.$$
 Then by part (1) of the theorem we have
$$\sum_{\vasi'\ngeq \vasi}\overline p_{\vasi'+\phi(\lk)}=\cale(R'(\widetilde n))-\cale
(\partial_1 R'(\widetilde n)\cup \partial_2 R'(\widetilde n)).$$
The right hand side is simplified as follows. 
First, notice that we may find $\widetilde n$ sufficiently high
in such a way, that if we choose a sequence
$\{{\bf j}_m\}_{m=0}^t$ from ${\bf j}_0=0$ to ${\bf j}_t=\vasi$ 
with ${\bf j}_{m+1}={\bf j}_m+1_{j(m)}$,
we have the following property:
\begin{center}
\noindent for every ${\bf j}' \in \partial_2 R'(\widetilde n)$ with ${\bf j}'\geq {\bf j}_m$ and
$({\bf j}')_{j(m)}=({\bf j}_m)_{j(m)}$ one has
$\overline w({\bf j}'+1_{j(m)})\leq \overline w({\bf j}')$.
\end{center}
Indeed, $(x({\bf j}')+\lk,E_{j(m)})$ is increasing in ${\bf j}'$ with fixed $j(m)$-th coefficient.
(Any  ${\bf j}' \in \partial_2 R'(\widetilde n)$ 
has `large' entries corresponding to coordinates
 $j$ when  $({\bf j}'-i_{\widetilde n})_j=0$, and `small' entry corresponding to $j(m)$. 
 Hence, when we increase the $j(m)$-th  entry by one,
the positivity of the quantities $(x({\bf j}')+\lk,E_{j(m)})$ is guaranteed  by the presence of `large' entries.)

Therefore, using the sequence $\{{\bf j}_m\}$ and
\ref{lem:con3}, there exists a contraction of $\partial_2 R'(\widetilde n)$ to
$\partial_1 R'(\widetilde n)\cap \partial_2 R'(\widetilde n)$ along which $\overline w$ is
non--increasing.
Then similarly as in (\ref{cale}), one get $\cale(\partial_2 R'(\widetilde n))=
\cale(\partial_1 R'(\widetilde n)\cap \partial_2 R'(\widetilde n))$, 
hence $\cale (\partial_1 R'(\widetilde n)\cup \partial_2 R'(\widetilde n))=
\cale(\partial_1 R'(\widetilde n))$ too.

Next, we claim that $\cale(R'(\widetilde n))=\cale(R(0,\vasi))$. Indeed,
using induction on the sequence $\{\vasi_n\}_{0\leq n<\widetilde n}$,
it is enough to show that $\cale(R'(n))=\cale(R'(n+1))$. This follows from
 \ref{ilem}, since for all $\overline I$ containing $j(n)$ and each
$(\vasi',\overline I)\in R'(n+1)\setminus R'(n)$ we have
$$\overline \omega(\vasi',\overline I)=\overline \omega(\vasi'+1_{j(n)},\overline I\setminus j(n)).$$
This ensures a combinatorial cancelation in the sum $\cale(R'(n+1))$, or an isomorphism in the
 corresponding lattice cohomologies, which gives the expected
equality.

With the same procedure applying to $\partial_1 R'(\widetilde n))$ we deduce the equality
$\cale(\partial_1 R'(\widetilde n))=\cale(\partial_1 R'(0))=-\overline \omega(\vasi)$. Hence the
identity follows.
\end{proof}

\begin{example}
In the reduced case, the expression $\overline{w}(\vasi)$ usually is a rather complicated
arithmetical quasipolynomial. E.g., assume that $G$ is a star--shaped graph whose central vertex has 
Euler decoration $b$ and the legs have Seifert invariants $(\alpha_j,\omega_j)_{j=1}^\ell$,
$0<\omega_j<\alpha_j$, ${\rm gcd}(\alpha_j,\omega_j)=1$. We fix the central vertex as the unique 
bad vertex. Then the lattice cohomology is completely determined by the 
sequence $\{\overline{w}(i)\}_{i\geq 0}$, for details see e.g. \cite{OSZINV}.

E.g., in the case of the canonical $spin^c$--structure, $\overline{w}(0)=0$ and 
$$\overline{w}(i+1)-\overline{w}(i)=1-ib-\sum_{j}\big\lceil \,i\omega_j/\alpha_j\,\big\rceil
\ \ \ (i\geq 0).$$ 

\end{example}

\begin{remark}
The fact that $\overline{w}(\vasi)$ is a quasipolynomial can be seen as follows.
Choose  $l\in \calS'\cap L$, $l=(\overline{l},l^*)$, such that $(l,E_j)=0$ for any $j\in \calj^*$.
Then one checks that $x(\vasi+n\overline{l})=x(\vasi)+nl$ for any $n\in \Z_{\geq 0}$, hence
$\overline{w}(\vasi+n\overline{l})=\chi_{k_r}(x(\vasi)+nl)$ is a polynomial in $n$.
  
\end{remark}

\section{Ehrhart theoretical interpretation of the\newline Seiberg--Witten invariant}\labelpar{s:Last}\
Let $G$ be a negative definite plumbing graph, a connected tree as in \ref{ss:11}.
Let  $\calN$ and $\cale$ be  the set of
nodes and  end--vertices as above. We assume that $\calN\not=\emptyset$.
If  $\delta_n$ denotes the valency of a node $n$, then $|\cale|=
2+\sum_{n\in \calN}( \delta_n-2)$.

We consider the
matrix $J$ with entries  $J_{nm}:=-(E^*_n,E^*_m)$ for $n,m\in\calN$.
By (\ref{ss:11}) it is a principal minor of $-\frI^{-1}$ (with rows and columns corresponding to the nodes).

Another incarnation of the matrix $J$ already appeared in subsection \ref{ss:NEC},
as the negative of the  inverse of the orbifold intersection matrix.
Indeed, let for any $n\in\calN$ take that component of $G\setminus
\cup_{m\in\calN\setminus n}\{m\}$
which contains $n$. It is a star--shaped graph,
let $e_n$ be its orbifold Euler number. Furthermore, for any
two nodes $n$ and $m$ which are connected by a chain, let $\alpha_{nm}$ be the determinant of
that chain (not including the nodes). Then define the orbifold intersection  matrix
(of size $|\calN|$)
as  $\frI^{orb}_{nn}=e_n$, $\frI^{orb}_{nm}=1/\alpha_{nm}$ if the two nodes $n\not=m$ are
connected by a chain,
and  $\frI^{orb}_{nm}=0$ otherwise; cf. \cite[4.1.4]{BN07} or \ref{ss:TNC}.
One can show (see \cite[4.1.4]{BN07}) that $\frI^{orb}$ is invertible, negative definite, and
$\det(-\frI^{orb})$ is the product of $\det(-\frI)$ with the determinants of all (maximal)
chains and legs of
$G$. This fact and  \ref{eq:DETsgr} imply that
$J=(-\frI^{orb})^{-1}$.

\subsection{The Ehrhart polynomial}\
In the sequel we assume that $L=L'$, that is $H=0$.

By  \ref{ss:cor}, $P^{(l)}$ sits in $\R^{|\cale|}$. Moreover, by 
Theorem \ref{ss:REST}, we can take
 $l$ of the form $l=\sum_{n\in \calN} \lambda_n E_n^*$, from the subcone of the Lipman cone
 generated by $\{E^*_n\}_{n\in\calN}$.

Then \ref{ss:REST} guarantees that  the associated polytope is $P^{(l)}=\bigcup_{n\in \calN}
P_n^{(l_n)}$, $P_n^{(l_n)}$ depending only on the component $l_n=-(l,E_n^*)$. Note that the
coefficients $\{\lambda_n\}_n$ and the entries $\{l_n\}_n$ are connected exactly by the
transformation law
$\left(l_n\right)_n=J \left(\lambda_n\right)_n$.

Take any chamber $\calC$ such that $int(\calC \cap \calS)\neq \emptyset$, as in \ref{cor:4.1}.
Let $\widehat \calL^{\calC}(P,\calT,(\lambda_n)_n)$
be the Ehrhart quasipolynomial  $\calL^{\calC}(P,\calT,(l_n)_n)$, associated with the
 denominator of $Z$, after changing the variables to $(\lambda_n)_n$ via $\left(l_n\right)_n=J \left(\lambda_n\right)_n$. It is convenient to normalize the coefficient of $\prod_n \lambda_n^{m_n}$
 by a factor $\prod _n m_n!$, hence we write
$$\widehat \calL^{\calC}(P,\calT,(\lambda_n)_n)=\sum_{\sum_n m_n\leq |\cale|
\atop m_n\geq 0; \ n\in \calN}  \widehat \fra^{\calC}_{(m_n)_n}\prod_n\frac{\lambda_n^{m_n}}{m_n!},$$
for certain  periodic functions $\widehat \fra^{\calC}_{(m_n)_n}$ in variables $(\lambda_n)_n$.
By \ref{eq:SUM}, \ref{cor:Taylor} and \ref{th:REST}
\begin{equation}\label{eq:Ehrcoeff}
 \chi\Big(\sum_{n\in \calN}\lambda_n E_n^*\Big)+{\rm pc}^{\calS}(Z)=\Delta((\lambda_n)_n),
\end{equation}
where
$$\Delta((\lambda_n)_n)= \sum_{0\leq k_n\leq \delta_n-2\atop \forall n\in\calN}
(-1)^{\sum_n k_n}\prod_n \binom{\delta_n-2}{k_n} \
\widehat \calL^{\calC}(P,\calT,(\lambda_n-k_n)_n)=$$
$$\sum_{\sum_n m_n\leq |\cale| \atop m_n\geq 0; \ n\in \calN}\ \Bigg(\sum_{0\leq p_n\leq m_n \atop
n\in\calN} (-1)^{\sum_n p_n}\cdot \prod_n\binom{m_n}{p_n}
\left( \sum_{k_n=0}^{\delta_n-2}(-1)^{k_n} \binom{\delta_n-2}{k_n} k_n^{p_n}\right)
\Bigg)\cdot \widehat \fra^{\calC}_{(m_n)_n}\prod_n\frac{\lambda_n^{m_n-p_n}}{m_n!}.$$
On the other hand, since
$\chi(l)=-(K+l,l)/2$,  the left hand side of (\ref{eq:Ehrcoeff})
is the quadratic function
$$\sum_{n,m\in\calN}(J_{nm}/2) \lambda_n\lambda_m+\sum_{n\in\calN}(-(K,E_n^*)/2)\lambda_n+{\rm pc}^{\calS}(Z).$$
Now we identify these coefficients with those of $\Delta((\lambda_n)_n)$ above. The additional
ingredient is the combinatorial formula (\ref{eq:PSz}), which also shows that for the non--zero
summands one necessarily has $p_n\geq \delta_n-2$ for any $n$. One gets the following result.
%
\begin{theorem}\label{th:MAINLAST}
\begin{eqnarray*}
 \widehat \fra^{\calC}_{(\delta_n,(\delta_m-2)_{m\neq n})}&=&J_{nn}\\
 \widehat \fra^{\calC}_{(\delta_n-1,\delta_m-1,(\delta_q-2)_{q\neq n,m})}&=&J_{nm} \ \
\mbox{for $n\not=m$}\\
 \widehat \fra^{\calC}_{(\delta_n-1,(\delta_m-2)_{m\neq n})}&=&
 \textstyle {
 -\frac{1}{2}(K,E_n^*)+\frac{1}{2}
 \sum_{m\in \calN}(\delta_m-2)J_{nm}}\\
\end{eqnarray*}
\begin{equation*}
\widehat \fra^{\calC}_{(\delta_n-2)_n}={\rm pc}^{\calS}(Z)-
\textstyle {
\sum_{n\in \calN} \frac{(\delta_n-2)(K,E_n^*)}{4}+\sum_{n\in \calN}
\frac{(\delta_n-2)(3\delta_n-7)J_{nn}}{24}
+ \sum_{n,m\in \calN \atop m\neq n}\frac{(\delta_n-2)(\delta_m-2)J_{nm}}{8}.}
\end{equation*}
\end{theorem}

Recall that   ${\rm pc}^{\calS}(Z)=-(K^2+|\calv|)/8-\lambda(M)$, where $\lambda(M)$ is the
Casson invariant of $M$. Hence $\widehat \fra^{\calC}_{(\delta_n-2)_n}$ equals
the normalized Casson invariant modulo some `easy terms'.

We emphasize that these formulae also  show  that the above coefficients are constants (as periodic functions
in $(\lambda_n)_n$) and independent of the chosen chamber $\calC$ in the Lipman cone.

\chapter{Seiberg--Witten and Ehrhart theoretical computations and examples}\labelpar{ch:4}\

\indent Applying the general theory developed in the previous chapter, we make detailed computations for
graphs with less than two nodes. As we have seen in \ref{ex:lens}, even in the special case of graphs 
without nodes (that is, the case of lens spaces), the description of the equivariant Ehrhart 
quasipolynomials is new.

In the one--node case (star--shaped graphs) we provide a detailed presentation of all the 
involved (Seiberg--Witten and Ehrhart) invariants, and we establish closed formulae in terms of the 
Seifert invariants. Here we make connection with already known topological results regarding 
the Seiberg--Witten invariants of Seifert 3--manifolds, and also with analytic invariants of 
weighted homogeneous singularities.

In the two node case again we make complete presentations in terms of the analogs of the
Seifert invariants of the chains and star--shaped subgraphs, including closed formulae for $\frsw(M)$.
But, this case has a very interesting additional surprise
in store. 

It turns out that the corresponding combinatorial series $Z(\bt)$ associated with
$G$, reduced to the two variables of the nodes, is the {\em Hilbert (characteristic)
series of an affine monoid of rank two (and some of its modules).}
In particular, the Seiberg--Witten invariant appears as the periodic constant
of Hilbert series associated with affine monoids (and certain modules indexed by $H$), and,
in some sense, measures the non--normality of these monoids.

At the end of the chapter, we provide some examples in which we demonstrate the calculation of the 
periodic constant (or equivalently, the normalized Euler characteristic of the lattice cohomology as well 
as the Seiberg--Witten invariant) from the topological Poincar\'e series $Z(\bt)$.
\section{The one--node case, star--shaped graphs}

\subsection{Seifert invariants and other notations}\labelpar{ss:seiferjel}\
Assume that the graph is star--shaped with $d$ legs.
Each leg is a chain  with normalized  Seifert  invariant $(\alpha_i,\omega_i)$,
where $0<\omega_i <\alpha_i$, gcd$(\alpha_i,\omega_i)=1$.
We also use $\omega_i'$ satisfying $\omega_i\omega_i'\equiv 1$ (mod $\alpha_i$), $0< \omega_i'<\alpha_i$.

If we consider the
Hirzebruch/negative continued fraction expansion, cf. (\ref{eq:HCF})
$$ \alpha_i/\omega_i=[b_{i1},\ldots, b_{i\nu_i}]=
b_{i1}-1/(b_{i2}-1/(\cdots -1/b_{i\nu_i})\cdots )\ \  \ \ (b_{ij}\geq 2),$$
then the $i^{\mathrm{th}}$ leg has $\nu_i$ vertices, say $v_{i1},\ldots, v_{i\nu_i}$,
 with decorations $-b_{i1},\ldots, -b_{i\nu_i}$, where
 $v_{i1}$ is connected by the central vertex. The corresponding base elements in $L$ are
$\{E_{ij}\}_{j=1}^{\nu_i}$.
Let  $b$ be the decoration of the  central vertex; this vertex also defines $E_0\in L$.
The plumbed 3--manifold $M$ associated with such a star--shaped graph has a Seifert structure. 
We will assume that $M$ is a rational homology sphere, or, equivalently, 
the central vertex has genus zero.

The classes
in $H=L'/L$ of the dual base elements are denoted by
$g_{ij}=[E^*_{ij}]$ and  $g_0=[E^*_0]$. For simplicity we also write
$E_i:=E_{i\nu_i}$ and $g_i:=g_{i\nu_i}$. A possible presentation of  $H$ is
\begin{equation}\label{eq:sei1}
H={\mathrm{ab}}\langle \, g_0,g_1,\ldots, g_d\,|\, -b\cdot g_0=\sum_{i=1}^d \omega_i\cdot g_i;
\, g_0=\alpha_i\cdot g_i\ (1\leq i\leq d)\rangle,
\end{equation}
cf. \cite{neumann.abel} (or use (\ref{eq:DET})).
The orbifold Euler number of $M$ is defined as $e=b+\sum_i\omega_i/\alpha_i$. The negative definiteness of
the intersection form implies $e<0$. We write $\alpha:=\mathrm{lcm}(\alpha_1,\ldots,\alpha_d)$,
$\frd:=|H|$ and
$\fro$ for the order of $g_0$ in $H$. One has (see e.g. \cite{neumann.abel})
\begin{equation}\label{eq:sei2}
\frd=\alpha_1\cdots\alpha_d|e|, \ \ \ \ \fro=\alpha|e|.
\end{equation}
Each leg has similar invariants as the graph of a lens space, cf. Example \ref{ex:lens}, and we can
introduce similar notation. For example, the determinant of the  $i^{\mathrm{th}}$ leg is $\alpha_i$.
We write $n^i_{j_1j_2}$ for  the determinant of the subchain  of the   $i^{\mathrm{th}}$ leg
connecting the vertices $v_{ij_1}$ and $v_{ij_2}$ (including these vertices too). Then, using the correspondence between intersection
pairing of the dual base elements and the determinants of the subgraphs, cf. (\ref{eq:DETsgr})
or \cite[11.1]{OSZINV}, one has
\begin{equation}\label{eq:DET}
\begin{array}{ll}
(a) \ \ (E^*_0,E^*_{ij}- n^i_{j+1,\nu_i}E^*_{i\nu_i})=0 \ \ & (b) \ \
g_{ij}=n^i_{j+1,\nu_i}g_{i\nu_i}\ \   (1\leq i\leq d,  \ 1\leq j\leq \nu_i)\\
(c) \ \ (E^*_i,E^*_0)=\frac{1}{\alpha_ie} \ \ & (d) \ \
(E^*_{0},E^*_{0})=\frac{1}{e} \ .
 \end{array}
\end{equation}
Part (b) also explains why we do not need to insert the generators $g_{ij}$ ($j<\nu_i$) in (\ref{eq:sei1}).

For any $l'\in L'$ we set $\tc(l'):=-(E^*_0,l')$, the $E_0$-coefficient of $l'$. Furthermore,
if $l'=c_0E^*_0+\sum_{i,j} c_{ij}E^*_{ij}\in L'$, then we define its {\em reduced transform }
 by $$l'_{red}:= c_0E^*_0+\sum_{i,j} c_{ij} \cdot n^i_{j+1,\nu_i}E^*_{i}.$$
By (\ref{eq:DET}) we get
$[l']=[l'_{red}]$ in $H$, $\tc(l')=\tc(l'_{red})$, and if
 $l'_{red}=\sum_{i=0}^dc_iE^*_i$, then $\tc(l'_{red})$ is
\begin{equation}\label{eq:ntilde}
\tc:=\frac{1}{|e|}\cdot \Big(\,c_0+
\sum_{i=1}^d\frac{c_i}{\alpha_i}\,\Big).
\end{equation}
If $h\in H$, and $l'_h\in L'$ is any of its lifting (that is, $[l'_h]=h$),
then $l'_{h,red}$ is also a lifting of the same $h$ with $\tc(l'_h)=\tc(l'_{h,red})$.
In general, $\tc=\tc(l'_h)$ depends on the lifting, nevertheless replacing $l'_h$ by $l'_h\pm E_0$ we modify
$\tc$ by $\pm 1$, hence we can always achieve
$\tc\in [0,1)$,  where it is  determined uniquely by $h$.
 For example, since $r_h\in \square$, its $E_0$--coefficient $\tc(r_h)$ is in
$[0,1)$.

  Finally, we consider
\begin{equation}\label{eq:discr}
\gamma:=\frac{1}{|e|}\cdot \Big( d-2-\sum_{i=1}^d \frac{1}{\alpha_i}\Big).
\end{equation}
It has several `names'.
 Since the canonical class is given by $K=-\sum_vE_v+\sum_{v}(\delta_v-2)E^*_v$,
by (\ref{eq:DET}) we get that the $E_0$--coefficient of $-K$ is $(K,E^*_0)=\gamma+1$. The number $-\gamma$ is
sometimes called the `log discrepancy' of $E_0$, $\gamma$ the `exponent' of the weighted
homogeneous germ $(X,0)$, and $\fro\gamma$ is the Goto--Watanabe $a$--invariant of the universal abelian
cover of $(X,0)$, see  \cite[(3.1.4)]{G-W} and \cite[(3.6.13)]{c-m}; while in
\cite{neumann.abel} $e\gamma$ appears as an orbifold Euler characteristic.

\subsection{Interpretation of Z(\bt)}\labelpar{ss:Z}\
By Theorem \ref{eq:REDPC}, for the periodic constant computation, we can
reduce ourself to the variable of the single node, it will be denoted by $t$.

First we analyze the equivariant rational function associated with  the denominator of $Z^e$
$$Z^{/H}(t)=\prod_{i=1}^d\,\big(1-t^{-(E^*_{i}, E^*_0)}[g_i]\big)^{-1}=
\sum_{x_1,\ldots , x_d\geq 0}\, t^{\,\sum_i
x_i/(\alpha_i|e|)}\ \Big[\,\sum _ix_ig_i\,\Big]\in \Z[[t^{1/\fro}]][H].
$$
 The right hand side of the above expression
can be transformed  as follows (cf. \cite[\S 3]{SWII}).
If we fix a lift $\sum_{i=0}^dc_iE_i^*$ of $h$
as above, then using the presentation (\ref{eq:sei1}) one gets that
$\sum_{i=1}^dx_ig_i$ equals $h$ if and only if there exist
$\ell,\ell_1,\ldots, \ell_d\in\Z$ such that
$$\begin{array}{lrll}
(a) \ & -c_0&= \ \ell_1+\cdots +\ell_d-\ell b &\\
(b) \ & x_i-c_i&= -\omega_i \ell-\alpha_i\ell_i & \ (i=1,\ldots, d).
\end{array}$$
Since $x_i\geq 0$, from (b) we get
$\tilde{\ell}_i:=\big\lfloor \frac{c_i-\omega_i\ell}{\alpha_i}\big\rfloor -\ell_i
\geq 0$.
Moreover, if we set for $\bc=(c_0,c_1,\ldots, c_d)$
\begin{equation}\label{eq:N}
N_\bc(\ell):=1+c_0-\ell b +\sum_{i=1}^d\Big\lfloor \frac{c_i-\omega_i\ell}{\alpha_i}\Big\rfloor,
\end{equation}
then the number of realizations of $h=\sum_ic_ig_i$ in the form $\sum_ix_ig_i$ is given by the number of
integers $(\tilde{\ell}_1,\ldots,\tilde{\ell}_d)$ satisfying
$\tilde{\ell}_i\geq 0$ and $\sum_i\tilde{\ell}_i=N_\bc(\ell)-1$.  This is
$\binom{N_\bc(\ell)+d-2}{d-1}$. Moreover, the non--negative
integer $\sum_ix_i/(\alpha_i|e|)$ equals $\ell+\tc$.
Therefore,
\begin{equation}\label{eq:ZH}
Z_h^{/H}(t)=\sum_{\ell\geq -\tc}\ \binom{N_\bc(\ell)+d-2}{d-1}\
 t^{\ell+\tc}.
\end{equation}
This expression is independent of the choice of $\bc=\{c_i\}_{i=0}^d$.
Similarly, for any function $\phi$, the expression
$\sum_{\ell\geq -\tc}\phi(N_\bc(\ell))t^{\ell+\tc}$ is independent of the choice of $\bc$, it depends only
on $h=\sum_ic_ig_i$.

Furthermore, one checks that $N_\bc(\ell)\leq 1+(\ell+\tc)|e|$, hence if $\ell+\tc<0$ then $N_\bc(\ell)\leq 0$,
therefore $\binom{N_\bc(\ell)+d-2}{d-1}=0$ as well. Hence, in (\ref{eq:ZH})  the inequality $\ell+\tc\geq 0$
below the sum, in fact, is not restrictive.

Next, we consider the numerator $(1-[g_0]t^{1/|e|})^{d-2}$ of $Z^e(t)$.
A similar computation as above done for $Z^e(t)$ (see \cite{neumann.abel} and \cite[\S 3]{SWII}),
or by multiplying (\ref{eq:ZH}) by the numerator and using $\sum_{k=0}^{d-2}(-1)^k\binom{d-2}{k}
\binom{N-k+d-2}{d-1}=\binom{N}{1}$,  gives
\begin{equation}\label{eq:Zt}
Z_h(t)
=\sum_{\ell\geq -\tc}\
\max\{0, N_\bc(\ell)\} \ t^{\ell+\tc}.
\end{equation}
In order to compute the periodic constant of $Z_h(t)$ we decompose $Z_h(t)$ into its
 `polynomial and negative degree parts', cf. \ref{PC}. Namely, we write $Z_h(t)=Z^{+}_h(t)+Z^{-}_h(t)$, where
\begin{equation}\label{eq:Sp}\begin{array}{l}
Z^{+}_h(t)=\sum_{\ell\geq -\tc}\ \max\big\{0, -N_\bc(\ell)\big\} \ t^{\ell+\tc} \ \
\mbox{(finite sum with positive exponents)}
\\ \  \ \\
Z^{-}_h(t)=\sum_{\ell\geq -\tc}\ N_\bc(\ell) \ t^{\ell+\tc}.
\end{array}
\end{equation} In  $Z^{-}_h$ it is convenient to
fix a choice with $\tc\in[0,1)$, hence  the summation is  over $\ell\geq 0$. Then
a computation shows that it is a rational function of negative degree
\begin{equation}\label{eq:Sp2}
Z^{-}_h(t)=\Big(\frac{1-e\tc}{1-t}-\frac{e\cdot t}{(1-t)^2}
-\sum_{i=1}^d\sum_{r_i=0}^{\alpha_i-1}\Big\{\frac{c_i-\omega_ir_i}{\alpha_i}\Big\}\, t^{r_i}\cdot
\frac{1}{1-t^{\alpha_i}}\Big)\cdot t^{\tc}.
\end{equation}
(This expression can be compared with the Laurent  expansion of $Z_h$ at $t=1$ which
was already considered
in the literature. Dolgachev, Pinkham, Neumann and Wagreich \cite{Do,P,neumann.abel,wa}
determine the first two terms (the pole part), while  \cite{SWII,OSZINV} the third terms as well.
Nevertheless the above $Z_h^{+}+Z_h^{-}$ decomposition does not coincide
with the `pole+regular part' decomposition of the Laurent expansion terms, and focuses on
 different aspects.)

Since the degree of $Z^{-}_h$ is negative (or by a direct computation)
 $\mathrm{pc}(Z^{-}_h)=0$, cf. \ref{PC}.

On the other hand, since $e<0$, in $Z^{+}_h(t)$ the sum is finite.
(The degree of $Z^{+}_0$ is $\leq \gamma$,
see e.g. \cite{NO2}. Since
$N_{{\bf c}(r_{h,red})}(\ell) \geq N_0(\ell)$,
  the degree of $Z_h^{+}$ is $\leq \gamma+\tc(r_h)$ too).
By  \ref{PC},
\begin{equation}\label{eq:PCpol}
\mathrm{pc}(Z_h)=Z_h^{+}(1)=\sum_{\ell\geq -\tc}\ \max\big\{0, -N_\bc(\ell)\big\}
\end{equation}
for {\it any} lifting $\bc$ of $h=\sum_ic_ig_i$.
In this sum the bound $\ell\geq -\tc$ is really restrictive.

We consider the non--equivariant version, the projection of $Z^e\in \Z[[t^{1/\fro}]][H]$
into $\Z[[t^{1/\fro}]]$ too
$$Z_{ne}(t)=\sum_hZ_h(t)=\frac{(1-t^{1/|e|})^{d-2}}{
\prod_{i=1}^d\,(1-t^{1/(|e|\alpha_i)})}\in \Z[[t^{1/\fro}]].
$$
We can get its $Z^{+}_{ne}+Z^{-}_{ne}$ decomposition either by summation of
$Z^{+}_{h}$ and $Z^{-}_{h}$, or as follows. Write
\begin{equation}\label{eq:NE}
Z_{ne}(t)=\frac{1}{(1-t^{1/|e|})^{2}} \prod_{i=1}^d\, \frac{1-t^{1/|e|}}{
1-t^{1/(|e|\alpha_i)}}=\frac{1}{(1-t^{1/|e|})^{2}}
\sum_{0\leq x_i<\alpha_i\atop 0\leq i\leq d} t^{\frac{1}{|e|}\cdot S(x)},
\end{equation}
where $S(x):=\sum_i \frac{x_i}{\alpha_i}$. Then its decomposition into
$Z_{ne}^{+}(t)+Z_{ne}^{-}(t)$ is
\begin{equation}\label{eq:NE1}
Z_{ne}^{-}(t)=
\sum_{0\leq x_i<\alpha_i\atop 0\leq i\leq d} t^{\frac{1}{|e|}\cdot \{S(x)\}}\cdot
\Big(\frac{1}{(1-t^{1/|e|})^{2}}-\frac{\lfloor S(x)\rfloor}{1-t^{1/|e|}}\Big)
\end{equation}
\begin{equation}\label{eq:NE2}
Z_{ne}^{+}(t)=
\sum_{0\leq x_i<\alpha_i\atop 0\leq i\leq d} t^{\frac{1}{|e|}\cdot \{S(x)\}}\cdot
\frac{t^{\frac{1}{|e|}\cdot \lfloor S(x)\rfloor}-\lfloor S(x)\rfloor
t^{\frac{1}{|e|}}+\lfloor S(x)\rfloor-1}{(1-t^{1/|e|})^{2}}.
\end{equation}
After dividing in $Z^+_{ne}(t)$ (or by L'Hospital rule), we get 
\begin{equation}\label{eq:NE3}
\mathrm{pc}(Z_{ne})=Z^+_{ne}(1)=\frac{1}{2}\cdot
\sum_{0\leq x_i<\alpha_i\atop 0\leq i\leq d} \lfloor S(x)\rfloor\cdot \lfloor S(x)-1\rfloor.
\end{equation}

\subsection{Analytic interpretations}\labelpar{ss:analy}\ 
Rational homology sphere negative definite
Seifert 3--manifolds  can be realized analytically as links of weighted homogeneous singularities,
or by equisingular deformations of weighted homogeneous singularities provided by splice--quotient
equations \cite{neumann.abel,nw-CIuac}.

Consider the smooth germ at the origin of  $\C^d$ with coordinate ring $\C\{z\}=\C\{z_1,\ldots, z_d\}$, where
$z_i$ corresponds to the $i^{\mathrm{th}}$ end. Then $H$ acts on it by the diagonal action
$h*z_i=\theta(g_i)(h)z_i$. Similarly, we can introduce a multidegree $deg(z_i)=E^*_i\in L'$,
hence the Poincar\'e series of $\C\{z\}$ associated with this multidegree is $\prod_i
(1-\bt^{E^*_i})^{-1}$. Moreover, considering the action of $H$ on it,
$\widetilde{Z}(\bt)=\prod_i
(1-[g_i]\bt^{E^*_i})^{-1}$ is the equivariant Poincar\'e series of $\C^d$, the invariant part
$\widetilde{Z}_0(\bt)$ being the Poincar\'e series of the corresponding quotient space $\C^d/H$.

In $\C^d$ one can consider the {\em splice equations} as follows. Consider a matrix $\{\lambda_{ij}\}_{ij}$
of full rank and of size $d\times(d-2)$. Then the equations
$\sum_{i=1}^d\lambda_{ij}z_i^{\alpha_i}=0$, for $j\in\{1,\ldots, d-2\}$, determine in $\C^d$ an isolated
complete intersection singularity $(Y,0)$ on which the group $H$ acts as well.
Then $(X,0)=(Y,0)/H$ is a normal surface singularity with the topological type of the Seifert manifold
we started with. The equivariant Poincar\'e series of $(Y,0)$ is $Z(\bt)$ (\cite{neumann.abel}).
For $(X,0)$, \cite{BN} proves the identity $P(\bt)=Z(\bt)$ mentioned in Subsection \ref{FM},
hence $Z(\bt)$ is also the Poincar\'e series of the equivariant divisorial filtration associated with all
the vertices.

Theorem~\ref{th:REST} reduces the filtration to the $\Z$--filtration: the
divisorial filtration associated with the central vertex.  In the weighted homogeneous case this
filtration is also induced by the weighted homogeneous equations. Then,
$Z^{/H}(t)$ is the Poincar\'e series of $\C^d/H$,
$Z(t)$  is the equivariant Poincar\'e series of $Y$, hence $Z_0(t)$ is the Poincar\'e series of $X$,
cf. \cite{Do,neumann.abel,P}.

By \ref{FM}, $\{\mathrm{pc}(Z_h)\}_{h\in H}$  are the equivariant
geometric genera of the universal abelian cover $Y$ of $X$, hence
$\mathrm{pc}(Z_0)$ and $\mathrm{pc}(Z_{ne})$ are the geometric genera
of $X$ and $Y$ respectively, cf. \cite{NCL}.

\subsection{Seiberg--Witten theoretical interpretations}\labelpar{ss:onenodeSW}\
Fix $h\in H$. Then, for any lifting $\sum_ic_ig_i$ of $h$, Corollary \ref{cor:4.1} and Equation
\ref{eq:PCpol} give
\begin{equation}\label{eq:pcrh}
  \mathrm{pc}(Z_h)=\sum_{\ell\geq -\tc}\ \max\big\{0, -N_\bc(\ell)\big\}=
-\frsw_{-h*\sigma_{can}}(M)
-\frac{(K+2r_h)^2+|\cV|}{8}.\end{equation}
Recall  that
$\sum_h\frsw_{-h*\sigma_{can}}(M)$ is the {\it Casson--Walker invariant}
$\lambda(M)$. Hence, for the non--equivariant version   we get
\begin{equation}\label{eq:pcrhne}
  \mathrm{pc}(Z_{ne})=\frac{1}{2}\cdot
\sum_{0\leq x_i<\alpha_i\atop 0\leq i\leq d} \lfloor S(x)\rfloor\cdot
\lfloor S(x)-1\rfloor=
-\lambda(M)-\frd\cdot\frac{K^2+|\cV|}{8}+\sum_h\chi(r_h).\end{equation}
For explicit formulae of $\lambda(M)$ and $K^2+|\cV|$ in terms of
Seifert invariants see e.g. \cite[2.6]{SWII}).

\begin{remark} (\ref{eq:pcrh})
can be compared with a known formulae of the Seiberg--Witten
invariants involving the representative $s_h$. This will also lead us to an
expression for $\chi(r_h)-\chi(r_s)$ in terms of $N_{{\bf c}}(\ell)$.
Let $\bc(s_h)=(c_0,\ldots,c_d)$ be the coefficients of $s_{h,red}$,
cf. \ref{ss:seiferjel}. The set of all reduced coefficients $\bc(s_h)$,
when $h$ runs in $H$, is characterized
in \cite[11.5]{OSZINV} by the inequalities
\begin{equation}\label{eq:shcar}
\left\{ \begin{array}{l}
c_0\geq 0, \ \ \alpha_i>c_i\geq 0 \ \ (1\leq i\leq d) \\
N_\bc(\ell)\leq 0 \ \ \mbox{for any $\ell<0$}.\end{array}\right.
\end{equation}
Moreover, for this special lifting $\bc(s_h)$ of $h$, in \cite[\S 11]{OSZINV} is proved
\begin{equation}\label{eq:pcsh}
\sum_{\ell\geq 0}\ \max\big\{0, -N_{\bc(s_h)}(\ell)\big\}=
-\frsw_{-h*\sigma_{can}}(M)
-\frac{(K+2s_h)^2+|\cV|}{8}.\end{equation}
Using the discussion from the end of \ref{ss:seiferjel}, this can be rewritten for {\it any} lifting
$\bc$ of $h$ as
\begin{equation}\label{eq:pcsh2}
\sum_{\ell\geq -\tc+\lfloor \tc(s_h)\rfloor}\ \ \max\big\{0, -N_{\bc}(\ell)\big\}=
-\frsw_{-h*\sigma_{can}}(M)
-\frac{(K+2s_h)^2+|\cV|}{8}.\end{equation}
This compared with (\ref{eq:pcrh}) gives for any lifting
$\bc$ of $h$
\begin{equation}\label{eq:pcsh3}
\sum_{-\tc+\lfloor \tc(s_h)\rfloor  >\ell\geq -\tc}\ \ \max\big\{0, -N_{\bc}(\ell)\big\}=
\chi(r_h)-\chi(s_h).\end{equation}
\end{remark}

\begin{example}\labelpar{ex:sh}
The sum in (\ref{eq:pcsh3}), in general, can be non--zero. This happens,
for example,  in the case of the link of a rational singularity whose universal abelian cover is not rational. 
Here is a concrete example, cf. \cite[4.5.4]{Ng}:
take the Seifert manifold with $b=-2$ and
three legs, all of them with Seifert invariants $(\alpha_i,\omega_i)=(3,1)$.
For  $ h=\sum_{i=1}^3g_i$ one has $s_h=\sum_{i=1}^3E_i^*$, the $E_0$-coefficient of $s_h$ is 1,
$r_h=s_h-E_0$, and $\chi(s_h)=0$, $\chi(r_h)=1$.
\end{example}

\subsection{Ehrhart theoretical interpretations}\labelpar{ss:Ehr}\
We fix $h\in H$ as above and {\it we assume that $\tc\in[0,1)$}. Note that $Z_h(t)$ has the form
$t^{\tc}\sum_{\ell\geq 0}p_\ell t^\ell$; here the exponents $\{\tc+\ell\}_{\ell\geq 0}$
are the possible $E_0$--coordinates of the elements $(r_h+L)\cap \calS'$.

Let us compute the counting  function for $Z_h$. If $S(t)=\sum_r p_rt^r$ is a series,
we write $Q(S)(r')=\sum_{r<r'}p_r$, for $r'\in\Q_{\geq 0}$.
\begin{lemma}\label{lemma:P} For any $n\in\N_{\geq 0}$ one has the following facts.

(a) \ $Q(Z_h)(n)=Q(Z_h)(n+\tc)$.

(b) \ $Q(Z_h^{+})(n)$ is a step function (hence piecewise polynomial)
with $$Q(Z_h^{+})(n)={\mathrm{pc}}(Z_h) \ \ \ \mbox{for $n> \mathrm{deg}(Z_h^{+})$}.$$

(c) \ $Q(Z_h^{-})(n)$ is a quasipolynomial:
\begin{equation}\label{eq:PSp2}
Q(Z^{-}_h)(n)=(1-e\tc)n-e\cdot \frac{n(n-1)}{2}
-\sum_{i=1}^d\sum_{r_i=0}^{\alpha_i-1}\Big\{\frac{c_i-\omega_ir_i}{\alpha_i}\Big\}\,
\Big\lceil \frac{n-r_i}{\alpha_i}\Big\rceil\end{equation}
\begin{equation*}
=-\frac{en^2}{2}+\frac{en}{2}(\gamma+1-2\tc)
-\sum_{i=1}^d\sum_{r_i=0}^{\alpha_i-1}\Big\{\frac{c_i-\omega_ir_i}{\alpha_i}\Big\}\,
\Big(\Big\{\frac{r_i-n}{\alpha_i}\Big\}-\frac{r_i}{\alpha_i}\Big).
\end{equation*}
\end{lemma}
In particular, if $n=m\alpha$  for  $m\in \Z$, and $n>\mathrm{deg}(Z^{+}_h)$,
then the double sum is zero, hence
\begin{equation}\label{eq:PP}
Q(Z_h)(n)=-\frac{en^2}{2}+\frac{en}{2}(\gamma+1-2\tc)+\mathrm{pc}(Z_h).
\end{equation}
This is compatible with the expression provided by Theorem \ref{th:JEMS} and Theorem
\ref{th:REST}. Indeed, let us fix any chamber $\calC$ such that $int(\calC\cap \calS')\not=\emptyset$, and
$\calC$ contains the ray $\calR=\R_{\geq 0}\cdot E^*_0$. Since the numerator of $f(\bt)$ is
$(1-\bt^{E^*_0})^{d-2}$, all the $b_k$--vectors belong to $\calR$. In particular, $\cap_k(b_k+\calC)$ intersects
$\calR$ along a semi--line $\calR^{\geq c}=\R_{\geq const}\cdot E^*_0$ of $\calR$. Since $Q_h(l')$
is quasipolynomial on $\cap_k(b_k+\calC)$, cf. (\ref{ex:qbar2}), and a restriction of it is determined
by  (\ref{eq:SUM}) whose right hand side is a quasipolynomial too, we obtain that the identity
(\ref{eq:SUM}) is valid on $\calR^{\geq c}$ as well.

Recall that for any $h\in H$ and $l'\in L'$ we have a unique $q_{l',h}\in\square $ with
$l'+q_{l',h}\in r_h+L$. Hence we get
\begin{equation}\label{eq:QPpol}
Q_h(l')=
-\frsw_{-h*\sigma_{can}}(M)-\frac{(K+2l'+2q_{l',h})^2+|\cV|}{8}\ \ \ \ (l'\in \calR^{\geq c}).
\end{equation}
The term $q_{l',h}$ is responsible for the  non--polynomial behavior. Nevertheless, if we assume that
 $l'=m\fro E^*_0\in\calR^{\geq c}\cap L$, $m\in \Z$, then $q_{l',h}=r_h$, hence by  (\ref{eq:pcrh})
 \begin{equation}\label{eq:PPP}
Q_h(l')=-\frac{(l',l'+K+2r_h)}{2}+\mathrm{pc}(Z_h).\end{equation}
By Theorem \ref{th:REST} $Q_h(l')$ from (\ref{eq:PPP}) depends only on
the $E_0$-coefficient of $l'=m\fro E_0^*$, which  is exactly $m\alpha$. One sees that in fact
(\ref{eq:PPP}) agrees with (\ref{eq:PP}) if we set $n=m\alpha$.

\medskip

The non--equivariant version can be obtained by summation of (\ref{eq:PP}). For this we need
$\sum_h\tc(r_h)$. Consider the group homomorphism $\varphi:H\to \Q/\Z$ given by
$h\mapsto [\tc(r_h)]$, or $ [E^*_v]\mapsto [-(E^*_0,E^*_v)]$. Its image is generated by the
classes of $1/(\alpha_i|e|)$, hence its order is $\fro$. Hence, $\tc(r_h)$ vanishes exactly
$\frd/\fro$ times (whenever $h\in \ker(\varphi)$). In all other cases $\tc(r_h)\not=0$, and
$\tc(r_h)+\tc(r_{-h})=1$. In particular, $2\sum_h\tc(r_h)=\frd-\frd/\fro$.
Therefore, the summation of (\ref{eq:PP}) provides
\begin{equation}\label{eq:PPne}
Q(Z_{ne})(n)=-\frac{\frd en^2}{2}+\frac{\frd en}{2}\Big(\gamma+\frac{1}{\fro}\Big)+\mathrm{pc}(Z_{ne})
\ \ \ \ \mbox{(for $n\in\alpha\Z$)}.
\end{equation}

Next, we will identify the coefficients of (\ref{eq:PP}) and (\ref{eq:PPne})
with the first three coefficient of the
Ehrhart quasipolynomial $\calL^\calC_h(\calT)$ via the identity (\ref{ex:qbar2}).

For simplicity we will assume that $\fro=1$, in particular all the $b_k$--vectors belong to $L$.

If $l'\in\calR$, then by Theorem \ref{th:REST} the counting function $\calL^\calC_h(\calT,l')$
of the polytope $P^{(l')}$ depends only on the $E_0$--coefficient of $l'$; let us denote this coefficient by $l'_0$.

Hence,  this $\calL^\calC_h(\calT,l'_0)$ is the Ehrhart quasipolynomial of the $d$--dimensional
simplicial polytope,
being its $h$--class counting function. Via (\ref{eq:DET}) the definition 
(\ref{eq:Pv}) of this polytope becomes
\begin{equation}\label{eq:P0}
P_0=\Big\{(x_1,\ldots,x_d)\in (\R_{\geq 0})^d\,:\, \sum_i\frac{x_i}{|e|\alpha_i}< l'_0\Big\}.
\end{equation}
Let
\begin{equation}\label{eq:fra}
\calL^\calC_h(\calT,l'_0)=\sum_{j=0}^d \, \fra_{h,j}(l'_0)\cdot \frac{(l'_0)^j}{j!}
\end{equation}
be the coefficients of the Ehrhart quasipolynomial: each $\fra_{h,j}(l'_0)$ is a periodic function in $l'_0$ and is normalized by $1/j!$.
Since the numerator of $f$ is $(1-t^{1/|e|})^{d-2}$, by
(\ref{ex:qbar2}) we obtain for $l'\in\calR$
\begin{equation}\label{eq:fraP}
Q_h(l')=\sum_{j=0}^d\, \fra_{h,j}(l'_0)\cdot \frac{1}{j!}\ \sum_{k=0}^{d-2} (-1)^k\binom{d-2}{k}\Big(l'_0-\frac{k}{|e|}\Big)^j.
\end{equation}
This equals the expression
(\ref{eq:QPpol}) above. The non--polynomial behavior of these
two expressions indicate that $\fra_j(l'_0)$ is indeed non--constant periodic, and can be determined explicitly.

Since we are interested primarily in the Seiberg--Witten invariant, namely in $\mathrm{pc}(Z_h)$, we perform
this explicit identification only via the expressions (\ref{eq:PP}) and (\ref{eq:PPP}). Hence, similarly
as in these cases, we take $l'=m\fro E_0^*\in\calR^{\geq c}\cap L$, and we identify (\ref{eq:PP})
with (\ref{eq:fraP}) evaluated for $l'$, whose $E_0$--coefficient is $l'_0=m\alpha=n$.
In this case $\fra_{h,j}(n)$ is a {\it constant}, denoted by $\fra_{h,j}$, and
\begin{equation}\label{eq:fraPP}
-\frac{en^2}{2}+\frac{ne}{2}(\gamma+1-2\tc)+\mathrm{pc}(Z_h)=
\sum_{j=0}^d \fra_{h,j}\cdot \frac{1}{j!}\ \sum_{k=0}^{d-2} (-1)^k\binom{d-2}{k}\Big(n-\frac{k}{|e|}\Big)^j.
\end{equation}
Here the following combinatorial expression is helpful (see e.g. \cite[p. 7-8]{PSz})
\begin{equation}\label{eq:PSz}
\frac{(-1)^d}{(d-2)!}\cdot \sum_{k=0}^{d-2}(-1)^k\binom{d-2}{k}k^j=\left\{
\begin{array}{ll}
0 & \mbox{if $j<d-2$},\\
1 & \mbox{if $j=d-2$},\\
(d-2)(d-1)/2 & \mbox{if $j=d-1$},\\
(d-2)(d-1)d(3d-5)/24 & \mbox{if $j=d$}.\end{array}\right.
\end{equation}
We obtain
\begin{equation}\label{eq:fra3a}
\begin{array}{rl}
\frac{\fra_{h,d}}{|e|^{d}}=&
\frac{1}{|e|}\\
\frac{\fra_{h,d-1}}{|e|^{d-1}}=&
\frac{d-2}{2|e|}-\frac{1}{2}(\gamma+1-2\tc)
\\ 
\frac{\fra_{h,d-2}}{|e|^{d-2}}=&
\mathrm{pc}(Z_h)+\frac{(d-2)(3d-7)}{24|e|}-\frac{d-2}{4}(\gamma+1-2\tc)
.
\end{array}
\end{equation}
In particular, the $\fra_{h,d-2}$ can be identified (up to
 `easy' extra terms) with  ${\rm pc}(Z_h)$ (with  analytical interpretation
$\dim (H^1(\widetilde{Y},\cO_{\widetilde{Y}})_{\theta(h)})$ and Seiberg--Witten
theoretical interpretation (\ref{eq:pcrh})).
The first coefficients can also be identified with the equivariant volume of $P_0$,
(a fact already known in the non--equivariant cases). Usually (in the non--equivariant case, and when
we count the points of all the facets)
the second coefficient can be related with the volumes of the facets.
Here we eliminate from this count some of the facets, and we are in the equivariant situation as well.

In the non--equivariant case, if $\sum_{j=0}^d \fra_{j}\frac{n^j}{j!}$ is the classical Ehrhart polynomial
of $P_0$, then
\begin{equation}\label{eq:fra3b}
\begin{array}{rl}
\frac{\fra_{d}}{|e|^d}=&
\prod_i\alpha_i\\
\frac{\fra_{d-1}}{|e|^{d-1}}=&
\prod_i\alpha_i \cdot\left(-\frac{1}{\alpha}+\sum_i\frac{1}{\alpha_i}\right)/2
\\ 
\frac{\fra_{d-2}}{|e|^{d-2}}=&
\prod_i\alpha_i \Big(
\frac{\mathrm{pc}(Z_{ne})}{\prod_i\alpha_i}-\frac{(d-2)(3d-5)}{24}+\frac{d-2}{4}(-\frac{1}{\alpha}+
\sum_i\frac{1}{\alpha_i})\Big).
\end{array}
\end{equation}
In this non--equivariant case the identities (\ref{eq:fra3b}) are valid even without
the assumption $\fro=1$ by Theorem \ref{th:MAINLAST}.

The formulae in (\ref{eq:fra3a}) and (\ref{eq:fra3b}) can be further simplified if we replace $P_0$
by $|e|P_0$, or if we substitute in the Ehrhart polynomial the new variable $\lambda:= |e|l_0'$ instead of $l_0'$; cf. Section \ref{s:Last}.

\section{The two--node case}\labelpar{s:TN}

\subsection{Notations and the group $H$}\labelpar{ss:TNC}\
We consider the graph $G$ from Figure \ref{fig:2nodes}.

\begin{figure}[h!]
\vspace{0.5cm}
\begin{center}
\begin{picture}(200,60)(80,15)
\put(100,40){\circle*{3}}
\put(140,40){\circle*{3}}
\put(100,40){\line(1,0){55}}
\put(171,40){\makebox(0,0){$\cdots$}}
\put(200,40){\circle*{3}}
\put(240,40){\circle*{3}}
\put(185,40){\line(1,0){55}}
\put(50,73){\circle*{3}}
\put(100,40){\line(-3,2){10}}
\put(50,73){\line(3,-2){10}}
\multiput(77,55)(-3,2){3}%
{\circle*{1}}

\multiput(70,43)(0,-3){3}%
{\circle*{1}}

\put(50,7){\circle*{3}}
\put(100,40){\line(-3,-2){10}}
\put(50,7){\line(3,2){10}}
\multiput(77,25)(-3,-2){3}%
{\circle*{1}}

\put(290,73){\circle*{3}}
\put(240,40){\line(3,2){10}}
\put(290,73){\line(-3,-2){10}}
\multiput(263,55)(3,2){3}%
{\circle*{1}}

\multiput(270,43)(0,-3){3}%
{\circle*{1}}

\put(290,7){\circle*{3}}
\put(240,40){\line(3,-2){10}}
\put(290,7){\line(-3,2){10}}
\multiput(263,25)(3,-2){3}%
{\circle*{1}}

\put(100,50){\makebox(0,0){$E_0$}}
\put(140,50){\makebox(0,0){$\overline E_1$}}
\put(200,50){\makebox(0,0){$\overline E_s$}}
\put(240,50){\makebox(0,0){$\widetilde E_0$}}
\put(35,75){\makebox(0,0){$E_1$}}\put(35,5){\makebox(0,0){$E_d$}}
\put(303,70){\makebox(0,0){$\widetilde E_1$}}
\put(303,5){\makebox(0,0){$\widetilde E_{\widetilde{d}}$}}

\multiput(10,85)(0,-8){12}{\line(0,-1){4}}
\multiput(10,85)(8,0){26}{\line(1,0){4}}
\multiput(214,85)(0,-8){12}{\line(0,-1){4}}
\multiput(214,-7)(-8,0){26}{\line(-1,0){4}}
\put(20,40){\makebox(0,0){$G_0$}}

\multiput(330,80)(-8,0){26}{\line(-1,0){4}}
\multiput(330,80)(0,-8){12}{\line(0,-1){4}}
\multiput(126,80)(0,-8){12}{\line(0,-1){4}}
\multiput(126,-13)(8,0){26}{\line(1,0){4}}
\put(320,35){\makebox(0,0){$\widetilde G_0$}}

\end{picture}
\end{center}
\vspace{1cm}
\caption{Graph with two nodes}
\label{fig:2nodes}
\end{figure}
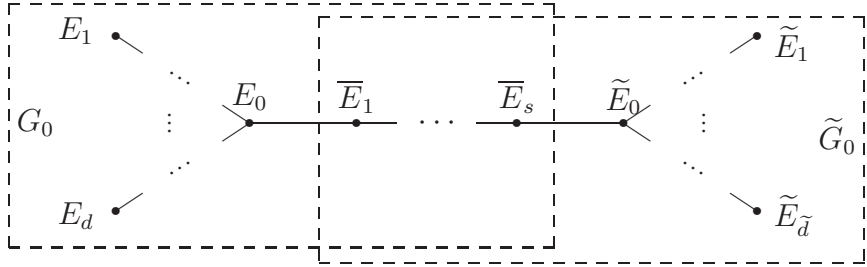

The nodes $E_0$ and $\widetilde{E}_0$ have decorations $b_0$ and $\widetilde{b}_0$ respectively.
Similarly as in the one--node case, we encode the decorations of maximal chains by
continued fraction expansions. In fact, it is convenient to consider the two maximal
star--shaped graphs $G_0$ and $\widetilde{G}_0$, and the corresponding normalized
Seifert invariants of their legs. Hence, let the
normalized Seifert invariants of the legs with ends $E_i$ ($1\leq i\leq d$) be  $(\alpha_i,\omega_i)$, while of the legs with ends  $\widetilde{E}_j$ ($1\leq j\leq \widetilde{d}$)
be $(\widetilde\alpha_j,\widetilde\omega_j)$.

The chain connecting the nodes, viewed in $G_0$ has normalized
Seifert invariants $(\alpha_0,\omega_0)$, while
viewed as a leg in $\widetilde{G}_0$, it has Seifert invariants
$(\alpha_0,\widetilde{\omega}_0)$. One has $\omega_0 \widetilde \omega_0=\alpha_0 \tau +1$.
Clearly, $\alpha_0$ is the determinant of the chain, and

\begin{picture}(400,40)(55,5)
\put(40,20){\makebox(0,0)[l]{$\omega_0:= \det(\hspace{23mm})$}}
\put(100,20){\circle*{3}}
\put(150,20){\circle*{3}}
\put(100,20){\line(1,0){10}}
\put(150,20){\line(-1,0){10}}
\put(125,20){\makebox(0,0){$\cdots$}}
\put(100,30){\makebox(0,0){$\overline E_2$}}
\put(150,30){\makebox(0,0){$\overline E_s$}}
\put(190,20){\makebox(0,0)[l]{$\widetilde{\omega}_0:= \det(\hspace{25mm})$}}
\put(250,20){\circle*{3}}
\put(300,20){\circle*{3}}
\put(250,20){\line(1,0){10}}
\put(300,20){\line(-1,0){10}}
\put(275,20){\makebox(0,0){$\cdots$}}
\put(250,30){\makebox(0,0){$\overline E_1$}}
\put(300,30){\makebox(0,0){$\overline E_{s-1}$}}

\put(346,20){\makebox(0,0)[l]{$\tau:= \det(\hspace{25mm}).$}}
\put(400,20){\circle*{3}}
\put(450,20){\circle*{3}}
\put(400,20){\line(1,0){10}}
\put(450,20){\line(-1,0){10}}
\put(425,20){\makebox(0,0){$\cdots$}}
\put(400,30){\makebox(0,0){$\overline E_2$}}
\put(450,30){\makebox(0,0){$\overline E_{s-1}$}}
\end{picture}

We denote the orbifold Euler numbers of the star--shaped subgraphs $G_0$
and $\widetilde G_0$ by
$$e=b_0+ \frac{\omega_0}{\alpha_0}+\sum_i \frac{\omega_i}{\alpha_i}  \ \ \ \mbox{and} \ \ \
\widetilde{e}=\widetilde b_0+
\frac{\widetilde\omega_0}{\alpha_0}+
\sum_j \frac{\widetilde\omega_j}{\widetilde\alpha_j}.$$
Consider the  {\it orbifold intersection matrix}
$\frI^{orb}=\left(
\begin{array}{lr}
 e & 1/\alpha_0 \\
 1/\alpha_0 & \widetilde{e}
\end{array}
\right)$, cf. \cite[4.1.4]{BN07}.

\noindent Then, the
negative definiteness of $\frI$ (or $G$)  implies
that  $\frI^{orb}$  is negative definite too, hence
$$\varepsilon:=\det \frI^{orb}=e \widetilde e - \frac{1}{\alpha_0^2}>0.$$
Then the determinant of the graph is $\det(G)=\det(-\frI)=\varepsilon\cdot \alpha_0 \prod_i \alpha_i \prod_j \widetilde{\alpha}_j$, cf. \cite{BN07}.

Using (\ref{eq:DETsgr}) we have the following intersection number of the
dual base elements:
\begin{equation}\label{eq:DET2node}
\begin{array}{rl}
(E^*_0)^2=\frac{\widetilde e}{\varepsilon}; \ \
(\widetilde E^*_0)^2=\frac{e}{\varepsilon}; \ \
(E^*_0,\widetilde E^*_0)=-\frac{1}{\alpha_0 \varepsilon}; \ \
(E^*_0,E^*_i)=\frac{\widetilde e}{\alpha_i \varepsilon}; \\
(E^*_0,\widetilde E^*_j)=-\frac{1}{\alpha_0 \widetilde \alpha_j \varepsilon}; \ \
(\widetilde E^*_0,E^*_i)=-\frac{1}{\alpha_0 \alpha_i \varepsilon}; \ \
(\widetilde E^*_0,\widetilde E^*_j)=\frac{e}{\widetilde \alpha_j \varepsilon}.
\end{array}
\end{equation}
Similarly as in \ref{ex:lens} or \ref{ss:seiferjel}, we can write
$n^i_{k_1,k_2}$, \  $\widetilde{n}^j_{k_1,k_2}$ resp. $\overline{n}_{k_1,k_2}$
 for the determinant of the
sub--chains  of the `left' $i^{th}$ leg, `right'  $j^{th}$ leg and connecting chain
 connecting the vertices $v_{k_1}$ and $v_{k_2}$. Let $\nu_i$ and $\tilde{\nu}_j$ be the
 number of vertices in the legs, cf. \ref{ss:seiferjel}. Then (with the standard notations, where
 $E_{i\ell}$ and $\widetilde{E}_{j\ell}$ are the vertices of the legs)
 one has the following slightly technical Lemma, but whose proof is standard based on the arithmetical
 properties of continued fractions:
\begin{lemma}\labelpar{lem:legeq}
(a) \
  $E^*_{i\ell}=n^i_{\ell+1,\nu_i} E^*_i+ \sum_{\ell<r\leq \nu_i}\frac{n^i_{1,\ell-1} n^i_{r+1,\nu_i}-
n^i_{1,r-1} n^i_{\ell+1,\nu_i}}{\alpha_i}E_{ir}$\ \ for any $1\leq \ell<\nu_i$.

(There is a similar formula for $\widetilde E^*_{j\ell}$.)

  (b) \  $\overline E^*_k=\overline{n}_{1,k-1}\overline E^*_1-\overline{n}_{2,k-1}E^*_0+
  \sum_{1\leq r<k} \frac{\overline{n}_{1,r-1} \overline{n}_{k+1,s}-\overline{n}_{1,k-1}
  \overline{n}_{r+1,s}}
{\alpha_0}\overline E_r \,$, for $1<k\leq s$.

(This is true even for $k=s+1$ with the identification $\overline{E}^*_{k+1}=\widetilde{E}^*_0$.)

\end{lemma}

Next, we give a presentation of $H=L'/ L$. Set $g_i:=[E^*_i]$ ($1\leq i\leq d$),
$\widetilde g_j:=[\widetilde E^*_j]$ ($1\leq j\leq \widetilde{d}$), $g_0:=[E^*_0]$ and
$\widetilde g_0:=[\widetilde E^*_0]$. Moreover we need to choose an additional generator
corresponding  to a vertex
sitting on the connecting chain: we choose $\overline{g}:=[\overline E^*_1]$ (this motivates
the choice in Lemma \ref{lem:legeq})(b) too).
The above lemma implies
\begin{equation}\label{eq:dual2nodes}
[E^*_{i\ell}]=n^i_{\ell+1,\nu_i}g_i, \ \ \ [\widetilde
E^*_{j\ell}]=\widetilde{n}^j_{\ell+1,\widetilde{\nu}_j}\widetilde g_j \ \ \
\mbox{and} \ \ \  [\overline{E}^*_k]=\overline{n}_{1,k-1}\overline{g}-\overline{n}_{2,k-1}g_0;
\end{equation}
and similar arguments as in the star--shaped case provides the following presentation for $H$
\begin{eqnarray}\label{eq:2Hpres}
{\textstyle  H={\mathrm{ab}}\langle \, g_0,\widetilde g_0, g_i, \widetilde g_j, \overline{g}
\,|\, g_0=\alpha_i\cdot g_i; \
\widetilde g_0 = \widetilde \alpha_j \cdot \widetilde g_j; \ \alpha_0 \cdot
\overline{g}=\omega_0 \cdot g_0 +\widetilde g_0;} \\ \nonumber {\textstyle
-\overline{g}-b_0\cdot g_0=\sum_i \omega_i \cdot g_i; \ -\widetilde \omega_0 \cdot
\overline{g} + \tau \cdot g_0
-\widetilde b_0 \cdot \widetilde g_0= \sum_j \widetilde \omega_j \cdot \widetilde g_j \rangle.}
\end{eqnarray}
Moreover,  for any  $l' \in L'$,
$$\textstyle{l'=c_0 E^*_0 + \widetilde c_0 \widetilde E^*_0+\sum_k \overline{c}_k
\overline{E}^*_k+\sum_{i,\ell} c_{i\ell}E^*_{i\ell}+\sum_{j\ell}
\widetilde c_{j\ell}\widetilde E^*_{j\ell},}$$
if we define its {\it reduced transform}  $l'_{red}$ by
\begin{equation*}
({c}_0-\sum_{k>1}
\overline{n}_{2,k-1}\overline{c}_k){E}^*_0 +\widetilde c_0 \widetilde E^*_0 + (\overline{c}_1+\sum_{k>1}\overline{n}_{1,k-1}\overline{c}_k)\overline{E}^*_1+
\sum_{i,\ell}
c_{i\ell}n^i_{\ell+1,\nu_i}E^*_i+ \sum_{j,\ell} \widetilde c_{j\ell}\widetilde n^j_{\ell+1,\widetilde{\nu}_j}\widetilde E^*_j,
\end{equation*}
then, by Lemma \ref{lem:legeq},  $[l']=[l'_{red}]$ in $H$.
Moreover,  if for any $l'\in L'$ we distinguish the $E_0$ and $\widetilde{E}_0$ coefficients,
that is,
we set $c(l'):=-(E^*_0,l')$ and $\widetilde c(l'):=-(\widetilde E^*_0,l')$,
then $c(l')=c(l'_{red})$
and $\widetilde c(l')=\widetilde c(l'_{red})$ as well. Lemma \ref{lem:legeq}(b) (applied for
$k=s+1$) provide these coefficients for $\overline{E}_1$:
\begin{equation}\label{eq:E1bar}
(\overline{E}_1^*,E_0^*)=\frac{1}{\varepsilon\alpha_0}\big(\omega_0\widetilde{e}-\frac{1}{\alpha_0}\big),
\ \ \ \
(\overline{E}_1^*,\widetilde{E}_0^*)=\frac{1}{\varepsilon\alpha_0}
\big(e-\frac{\omega_0}{\alpha_0}\big).
\end{equation}
We will use the coefficients $\bc=(c_0,\widetilde c_0,
\overline c, c_i,\widetilde c_j)$ to write an element
$l'_{red}=c_0E^*_0+ \widetilde{c}_0\widetilde E^*_0+ \overline{c}\overline{E}^*_1+
\sum_ic_iE^*_i+\sum_j\widetilde{c}_j\widetilde E^*_j$. Then (\ref{eq:DET2node}) and (\ref{eq:E1bar})
imply that
\begin{equation}\label{eq:AA}
\Bigg( \begin{array}{c}  c\\   \widetilde{c} \end{array} \Bigg)=
\Bigg( \begin{array}{c}  c(l'_{red})\\   \widetilde{c}(l'_{red})
 \end{array} \Bigg)
 = (-\frI^{orb})^{-1}\cdot \Bigg( \begin{array}{c}   A\\
\widetilde A   \end{array} \Bigg)
=\frac{1}{\varepsilon}
\Bigg( \begin{array}{lr}
 -\widetilde{e} & 1/\alpha_0 \\
 1/\alpha_0 & -e
\end{array} \Bigg)
\cdot \Bigg( \begin{array}{c} A\\
\widetilde A \end{array} \Bigg),
\end{equation}
where
\begin{equation*}
 A:=c_0+\sum_i \frac{c_i}{\alpha_i}+\frac{\omega_0}{\alpha_0} \overline{c}, \ \ \ \ \ \ \
 \widetilde A:=\widetilde c_0+\sum_j \frac{\widetilde c_j}{\widetilde \alpha_j}+\frac{1}{\alpha_0} \overline{c}.
\end{equation*}
Therefore, any
$h\in H$ has a lift of type  $l'_{h,red}$. Although the corresponding coefficients  $c$ and
$\widetilde c$ depend on the lift, by adding $\pm E_0$ and $\pm \widetilde E_0$ to $l'_{h,red}$ we can achieve
$c,\widetilde c \in [0,1)$, and these values
are uniquely determined by $h$. For example, the reduced  transform $(r_h)_{red}$ of
$r_h$ satisfies  $c((r_h)_{red})=c(r_h)\in[0,1) $ and $\widetilde c((r_h)_{red})=\widetilde c(r_h)
\in[0,1)$ since $r_h \in \square$.

As we will see, for different elements of $h\in H$, we have to shift the rank two lattices  by 
vectors of type $(c,\widetilde{c})$, hence the vectors 
$(c,\widetilde c)$ will play a crucial role later.

\subsection{Interpretation of $Z(\bt)$}\
If we wish to compute the periodic constant of $Z^e(\bt)$, by Theorem
 \ref{th:REST} we can eliminate all the variables of
$Z^e(\bt)$ except  the variables of the nodes; these remaining two variables
are denoted by $(t,\wtt)$.
Therefore the equivariant form of reciprocal of the denominator is
\begin{eqnarray*}
Z^{/H}(t,\wtt) &=& \prod_{i}\,\big(1-t^{-(E^*_{i}, E^*_0)}\wtt^{-(E^*_{i}, \widetilde E^*_0)}
[g_i]\big)^{-1}\cdot
\prod_{j}\,\big(1-t^{-(\widetilde E^*_{j}, E^*_0)}\wtt^{-(\widetilde E^*_{j},
\widetilde E^*_0)}[\widetilde g_j]
\big)^{-1}\\
&=& \sum_{x_i,\widetilde x_j\geq 0}\, t^{\,\frac{-\widetilde e}{\varepsilon}\sum_i
\frac{x_i}{\alpha_i} + \frac{1}{\alpha_0 \varepsilon} \sum_j \frac{\widetilde x_j}{\widetilde \alpha_j}}
\,\, \wtt^{\ \frac{1}{\alpha_0 \varepsilon}\sum_i
\frac{x_i}{\alpha_i} + \frac{-e}{\varepsilon} \sum_j \frac{\widetilde x_j}
{\widetilde \alpha_j}}\ \big[\,\textstyle{\sum _i} x_ig_i
+\textstyle{\sum_j} \widetilde x_j \widetilde g_j\big].
\end{eqnarray*}
We fix a lift $c_0 E^*_0+\widetilde c_0 \widetilde E^*_0 +\overline c\, \overline{E}^*_1+
\sum_i c_i E^*_i+\sum_j \widetilde c_j \widetilde E^*_j$ of $h$.  Then
the class of $\sum _i x_iE^*_i+\sum_j \widetilde x_j \widetilde E^*_j$
equals $h$ if and only if its  difference with
the lift is a linear combination of the relation in \ref{eq:2Hpres}. In other words,
if  there exist $\ell_0,\widetilde \ell_0, \overline \ell,\ell_i,
\widetilde \ell_j \in \Z$ such that
$$\begin{array}{lrlllrll}
(a) \ & -c_0&= \sum_i \ell_i-b_0 \ell_0 +\tau \widetilde \ell_0 +\omega_0 \overline \ell \hspace{5mm}&
(c) \ & x_i-c_i&= -\omega_i \ell_0-\alpha_i \ell_i & \ (i=1,\ldots, d)
\\
(b) \ & -\widetilde c_0&= \sum_j \widetilde \ell_j-\widetilde b_0 \widetilde\ell_0 +\overline \ell &
(d) \ & \widetilde x_j-\widetilde c_j&= -\widetilde \omega_j \widetilde\ell_0-\widetilde\alpha_j
\widetilde\ell_j & \ (j=1,\ldots, \widetilde d)
\\
(e) \ & -\overline c&= -\ell_0-\widetilde\omega_0 \widetilde\ell_0-\alpha_0 \overline \ell.
\end{array}$$
From (e) we deduce that
\begin{equation}\label{eq:CONGR}
\ell_0+\widetilde\omega_0 \widetilde\ell_0\equiv \overline c \,(\mbox{mod} \,\alpha_0).
\end{equation}
Since $x_i,\widetilde x_j\geq 0$, (c) and (d) implies $\frac{c_i-\omega_i\ell_0}{\alpha_i}\geq \ell_i$,
$\frac{\widetilde c_j-\widetilde\omega_j \widetilde\ell_0}{\widetilde\alpha_j}\geq \widetilde\ell_j$.
Recall also that $\omega_0 \widetilde \omega_0=\alpha_0 \tau +1$.
Therefore if we set $m_i:=\lfloor \frac{c_i-\omega_i\ell_0}{\alpha_i} \rfloor-\ell_i$ and $\widetilde m_j:=
\lfloor \frac{\widetilde c_j-\widetilde\omega_j \widetilde\ell_0}{\widetilde\alpha_j}\rfloor-\widetilde \ell_j$
 non--negative integers then the number of the realization of $h$ in the form $\sum_i x_ig_i+\sum_j\widetilde x_j\widetilde g_j$
is determined by the number of non--negative  integral $(d+\widetilde d)$--tuples $(m_i,\widetilde m_j)$
satisfying
$$\begin{array}{lrlr}
  & N_\bc(\ell_0,\widetilde \ell_0)&:= c_0+\frac{\omega_0}{\alpha_0}\overline c -(b_0+\frac{\omega_0}{\alpha_0})
\ell_0- \frac{1}{\alpha_0}\widetilde\ell_0 + \sum_i \lfloor \frac{c_i-\omega_i\ell_0}{\alpha_i}
\rfloor& = \sum_i m_i \, , \\
& \widetilde N_\bc(\ell_0,\widetilde \ell_0)&:= \widetilde c_0+\frac{1}{\alpha_0}\overline c -
(\widetilde b_0+\frac{\widetilde\omega_0}{\alpha_0}) \widetilde\ell_0- \frac{1}{\alpha_0}\ell_0 +
\sum_j \lfloor \frac{\widetilde c_j-\widetilde\omega_j\widetilde\ell_0}{\widetilde\alpha_i} \rfloor& =
\sum_j \widetilde m_j \, .
  \end{array}
$$
This number is $\binom{N_\bc(\ell_0,\widetilde \ell_0)+d-1}{d-1}\binom{\widetilde N_\bc(\ell_0,\widetilde \ell_0)
+\widetilde d-1}{\widetilde d-1}$ if $N_\bc$ and $\widetilde N_\bc$ are non--negative, otherwise it is $0$.
Note that (\ref{eq:CONGR}) guarantees that both $N_\bc$ and $\widetilde N_\bc $ are integers.
Furthermore, (c) and (d) and (\ref{eq:AA}) show  that
the exponent of $t$ and $\wtt$ in the formula of $Z_h^{/H}(t,\wtt)$
are  equal to $\ell_0+c$ and  $\widetilde \ell_0 +\widetilde c$ respectively. Hence
$$Z_h^{/H}(t,\wtt)=\sum  
\ \binom{N_\bc(\ell,\widetilde \ell)+d-1}{d-1}
\ \binom{\widetilde N_\bc(\ell,\widetilde \ell)+\widetilde d-1}{\widetilde d-1}\
 t^{\ell+c} \ \wtt^{\widetilde \ell +\widetilde c},
$$
where the sum runs over $(\ell,\widetilde \ell)\in \Z^2$ with  $\ell+\widetilde\omega_0 \widetilde\ell \equiv
\overline c \ (\mbox{mod} \ \alpha_0)$.\\
The numerator of $Z(t,\wtt)$ is
$\big(1-t^{-(E^*_{0}, E^*_0)}\wtt^{-(E^*_{0}, \widetilde E^*_0)}[g_0]\big)^{d-1}\cdot
\big(1-t^{-(\widetilde E^*_{0}, E^*_0)}\wtt^{-(\widetilde E^*_{0}, \widetilde E^*_0)}[\widetilde g_0]
\big)^{\widetilde{d}-1}$.
Hence we get $Z^e$ by multiplying this expression by $\sum_hZ_h^{/H}[h]$.
Recall that
 $h=c_0 g_0+\widetilde c_0 \widetilde g_0 +\overline c\, \overline{g}+
\sum_i c_i g_i+\sum_j \widetilde c_j \widetilde g_j$ is paired  with ${\bf c}$. Set
$h':=h+ k g_0+\widetilde k \widetilde g_0$ which corresponds  to ${\bf c'}={\bf c}+(k,
\widetilde k,0,0,0)$. Hence $Z_{h'}[h']$ is the next  sum according to the  decompositions
$h'=h+ k g_0+\widetilde k \widetilde g_0$:
$$
\sum_{k=0}^{d-1}(-1)^k\binom{d-1}{k}\sum_{\widetilde k=0}^{\widetilde d-1}
(-1)^{\widetilde k}\binom{\widetilde d-1}{\widetilde k} \cdot \hspace{3cm} $$
$$\sum_h \
\Bigg(
\sum_{\ell+\widetilde\omega_0 \widetilde\ell \equiv
\overline c \ (\alpha_0)}\ \binom{N_\bc(\ell,\widetilde \ell)+d-1}{d-1}
\ \binom{\widetilde N_\bc(\ell,\widetilde \ell)+\widetilde d-1}{\widetilde d-1}\
 t^{\ell+c+\frac{-\widetilde e k+\widetilde k/\alpha_0}{\varepsilon}}
  \ \wtt^{\ \widetilde \ell +\widetilde c+
 \frac{-e\widetilde k+k/\alpha_0}{\varepsilon}}\Bigg)\ [h']
$$
$$
=\sum_{k=0}^{d-1}(-1)^k\binom{d-1}{k}\sum_{\widetilde k=0}^{\widetilde d-1}
(-1)^{\widetilde k}\binom{\widetilde d-1}{\widetilde k} \cdot \hspace{3cm}$$
$$\sum_{h}\Bigg(
\sum_{\ell+\widetilde\omega_0 \widetilde\ell \equiv
\overline c \ (\alpha_0)}\ \binom{N_{\bc'}(\ell,\widetilde \ell)-k+d-1}{d-1}
\ \binom{\widetilde N_{\bc'}(\ell,\widetilde \ell)-\widetilde k +\widetilde d-1}{\widetilde d-1}\
 t^{\ell+c'} \ \wtt^{\widetilde \ell +\widetilde c'}\Bigg)\  [h'].$$
Rearranging and using the combinatorial formula
$\sum_{k=0}^{d-1}(-1)^k\binom{N-k+d-1}{d-1}\binom{d-1}{k}=1$ for $N\geq 0$ and $=0$ otherwise,
we get the following.
\begin{theorem}\label{th:UJZH}
For any $h\in H$ one has
\begin{equation}\label{eq:ZHTT}
 Z_h(t,\wtt)=\sum_{(\ell,\widetilde \ell)\in \cS_\bc} t^{\ell+c} \ \wtt^{ \ \widetilde\ell+\widetilde c},
 \ \ \ \ \mbox{where}
\end{equation}
\begin{equation}\label{eq:module}
\cS_\bc:=\left\{ (\ell,\widetilde \ell) \in \Z^2 \ : \ N_\bc(\ell,\widetilde \ell)\geq 0, \
\widetilde N_\bc(\ell,\widetilde \ell)\geq 0\ \ \mbox{and} \
\ \ell+\widetilde\omega_0 \widetilde\ell \equiv
\overline c \ (\mbox{mod} \ \alpha_0)
\right\}.
\end{equation}\end{theorem}
It is straightforward to verify that
the right hand side of (\ref{eq:ZHTT})  does not depend on the choice of $\bc$, it depends only on $h$.
The identity (\ref{eq:ZHTT}) is remarkable: it realizes the bridge between the series $Z^e$ and the 
equivariant Hilbert series of {\it affine monoids and their modules}.

\subsection{The structure of $\cS_\bc$}\
Recall that for any $h \in H$ we consider a lift of $h$ identified by a certain
${\bf c}$ which detemines the pair $(c,\widetilde c)$ (cf.  (\ref{eq:AA})), and the integers
 $N_\bc(\frl)$ and $\widetilde N_\bc(\frl)$, where  $\frl=(\ell,\well) \in \Z^2$.
We define $$\Z^2({\bf c}):= \{(\ell,\well)\in \Z^2\,:\,
\ell+\widetilde \omega_0 \well \equiv \overline c \
(\mbox{mod} \ \alpha_0)\}.$$

If  $h=0$ then we always choose the zero lift with ${\bf c}={\bf 0}$.

If, in the definition of  $N_\bc(\frl)$  and  $\widetilde N_\bc(\frl)$, 
we replace each $[y]$ by $y$, we get  the entries of
$$\Bigg(\begin{array}{c}  A-e\ell_0-\widetilde \ell/\alpha_0 \\ \widetilde A-\ell_0/
\alpha_0-\widetilde e \widetilde \ell_0 \end{array}\Bigg)=
  -\frI^{orb}\Bigg( \begin{array}{c} \ell+c \\   \well+\widetilde c \end{array}\Bigg).$$
This motivates to define
\begin{equation}
\overline\calS_\bc:=\Big\{ \frl\in \zc \ : \ -\frI^{orb}
\Bigg( \begin{array}{c}
 \ell+c \\
 \well+\widetilde c
\end{array}\Bigg) \geq 0
\Big\}.
\end{equation}
Clearly $\calS_{{\bf c}}\subset \overline{\calS}_{{\bf c}}$.
We also consider
$\calC^{orb}$,  the real cone $\{\frl\in\R^2\ :\ -\frI^{orb}\cdot \frl\geq 0\}$.  Then
$\overline \calS_{\bf c}=\big(\calC^{orb}-(c,\widetilde{c})\big)\cap
\zc$.

\begin{lemma}\labelpar{lem:fingen} (1) \ $\calS_{\bf 0}$ and $\overline \calS_{\bf 0}$ are affine monoids.
$\overline \calS_{\bf 0}$ is the normalization of  $\calS_{\bf 0}$.

(2) \  $\calS_{{\bf c}}$ and $ \overline{\calS}_{{\bf c}}$ are  finitely generated $\calS_{\bf 0}$-modules,
$\calS_{{\bf c}}$ is a submodule of $ \overline{\calS}_{{\bf c}}$.
\end{lemma}
\begin{proof} (1) is elementary. By Corollary \cite[2.12]{BG}
 $\overline \calS_\bc$ is finitely generated over $\overline\calS_0$, but
$\overline \calS_0$ itself  is finitely generated as an $\calS_0$ module.
\end{proof}
\begin{lemma}\labelpar{lem:VV} There exists $\frv_1$ and $\frv_2$ elements of $\Z^2$ with the following properties:

(a) $\frv_1$ and $\frv_2$ belong to $\calS_{{\bf 0}}$ and
 $\R_{\geq 0}\frv_1+\R_{\geq 0}\frv_2=\calC^{orb}$.

 (b) For any $\frl\in\overline{\calS}_{{\bf c}}$ one has:
 $$
 \begin{array}{ll}
 (i) \ \ N_{{\bf c}}(\frl+\frv_1)=N_{{\bf c}}(\frl); \ \ & \ \
 (\widetilde{i}) \ \ \widetilde{N}_{{\bf c}}(\frl+\frv_2)=\widetilde{N}_{{\bf c}}(\frl);\\
 (ii) \ \ N_{{\bf c}}(\frl+\frv_2)\geq 0; \ \ & \ \
 (\widetilde{ii}) \ \ \widetilde{N}_{{\bf c}}(\frl+\frv_1)\geq 0.
 \end{array}
 $$
\end{lemma}
\begin{proof} We choose $\frv_1$ and $\frv_2$ such that
$\widetilde N_{{\bf 0}}(\frv_1)\geq \widetilde d-1$ and
$N_{{\bf 0}}(\frv_2)\geq d-1$, and with

(A)  $\frv_1=(\ell_1,\well_1)\in \zc$
such that $\{-\omega_i\ell_1/\alpha_i\}=0$ for all $i$, and  $N_{{\bf 0}} (\frv_1)=0$;

(B) $\frv_2=(\ell_2,\well_2)\in \zc$
such that $\{-\widetilde{\omega}_j\widetilde{\ell}_2/\widetilde{\alpha}_j\}=0$
for all $j$, and $\widetilde{N}_{{\bf 0}} (\frv_2)=0$.

\noindent  Then $\frv_1$ and $\frv_2$ satisfy (a), and (b)($i$), and (b)($\widetilde{i}$).
Furthermore,
note that $N_{{\bf c}}(\frl+\frv_2)\geq  N_{{\bf c}}(\frl)+N_{{\bf 0}}(\frv_2)$
and for any  $\frl\in\overline{\calS}_{{\bf c}}$
one has $N_{{\bf c}}(\frl)\geq -(d-1)$,
hence all the conditions will be satisfied.
\end{proof}
\begin{remark}
Usually, the `universal restrictions'  $\widetilde N_{{\bf 0}}(\frv_1)\geq \widetilde d-1$
  and $N_{{\bf 0}}(\frv_2)\geq d-1$  in the proof of
  Lemma \ref{lem:VV}  provide  rather `large' vectors  $\frv_1$ and $\frv_2$. Nevertheless,
usually much smaller vectors also satisfy (a) and (b). Here is another choice.
Besides (A) and (B) we impose the following:

(C) Let $\Box=\Box(\frv_1,\frv_2)=
 \left\{ \frl=q_1\frv_1+q_2\frv_2 \ : \ 0\leq q_1,q_2<1 \right\}$ be
  the semi--open cube in $\calC^{orb}$. Then we require $N_{{\bf 0}}(\frv_2)\geq 0$ and
  $N_{{\bf c}}(\frl_\Box+\frv_2)\geq 0$ for any
  $\frl_\Box\in (\Box-(c,\widetilde c))\cap\zc$; and symmetrically:
  $\widetilde N_{{\bf 0}}(\frv_1)\geq 0$ and
  $\widetilde N_{{\bf c}}(\frl_\Box+\frv_1)\geq 0$ for any
   $\frl_\Box\in (\Box-(c,\widetilde c))\cap\zc$.
\end{remark}

The wished inequality for any $\frl\in\overline{\calS}_{{\bf c}}$ then follows from
$N_{{\bf c}}(\frl_\Box +k_1\frv_1+k_2\frv_2+\frv_2)=N_{{\bf c}}(\frl_\Box +k_2\frv_2+\frv_2)\geq
N_{{\bf c}}(\frl_\Box +\frv_2)+k_2N_{{\bf 0}}(\frv_2)$ (and its symmetric version).

In the sequel the next two subsets of  $\overline{\calS}_{{\bf c}}$ will be crucial.
\begin{eqnarray*}
\calS^-_{{\bf c},1}
:=&\hspace*{-2mm}
\big\{ \frl\in (\Box-(c,\widetilde c))\cap\zc \ : \ N_\bc(\frl)<0\big\} , \\
\calS^-_{{\bf c},2}:=&\hspace*{-2mm}
\big\{\frl\in (\Box-(c,\widetilde{c}))\cap\zc \ : \ \widetilde N_\bc(\frl)
<0\big\}.
\end{eqnarray*}
Again, both sets $\calS^-_{{\bf c},1}$ and $\calS^-_{{\bf c},2}$ are independent of the choice of ${\bf c}$, they
depend only on $h$.

\begin{proposition}\labelpar{prop:str} Let $\frv_1$ and $\frv_2$
be as in Lemma \ref{lem:VV}. Then
\begin{eqnarray*}
 &(1)&  \overline\calS_\bc= \bigsqcup_{\frl\in (\Box-(c,\widetilde{c}))\cap \zc} \frl+\Z_{\geq 0}\frv_1+\Z_{\geq 0}\frv_2 \\
 &(2)&  \overline\calS_\bc\setminus \calS_\bc = \big(\bigsqcup_{\frl\,\in \calS^-_{\bc,1}}
 \frl+\Z_{\geq 0}\frv_1\big) \ \cup\ \big(\bigsqcup_{\frl\,\in \calS^-_{\bc,2}}
 \frl+\Z_{\geq 0}\frv_2\big), \\
 & &\mbox{but} \ \ \big(\bigsqcup_{\frl\,\in \calS^-_{\bc,1}}
 \frl+\Z_{\geq 0}\frv_1\big) \ \cap\ \big(\bigsqcup_{\frl\,\in \calS^-_{\bc,2}}
 \frl+\Z_{\geq 0}\frv_2\big)= \bigsqcup_{\frl\,\in \calS^-_{\bc,1} \cap \calS^-_{\bc,2}} \frl \ .
\end{eqnarray*}
\end{proposition}
\begin{proof} The statements follow from the choice of $\frv_1$ and $\frv_2$ and
properties (a) and (b) of Lemma \ref{lem:VV}.
Compare also with the structure theorem  \cite[4.36]{BG} of $\calS_{{\bf 0}}$ modules. \end{proof}

\subsection{The periodic constant and $\frsw$ in the equivariant case. }\
Set $\bt=(t,\wtt)$.
Using (\ref{eq:ZHTT}) and Proposition \ref{prop:str}
one can write $Z_h(\bt)/\bt^{(c,\widetilde c)}$ in the next form:
$$\sum_{\frl\in (\Box-(c,\widetilde c))\cap \Z^2\cap (\equiv_\bc)}
\frac{\bt^\frl}{(1-\bt^{\frv_1})(1-\bt^{\frv_2})}
-\sum_{\frl\in \calS^-_{\bc,1}}\frac{\bt^\frl}{1-\bt^{\frv_1}} -\sum_{\frl\in \calS^-_{\bc,2}}
\frac{\bt^\frl}{1-\bt^{\frv_2}}+\sum_{\frl\in \calS^-_{\bc,1}\cap \calS^-_{\bc,2}}\bt^\frl.$$

Next, we apply the decomposition established in subsection \ref{ss:POLPART}.
Here it is important to {\it choose ${\bf c}$ in such a way that
$c\in[0,1)$ and $\widetilde c\in[0,1)$}.

Note that $\frv_1\in\R_{>0}(1/\alpha_0,-e)$ and $\frv_2\in\R_{>0}(-\widetilde e,1/\alpha_0)$,
hence $\frv_2$ sits in the cone determined by $\frv_1$ and $(1,0)$. Then, as in \ref{ss:POLPART},
we set $\Xi_1:=\{(\ell,\widetilde \ell)\,:\, 0\leq \ell < \mbox{first coordinate of $\frv_1$}\}$
and $\Xi_2:=\{(\ell,\widetilde \ell)\,:\, 0\leq \widetilde \ell < \mbox{second  coordinate of $\frv_2$}\}$, and for any $\frl\in \calS^-_{\bc,i}$  the unique
$n_{\frl,i}$ such that $\frl-n_{\frl,i}\frv_i\in\Xi_i$, $i=1,2$.
Then subsection \ref{ss:POLPART} provides the following decomposition
\begin{equation*}\label{eq:2nZ+}
\begin{array}{l}
 Z^+_h(\bt)=\bt^{(c,\widetilde c)}
 \left( \sum_{\frl\in \calS^-_{\bc,1}}\sum_{j=1}^{n_{\frl,1}}\bt^{\frl-j\frv_1} +
 \sum_{\frl\in \calS^-_{\bc,2}}
\sum_{j=1}^{n_{\frl,2}}\bt^{\frl-j\frv_2}+\sum_{\frl\in \calS^-_{\bc,1}\cap
\calS^-_{\bc,2}}\bt^\frl
\right) \ \\
Z^-_h(\bt)=\bt^{(c,\widetilde c)} \left( \sum_{\frl\in (\Box-(c,\widetilde c))\cap \Z^2\cap (\equiv_\bc)}
\frac{\bt^\frl}{(1-\bt^{\frv_1})(1-\bt^{\frv_2})} -\sum_{\frl\in \calS^-_{\bc,1}}\frac{\bt^{\frl-n_{\frl,1}\frv_1}}{1-\bt^{\frv_1}}
-\sum_{\frl\in \calS^-_{\bc,2}}\frac{\bt^{\frl-n_{\frl,2}\frv_2}}{1-\bt^{\frv_2}} \right) \ .
\end{array}
\end{equation*}
Therefore, by \ref{rem:fh} and Theorem \ref{th:POLPART} we get
$${\rm pc}^{\calC^{orb}}_h(Z)=
{\rm pc}^{\calC^{orb}}(Z_h(\bt)/\bt^{(c,\widetilde c)})=Z^+_h(1,1)=
\sum_{\frl\in \calS^-_{\bc,1}}n_{\frl,1}+
\sum_{\frl\in \calS^-_{\bc,2}}n_{\frl,2}+
|\calS^-_{\bc,1} \cap \calS^-_{\bc,2}| \ .$$

\begin{corollary} Choose  ${\bf c}$ in such a way that
$c\in[0,1)$ and $\widetilde c\in[0,1)$. Then one has the following combinatorial formula for the
  normalized Seiberg--Witten invariant of $M$
 \begin{equation*}
-\frac{(K+2r_h)^2+|\cV|}{8}-\frsw_{-h*\sigma_{can}}(M)=
\sum_{\frl\in \calS^-_{\bc,1}}n_{\frl,1}+
\sum_{\frl\in \calS^-_{\bc,2}}n_{\frl,2}+
|\calS^-_{\bc,1} \cap \calS^-_{\bc,2}|.
\end{equation*}
\end{corollary}
\begin{proof}
Use Corollary \ref{cor:4.1}, the Theorem \ref{th:REST} and the above computation.
\end{proof}

\subsection{The periodic constant and  $\lambda(M)$ in the non--equivariant case}\labelpar{ss:NEC}\
Though the non--equivariant $ Z_{ne}$  can be obtained by the sum $\sum_hZ_h$ treated in the previous subsection, here we
provide a more direct procedure, which leads to a new formula.
Write $J:=(-\frI^{orb})^{-1}$ and $\bt^{\binom{a}{b}}$ for $t^a\wtt^{b}$.
Applying the reduction \ref{th:REST} for the definition
(\ref{eq:INTR}) of $Z$, we get
\begin{eqnarray*}
 Z_{ne}(\bt)=\frac{(1-\bt^{J\binom{1}{0}})^{d-1}(1-\bt^{J\binom{0}{1}})^{\widetilde d-1}}
{\prod_i(1-\bt^{J\binom{1/\alpha_i}{0}})\prod_j(1-\bt^{J\binom{0}{1/\widetilde \alpha_j}})}.
\end{eqnarray*}
Set
$S(x):= \sum_i x_i/\alpha_i$ and $\widetilde S(\widetilde x):= \sum_j \widetilde{x}_j/\widetilde{\alpha}_j$.
 Similarly as in \ref{eq:NE}, $Z_{ne}(\bt)$ can be written as
\begin{eqnarray*}
\sum_{0\leq x_i<\alpha_i, 0\leq i\leq d\atop 0\leq \widetilde x_j<\widetilde\alpha_j,
0\leq j \leq\widetilde d}f(x,\widetilde x), \ \ \mbox{where} \ \ f(x,\widetilde x)=
 \frac{\bt^{J\binom{S(x)}{\widetilde S(\widetilde x)}}}{(1-\bt^{J\binom{1}{0}})
(1-\bt^{J\binom{0}{1}})}.
\end{eqnarray*}
By the substitution $u_1=\bt^{J\binom{1}{0}}$ and $u_2 =\bt^{J\binom{0}{1}}$, $f(x,\widetilde x)$ transforms into  $u_1^{S(x)}u_2^{\widetilde S(\widetilde x)}/ (1-u_1)(1-u_2)$.
The division of this fraction (with remainder) is elementary, hence $f(x,\widetilde x)$ equals
$$
\bt^{J\binom{S_{rat}}{\widetilde S_{rat}}} \left( \sum_{n=0}^{S_{int}-1}
\sum_{k=0}^{\widetilde S_{int}-1}\bt^{J\binom{n}{k}}
-\sum_{k=0}^{S_{int}-1}\frac{\bt^{J\binom{k}{0}}}{1-\bt^{J\binom{0}{1}}} -\sum_{\widetilde k=0}^{\widetilde S_{int}-1}
\frac{\bt^{J\binom{0}{\widetilde k}}}{1-\bt^{J\binom{1}{0}}}+ \frac{1}{(1-\bt^{J\binom{1}{0}})
(1-\bt^{J\binom{0}{1}})}\right)\ ,$$
where $S_{int}:=\lfloor S(x)\rfloor$, $\widetilde S_{int}:=\lfloor \widetilde
S(\widetilde x)\rfloor$,
$S_{rat}:=\{S(x)\}$ and $\widetilde S_{rat}:=\{\widetilde S(\widetilde x)\}$.

Then, by \ref{lem:d2} 
${\rm pc}^{\calC^{orb}}(\bt^{J\binom{S_{rat}}{\widetilde S_{rat}}}/ (1-\bt^{J\binom{1}{0}})
(1-\bt^{J\binom{0}{1}}))=0$. Moreover, \ref{ss:POLPART} gives a unique integer $s(k)\geq 0$
for $k\in \{0,\dots,S_{int}-1\}$ such that $\bt^{J\binom{k+S_{rat}}{-s(k)+\widetilde S_{rat}}}/1-\bt^{J\binom{0}{1}}$
has vanishing periodic constant with respect to $\calC^{orb}$. It turns out that $s(k)=\lfloor
-\widetilde e \alpha_0(k+S_{rat})+\widetilde S_{rat}\rfloor$.
Similarly $s(\widetilde k)=\lfloor -e\alpha_0(\widetilde k+
\widetilde S_{rat})+S_{rat}\rfloor$ in the case of
$\bt^{J \binom{-s(\widetilde k)+S_{rat}}{\widetilde k+
\widetilde S_{rat}}}/1-\bt^{J\binom{1}{0}}$.
Therefore, by \ref{th:POLPART}, for
\begin{equation*}
  \mathrm{pc}(Z_{ne})=
-\lambda(M)-\frd\cdot\frac{K^2+|\cV|}{8}+\sum_h\chi(r_h)
\end{equation*}
we get
\begin{equation*}
\sum_{0\leq x_i<\alpha_i, 0\leq i\leq d\atop 0\leq \widetilde x_j<\widetilde\alpha_j,
0\leq j \leq\widetilde d} \  \Big( S_{int}\widetilde S_{int} + \sum_{k=0}^{S_{int}-1}
 \lfloor -\widetilde e
\alpha_0(k+S_{rat})+\widetilde S_{rat}\rfloor
+ \sum_{\widetilde k=0}^{\widetilde S_{int}-1} \lfloor -e\alpha_0(\widetilde k+ \widetilde S_{rat})+S_{rat}\rfloor  \ \Big).
\end{equation*}

\subsection{Ehrhart theoretical interpretation}\
In general, 
in contrast with the one--node case \ref{ss:Ehr}, the direct determination
 of
the counting function of $Z_h(\bt)$, or equivalently, of the complete equivariant
 Ehrhart quasipolynomial associated with the corresponding polytope,  is rather hard.
Nevertheless, those coefficients which are relevant to us (e.g. those ones which
contain the information about the Seiberg--Witten invariants of the 3--manifold)
can be identified using  the right hand side of (\ref{eq:SUM}).
The computation is more transparent
when $L'=L$. In  that case, the two--variable Ehrhart polynomial has degree $d+\widetilde d$,
and a specific $d+\widetilde d-2$ degree coefficient
 is exactly the normalized Seiberg--Witten invariant of the 3--manifold.
We will not provide here the formulae, since
this identification was already established for any
negative definite plumbing graph with arbitrary number of nodes, see
Section \ref{s:Last}, where several other coefficients were computed as well.

\section{Examples}

\begin{eexample}\labelpar{ex1} Consider the following plumbing graph. 
\begin{figure}[h!]
\vspace{0.5cm}
\begin{center}
\begin{picture}(140,60)(80,15)
\put(110,40){\circle*{3}}
\put(140,40){\circle*{3}}
\put(170,40){\circle*{3}}
\put(200,40){\circle*{3}}
\put(80,60){\circle*{3}}
\put(80,20){\circle*{3}}
\put(200,60){\circle*{3}}
\put(200,20){\circle*{3}}
\put(110,40){\line(1,0){90}}
\put(110,40){\line(-3,2){30}}
\put(110,40){\line(-3,-2){30}}
\put(170,40){\line(3,2){30}}
\put(170,40){\line(3,-2){30}}
\put(30,20){\makebox(0,0){$E_2$}}
\put(30,60){\makebox(0,0){$E_1$}}
\put(250,40){\makebox(0,0){$\widetilde E_2$}}
\put(250,60){\makebox(0,0){$\widetilde E_1$}}
\put(250,20){\makebox(0,0){$\widetilde E_3$}}
\put(85,65){\makebox(0,0){$-2$}}
\put(85,15){\makebox(0,0){$-3$}}
\put(210,35){\makebox(0,0){$-5$}}
\put(190,65){\makebox(0,0){$-5$}}
\put(190,15){\makebox(0,0){$-5$}}
\put(115,50){\makebox(0,0){$-1$}}
\put(140,50){\makebox(0,0){$-9$}}
\put(165,50){\makebox(0,0){$-1$}}

\end{picture}
\end{center}
\caption{Graph for Example \ref{ex1}.}
\end{figure}
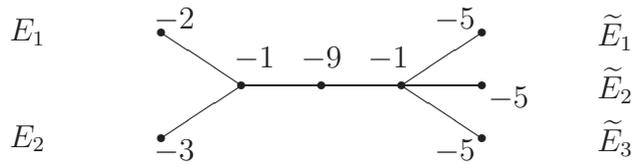
\noindent

The corresponding Seifert invariants are  $\alpha_1=2$, $\alpha_2=3$, $\widetilde{\alpha}_j=5$, $\alpha_0=9$ and
$\omega_i=\widetilde\omega_j=\omega_0=\widetilde\omega_0=1$ for all $i$ and  $j$.
Hence  $e=-1/18$, $\widetilde e=-13/45$ and $\varepsilon=1/(3^3\cdot 10)$.
For $h=0$ we choose $\bc=0$. Then
\begin{equation*}
\calS_0=\left\{
\begin{array}{l} (\ell,\well)\in \Z^2 \\
 8\ell - \well +9\cdot([\frac{-\ell}{2}]+[\frac{-\ell}{3}])\geq 0  \\
 8\well -\ell +27\cdot[\frac{-\well}{5}]\geq 0\\
 \ell+\well \equiv 0 \ (\mbox{mod} \ 9 )
\end{array}
\right\} \ \ \mbox{and} \ \ \
\overline\calS_0=\left\{
\begin{array}{l} (\ell,\well)\in \Z^2 \\
 \ell-2\well \geq 0 \\
 -5\ell+13\well \geq 0 \\
 \ell+\well \equiv 0 \ (\mbox{mod} \ 9 )
\end{array}
\right\} .
\end{equation*}
If we take the generators $\frv_1=(60,30)$ and $\frv_2=(26,10)$ (via conditions (A)-(B)-(C) following
Lemma \ref{lem:VV}), one can calculate explicitly the sets
\begin{equation*}
\calS^-_{{\bf 0},1}=\left\{
\begin{array}{l}
(13,5),(19,8),(25,11),\\
(31,14),(37,17),(43,20),\\
(49,23),(55,26),(61,29),\\
(67,32)
\end{array}
\right\} \ \mbox{and} \ \
\calS^-_{{\bf 0},2}=\left\{
\begin{array}{l}
(6,3),(19,8),(12,6),\\
(25,11),(24,12),(37,17),\\
(42,21),(55,26)
\end{array}\right\} .
\end{equation*}
This generates the next counting function of
$\overline \calS_{{\bf 0}}\setminus \calS_{{\bf 0}}$, namely
$\sum_{(\ell, \well)\in \overline \calS_{{\bf 0}}\setminus \calS_{{\bf 0}}}
\ t^\ell \widetilde{t}^{\well}\ =$
\begin{equation*}
\begin{array}{l}
\sum_{(\ell, \well)\in \overline \calS_{{\bf 0}}\setminus \calS_{{\bf 0}}}
\ t^\ell \widetilde{t}^{\well}\ =
\frac{t^{13}\wtt^5+t^{19}\wtt^8+t^{25}\wtt^{11}+t^{31}\wtt^{14}+t^{37}\wtt^{17}+
t^{43}\wtt^{20}+t^{49}\wtt^{23}+t^{55}\wtt^{26}+t^{61}\wtt^{29}+t^{67}\wtt^{32}}{1-t^{60}\wtt^{30}}+\\
 + \frac{t^6\wtt^3+t^{12}\wtt^6+t^{19}\wtt^8+t^{24}\wtt^{12}+t^{25}\wtt^{11}+t^{37}\wtt^{17}
 +t^{42}\wtt^{21}+t^{55}\wtt^{26}}{
 1-t^{26}\wtt^{10}}
 -t^{19}\wtt^8-t^{25}\wtt^{11}-t^{37}\wtt^{17}-t^{55}\wtt^{26} \ ,
\end{array}
\end{equation*}
which by \ref{eq:2nZ+} provides  $Z^+_0(t,\wtt)=t\wtt^{-1}+t^3\wtt^{2}+t^{-2}\wtt^2+t^{-1}\wtt+t^{11}\wtt^7+t^{16}\wtt^{11}+t^{-10}\wtt+
t^{29}\wtt^{16}+t^{3}\wtt^{6}+t^{19}\wtt^8+t^{25}\wtt^{11}+t^{37}\wtt^{17}+t^{55}\wtt^{26}$.
Hence ${\rm pc}^{\calC^{orb}}_0(Z)=Z^+_0(1,1)=13$.

It can be verified that there exists a splice--quotient type normal surface singularity whose
link is given by the above graph. It is a complete intersection in $(\C^4,0)$ with equations
$z^3+(y_2+2y_3)^2-y_1y_2(2y_2+3y_3)=y_1^5+(2y_2+3y_3)y_2y_3=0$.
Its geometric genus is 13 according to the above computation and \cite{NO}.
\end{eexample}

\begin{eexample}\labelpar{ex2} Let $G$ be the graph in Figure \ref{fig:pex2}.
\begin{figure}[h!]
\begin{center}
\begin{picture}(130,60)(80,15)
\put(80,40){\circle*{3}}
\put(110,40){\circle*{3}}
\put(140,40){\circle*{3}}
\put(170,40){\circle*{3}}
\put(200,40){\circle*{3}}
\put(80,60){\circle*{3}}
\put(80,20){\circle*{3}}
\put(200,60){\circle*{3}}
\put(200,20){\circle*{3}}
\put(80,40){\line(1,0){120}}
\put(110,40){\line(-3,2){30}}
\put(110,40){\line(-3,-2){30}}
\put(170,40){\line(3,2){30}}
\put(170,40){\line(3,-2){30}}
\put(30,40){\makebox(0,0){$E_2$}}
\put(30,20){\makebox(0,0){$E_3$}}
\put(30,60){\makebox(0,0){$E_1$}}
\put(250,40){\makebox(0,0){$\widetilde E_2$}}
\put(250,60){\makebox(0,0){$\widetilde E_1$}}
\put(250,20){\makebox(0,0){$\widetilde E_3$}}
\put(170,10){\makebox(0,0){$\widetilde E_0$}}
\put(140,10){\makebox(0,0){$E$}}
\put(110,10){\makebox(0,0){$E_0$}}
\put(70,35){\makebox(0,0){$-5$}}
\put(85,65){\makebox(0,0){$-5$}}
\put(85,15){\makebox(0,0){$-5$}}
\put(210,35){\makebox(0,0){$-5$}}
\put(190,65){\makebox(0,0){$-5$}}
\put(190,15){\makebox(0,0){$-5$}}
\put(115,50){\makebox(0,0){$-1$}}
\put(140,50){\makebox(0,0){$-7$}}
\put(165,50){\makebox(0,0){$-1$}}
\end{picture}
\end{center}
\caption{Graph for Example \ref{ex2}.}
\label{fig:pex2}
\end{figure}

The corresponding generalized Seifert invariants are $\alpha_i=\widetilde\alpha_j=5$, $\omega_0=
\widetilde\omega_0=\omega_i=\widetilde\omega_j=1$, $e=\widetilde e=-9/35$ and $\varepsilon=8/(7\cdot 35)$
for all $i,j\in \{ 1,\ldots,3\}$.
Let $h\in H$ determined by the following coefficients: $c_0=-2$, $\widetilde c_0=1$, $\overline c=2$,
$c_i=3$ and $\widetilde c_j=-2$ for any $i,j$. Then $(c,\widetilde c)=(3/4,3/4)$ which is
uniquely determined by the $h$. It is immediate that
\begin{equation*}
\calS_\bc=\left\{
\begin{array}{l}
 (\ell,\well)\in \Z^2\\
 6\ell - \well +21\cdot[\frac{3-\ell}{5}]\geq 12  \\
 6\well -\ell +21\cdot[\frac{-2-\well}{5}]\geq 9\\
 \ell+\well \equiv 2 \ (\mbox{mod} \ 7 )
\end{array}
\right\} \ \mbox{and} \ \ 
\overline\calS_\bc=\left\{
\begin{array}{l}
  (\ell,\well)\in \Z^2 \\
 9\ell-5\well \geq -3 \\
 -5\ell+9\well \geq -3 \\
 \ell+\well \equiv 2 \ (\mbox{mod} \ 7 )
\end{array}
\right\} \ .
\end{equation*}
If we choose $\frv_1:=(5,9)$ and $\frv_2:=(9,5)$ as generators for $\calC^{orb}$, one can
calculate $\calS^-_{{\bf c},1}$ and $\calS^-_{{\bf c},2}$ explicitly, i.e.
$$\calS^-_{{\bf c},1}=\left\{ (1,1),(4,5),(5,4),(9,7)\right\} \ \ \mbox{and} \ \ 
\calS^-_{{\bf c},2}=\left\{ (1,1),(4,5),(5,4),(7,9)
\right\} \ .$$
Therefore, the counting function of $\overline\calS_\bc\setminus \calS_\bc$ is 
$$-t^{3/4}\wtt^{3/4}\big( (t\wtt+t^4\wtt^5+t^5\wtt^4+t^9\wtt^7)/(1-t^5\wtt^9) 
+ (t\wtt+t^4\wtt^5+t^5\wtt^4+t^7\wtt^9)/(1-t^9\wtt^5)-t\wtt-t^4\wtt^5-t^5\wtt^4 \big).$$
Finally, using \ref{eq:2nZ+} we get 
$Z^+_h(t,\wtt)=-t^{3/4}\wtt^{3/4}(-\wtt^5-t^4\wtt^{-2}-t^{-5}-t^{-2}\wtt^4-t\wtt-t^4\wtt^5-t^5\wtt^4)$, 
hence ${\rm pc}^{\calC^{orb}}_h(Z)=Z^+_h(1,1)=7$.
\end{eexample}

\chapter{Lattice cohomological calculations and examples}\labelpar{ch:5}\

\indent N\'emethi's very first article on lattice cohomology \cite{OSZINV} presents a method, using 
{\em graded roots} (see \cite[\textsection 3]{OSZINV}), to compute 
the lattice cohomology in the case when the negative 
definite plumbing graph $G$ is almost rational, i.e. has only one bad vertex. 

In this case, as the Reduction Theorem 
\ref{red} shows, $\overline L=\Z_{\geq 0}$ and only $\bH^0$ might be non--zero. Moreover, one can find a bound 
$i_m\in\Z_{\geq 0}$, such that $\{\overline w(i)\}_{0\leq i\leq i_m}$ contains all the lattice 
cohomological data 
of $G$. Hence, it is enough to determine how the function $\overline w$ behaves along the `interval' $[0,i_m]$.

As an example, one can look at the case, when $M$ is a Seifert 3--manifold ($G$ is star--shaped). Then 
$\overline w(i+1)-\overline w(i)$, hence the lattice cohomology itself, can be calculated using the 
normalized Seifert invariants. Moreover, the sum $\sum\max\{0,\overline w(i)-\overline w(i+1)\}$, 
or equivalently, the Euler characteristic of the reduced lattice cohomology, equals the 
{\em Dolgachev--Pinkham invariant} (cf. \cite{P},\cite[11.14]{OSZINV}). Therefore, it gives the geometric 
genus of a normal surface singularity, which admits $M$ as its link and a good $\C^*$--action. 

This example automatically connects us to the Seiberg--Witten invariant conjecture 
(\ref{ss:SWIC} and \ref{ss:SWIrev}). 
Note also that in this almost rational case, the interval $[0,i_m]$ can not be simplified further. In 
other words, one can say that $[0,i_m]$ is the {\em minimal reduction} of the original lattice $L$. On the other hand, 
in the case of more bad vertices, a (multi)rectangle $R(0,\vasi_m)$ can be `reduced' further.

In this chapter, we will make some calculations and illustrations of the lattice cohomology for graphs 
having only two nodes.  
The first part applies the Reduction Theorem \ref{red}, 
calculates the special cycles $x(i,j)$ and their weights in terms of the normalized Seifert invariants 
of the maximal 
star--shaped subgraphs. We continue with the characterization of the optimal bound 
$\vasi_m$ and prove that the rectangle 
$R(0,\vasi_m)$ contains all the lattice cohomological informations. Notice that this can be generalized 
to arbitrary bad vertices as well.

In the second part, we provide some examples with figures to illustrate their lattice cohomology. 


\section{Graphs with 2 nodes}

\begin{figure}[h]
\begin{center}
\begin{picture}(200,60)(70,30)
\put(100,40){\circle*{3}}
\put(140,40){\circle*{3}}
\put(100,40){\line(1,0){55}}
\put(171,40){\makebox(0,0){$\cdots$}}
\put(200,40){\circle*{3}}
\put(240,40){\circle*{3}}
\put(185,40){\line(1,0){55}}
\put(50,73){\circle*{3}}
\put(100,40){\line(-3,2){10}}
\put(50,73){\line(3,-2){10}}
\multiput(77,55)(-3,2){3}%
{\circle*{1}}

\multiput(70,43)(0,-3){3}%
{\circle*{1}}

\put(50,7){\circle*{3}}
\put(100,40){\line(-3,-2){10}}
\put(50,7){\line(3,2){10}}
\multiput(77,25)(-3,-2){3}%
{\circle*{1}}

\put(290,73){\circle*{3}}
\put(240,40){\line(3,2){10}}
\put(290,73){\line(-3,-2){10}}
\multiput(263,55)(3,2){3}%
{\circle*{1}}

\multiput(270,43)(0,-3){3}%
{\circle*{1}}

\put(290,7){\circle*{3}}
\put(240,40){\line(3,-2){10}}
\put(290,7){\line(-3,2){10}}
\multiput(263,25)(3,-2){3}%
{\circle*{1}}

\put(100,50){\makebox(0,0){$E_0$}}
\put(100,30){\makebox(0,0){$b_0$}}
\put(45,55){\makebox(0,0){$\{E^l_v\}_{v\in\overline{1,s_l}}$}}
\put(40,25){\makebox(0,0){$(\alpha_l,\omega_l)_{l\in\overline{1,d}}$}}
\put(170,55){\makebox(0,0){$\{E_v\}_{v\in\overline{1,s}}$}}
\put(170,25){\makebox(0,0){$(\alpha,\omega,\widetilde \omega)$}}
\put(240,50){\makebox(0,0){$\widetilde E_0$}}
\put(240,30){\makebox(0,0){$\widetilde b_0$}}
\put(300,55){\makebox(0,0){$\{\widetilde E^l_v\}_{v\in\overline{1,\widetilde s_l}}$}}
\put(300,25){\makebox(0,0){$(\widetilde\alpha_l,\widetilde\omega_l)_{l\in\overline{1,\widetilde d}}$}}

\multiput(10,85)(0,-8){12}{\line(0,-1){4}}
\multiput(10,85)(8,0){26}{\line(1,0){4}}
\multiput(214,85)(0,-8){12}{\line(0,-1){4}}
\multiput(214,-7)(-8,0){26}{\line(-1,0){4}}
\put(20,77){\makebox(0,0){$G_0$}}

\multiput(330,80)(-8,0){26}{\line(-1,0){4}}
\multiput(330,80)(0,-8){12}{\line(0,-1){4}}
\multiput(126,80)(0,-8){12}{\line(0,-1){4}}
\multiput(126,-13)(8,0){26}{\line(1,0){4}}
\put(320,-4){\makebox(0,0){$\widetilde G_0$}}
\end{picture}
\end{center}
\vspace{1.5cm}
\caption{Seifert invariants in the two--node case}
\end{figure}
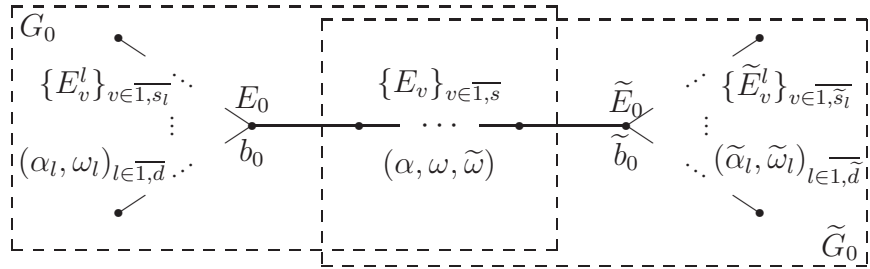

Similarly as in Section \ref{s:TN}, consider a negative definite plumbing graph $G$ with 
nodes $E_0$ and $\widetilde{E}_0$, which have decorations $b_0$ and 
$\widetilde{b}_0$ respectively (as it is shown in Figure 6.1). One has $d$ legs connected to $E_0$, 
$\widetilde d$ legs connected to $\widetilde E_0$, and a chain connecting the two nodes. 

Recall that one can consider the two maximal
star--shaped subgraphs $G_0$ and $\widetilde G_0$, and their normalized
Seifert invariants: we encode the legs with vertices $\{ E_v^l\}_{v\in \overline{1,s_l}}$ by 
$(\alpha_l,\omega_l)$ for any $l\in \overline{1,d}$ , while the legs with vertices $\{\widetilde{E}_v^l\}
_{v\in \overline{1,\widetilde s_l}}$ by $(\widetilde\alpha_l,\widetilde\omega_l)$ for 
$l\in \overline{1,\widetilde{d}}$.
For any $1\leq v\leq w\leq s_l$, we may define $n_{v,w}^l$ as the determinant of the chain 
starting from $E_v^l$ and ending with $E_w^l$, and similarly $\widetilde n_{v,w}^l$ as well. Notice that 
$\omega_l=n_{2,s_l}^l$ and respectively $\widetilde\omega_l=\widetilde n_{2,\widetilde s_l}^l$. We also set 
$n_{v,v-1}^l=\widetilde n_{v,v-1}^l:=1$ and $n_{v,w}^l=\widetilde n_{v,w}^l:=0$ for $w<v-1$, 
cf. \cite[10.2]{OSZINV}.
The chain connecting the nodes, viewed in $G_0$, has normalized
Seifert invariants $(\alpha,\omega)$, while it has $(\alpha,\widetilde{\omega})$, 
viewed as a leg in $\widetilde G_0$. One satisfies $\omega \widetilde \omega=\alpha \tau +1$ and, if we define 
the integers $n_{vw}$ for this chain too, 
then $\omega=n_{2,s}$, $\widetilde\omega=n_{1,s-1}$ and $\tau=n_{2,s-1}$.

The orbifold Euler numbers of the star--shaped subgraphs and the determinant of $G$ can be calculated via 
the formulae
$$e=b_0+ \frac{\omega}{\alpha}+\sum_{l=1}^d \frac{\omega_l}{\alpha_l},\ \ \
\widetilde{e}=\widetilde b_0+
\frac{\widetilde\omega}{\alpha}+
\sum_{l=1}^{\widetilde d} \frac{\widetilde\omega_l}{\widetilde\alpha_l}\ \ \ \mbox{and} \ \ \
\det(G)=\varepsilon\cdot \alpha_0 \prod_{l=1}^d \alpha_l 
\prod_{l'=1}^{\widetilde d} \widetilde{\alpha}_{l'}$$
where $\varepsilon:=\det \frI^{orb}=e \widetilde e - \frac{1}{\alpha_0^2}>0$ (see \ref{ss:TNC}).

\subsection{Reduction and the cycles $x(i,j)$}\labelpar{xiw}
\begin{lemma}
 G is a $2$--rational graph.
\end{lemma}

\begin{proof}
 We may assume that for all the vertices of $G$, except $E_0$ and $\widetilde E_0$, we have $-b\geq \delta$, 
 where $b$ is the weight and $\delta$ is the valency of the vertex. Otherwise, we blow down first all these 
 non--nodes with weight $-1$. Then if we replace $b_0$ and $\widetilde b_0$ with $-d-1$ and $-\widetilde d-1$, 
 the Laufer algorithm \ref{La} shows that we get a rational graph.
\end{proof}

Now we can apply the reduction procedure associated with the two bad vertices $E_0$ and 
$\widetilde E_0$. This says that 
$$\bH^*(G,k_r)\cong\bH^*(\overline{L},\overline{w}[k]),$$
where $\overline{L}=\Z_{\geq 0}^2$ and $\overline{w}[k](i,j):=\chi_{k_r}(x(i,j))$ for all $(i,j)\in \Z_{\geq 0}^2$.

{\em For simplicity, from now on we assume that $[k]$ is the canonical class, hence $k_r=k_{can}$ and 
$l'_{[k_{can}]}=0$.} Since any kind of object will be associated with this class, 
we will omit $k$ from the notations. 

In the sequel we start to determine 
the cycles $x(i,j)$ and their $\chi$--values. 

\begin{proposition}\labelpar{eq:xij}
Assume that 
 \begin{equation*}
x(i,j)=i E_0+j \widetilde E_0+\sum_{v=1}^s m_v E_v
+ \sum_{1\leq v \leq s_l \atop 1\leq l \leq d} 
m_v^l E_v^l 
+ \sum_{1\leq v \leq \widetilde s_l \atop 1\leq l \leq \widetilde d} 
\widetilde m_v^l \widetilde E_v^l, 
 \end{equation*}
where $m_v$,$m_v^l$ and $\widetilde m_v^l$ denote the coefficients of the corresponding $E_v$, $E_v^l$ 
and $\widetilde E_v^l$ (we set also $m_0=m_0^l=i$ and $m_{s+1}=\widetilde m_0^l=j$).
Then these coefficients can be calculated by the following recursive formulae 
\begin{align*}
&(a)& \ \  m_v=\left\lceil \frac{m_{v-1}\cdot n_{v+1,s}+j}{n_{v,s}} \right\rceil=
\left\lceil \frac{i+m_{v+1}\cdot n_{1,v-1}}{n_{1,v}} \right\rceil \ \ \mbox{for} \ \ v\in\{1,\ldots,s\};\\
&(b)& \ \  m_v^l=\left\lceil \frac{m_{v-1}^l\cdot n_{v+1,s_l}^l}{n_{v,s_l}^l} \right\rceil \ \ 
\mbox{for} \ \ l\in \{1,\dots,d\}\ \ \mbox{and} \ \ v\in\{1,\ldots,s_l\};\\ 
&(\widetilde b)&\ \ \widetilde m_v^l=\left\lceil \frac{\widetilde m_{v-1}^l\cdot 
\widetilde n_{v+1,\widetilde s_l}^l}
{\widetilde n_{v,\widetilde s_l}^l} \right\rceil \ \ 
\mbox{for} \ \ l\in \{1,\dots,\widetilde d\}\ \ \mbox{and} \ \ v\in\{1,\ldots,\widetilde s_l\}.
\end{align*}
\end{proposition}

\begin{proof}
 We use the interpretation of $x(i,j)$ from \ref{rem:xi}. This claims that 
$x(i,j)^*$ is the minimal solution 
 for the following system 
 of inequalities:
\begin{equation}\label{eq:matrsyst}
 -C\cdot x(i,j)^*\geq -B^t\cdot\Bigg(
 \begin{array}{c}
  i\\
  j
 \end{array}\Bigg),
\end{equation} 
where $C$ is the intersection matrix of the graph obtained from $G$ by deleting the bad vertices 
$E_0, \widetilde E_0$  and all their adjacent edges. 
Hence we may write
$$C=\left(
\begin{array}{ccc}
 \ddots &   0            &   0\\
 0     & \frI_{leg} &   0\\
 0     &   0            &  \ddots
\end{array}
\right),
$$
where the diagonal of the block structure contains the intersection matrices of the legs of $G$.
Since these blocks/legs do not interact, one can split the system 
(\ref{eq:matrsyst}) and look for each leg separately.

Therefore, the problem reducts to finding the minimal integral solutions of the system
\begin{equation}
 \left( 
\begin{array}{ccccc}
 k_1 & -1 & 0 & 0 & 0\\
 -1 & k_2 & -1 & 0 & 0\\
 0 & -1 & \ddots & -1 & 0\\
 0 & 0 & -1 & k_{s-1} & -1\\
 0 & 0 & 0 & -1 & k_s\\ 
\end{array}
\right)\cdot 
\left(
\begin{array}{c}
x_1\\
x_2\\
\vdots\\
x_{s-1}\\
x_s\\
\end{array}
\right)\geq 
\left(
\begin{array}{c}
a\\
0\\
\vdots\\
0\\
b\\
\end{array}
\right),
\end{equation}
where $-k_t$ denotes the weight of $E_t$ on a chain with vertices $\{E_1,\ldots,E_s\}$, 
$a$ and $b$ are some integral parameters. We have to observe, that if we multiply the $t^{th}$ row by 
$n_{t+1,s}$ and use the equality 
$k_t n_{t+1,s}-n_{t+2,s}=n_{t,s}$ (consequence of Lemma \ref{lem:legeq} or \cite[10.2]{OSZINV}), then 
it gives an equivalent system 
$$
n_{t,s}x_t-n_{t+1,s}x_{t-1}\geq n_{t+1,s}x_{t+1}-n_{t+2,s}x_t \ \ \ \mbox{for} \ \ 1\leq t \leq s, 
$$ where we set $x_{-1}:=a$ and $x_{s+1}:=b$. Its minimal solutions can be calculated recursively. 
Indeed, the minimal solution for $x_t$ does 
depend only on $x_{t-1}$ and it is determined by the inequality 
$n_{t,s}x_t-n_{t+1,s}x_{t-1}\geq b$. 

Therefore, $x_t=\lceil (x_{t-1} n_{t+1,s}+b)/n_{t,s}\rceil$. In 
particular, $x_1=\lceil (a\omega +b)/\alpha\rceil$. Notice that one can achieve the solutions 
from the other way around, if we multiply the $t^{th}$ row by $n_{1,t-1}$. In this case, we get 
$x_t=\lceil (a+x_{t+1} n_{1,t-1})/n_{1,t}\rceil$, in particular, 
$x_s=\lceil (a+b\widetilde\omega)/\alpha\rceil$.

Then, it is straightforward that if we choose $a=i$ and $b=j$, we find $m_v$, for $a=i$ and $b=0$ we get $m_v^l$
and finally $a=0$ and $b=j$ gives $\widetilde m_v^l$. 
\end{proof}

\begin{remark}
\begin{itemize}
\item[(a)] We can compare this formula with \cite[11.11]{OSZINV}, since if we take $j=0$ (resp. $i=0$) we get 
the special cycles $x(i)$ (resp. $x(j)$) associated with the almost rational graph $G_0$ (resp. $\widetilde G_0$).

\item[(b)] Since one can get a recursive formula for $m_v$ from both directions 
(either starting from $E_0$ or from $\widetilde E_0$), this gives interesting arithmetical relations 
between the coefficients.
\item[(c)] Notice that in general, the recursive formula for $m_v$ can not be simplified to 
$m'_v:=\lceil (in_{v+1,s}+j n_{1,v-1})/\alpha\rceil$, except $v\in\{1,s\}$. 
E.g., one can imagine a leg connecting to 
only one bad vertex $E_0$ (that is $j=0$), 
as it is shown in the next picture.
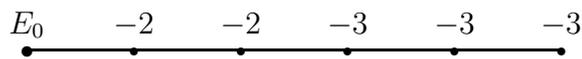
\begin{center}
\begin{picture}(200,50)(50,30)
\put(60,40){\circle*{4}}
\put(100,40){\circle*{3}}
\put(60,40.5){\line(1,0){200}}
\put(140,40){\circle*{3}}
\put(180,40){\circle*{3}}
\put(220,40){\circle*{3}}
\put(260,40){\circle*{3}}
\put(60,50){\makebox(0,0){$E_0$}}
\put(100,50){\makebox(0,0){$-2$}}
\put(140,50){\makebox(0,0){$-2$}}
\put(180,50){\makebox(0,0){$-3$}}
\put(220,50){\makebox(0,0){$-3$}}
\put(260,50){\makebox(0,0){$-3$}}
\end{picture}
\end{center}
Then if we choose $i=5$, one can calculate easily that $m_2=3,m_3=2$ and $m_4=1$. But 
$m'_2=3,m'_3=1$ and $m'_4=1$, which do not satisfy the needed inequality $-x_2+3x_3-x_4\geq 0$.
 
\item[(d)] However, it turns out that the weight of the cycle $x'(i,j)$ (with coefficients $m'_v$) 
equals the weight of $x(i,j)$, since the only coefficients which contribute to the weight, are 
the ones `around' the bad vertices.
\end{itemize}
\end{remark}

The general result \ref{propF1} and the previous formula provides 
$\overline w(i,j)$ for any $(i,j)\in \overline L$.

\begin{corollary}\labelpar{eq:w}
\noindent
\begin{itemize}
 \item[(a)] $\Delta_1(i,j):=\overline w(i+1,j)-\overline w(i,j)=1-i b_0-
\left\lceil \frac{i \omega+j}{\alpha}\right\rceil-
\sum_{l=1}^d \left\lceil \frac{i \omega_l}{\alpha_l}\right\rceil$ and 

 \item[(b)] $\Delta_2(i,j):=\overline w(i,j+1)-\overline w(i,j)=1-j \widetilde b_0-
\left\lceil \frac{i+j \widetilde \omega}{\alpha}\right\rceil-
\sum_{l=1}^{\widetilde d} \left\lceil \frac{j \widetilde\omega_l}{\widetilde\alpha_l}\right\rceil$.
\end{itemize}
Moreover, 
\begin{eqnarray*}
\overline w(i,j)=i+j-\frac{i(i-1)}{2}b_0-\frac{j(j-1)}{2}\widetilde b_0
&-&\sum_{q=0}^{i-1}\left(\left\lceil \frac{q\omega+j}{\alpha}\right\rceil+
\sum_{l=1}^d \left\lceil \frac{q\omega_l}{\alpha_l}\right\rceil\right)\\
&-&\sum_{q=0}^{j-1}\left(\left\lceil \frac{q\widetilde\omega}{\alpha}\right\rceil+
\sum_{l=1}^{\widetilde d} \left\lceil \frac{q\widetilde\omega_l}{\widetilde\alpha_l}\right\rceil\right).
\end{eqnarray*}

\end{corollary}

\begin{proof}
 The formulae follow from \ref{lemF1} and \ref{eq:xij}. The formula for $\overline w$ is calculated 
by inductions on $i$ and on $j$. Hence, when we change the order of the inductions, we get a similar 
formula for $\overline w$.
\end{proof}

\subsection{The optimal bound for the reduced lattice}\
Consider the subset $\sol$ of $\overline L$ with the following definition:
\begin{equation}\label{def:Sol}
 \sol:=\left\{ (i,j)\in \overline L_{>0} \ : \ 
\Delta_1(i-1,j)<0, \ 
\Delta_2(i,j-1)<0
\right\}.
\end{equation}

Set $d(i):=1-i(b_0+\omega/\alpha)-\sum_{l=1}^d \lceil i\omega_l/\alpha_l \rceil$ and similarly 
$\widetilde d(j):=1-j(\widetilde b_0+\widetilde\omega/\alpha)-\sum_{l=1}^{\widetilde d} \lceil j
\widetilde\omega_l/\widetilde\alpha_l \rceil$.
Then using the explicit formulae \ref{eq:w} of $\Delta_1$ and $\Delta_2$, one can see that $(i,j)\in\sol$ 
if and only if the following system of inequalities holds:
\begin{equation}\label{eq:Sol1}\left\{
 \begin{array}{c}
  i\geq \alpha \cdot\widetilde d(j-1)+1\\ 
  j\geq \alpha \cdot d(i-1)+1 \ .
 \end{array}
 \right.
\end{equation}
Indeed, using the formulae from \ref{eq:w} and the definition of the ceiling function $\lceil \ \rceil$, e.g., 
the first inequality $\Delta_1(i-1,j)<0$ is equivalent with 
$((i-1)\omega+j)/\alpha>1-(i-1)b_0-\sum_{l=0}^d\lceil(i-1)\omega_l/\alpha_l\rceil$. Then we multiply 
by $\alpha$ and use that the expressions on both sides are integers. 

We write $\alpha_0$, respectively $\widetilde\alpha_0$, for the least common multiple of the 
numbers $\alpha_l$ for any $l\in \overline{1,d}$, respectively of $\widetilde\alpha_l$ for all 
$l\in \overline{1,\widetilde d}$. Then $(i,j)$ can be written in 
the form $(\alpha_0 q+i_0,\widetilde\alpha_0 \widetilde q+j_0)$ for some 
$q,\widetilde q \in \Z_{\geq 0}$ and 
$(i_0,j_0)\in \{0,\dots,\alpha_0-1\}\times \{0,\dots,\widetilde\alpha_0-1\}$.\\
(\ref{eq:Sol1}) implies that for a fixed $(i_0,j_0)$, if $(q,\widetilde q)$ satisfies the 
system
\begin{equation}\label{eq:SI}\left\{
 \begin{array}{c}
  (\alpha \alpha_0 e)\cdot q + \widetilde\alpha_0 \cdot \widetilde q\geq \alpha \cdot d(i_0-1)-(j_0-1)\\
  \alpha_0 \cdot q + (\alpha\widetilde\alpha_0 \widetilde e)\cdot \widetilde q
\geq \alpha\cdot \widetilde d(j_0-1)-(i_0-1) \ ,
 \end{array}
\right.\tag{$SI_{(i_0,j_0)}$}
\end{equation}
then $(\alpha_0 q+i_0,\widetilde\alpha_0 \widetilde q+j_0)$ belongs to $\sol$. 

The next lemmas provide some important properties of this set and give the optimal bound for the lattice 
cohomological data, in the sense mentioned in the introductory part of this chapter. 

\begin{lemma}
 The number of elements in $\sol$ is finite.
\end{lemma}

\begin{proof}
 We multiply the first (resp. second) inequality in (\ref{eq:SI}) with the positive number 
$-\alpha \widetilde e$ (resp. $-\alpha e$), and sum up the two inequalities. Then, using that 
$-\alpha^2\alpha_0 \varepsilon <0$ (resp. $-\alpha^2\widetilde\alpha_0 \varepsilon <0$), we get 
\begin{equation}\label{eq:vartheta}
q\leq \vartheta(i_0,j_0) \ \ \ \mbox{and} \ \ \ \widetilde q\leq \widetilde\vartheta(i_0,j_0),
\end{equation}
where 
$\vartheta(i_0,j_0):=\left\lfloor\big(\alpha^2\widetilde e \cdot d(i_0-1)-\alpha(\widetilde e (j_0-1)+
\widetilde d(j_0-1))+ i_0-1\big)/\alpha^2 \alpha \varepsilon\right\rfloor$, and symmetrically one 
defines 
$\widetilde\vartheta(i_0,j_0)$ as well. It is enough to look at the cases when $\vartheta(i_0,j_0)$ and 
$\widetilde\vartheta(i_0,j_0)$ are non--negative.
The facts, that the number of possible pairs $(i_0,j_0)$ is finite and each of them can be completed 
with finitely many solutions  
$(q,\widetilde q)\in \{0,\dots,\vartheta(i_0,j_0)\}\times \{0,\dots,\widetilde\vartheta(i_0,j_0)\}$, 
proves that $\sol$ has only finitely many elements. 
\end{proof}

\begin{lemma}\labelpar{lem:solelm}
\begin{itemize}
\item[(a)] If $\vasi_1=(i_1,j_1)$ and $\vasi_2=(i_2,j_2)$ are elements of $\sol$, then 
$\max\{\vasi_1,\vasi_2\}:=(\max\{i_1,i_2\},\max\{j_1,j_2\}) \in \sol$ too. In particular, there exists 
an element $\vasi_m=(i_m,j_m)$ which is the (unique) maximum of $\sol$.
\item[(b)] We have the isomorphism
$$\bH^*(G,k_{can})\cong\bH^*(R(0,\vasi_m),\overline w).$$
\end{itemize}
\end{lemma}

\begin{proof}
For part (a), notice that the formulae in Corollary \ref{eq:w} imply that 
$\Delta_1(i_1-1,j)<0$, $\Delta_2(i,j_1-1)<0$ for any 
$j\geq j_1$ and $i\geq i_1$. Similarly, one gets $\Delta_1(i_2-1,j')<0$ and $\Delta_2(i',j_2-1)<0$ for any 
$j'\geq j_2$ and $i'\geq i_2$. Hence, $\Delta_1(\max\{i_1-1,i_2-1\},\max\{j_1,j_2\})<0$ and 
$\Delta_2(\max\{i_1,i_2\},\max\{j_1-1,j_2-1\})<0$ as well.

In the case of (b), we pick an element $(i,j)$ which does not belong to $\sol$. By (\ref{def:Sol}), 
it satisfies at 
least one of the inequalities: $\Delta_1(i-1,j)\geq 0$ and $\Delta_2(i,j-1)\geq 0$. Without loss of 
generality, we may assume that $\Delta_1(i-1,j)\geq 0$. 

Then one can consider the natural inclusion 
$\iota:R[0,(i-1,j)]\longrightarrow R[0,(i,j)]$ and its retract $\rho:R[0,(i,j)]\longrightarrow 
R[0,(i-1,j)]$, where $\rho|_{R[0,(i-1,j)]}$ is the identity map and $\rho(i,j')=(i-1,j')$ for every 
$0\leq j'\leq j$.

Since the explicit formula \ref{eq:w}(a) implies $\Delta_1(i-1,j')\geq \Delta_1(i-1,j)$ for $j'\leq j$, 
for any $N$ the inclusion 
$\iota_N:S_N\cap R[0,(i-1,j)]\to S_N\cap R[0,(i,j)]$ and its retract 
$\rho_N:S_N\cap R[0,(i,j)]\to S_N\cap R[0,(i-1,j)]$ induce isomorphism at the level of 
simplicial cohomology. Hence, \ref{geomdef} implies 
$$\bH^*(R[0,(i,j)],\overline w)\cong \bH^*(R[0,(i-1,j)],\overline w).$$

Choose $\vasi\in \overline L$ as in the Lemma \ref{ilem}, for which there exists a sequence 
$\{\vasi_n=(i_n,j_n)\}_{n\geq 0}$ with the following properties: $\vasi_0=\vasi$, $\vasi_{n+1}$ 
is either $(i_n+1,j_n)$ or $(i_n,j_n+1)$, $\vasi_n$ tends to the infinity, and 
for any $\vasi_n'=(i_n',j_n')\leq\vasi_n$ with $i_n'=i_n$ (respectively $j_n'=j_n$), 
one has $\Delta_1(i_n,j_n)>0$ (respectively $\Delta_2(i_n,j_n)>0$). Moreover, 
we have $\bH^*(\overline L,\overline w)\cong\bH^*(R(0,\vasi),\overline w)$ as well. 

Therefore, if $\vasi \in \sol$, then this argument implies that this is 
the maximum point and we are done. Otherwise, we apply the above procedure. 
The procedure stops after finitely many steps when 
arrives to $\vasi_m \in \sol$ (by the same argument as before) and we get  
$$\bH^*(R(0,\vasi),\overline w)\cong \bH^*(R(0,\vasi_m),\overline w).$$
\end{proof}
\begin{remark}
\begin{itemize}
\item[(a)]
Notice that, with the previous lemma, we give a `bound' for $\bH^*(\overline L,\overline w)$. 
In other words, all the lattice cohomological data is concentrated into $R(0,\vasi_m)$.
Moreover, the proof emphasizes that this bound is {\em optimal}.
\item[(b)]
If we look at $\chi(l)=-(l,l+\K)/2$ as a real 
function on $L\otimes\R$, then $\chi$ is increasing if $l\geq -\K$. Hence, it can be shown that 
the lattice cohomology is concentrated into $R(0,-\K)$.
Let $\vasi_{can}:=(\lfloor(-\K,E_0^*)\rfloor,\lfloor(-\K,\widetilde E_0^*) \rfloor)$, i.e. the 
projection of $\lfloor -\K \rfloor$ (the floor function $\lfloor \ \rfloor$ is taken componentwise) via 
$\phi:L\to \overline L$. Therefore, $R(0,\vasi_{can})$ has the same lattice cohomology as $\overline L$.
Notice that almost all the examples in 
\ref{s:examples} have the property $\vasi_m=\vasi_{can}$, however, this is not the case in general, see 
Example \ref{ex:boundlesscan}. 

\end{itemize}
\end{remark}

\section{Examples}\labelpar{s:examples}\
In this section we provide some examples and their lattice cohomology calculations 
using illustrating pictures on the weight structure of the corresponding $R(0,\vasi_m)$. 

These pictures have the lattice point 
$(0,0)$ at the lower left corner, the horizontal direction is the direction of $i$, 
while the vertical is for $j$.  
The red frames highlight the generators of $\bH^0$ and the dashed red frames are for marking 
$\bH^1$. We also display a chosen minimal reduction set, i.e. a (non--unique) subset of $R(0,\vasi_m)$, 
which contains the lattice cohomology information, and it is a minimal set with respect to this property. As a consequence, 
we read off $eu(\bH^*)$, $eu(\bH^0)$ and $eu(\gamma_{min}):=\min_{\gamma}eu(\gamma,\K)$, where $\gamma$ connects 
$0$ with $\vasi_m$, and in most cases, we discuss their relations.

\begin{eexample}\labelpar{ex:6.2.1}
 Let us consider the graph from Figure \ref{fig:gex1}. 
\begin{figure}[ht!]
\begin{center}
\begin{picture}(140,60)(75,15)
\put(110,40){\circle*{3}}
\put(140,40){\circle*{3}}
\put(170,40){\circle*{3}}
\put(200,40){\circle*{3}}
\put(80,60){\circle*{3}}
\put(80,40){\circle*{3}}
\put(80,20){\circle*{3}}
\put(200,60){\circle*{3}}
\put(200,20){\circle*{3}}
\put(110,40){\line(1,0){90}}
\put(110,40){\line(-3,2){30}}
\put(110,40){\line(-1,0){30}}
\put(110,40){\line(-3,-2){30}}
\put(170,40){\line(3,2){30}}
\put(170,40){\line(3,-2){30}}
\put(85,65){\makebox(0,0){$-5$}}
\put(70,40){\makebox(0,0){$-5$}}
\put(85,15){\makebox(0,0){$-5$}}
\put(210,40){\makebox(0,0){$-5$}}
\put(190,65){\makebox(0,0){$-5$}}
\put(190,15){\makebox(0,0){$-5$}}
\put(115,50){\makebox(0,0){$-1$}}
\put(140,50){\makebox(0,0){$-11$}}
\put(165,50){\makebox(0,0){$-1$}}
\end{picture}
\end{center}
\caption{Graph for Example \ref{ex:6.2.1}}
\label{fig:gex1}
\end{figure}
The reduction of its lattice is simple in the 
sense that the bound (see Lemma \ref{lem:solelm}) $\vasi_m=(7,7)$ is small. Notice that it is also 
equal to $\vasi_{can}$, since 
$\lfloor(-\K,E_0^*)\rfloor=\lfloor(-\K,\widetilde E_0^*)\rfloor=7$, where $E_0$ and $\widetilde E_0$ 
represent the nodes, as ususal. 
Therefore, Figure \ref{fig:ex1} presents $R[0,(7,7)]$, from where one can read off the full lattice 
cohomological data. Hence, 
$$\bH^0(G,\K)=\calT_{-10}^+ \oplus \calT_{-2}(1)\oplus \calT_0(1) \ \ \mbox{and} \ \ 
\bH^1(G,\K)=\calT_0(1).$$ Moreover, these imply that $eu(\bH^0)=7$, $eu(\bH^*)=6$ and the minimal 
reduction helps us to see $eu(\gamma_{min})=6$. 
\begin{figure}[ht!]
\psset{unit=1cm}
\begin{center}
\begin{pspicture}(0,0.5)(0,6)
\begin{psmatrix}[colsep=0.5cm,rowsep=0.2cm]
 1 & 1 & -1 & -2 & -2 & -1 & 0 & -1 \\
 1 & 1 & -1 & -2 & -2 & -1 & 1 & 0 \\
 -1 & -1 & -3 & -4 & -4 & -3 & -1 & -1 \\
 -2 & -2 & -4 & -5 & -5 & -4 & -2 & -2 \\
 -2 & -2 & -4 & -5 & -5 & -4 & -2 & -2 \\
 -1 & -1 & -3 & -4 & -4 & -3 & -1 & -1 \\
 1 & 1 & -1 & -2 & -2 & -1 & 1 & 1 \\
 0 & 1 & -1 & -2 & -2 & -1 & 1 & 1\\
\ncline[nodesep=0.1cm]{8,1}{7,1}
\ncline[nodesep=0.1cm]{7,1}{7,2}
\ncline[nodesep=0.1cm]{7,2}{6,2}
\ncline[nodesep=0.1cm]{6,2}{5,2}
\ncline[nodesep=0.1cm]{5,2}{5,3}
\ncline[nodesep=0.1cm]{5,3}{5,4}
\ncline[nodesep=0.1cm]{5,4}{5,5}
\ncline[nodesep=0.1cm]{5,5}{4,5}
\ncline[nodesep=0.1cm]{4,5}{3,5}
\ncline[nodesep=0.1cm]{3,5}{3,6}
\ncline[nodesep=0.1cm]{3,6}{3,7}
\ncline[nodesep=0.1cm]{3,7}{3,8}\ncline[nodesep=0.1cm]{3,8}{2,8}\ncline[nodesep=0.1cm]{2,8}{1,8}
\ncline[nodesep=0.1cm]{3,6}{2,6}
\ncline[nodesep=0.1cm]{2,6}{1,6}\ncline[nodesep=0.1cm]{1,6}{1,7}\ncline[nodesep=0.1cm]{1,7}{1,8}
\psframe[linecolor=red,linearc=0.2](3.6,3.9)(2.3,2.7)
\psframe[linecolor=red,linearc=0.2](0.22,1.1)(-0.2,0.6)
\psframe[linecolor=red,linearc=0.2](6.1,6)(5.7,5.6)
\psframe[linecolor=red,linestyle=dashed,linearc=0.2](5.5,5.4)(4.6,4.85)
\end{psmatrix}
\end{pspicture}
\end{center}
\caption{Lattice cohomology of Example 6.2.1}
\label{fig:ex1}
\end{figure}
\end{eexample}

\begin{eexample}\labelpar{ex:boundlesscan}

\begin{figure}[ht!]
\centering
\begin{picture}(140,60)(75,15)
\put(110,40){\circle*{3}}
\put(140,40){\circle*{3}}
\put(170,40){\circle*{3}}
\put(200,40){\circle*{3}}
\put(80,60){\circle*{3}}
\put(80,40){\circle*{3}}
\put(80,20){\circle*{3}}
\put(200,60){\circle*{3}}
\put(200,20){\circle*{3}}
\put(110,40){\line(1,0){90}}
\put(110,40){\line(-3,2){30}}
\put(110,40){\line(-1,0){30}}
\put(110,40){\line(-3,-2){30}}
\put(170,40){\line(3,2){30}}
\put(170,40){\line(3,-2){30}}
\put(85,65){\makebox(0,0){$-4$}}
\put(70,40){\makebox(0,0){$-4$}}
\put(85,15){\makebox(0,0){$-4$}}
\put(210,40){\makebox(0,0){$-5$}}
\put(190,65){\makebox(0,0){$-5$}}
\put(190,15){\makebox(0,0){$-5$}}
\put(115,50){\makebox(0,0){$-1$}}
\put(140,50){\makebox(0,0){$-11$}}
\put(165,50){\makebox(0,0){$-1$}}
\end{picture}
\vspace{0.5cm}
\caption{Graph for Example \ref{ex:boundlesscan}}
\label{fig:gex2}
\end{figure}
Let $G$ be similar as in the previous example, except we increase the weights of the legs on the right side 
(Figure \ref{fig:gex2}). 
Then, as we will see in the sequel, the structure is much more tricky. 
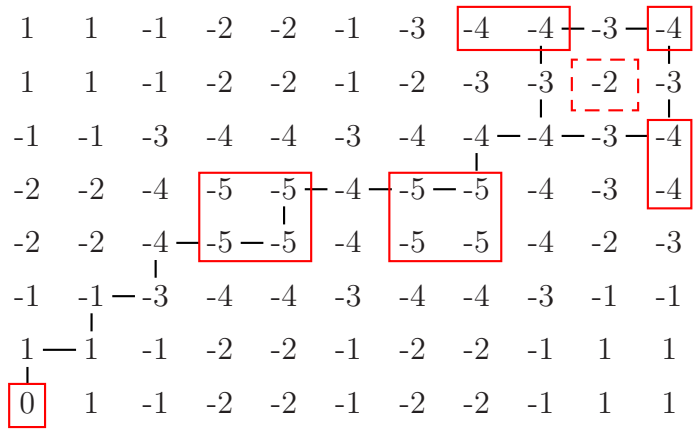
\begin{figure}[ht!]
\psset{unit=1cm}
\begin{center}
\begin{pspicture}(0,0.5)(0,6)
\begin{psmatrix}[colsep=0.5cm,rowsep=0.2cm]
 1 & 1 & -1 & -2 & -2 & -1 & -3 & -4 & -4 & -3 & -4 \\
 1&1&-1&-2&-2&-1&-2&-3&-3&-2&-3\\
 -1&-1&-3&-4&-4&-3&-4&-4&-4&-3&-4\\
 -2&-2&-4&-5&-5&-4&-5&-5&-4&-3&-4\\
 -2&-2&-4&-5&-5&-4&-5&-5&-4&-2&-3\\
 -1&-1&-3&-4&-4&-3&-4&-4&-3&-1&-1\\
 1&1&-1&-2&-2&-1&-2&-2&-1&1&1\\
 0&1&-1&-2&-2&-1&-2&-2&-1&1&1
\ncline[nodesep=0.1cm]{8,1}{7,1}
\ncline[nodesep=0.1cm]{7,1}{7,2}
\ncline[nodesep=0.1cm]{7,2}{6,2}
\ncline[nodesep=0.1cm]{6,2}{6,3}
\ncline[nodesep=0.1cm]{6,3}{5,3}
\ncline[nodesep=0.1cm]{5,3}{5,4}
\ncline[nodesep=0.1cm]{5,4}{5,5}
\ncline[nodesep=0.1cm]{5,5}{4,5}
\ncline[nodesep=0.1cm]{4,5}{4,6}
\ncline[nodesep=0.1cm]{4,6}{4,7}
\ncline[nodesep=0.1cm]{4,7}{4,8}
\ncline[nodesep=0.1cm]{4,8}{3,8}
\ncline[nodesep=0.1cm]{3,8}{3,9}\ncline[nodesep=0.1cm]{3,9}{2,9}\ncline[nodesep=0.1cm]{2,9}{1,9}
\ncline[nodesep=0.1cm]{1,9}{1,10}\ncline[nodesep=0.1cm]{1,10}{1,11}
\ncline[nodesep=0.1cm]{3,9}{3,10}\ncline[nodesep=0.1cm]{3,10}{3,11}\ncline[nodesep=0.1cm]{3,11}{2,11}
\ncline[nodesep=0.1cm]{2,11}{1,11}
\psframe[linecolor=red,linearc=0.2](-6.3,3.2)(-4.8,2)
\psframe[linecolor=red,linearc=0.2](-3.8,3.2)(-2.3,2)
\psframe[linecolor=red,linearc=0.2](-2.9,5.4)(-1.4,4.8)
\psframe[linecolor=red,linearc=0.2](-0.4,3.9)(0.2,2.7)
\psframe[linecolor=red,linearc=0.2](-8.8,-0.2)(-8.3,0.4)
\psframe[linecolor=red,linearc=0.2](-0.4,5.4)(0.2,4.8)
\psframe[linecolor=red,linestyle=dashed,linearc=0.2](-1.4,4.7)(-0.5,4)
\end{psmatrix}
\end{pspicture}
\end{center}
\caption{Lattice cohomology for Example \ref{ex:boundlesscan}}
\label{fig:ex2}
\end{figure}
Notice that the coefficients of 
$-\K$ corresponding to the nodes are $122/9$ and $83/9$, hence $\vasi_{can}=(13,9)$. On the other hand, 
the bound is $(10,7)$, which is, in this case, smaller than $\vasi_{can}$. 
One can read off the lattice cohomology of $R[0,(10,7)]$ from Figure \ref{fig:ex2}. Namely, 
$$\bH^0(G,\K)=\calT_{-10}^+ \oplus \calT_{-10}(1)\oplus \calT_{-8}(1)^3\oplus \calT_0(1) 
\ \ \mbox{and} \ \ \bH^1(G,\K)=\calT_{-6}(1).$$
Then $eu(\gamma_{min})=eu(\bH^*)=9<eu(\bH^0)=10$.
\end{eexample}

\begin{eexample}\labelpar{ex:6.2.3}
In this example, we put two vertices on one of the legs and consider the graph in Figure \ref{fig:gex3}.
\begin{figure}[ht!]
\centering
\begin{picture}(140,60)(75,15)
\put(110,40){\circle*{3}}
\put(140,40){\circle*{3}}
\put(170,40){\circle*{3}}
\put(210,40){\circle*{3}}
\put(70,67){\circle*{3}}
\put(90,53){\circle*{3}}
\put(70,40){\circle*{3}}
\put(70,13){\circle*{3}}
\put(210,67){\circle*{3}}
\put(210,13){\circle*{3}}
\put(110,40){\line(1,0){100}}
\put(110,40){\line(-3,2){40}}
\put(110,40){\line(-1,0){40}}
\put(110,40){\line(-3,-2){40}}
\put(170,40){\line(3,2){40}}
\put(170,40){\line(3,-2){40}}
\put(70,72){\makebox(0,0){$-2$}}
\put(90,60){\makebox(0,0){$-3$}}
\put(60,40){\makebox(0,0){$-5$}}
\put(70,8){\makebox(0,0){$-5$}}
\put(220,40){\makebox(0,0){$-5$}}
\put(210,72){\makebox(0,0){$-5$}}
\put(210,8){\makebox(0,0){$-5$}}
\put(115,50){\makebox(0,0){$-1$}}
\put(140,50){\makebox(0,0){$-14$}}
\put(165,50){\makebox(0,0){$-1$}}
\end{picture}
\vspace{0.5cm}
\caption{Graph for Example \ref{ex:6.2.3}}
\label{fig:gex3}
\end{figure}
The bound is $(12,7)$, $R[0,(12,7)]$ is shown by Figure \ref{fig:ex3}, hence the lattice cohomology is 
$$\bH^0(G,\K)=\calT_{-12}^+ \oplus \calT_{-12}(1)\oplus \calT_{-8}(1)^3\oplus \calT_0(1) 
\ \ \mbox{and} \ \ \bH^1(G,\K)=\calT_{-6}(1).$$ 
Therefore, these formulae and the minimal reduction set shows that 
$eu(\gamma_{min})=eu(\bH^*)=10$ and $eu(\bH^0)=11$.
\begin{figure}[ht!]
\psset{unit=1cm}
\begin{center}
\begin{pspicture}(0,0.5)(0,6)
\begin{psmatrix}[colsep=0.5cm,rowsep=0.2cm]
 1 & 1 & -1 & -2 & -3 & -3 & -2 & -3 & -3 & -4 & -4 & -3 & -4 \\
 1 & 1 & -1 & -2 & -3 & -3 & -2 & -3 & -3 & -3 & -3 & -2 & -3 \\
 -1 & -1 & -3 & -4 & -5 & -5 & -4 & -5 & -5 & -5 & -4 & -3 & -4 \\
 -2 & -2 & -4 & -5 & -6 & -6 & -5 & -6 & -6 & -6 & -5 & -3 & -4 \\
 -2 & -2 & -4 & -5 & -6 & -6 & -5 & -6 & -6 & -6 & -5 & -3 & -3 \\
 -1 & -1 & -3 & -4 & -5 & -5 & -4 & -5 & -5 & -5 & -4 & -2 & -2 \\
 1 & 1 & -1 & -2 & -3 & -3 & -2 & -3 & -3 & -3 & -2 & 0 & 0 \\
 0 & 1 & -1 & -2 & -3 & -3 & -2 & -3 & -3 & -3 & -2 & 0 & 0
\ncline[nodesep=0.1cm]{8,1}{7,1}
\ncline[nodesep=0.1cm]{7,1}{7,2}
\ncline[nodesep=0.1cm]{7,2}{6,2}
\ncline[nodesep=0.1cm]{6,2}{6,3}
\ncline[nodesep=0.1cm]{6,3}{5,3}
\ncline[nodesep=0.1cm]{5,3}{5,4}
\ncline[nodesep=0.1cm]{5,4}{5,5}
\ncline[nodesep=0.1cm]{5,5}{4,5}
\ncline[nodesep=0.1cm]{4,5}{4,6}
\ncline[nodesep=0.1cm]{4,6}{4,7}
\ncline[nodesep=0.1cm]{4,7}{4,8}
\ncline[nodesep=0.1cm]{4,8}{4,9}\ncline[nodesep=0.1cm]{4,9}{4,10}
\ncline[nodesep=0.1cm]{4,10}{3,10}\ncline[nodesep=0.1cm]{3,10}{3,11}
\ncline[nodesep=0.1cm]{3,11}{2,11}\ncline[nodesep=0.1cm]{2,11}{1,11}
\ncline[nodesep=0.1cm]{1,11}{1,12}\ncline[nodesep=0.1cm]{1,12}{1,13}
\ncline[nodesep=0.1cm]{3,11}{3,12}\ncline[nodesep=0.1cm]{3,12}{3,13}
\ncline[nodesep=0.1cm]{3,13}{2,13}\ncline[nodesep=0.1cm]{2,13}{1,13}
\psframe[linecolor=red,linearc=0.2](-10.5,-0.2)(-10,0.4)
\psframe[linecolor=red,linearc=0.2](-7.1,3.2)(-5.6,2)
\psframe[linecolor=red,linearc=0.2](-4.6,3.2)(-2.3,2)
\psframe[linecolor=red,linearc=0.2](-2.9,5.4)(-1.4,4.8)
\psframe[linecolor=red,linearc=0.2](-0.4,3.9)(0.2,2.7)
\psframe[linecolor=red,linearc=0.2](-0.4,5.4)(0.2,4.8)
\psframe[linecolor=red,linestyle=dashed,linearc=0.2](-1.4,4.7)(-0.5,4)
\end{psmatrix}
\end{pspicture}
\end{center}
\caption{Lattice cohomology of Example \ref{ex:6.2.3}}
\label{fig:ex3}
\end{figure}
\end{eexample}

\begin{eexample}\labelpar{ex:6.2.4}
 Now, we take an example when the graph has two vertices on the chain connecting the two nodes. Let 
$G$ be as in Figure \ref{fig:gex4}.
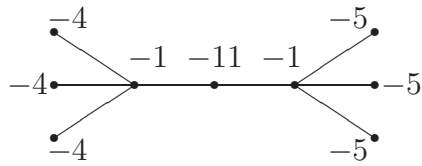
\begin{figure}[h!]
\centering
\begin{picture}(140,60)(75,15)
\put(110,40){\circle*{3}}
\put(140,40){\circle*{3}}
\put(170,40){\circle*{3}}
\put(200,40){\circle*{3}}
\put(80,60){\circle*{3}}
\put(80,20){\circle*{3}}
\put(230,60){\circle*{3}}
\put(230,20){\circle*{3}}
\put(110,40){\line(1,0){90}}
\put(110,40){\line(-3,2){30}}
\put(110,40){\line(-3,-2){30}}
\put(200,40){\line(3,2){30}}
\put(200,40){\line(3,-2){30}}
\put(85,65){\makebox(0,0){$-2$}}
\put(85,15){\makebox(0,0){$-3$}}
\put(220,65){\makebox(0,0){$-2$}}
\put(220,15){\makebox(0,0){$-3$}}
\put(115,50){\makebox(0,0){$-1$}}
\put(195,50){\makebox(0,0){$-1$}}
\put(140,50){\makebox(0,0){$-7$}}
\put(165,50){\makebox(0,0){$-8$}}
\end{picture}
\vspace{0.5cm}
\caption{Graph for Example \ref{ex:6.2.4}}
\label{fig:gex4}
\end{figure}
Then, $R[0,(20,14)]$ has to be checked for the lattice cohomology calculation, as it is in Figure 
\ref{fig:ex4}. The cohomology is 
$$\bH^0(G,\K)=\calT_{-2}^+ \oplus \calT_{-2}(1)^3 \oplus \calT_{0}(1) 
\ \ \mbox{and} \ \ \bH^1(G,\K)=\calT_{0}(1)^2.$$ However, the minimal set for the reduction is much 
more interesting: contains the set $R[(6,6),(14,8)]$, which can not be reduced further and contains 
the two $\bH^1$ generators. Together with the cohomology modules show that 
$$eu(\bH^*)=4<eu(\gamma_{min})=5<eu(\bH^0)=6.$$ 
\begin{figure}[ht!]
\psset{unit=1cm}
\begin{center}
\begin{pspicture}(0,0.5)(0,6)
\begin{psmatrix}[colsep=0.4cm,rowsep=0.1cm]
 2 & 2 & 1 & 1 & 1 & 1 & 1 & 1 & 0 & 0 & 0 & 0 & 0 & 1 & 0 & 0 & 0 & 0 & 0 & 1 & 0 \\
 2 & 2 & 1 & 1 & 1 & 1 & 1 & 1 & 0 & 0 & 0 & 0 & 0 & 1 & 0 & 0 & 0 & 0 & 0 & 1 & 1 \\
 1 & 1 & 0 & 0 & 0 & 0 & 0 & 0 & -1 & -1 & -1 & -1 & -1 & 0 & -1 & -1 & -1 & -1 & -1 & 0 & 0 \\
 1 & 1 & 0 & 0 & 0 & 0 & 0 & 0 & -1 & -1 & -1 & -1 & -1 & 0 & -1 & -1 & -1 & -1 & -1 & 0 & 0 \\
 1 & 1 & 0 & 0 & 0 & 0 & 0 & 0 & -1 & -1 & -1 & -1 & -1 & 0 & -1 & -1 & -1 & -1 & -1 & 0 & 0 \\
 1 & 1 & 0 & 0 & 0 & 0 & 0 & 0 & -1 & -1 & -1 & -1 & -1 & 0 & -1 & -1 & -1 & -1 & -1 & 0 & 0 \\
 1 & 1 & 0 & 0 & 0 & 0 & 0 & 0 & -1 & -1 & -1 & -1 & -1 & 0 & -1 & -1 & -1 & -1 & -1 & 0 & 0 \\
 1 & 1 & 0 & 0 & 0 & 0 & 0 & 1 & 0 & 0 & 0 & 0 & 0 & 1 & 0 & 0 & 0 & 0 & 0 & 1 & 1 \\
 0 & 0 & -1 & -1 & -1 & -1 & -1 & 0 & -1 & -1 & -1 & -1 & -1 & 0 & 0 & 0 & 0 & 0 & 0 & 1 & 1 \\
 0 & 0 & -1 & -1 & -1 & -1 & -1 & 0 & -1 & -1 & -1 & -1 & -1 & 0 & 0 & 0 & 0 & 0 & 0 & 1 & 1 \\
 0 & 0 & -1 & -1 & -1 & -1 & -1 & 0 & -1 & -1 & -1 & -1 & -1 & 0 & 0 & 0 & 0 & 0 & 0 & 1 & 1 \\
 0 & 0 & -1 & -1 & -1 & -1 & -1 & 0 & -1 & -1 & -1 & -1 & -1 & 0 & 0 & 0 & 0 & 0 & 0 & 1 & 1 \\
 0 & 0 & -1 & -1 & -1 & -1 & -1 & 0 & -1 & -1 & -1 & -1 & -1 & 0 & 0 & 0 & 0 & 0 & 0 & 1 & 1 \\
 1 & 1 & 0 & 0 & 0 & 0 & 0 & 1 & 0 & 0 & 0 & 0 & 0 & 1 & 1 & 1 & 1 & 1 & 1 & 2 & 2 \\
 0 & 1 & 0 & 0 & 0 & 0 & 0 & 1 & 0 & 0 & 0 & 0 & 0 & 1 & 1 & 1 & 1 & 1 & 1 & 2 & 2
\ncline[nodesep=0.1cm]{15,1}{14,1}\ncline[nodesep=0.1cm]{14,1}{14,2}\ncline[nodesep=0.1cm]{14,2}{13,2}
\ncline[nodesep=0.1cm]{13,2}{12,2}\ncline[nodesep=0.1cm]{12,2}{11,2}\ncline[nodesep=0.1cm]{11,2}{10,2}
\ncline[nodesep=0.1cm]{10,2}{9,2}\ncline[nodesep=0.1cm]{9,2}{9,3}\ncline[nodesep=0.1cm]{9,3}{9,4}
\ncline[nodesep=0.1cm]{9,4}{9,5}\ncline[nodesep=0.1cm]{9,5}{9,6}\ncline[nodesep=0.1cm]{9,6}{9,7}

\ncline[nodesep=0.1cm]{9,7}{8,7}\ncline[nodesep=0.1cm]{8,7}{7,7}\ncline[nodesep=0.1cm]{7,7}{7,8}
\ncline[nodesep=0.1cm]{7,8}{7,9}\ncline[nodesep=0.1cm]{7,9}{7,10}\ncline[nodesep=0.1cm]{7,10}{7,11}
\ncline[nodesep=0.1cm]{7,11}{7,12}\ncline[nodesep=0.1cm]{7,12}{7,13}\ncline[nodesep=0.1cm]{7,13}{7,14}
\ncline[nodesep=0.1cm]{7,14}{7,15}

\ncline[nodesep=0.1cm]{9,7}{9,8}
\ncline[nodesep=0.1cm]{9,8}{9,9}\ncline[nodesep=0.1cm]{9,9}{9,10}\ncline[nodesep=0.1cm]{9,10}{9,11}
\ncline[nodesep=0.1cm]{9,11}{9,12}\ncline[nodesep=0.1cm]{9,12}{9,13}\ncline[nodesep=0.1cm]{9,13}{9,14}
\ncline[nodesep=0.1cm]{9,14}{9,15}\ncline[nodesep=0.1cm]{9,15}{8,15}\ncline[nodesep=0.1cm]{8,15}{7,15}

\ncline[nodesep=0.1cm]{7,15}{7,16}\ncline[nodesep=0.1cm]{7,16}{7,17}
\ncline[nodesep=0.1cm]{7,17}{7,18}\ncline[nodesep=0.1cm]{7,18}{7,19}
\ncline[nodesep=0.1cm]{7,19}{6,19}\ncline[nodesep=0.1cm]{6,19}{5,19}
\ncline[nodesep=0.1cm]{5,19}{4,19}\ncline[nodesep=0.1cm]{4,19}{3,19}
\ncline[nodesep=0.1cm]{3,19}{2,19}\ncline[nodesep=0.1cm]{2,19}{2,20}
\ncline[nodesep=0.1cm]{2,20}{1,20}\ncline[nodesep=0.1cm]{1,20}{1,21}

\psframe[linecolor=red,linearc=0.2](-14.5,-0.2)(-14,0.4)
\psframe[linecolor=red,linearc=0.2](-13.4,4)(-9.6,1.1)
\psframe[linecolor=red,linearc=0.2](-9,4)(-5.3,1.1)
\psframe[linecolor=red,linearc=0.2](-9,7.7)(-5.3,4.8)
\psframe[linecolor=red,linearc=0.2](-4.8,7.7)(-1,4.8)
\psframe[linecolor=red,linearc=0.2](-0.4,9)(0.2,8.4)
\psframe[linecolor=red,linestyle=dashed,linearc=0.2](-5.4,4.7)(-4.7,4.1)
\psframe[linecolor=red,linestyle=dashed,linearc=0.2](-9.7,4.7)(-9,4.1)
\end{psmatrix}
\end{pspicture}
\end{center}
\caption{Lattice cohomology of Example \ref{ex:6.2.4}}
\label{fig:ex4}
\end{figure}

\end{eexample}

\begin{eexample}\labelpar{ex:6.2.5}
In the previous examples, the generators of $\bH^1$ have a very special `shape' (that is, the loop is the 
smallest possible, containing only one lattice point). 
We provide a graph, shown in Figure \ref{fig:gex5}, which is interesting 
in this sense, i.e. it has more complicated $\bH^1$ 
generators. 
\begin{figure}[ht!]
\centering
\begin{picture}(140,60)(75,15)
\put(110,40){\circle*{3}}
\put(140,40){\circle*{3}}
\put(170,40){\circle*{3}}
\put(200,40){\circle*{3}}
\put(80,60){\circle*{3}}
\put(80,40){\circle*{3}}
\put(80,20){\circle*{3}}
\put(200,60){\circle*{3}}
\put(200,20){\circle*{3}}
\put(110,40){\line(1,0){90}}
\put(110,40){\line(-3,2){30}}
\put(110,40){\line(-1,0){30}}
\put(110,40){\line(-3,-2){30}}
\put(170,40){\line(3,2){30}}
\put(170,40){\line(3,-2){30}}
\put(85,65){\makebox(0,0){$-3$}}
\put(70,40){\makebox(0,0){$-4$}}
\put(85,15){\makebox(0,0){$-5$}}
\put(210,40){\makebox(0,0){$-3$}}
\put(190,65){\makebox(0,0){$-4$}}
\put(190,15){\makebox(0,0){$-5$}}
\put(115,50){\makebox(0,0){$-1$}}
\put(140,50){\makebox(0,0){$-10$}}
\put(165,50){\makebox(0,0){$-1$}}
\end{picture}
\vspace{0.5cm}
\caption{Graph for Example \ref{ex:6.2.5}}
\label{fig:gex5}
\end{figure}
Since the bound for the reduced lattice is big and can not be illustrated here with a picture, 
we give only 
the cohomology modules. $\bH^0(G,\K)=\calT_{-48}^+ \oplus\calT_{0}(1) \oplus \calT_{-26}(1) \oplus 
\calT_{-30}(1)\oplus \calT_{-36}(1)\oplus \calT_{-42}(1)\oplus \calT_{-44}(1)^2
\oplus \calT_{-46}(1)^2$ and 
$\bH^1(G,\K)=\calT_{-24}(1)\oplus\calT_{-40}(1)\oplus\calT_{-42}(1)\oplus\calT_{-44}(1)$, 
where the last three components are generated by the $1$--cycles from Figure \ref{fig:ex5}. The lower 
left corners of the blocks are in positions $(24,24)$, $(39,39)$ and $(44,44)$.
\begin{figure}[ht!]
\psset{unit=1cm}
\begin{center}
\begin{pspicture}(0,0.5)(0,2)
\begin{psmatrix}[colsep=0.5cm,rowsep=0.1cm]
 -20 & -20 & -21 & -22 & \ \ \ & -21 & -21 & -21 & -22 & \ \ \ & -22 & -22 & -22 & -23 \\
 -20 & -19 & -20 & -21 & \ \ \ & -21 & -20 & -20 & -21 & \ \ \ &  -22 & -21 & -21 & -22 \\
 -20 & -19 & -19 & -20 & \ \ \ & -22 & -21 & -20 & -21 & \ \ \ & -23 & -22 & -21 & -22 \\
 -21 & -20 & -20 & -20 & \ \ \ & -23 & -22 & -21 & -21 & \ \ \ & -24 & -23 & -22 & -22
\psline[linecolor=red,linestyle=dashed,linearc=0.2](-2.8,1.6)(-1,1.6)
\psline[linecolor=red,linestyle=dashed,linearc=0.2](-1,1.6)(-1,0.5)
\psline[linecolor=red,linestyle=dashed,linearc=0.2](-1,0.5)(-1.8,0.5)
\psline[linecolor=red,linestyle=dashed,linearc=0.2](-1.8,0.5)(-1.8,1)
\psline[linecolor=red,linestyle=dashed,linearc=0.2](-1.8,1)(-2.8,1)
\psline[linecolor=red,linestyle=dashed,linearc=0.2](-2.8,1)(-2.8,1.6)

\psline[linecolor=red,linestyle=dashed,linearc=0.2](-7.8,1.6)(-5.9,1.6)
\psline[linecolor=red,linestyle=dashed,linearc=0.2](-5.9,1.6)(-5.9,0.5)
\psline[linecolor=red,linestyle=dashed,linearc=0.2](-5.9,0.5)(-6.8,0.5)
\psline[linecolor=red,linestyle=dashed,linearc=0.2](-6.8,0.5)(-6.8,1)
\psline[linecolor=red,linestyle=dashed,linearc=0.2](-6.8,1)(-7.8,1)
\psline[linecolor=red,linestyle=dashed,linearc=0.2](-7.8,1)(-7.8,1.6)

\psline[linecolor=red,linestyle=dashed,linearc=0.2](-12.8,1.6)(-11.8,1.6)
\psline[linecolor=red,linestyle=dashed,linearc=0.2](-11.8,1.6)(-11.8,1)
\psline[linecolor=red,linestyle=dashed,linearc=0.2](-11.8,1)(-10.9,1)
\psline[linecolor=red,linestyle=dashed,linearc=0.2](-10.9,1)(-10.9,0.5)
\psline[linecolor=red,linestyle=dashed,linearc=0.2](-10.9,0.5)(-12.8,0.5)
\psline[linecolor=red,linestyle=dashed,linearc=0.2](-12.8,0.5)(-12.8,1.6)
\end{psmatrix}
\end{pspicture}
\end{center}
\caption{The `shape' of some $\bH^1$ generators in Example \ref{ex:6.2.5}}
\label{fig:ex5}
\end{figure}
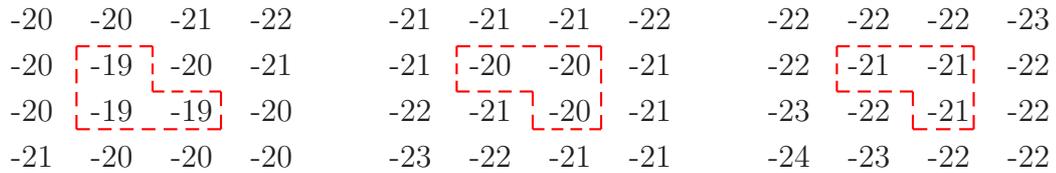
\end{eexample}

\begin{eexample}\labelpar{example:superisol}
 The last example provides a counterexample for the SWI Conjecture. In other words, for the graph 
$G$ given in Figure \ref{fig:gex6}, there exists an analytic realization for which the $p_g>eu(\bH^*)$.
\begin{figure}[ht!]
\centering
\begin{picture}(140,60)(75,15)
\put(110,40){\circle*{3}}
\put(140,40){\circle*{3}}
\put(170,40){\circle*{3}}
\put(200,40){\circle*{3}}
\put(230,40){\circle*{3}}
\put(260,40){\circle*{3}}
\put(80,40){\circle*{3}}
\put(50,40){\circle*{3}}
\put(110,10){\circle*{3}}
\put(170,10){\circle*{3}}
\put(110,40){\line(1,0){150}}
\put(110,40){\line(-1,0){60}}
\put(110,40){\line(0,-1){30}}
\put(170,40){\line(0,-1){30}}
\put(80,50){\makebox(0,0){$-2$}}
\put(50,50){\makebox(0,0){$-2$}}
\put(100,10){\makebox(0,0){$-4$}}
\put(180,10){\makebox(0,0){$-2$}}
\put(200,50){\makebox(0,0){$-3$}}
\put(230,50){\makebox(0,0){$-2$}}
\put(260,50){\makebox(0,0){$-2$}}
\put(110,50){\makebox(0,0){$-1$}}
\put(140,50){\makebox(0,0){$-31$}}
\put(170,50){\makebox(0,0){$-1$}}
\end{picture}
\vspace{0.5cm}
\caption{Graph for Example \ref{example:superisol}}
\label{fig:gex6}
\end{figure}
This example appeared in \cite[pg. 6]{LMN} and \cite[7.3.3]{Nlat}, since this topological type admits 
a superisolated hypersurface singularity with geometric genus $p_g^{(si)}=10$. On the other hand, 
if we take the complete  
intersection $\{z_1^3+z_2^4+z_3^5z_4=z_3^7+z_4^2+z_1^4z_2=0\}\subset (\C^4,0)$ divided by the 
diagonal $\Z_5$--action $(p^2,p^4,p,p)$, we get a splice--quotient type singularity with geometric genus 
$p_g^{(sq)}=8$. For other `generic' analytic types, $p_g$ drops even more.

Then we can analyze the lattice cohomological structure using the reduction to the nodes. The bound is 
$(30,34)$. Hence, in Figure \ref{fig:ex6} we show how the weight structure of $R[0,(30,34)]$ looks like.
Since this rectangle is rather big, the figure is constructed in a way that the point $(0,0)$ stays 
at the upper left corner, the vertical direction stands for $i$ and the 
horizontal is the direction of $j$. Moreover, the chosen minimal reduction set is visualized 
by the bold face characters.
One can read that 
$$\bH^0(G,\K)=\calT_{-10}^+ \oplus \calT_{-10}(3) \oplus \calT_{0}(1)^2 
\ \ \mbox{and} \ \ \bH^1(G,\K)=\calT_{-4}(1)^2.$$
This implies that $eu(\bH^0)=10$, $eu(\bH^*)=8$ and by traveling along the minimal reduction set one calculates 
$eu(\gamma_{min})=10$. Therefore, $p_g^{(si)}=eu(\gamma_{min})>eu(\bH^*)$, which also shows that the 
superisolated hypersurface structure is `extremal' in the sense that its geometric genus hits the 
maximum of possible values (cf. \ref{ss:SWIrev}).  
\begin{figure}[p]
\psset{unit=1cm}
\begin{center}
\begin{pspicture}(0,0.5)(0,6)
\begin{psmatrix}[colsep=0.04cm,rowsep=0.05cm]
\bf 0 &\bf 1 & 0 & 0 & -1 & -1 & -2 & -2 & -2 & -2 & -2 & -2 & -2 & -2 & -2 & -1 & -1 & 0 & 0 & 1 & 1 & 2 & 3 & 4 & 5 & 6 & 7 & 8 & 9 & 11 & 12 & 14 & 15 & 16 & 16 \\
 1 &\bf 1 & 0 & 0 & -1 & -1 & -2 & -2 & -2 & -2 & -2 & -2 & -2 & -2 & -2 & -1 & -1 & 0 & 0 & 1 & 1 & 2 & 3 & 4 & 5 & 6 & 7 & 8 & 9 & 11 & 12 & 14 & 14 & 15 & 15 \\
 0 &\bf 0 &\bf -1 &\bf -1 &\bf -2 &\bf -2 &\bf -3 &\bf -3 &\bf -3 &\bf -3 &\bf -3 &\bf -3 &\bf -3 &\bf -3 &\bf -3 & -2 & -2 & -1 & -1 & 0 & 0 & 1 & 2 & 3 & 4 & 5 & 6 & 7 & 8 & 10 & 11 & 12 & 12 & 13 & 13 \\
 -1 & -1 & -2 & -2 & -3 & -3 & -4 & -4 & -4 & -4 & -4 & -4 & -4 & -4 &\bf -4 & -3 & -3 & -2 & -2 & -1 & -1 & 0 & 1 & 2 & 3 & 4 & 5 & 6 & 7 & 9 & 9 & 10 & 10 & 11 & 11 \\
 -1 & -1 & -2 & -2 & -3 & -3 & -4 & -4 & -4 & -4 & -4 & -4 & -4 & -4 &\bf -4 & -3 & -3 & -2 & -2 & -1 & -1 & 0 & 1 & 2 & 3 & 4 & 5 & 6 & 7 & 8 & 8 & 9 & 9 & 10 & 10 \\
 -1 & -1 & -2 & -2 & -3 & -3 & -4 & -4 & -4 & -4 & -4 & -4 & -4 & -4 &\bf -4 & -3 & -3 & -2 & -2 & -1 & -1 & 0 & 1 & 2 & 3 & 4 & 5 & 6 & 6 & 7 & 7 & 8 & 8 & 9 & 9 \\
 -2 & -2 & -3 & -3 & -4 & -4 & -5 & -5 & -5 & -5 & -5 & -5 & -5 & -5 &\bf -5 & -4 & -4 & -3 & -3 & -2 & -2 & -1 & 0 & 1 & 2 & 3 & 4 & 4 & 4 & 5 & 5 & 6 & 6 & 7 & 7 \\
 -2 & -2 & -3 & -3 & -4 & -4 & -5 & -5 & -5 & -5 & -5 & -5 & -5 & -5 &\bf -5 & -4 & -4 & -3 & -3 & -2 & -2 & -1 & 0 & 1 & 2 & 3 & 3 & 3 & 3 & 4 & 4 & 5 & 5 & 6 & 6 \\
 -2 & -2 & -3 & -3 & -4 & -4 & -5 & -5 & -5 & -5 & -5 & -5 & -5 & -5 &\bf -5 & -4 & -4 & -3 & -3 & -2 & -2 & -1 & 0 & 1 & 2 & 2 & 2 & 2 & 2 & 3 & 3 & 4 & 4 & 5 & 5 \\
 -2 & -2 & -3 & -3 & -4 & -4 & -5 & -5 & -5 & -5 & -5 & -5 & -5 & -5 &\bf -5 & -4 & -4 & -3 & -3 & -2 & -2 & -1 & 0 & 1 & 1 & 1 & 1 & 1 & 1 & 2 & 2 & 3 & 3 & 4 & 4 \\
 -2 & -2 & -3 & -3 & -4 & -4 & -5 & -5 & -5 & -5 & -5 & -5 & -5 & -5 &\bf -5 & -4 & -4 & -3 & -3 & -2 & -2 & -1 & 0 & 0 & 0 & 0 & 0 & 0 & 0 & 1 & 1 & 2 & 2 & 3 & 3 \\
 -2 & -2 & -3 & -3 & -4 & -4 & -5 & -5 & -5 & -5 & -5 & -5 & -5 & -5 &\bf -5 & -4 & -4 & -3 & -3 & -2 & -2 & -1 & -1 & -1 & -1 & -1 & -1 & -1 & -1 & 0 & 0 & 1 & 1 & 2 & 2 \\
 -2 & -2 & -3 & -3 & -4 & -4 & -5 & -5 & -5 & -5 & -5 & -5 & -5 & -5 &\bf -5 &\bf -4 &\bf -4 &\bf -3 &\bf -3 &\bf -2 & -2 & -2 & -2 & -2 & -2 & -2 & -2 & -2 & -2 & -1 & -1 & 0 & 0 & 1 & 1 \\
 -1 & -1 & -2 & -2 & -3 & -3 & -4 & -4 & -4 & -4 & -4 & -4 & -4 & -4 &\bf -4 &\bf -3 &\bf -3 &\bf -2 &\bf -2 &\bf -1 &\bf -2 & -2 & -2 & -2 & -2 & -2 & -2 & -2 & -2 & -1 & -1 & 0 & 0 & 1 & 1 \\
 -1 & -1 & -2 & -2 & -3 & -3 & -4 & -4 & -4 & -4 & -4 & -4 & -4 & -4 &\bf -4 &\bf -3 &\bf -3 &\bf -2 &\bf -2 &\bf -2 &\bf -3 & -3 & -3 & -3 & -3 & -3 & -3 & -3 & -3 & -2 & -2 & -1 & -1 & 0 & 0 \\
 -1 & -1 & -2 & -2 & -3 & -3 & -4 & -4 & -4 & -4 & -4 & -4 & -4 & -4 &\bf -4 &\bf -3 &\bf -3 &\bf -2 &\bf -3 &\bf -3 &\bf -4 & -4 & -4 & -4 & -4 & -4 & -4 & -4 & -4 & -3 & -3 & -2 & -2 & -1 & -1 \\
 0 & 0 & -1 & -1 & -2 & -2 & -3 & -3 & -3 & -3 & -3 & -3 & -3 & -3 &\bf -3 &\bf -2 &\bf -2 &\bf -2 &\bf -3 &\bf -3 &\bf -4 & -4 & -4 & -4 & -4 & -4 & -4 & -4 & -4 & -3 & -3 & -2 & -2 & -1 & -1 \\
 1 & 1 & 0 & 0 & -1 & -1 & -2 & -2 & -2 & -2 & -2 & -2 & -2 & -2 &\bf -2 &\bf -1 &\bf -2 &\bf -2 &\bf -3 &\bf -3 &\bf -4 & -4 & -4 & -4 & -4 & -4 & -4 & -4 & -4 & -3 & -3 & -2 & -2 & -1 & -1 \\
 1 & 1 & 0 & 0 & -1 & -1 & -2 & -2 & -2 & -2 & -2 & -2 & -2 & -2 &\bf -2 &\bf -2 &\bf -3 &\bf -3 &\bf -4 &\bf -4 &\bf -5 & -5 & -5 & -5 & -5 & -5 & -5 & -5 & -5 & -4 & -4 & -3 & -3 & -2 & -2 \\
 2 & 2 & 1 & 1 & 0 & 0 & -1 & -1 & -1 & -1 & -1 & -1 & -1 & -1 & -2 & -2 & -3 & -3 & -4 & -4 &\bf -5 & -5 & -5 & -5 & -5 & -5 & -5 & -5 & -5 & -4 & -4 & -3 & -3 & -2 & -2 \\
 3 & 3 & 2 & 2 & 1 & 1 & 0 & 0 & 0 & 0 & 0 & 0 & 0 & -1 & -2 & -2 & -3 & -3 & -4 & -4 &\bf -5 & -5 & -5 & -5 & -5 & -5 & -5 & -5 & -5 & -4 & -4 & -3 & -3 & -2 & -2 \\
 4 & 4 & 3 & 3 & 2 & 2 & 1 & 1 & 1 & 1 & 1 & 1 & 0 & -1 & -2 & -2 & -3 & -3 & -4 & -4 &\bf -5 & -5 & -5 & -5 & -5 & -5 & -5 & -5 & -5 & -4 & -4 & -3 & -3 & -2 & -2 \\
 5 & 5 & 4 & 4 & 3 & 3 & 2 & 2 & 2 & 2 & 2 & 1 & 0 & -1 & -2 & -2 & -3 & -3 & -4 & -4 &\bf -5 & -5 & -5 & -5 & -5 & -5 & -5 & -5 & -5 & -4 & -4 & -3 & -3 & -2 & -2 \\
 6 & 6 & 5 & 5 & 4 & 4 & 3 & 3 & 3 & 3 & 2 & 1 & 0 & -1 & -2 & -2 & -3 & -3 & -4 & -4 &\bf -5 & -5 & -5 & -5 & -5 & -5 & -5 & -5 & -5 & -4 & -4 & -3 & -3 & -2 & -2 \\
 7 & 7 & 6 & 6 & 5 & 5 & 4 & 4 & 4 & 3 & 2 & 1 & 0 & -1 & -2 & -2 & -3 & -3 & -4 & -4 &\bf -5 &\bf -5 &\bf -5 &\bf -5 &\bf -5 &\bf -5 &\bf -5 &\bf -5 &\bf -5 & -4 & -4 & -3 & -3 & -2 & -2 \\
 9 & 9 & 8 & 8 & 7 & 7 & 6 & 6 & 5 & 4 & 3 & 2 & 1 & 0 & -1 & -1 & -2 & -2 & -3 & -3 & -4 & -4 & -4 & -4 & -4 & -4 & -4 & -4 &\bf -4 & -3 & -3 & -2 & -2 & -1 & -1 \\
 10 & 10 & 9 & 9 & 8 & 8 & 7 & 6 & 5 & 4 & 3 & 2 & 1 & 0 & -1 & -1 & -2 & -2 & -3 & -3 & -4 & -4 & -4 & -4 & -4 & -4 & -4 & -4 &\bf -4 & -3 & -3 & -2 & -2 & -1 & -1 \\
 11 & 11 & 10 & 10 & 9 & 9 & 7 & 6 & 5 & 4 & 3 & 2 & 1 & 0 & -1 & -1 & -2 & -2 & -3 & -3 & -4 & -4 & -4 & -4 & -4 & -4 & -4 & -4 &\bf -4 &\bf -3 &\bf -3 &\bf -2 &\bf -2 & -1 & -1 \\
 13 & 13 & 12 & 12 & 11 & 10 & 8 & 7 & 6 & 5 & 4 & 3 & 2 & 1 & 0 & 0 & -1 & -1 & -2 & -2 & -3 & -3 & -3 & -3 & -3 & -3 & -3 & -3 & -3 & -2 & -2 & -1 &\bf -1 &\bf 0 & 0 \\
 15 & 15 & 14 & 14 & 12 & 11 & 9 & 8 & 7 & 6 & 5 & 4 & 3 & 2 & 1 & 1 & 0 & 0 & -1 & -1 & -2 & -2 & -2 & -2 & -2 & -2 & -2 & -2 & -2 & -1 & -1 & 0 & 0 &\bf 1 &\bf 1 \\
 16 & 16 & 15 & 14 & 12 & 11 & 9 & 8 & 7 & 6 & 5 & 4 & 3 & 2 & 1 & 1 & 0 & 0 & -1 & -1 & -2 & -2 & -2 & -2 & -2 & -2 & -2 & -2 & -2 & -1 & -1 & 0 & 0 & 1 &\bf 0

\psframe[linecolor=red,linearc=0.2](-15.4,16.6)(-14.9,17.2)
\psframe[linecolor=red,linearc=0.2](-12.7,13.7)(-8.7,10)
\psframe[linecolor=red,linearc=0.2](-6.6,7.1)(-2.6,3.3)
\psframe[linecolor=red,linearc=0.2](-0.4,0.4)(0.2,-0.2)
\psframe[linecolor=red,linestyle=dashed,linearc=0.2](-7,9.9)(-6.5,9.3)
\psframe[linecolor=red,linestyle=dashed,linearc=0.2](-8.7,7.7)(-8.2,7.1)
\end{psmatrix}
\end{pspicture}
\end{center}
\caption{Lattice cohomology of Example \ref{example:superisol}}
\label{fig:ex6}
\end{figure}
\end{eexample}


\begin{thebibliography}{999}

\bibitem{AGV} Arnold, V.I., Gussein--Zade, S.M. and  Varchenko, A.N.:
Singularities of differentiable
maps volume 2 (engl.vers.), {\em Monographs in mathematics, Birkh\"auser} {\bf 83}, 1988.

\bibitem{Artin62} Artin, M.:
Some numerical criteria for contractibility of curves on algebraic surfaces, 
{\em  Amer. J. of Math.} {\bf 84} (1962),  485--496.

\bibitem{Artin66} Artin, M.:
On isolated rational singularities of surfaces, 
{\em Amer. J. of Math.} {\bf 88} (1966), 129--136.

\bibitem{BPVVbook} Barth, W., Peters, C. and Van de Ven, A.: Compact complex surfaces, 
{\em Springer--Verlag}, Berlin \& Heidelberg, 1984.

\bibitem{Bar} Barvinok, A.I.:
 Computing the Ehrhart polynomial of a convex lattice
 polytope, {\it Discrete Comput. Geom.} {\bf  12} (1994), 35--48.

\bibitem{BP}  Barvinok, A. and Pommersheim, J.: An algorithmic theory of lattice points in
polyhedra, {\em New Perspectives in Algebraic Combinatorics} {\bf 38} (1999), 91--147.

\bibitem{Beck_c} Beck, M.: A closer look at lattice points in rational simplices, 
{\em Electron. J. Combin.} {\bf 6} (1999), Research Paper 37, 9 pp. (electronic).

\bibitem{Beck_m} Beck, M.: Multidimensional Ehrhart reciprocity, 
{\em J. Combin. Theory Ser. A} {\bf 97} (2002), no. 1, 187--194.

\bibitem{Beck_p} Beck, M.: The Partial-Fractions Method for Counting Solutions to Integral
Linear Systems, {\em Disc. \& Comp. Geom.} (2004), 437--446.

\bibitem{BR1} Beck, M., Robins, S.: Explicit and efficient formulas for the
 lattice point count inside rational polygons, {\it Discrete \& Computational
 Geometry} {\bf  27}  (2002), 443--459.

\bibitem{BR2} Beck, M.,  Robins, S.: Computing the Continuous Discretely:
 Integer-point enumeration in polyhedra, {\it Undergraduate Texts in
 Mathematics}, {\em Springer-Verlag}, New York 2007.

\bibitem{BDR} Beck, M., Diaz, R., Robins, S.:  The Frobenius problem,
 rational polytopes, and Fourier-Dedekind sums, {\it Journal of Number
 Theory}, {\bf  96} (2002), 1--21.

\bibitem{BN07} Braun, G. and N\'emethi, A.:
Invariants of Newton non--degenerate surface singularities, {\em Compos. Math.} {\bf 143} (2007),
1003--1036.

\bibitem{BN} Braun, G. and  N\'emethi, A.:
Surgery formula for the Seiberg-Witten invariants
of negative definite plumbed 3-manifolds,
{\em Journal f\"ur die Reine und angewandte Mathematik} {\bf 638} (2010), 189--208.

\bibitem{B66a} Brieskorn, E.: Examples of singular normal complex spaces which are topological 
manifolds, {\em Proc. Nat. Acad. Sci.} {\bf 55} (1966), 1395--1397.

\bibitem{B66b} Brieskorn, E.: Beispiele zur Differentialtopologie von Singularit\"aten, 
{\em Invent. Math.} {\bf 2} (1966), 1--14.

\bibitem{BK} Brieskorn, E. and Kn\"orrer, H.: 
Plane algebraic curves, 
{\em Birkh\"auser}, Boston, 1986.

\bibitem{BV} Brion, M. and Vergne, M.: Residue formulae, vector partition functions and
lattice points in rational polytopes, {\em J. Amer. Math. Soc.} {\bf 10} (1997), 797--833.

\bibitem{c-m}
Bruns, W. and Herzog, J.: Cohen--Macaulay rings (Revised edition),
{\em Cambridge Stud. Adv. Math.}, vol. 39, Cambridge Univ. Press, Cambridge (1998).

\bibitem{BG} Bruns, W. and Gubeladze, J.: Polytopes, rings and K-theory, 
{\em Springer Mon. in Math. Ser.}, XIV, Springer (2009).


\bibitem{CDG} Campillo, A., Delgado, F. and Gusein-Zade, S. M.:
Poincar\'e series of a rational surface singularity, {\em Invent. Math.} {\bf 155} (2004),
no. 1, 41--53.

\bibitem{CDGEq} Campillo, A., Delgado, F. and Gusein-Zade, S. M.:
Universal abelian covers of rational
surface singularities and multi-index filtrations,
{\em Funk. Anal. i Prilozhen.} {\bf 42} (2008), no. 2, 3--10.

\bibitem{CL} Clauss, Ph. and Loechner, V.: Parametric Analysis of Polyedral Iteration Spaces,
{\em J. of VLSI Signal Proc.} (1998), 1--16.

\bibitem{CHR} Cutkosky, S. D.,  Herzog, J. and Reguera, A.:
Poincar\'e series of resolutions of surface singularities,
{\em Trans. Amer. Math. Soc.} {\bf 356} (2004), no. 5,  1833--1874.

\bibitem{DR}
Diaz, R.,  Robins, S.: The Ehrhart polynomial of a lattice
 polytope, {\it Annals of Mathematics}, {\bf 145} (1997), 503--518.

\bibitem{Dimca} Dimca, A.: Singularities and Topology of Hypersurfaces,
{\em Universitext, Springer--Verlag}, 1992.

\bibitem{Do} Dolgachev, I.V.: On the link space of a Gorenstein quasihomogeneous
surface singularity, {\em Math Annalen} {\bf 265} (1983), 529--540.

\bibitem{D78} Durfee, A.H.: The signature of smoothings of complex surface singularities,
{\em Math. Ann.} {\bf 232} (1978), no. 1, 85--98.

\bibitem{EN} Eisenbud, D. and Neumann, W.: Three-dimensional
link theory and invariants of plane curve singularities,
{\em Ann. of Math. Studies} {\bf 110}, Princeton University Press, 1985.

\bibitem{G62} Grauert, H.: \"Uber Modifikationen und exzeptionelle analytische Mengen, 
{\em Math. Annalen} {\bf 146} (1962), 331--368.

\bibitem{Greene} Greene, J.: A surgery triangle for lattice cohomology, 
{\em Alg. \& Geom. Topology} {\bf 13} (2013), 441--451.

\bibitem{GS} Gompf, R.E. and Stipsicz, I.A.:  An Introduction to
$4$-Manifolds and Kirby Calculus, {\em Graduate Studies in Mathematics} vol. {\bf 20},
Amer. Math. Soc., 1999.

\bibitem{GN} Gorsky, E. and N\'emethi, A.: Poincar\'e series of algebraic links and lattice homology,
{\em arXiv:1301.7636 [math.AG]} (2013).

\bibitem{G-W} Goto, S. and Watanabe, K.: On graded rings I., 
{\em J. Math. Soc. Japan} {\bf 30} no. 2 (1978), 179--213.


\bibitem{HNK71} Hirzebruch, F., Neumann, W. and Koh, S.S.: Differentiable manifolds and quadratic forms, 
{\em Marcel Dekker, Inc. New York}, 1971.

\bibitem{KM} Kronheimer, P. and Mrowka, T.: Monopoles and three--manifolds, 
{\em New Math. Mon., Cambridge Univ. Press}, 2010.

\bibitem{Ehrhart} L\'aszl\'o, T. and N\'emethi, A.: Ehrhart theory of polytopes and Seiberg--Witten 
invariants of plumbed 3--manifolds, {\em arXiv:1211.2539 [math.AG]} (2012).

\bibitem{redthm} L\'aszl\'o, T. and N\'emethi, A.: The reduction theorem for lattice cohomology, 
{\em arXiv:1302.4716 [math.GT]} (2013).

\bibitem{Lauferb} Laufer, H.B.: Normal two--dimensional singularities, {\em Ann. of Math. Studies} {\bf 71}, 
Princeton University Press, 1971.

\bibitem{Laufer72} Laufer, H.B.: On rational singularities,
{\em Amer. J. of Math.} {\bf 94} (1972), 597-608.

\bibitem{Laufer73} Laufer, H.B.: Taut two--dimensional singularities, 
{\em Math. Ann.} {\bf 205} (1973), 131--164.

\bibitem{Laufer77} Laufer, H.B.: On minimally elliptic singularities,
{\em Amer. J. of Math.} {\bf 99} (1977), 1257--1295.

\bibitem{Lescop} Lescop, C.: Global surgery formula for the Casson--Walker
invariant, {\em Ann. of Math. Studies}  {\bf 140}, Princeton Univ. Press, 1996.

\bibitem{LeDT} L\^{e}, D.T.: Singularit\'es isol\'ees des intersections compl\'etes, 
Introduction \'a la th\'eorie des singularit\'es, I, 1--48, 
{\em Travaux en Cours} {\bf 36}, Hermann, Paris, 1988. 

\bibitem{Lim} Lim, Y.: Seiberg--Witten invariants for $3$--manifolds in the 
case $b_1=0$ or $1$, {\em Pacific J. of Math.} {\bf 195} no. 1 (2000), 179--204.

\bibitem{LMN} Luengo-Velasco, I., Melle-Hern\'andez, A. and N\'emethi, A.:
Links and analytic invariants of superisolated singularities,
{\em Journal of Alg. Geom.} {\bf 14} (2005), 543--565.

\bibitem{MW} Marcolli, M. and Wang, B.L.: Seiberg--Witten invariant and the Casson--Walker 
invariant for rational homology $3$--spheres, {\em Geom. Dedicata} {\bf 91} (2002), 45--58.

\bibitem{M95} Mellor, B.: $Spin^c$--manifolds,
{\em http://myweb.lmu.edu/bmellor}, unpublished exposition, 1995.

\bibitem{MN} Mendris, R. and N\'emethi, A.: The link of ${f(x,y)+z^n=0}$ and Zariski's Conjecture, 
{\em Compositio Math.} {\bf 141}(2) (2005), 502--524.

\bibitem{M} Mendris, R.: The link of suspension singularities and Zariski's conjecture, 
{\em PhD Dissertation}, The Ohio State University, 2003.

\bibitem{MT}
Merle, M. and  Teissier, B.:
Conditions d'adjonction, d'apr\'es DuVal,  in {\em S\'eminaire sur les Singularit\'es des Surfaces}, 
Lecture Notes in Math., {\bf 777} (1980), Springer, Berlin, 229--245.

\bibitem{M78} McMullen, P.: Lattice invariant valuations on rational polytopes, {\em Arch. Math.},
{\bf 31} (1978), 509--516.

\bibitem{Milnorbook} Milnor, J.: Singular points of complex
hypersurfaces, {\em Ann. of Math. Studies} {\bf 61}, Princeton
Univ. Press, 1968.

\bibitem{Morgan} Morgan, J. W.: The Seiberg--Witten equations and applications to the topology of 
smooth four--manifolds, {\em Mathematical Notes} {\bf 44}, Princeton University Press, 1996.

\bibitem{Mum61} Mumford, D.: The topology of normal singularities of an algebraic surface and 
criterion for simplicity, {\em IHES Publ. Math.} {\bf 9} (1961), 5--22.


\bibitem{Nem99} N\'emethi, A.: ``Weakly'' elliptic Gorenstein singularities of surfaces, 
{\em Invent. Math.} {\bf 137} (1999), 145--167.

\bibitem{NSW} N\'emethi, A.: The Seiberg-Witten invariants of negative definite plumbed 3-manifolds,
{\em  J. Eur. Math. Soc.} {\bf 13} No. 4 (2011), 959--974.

\bibitem{NPS} N\'emethi, A.: Poincar\'e series associated with surface singularities, in Singularities I, 271--297,
{\em Contemp. Math.} {\bf 474}, Amer. Math. Soc., Providence RI, 2008.

\bibitem{NCL} N\'emethi, A.: The cohomology of line bundles of splice-quotient singularities,
{\em Advances in Math.} {\bf 229} 4 (2012), 2503--2524.

\bibitem{Ninv} N\'emethi, A.: Invariants of normal surface singularities, 
{\em Proceedings of the Conference: 
Real and Complex Singularities, Sao Carlos, Brazil, August 2002}, 
{\em Contemporary Mathematics} {\bf 354} (2004), 161--208.

\bibitem{OSZINV}  N\'emethi, A.: On the Ozsv\'ath-Szab\'o
invariant of negative definite plumbed 3-manifolds,
{\em Geometry and Topology}, {\bf 9} (2005), 991--1042.

\bibitem{Nsur}  N\'emethi, A.: On the Heegaard Floer homology of $S^3_{-d}(K)$ and
unicuspidal rational plane curves,  {\em Fields Institute
Communications} Vol. {\bf 47} (2005), 219--234;
 ``Geometry and Topology of Manifolds'',
Editors: H.U. Boden, I. Hambleton, A.J. Nicas and B.D. Park,
 (Proceedings of the  Conference at McMaster University, May 2004).

\bibitem{Nsurrat}  N\'emethi, A.: On the Heegaard Floer homology of
$S^3_{-p/q}(K)$, math.GT/0410570, publishes as part of \cite{Ng}.

\bibitem{Nlat}  N\'emethi, A.: Lattice cohomology of normal surface singularities,
 {\em Publ. RIMS. Kyoto Univ.}, {\bf 44} (2008), 507--543.

\bibitem{Nexseq} N\'emethi, A.: Two exact sequences for lattice cohomology,
Proceedings of the conference organized to honor H. Moscovici's 65th birthday,
{\em Contemporary Math.} {\bf 546} (2011), 249--269.

\bibitem{Ng} N\'emethi, A.:  Graded roots and singularities,
 Proceedings {\em Advanced
School and Workshop} on {\em Singularities in Geometry and
Topology} ICTP (Trieste, Italy),
 World Sci. Publ., Hackensack, NJ, 2007, 394--463.

\bibitem{Nfive} N\'emethi, A.: Five lectures on normal surface singularities,
lectures at the Summer School in {\em Low dimensional topology} Budapest,
Hungary, 1998; Bolyai Society Math. Studies {\bf 8} (1999), 269--351.

\bibitem{Nline} N\'emethi, A.: Line bundles associated with normal surface singularities, 
{\em arXiv:math/0310084 [math.AG]}, published as part of \cite{Ng}.

\bibitem{SWI} N\'emethi, A. and Nicolaescu, L.I.:
 Seiberg--Witten invariants and surface singularities,
{\em Geometry and Topology} {\bf 6} (2002), 269--328.

\bibitem{SWII} N\'emethi, A. and Nicolaescu, L.I.:
Seiberg--Witten invariants and surface singularities II (singularities with good $\C^*$-action),
{\em Journal of London Math. Soc.} {\bf 69} (2) (2004), 593--607.

\bibitem{SWIII} N\'emethi, A. and Nicolaescu, L.I.:
Seiberg--Witten invariants and surface singularities: Splicings and cyclic covers,
{\em Selecta Mathematica} {\bf 11} nr. 3-4 (2005), 399--451.

\bibitem{NO} N\'emethi, A. and Okuma, T.:
The Seiberg--Witten invariant conjecture for splice--quotients,
{\em Journal of London Math. Soc.} {\bf 28} (2008), 143--154.

\bibitem{NO1} N\'emethi, A. and Okuma, T.:
On the Casson invariant conjecture of Neumann--Wahl, {\em Journal of Algebraic Geometry}
{\bf 18} (2009), 135--149.

\bibitem{NO2} N\'emethi, A. and Okuma, T.:
The embedding dimension of weighted homogeneous surface singularities,
{\em J. of Topology} {\bf 3} (2010), 643--667.

\bibitem{NR} N\'emethi, A. and Rom\'an, F.:
The lattice cohomology of $S^3_{-d}(K)$, Proceedings of
the `Recent Trends on Zeta Functions in Algebra and Geometry',
2010 Mallorca (Spain), {\em Contemporary Mathematics}.

\bibitem{Nem10} N\'emethi, A.: Some properties of the lattice cohomology, 
{\em Proceedings of the `Geometry Conference' meeting}, organized at Yamagata University, Japan, 
September 2010.

\bibitem{NP} Neumann, W.D.: A calculus for plumbing applied to the
topology of complex surface singularities and degenerating
complex curves, {\em Transactions of the AMS} {\bf 268} (2)
(1981), 299--344.

\bibitem{neumann.abel}
Neumann, W.D.: Abelian covers of quasihomogeneous surface singularities,
{\em Singularities, Part 2 (Arcata, Calif., 1981), Proc. Sympos. Pure Math.}, vol. 40,
Amer. Math. Soc., Providence, RI (1983), 233--243.

\bibitem{NW} Neumann, W.D. and Wahl, J.: Casson invariant of links of singularities,
{\em Comment. Math. Helv.} {\bf 65} no. 1 (1990), 58--78.

\bibitem{nw-CIuac} Neumann, W.D. and Wahl, J.: 
Complete intersection singularities of splice type as universal abelian covers, 
{\em Geom. Topol.} {\bf 9} (2005), 699--755.

\bibitem{Nic00_1} Nicolaescu, L.: Notes on Seiberg-Witten theory, 
{\em AMS Graduate Series in Math. Monograph} {\bf 28}, xii+ 482pp, 2000.

\bibitem{Nic00_2} Nicolaescu, L.: Seiberg--Witten invariants of rational homology $3$--spheres, 
{\em http://www3.nd.edu/~lnicolae/swc.pdf}, unpublished manuscript, 2000.

\bibitem{Nic04} Nicolaescu, L.: Seiberg--Witten invariants of rational homology $3$--spheres, 
{\em Comm. in Cont. Math.} {\bf 6} no. 6 (2004), 833--866. 

\bibitem{Opg} Okuma, T.: The geometric genus of splice quotient singularities,
{\em Trans. Amer. Math. Soc.} {\bf 360} no. 12 (2008), 6643--6659.

\bibitem{OSz00} Ozsv\'ath, P.S. and Szab\'o, Z.: The theta divisor and 
the Casson--Walker invariant, {\em arXiv:math/0006194 [math.GT]} (2000).

\bibitem{OSzP} Ozsv\'ath, P.S. and Szab\'o, Z.: On the Floer
homology of plumbed three-manifolds, {\em Geom. Topol.} {\bf 7} (2003),
185--224.

\bibitem{OSz} Ozsv\'ath, P.S. and Szab\'o, Z.:
Holomorphic disks and topological invariants for closed three-manifolds,
{\em Ann. of Math.} {\bf 159} (2) no. 3 (2004), 1027--1158.

\bibitem{OSz7} Ozsv\'ath, P.S. and Szab\'o, Z.: Holomorphic discs
and three--manifold invariants: properties and applications,
{\em Annals of Math.} {\bf 159} (2004), 1159--1245.

\bibitem{OSz06a} Ozsv\'ath, P.S. and Szab\'o, Z.: An Introduction to Heegaard Floer homology, 
{\em Floer homology, gauge theory, and low--dimensional topology} {\bf 5} (2006), 3--27.

\bibitem{OSz06b} Ozsv\'ath, P.S. and Szab\'o, Z.: Lectures on Heegaard Floer homology, 
{\em Floer homology, gauge theory, and low-dimensional topology} {\bf 5} (2006), 29--70.

\bibitem{OSSz1} Ozsv\'ath, P., Stipsicz, A. and Szab\'o, Z.:
A spectral sequence on lattice homology, {\em arXiv:1206.1654 [math.GT]} (2012).

\bibitem{OSSz2} Ozsv\'ath, P., Stipsicz, A. and Szab\'o, Z.:
Knots in lattice homology, {\em arXiv:1208.2617 [math.GT]} (2012).

\bibitem{OSSz3} Ozsv\'ath, P., Stipsicz, A. and Szab\'o, Z.:
Knot lattice homology in L-spaces, {\em arXiv:1207.3889 [math.GT]} (2012).

\bibitem{P} Pinkham, H.: Normal surface singularities with $\C^*$ action, 
{\em Math. Ann.} {\bf 117} (1977), 183--193.

\bibitem{PSz} P\'olya, G. and Szeg\H{o}, G.: Problems and theorems in
analysis I., Classics in Mathematics, {\em Springer} (1998).

\bibitem{Rus} Rustamov, R.: A surgery formula for renormalized
Euler characteristic of Heegaard Floer homology, math.GT/0409294.

\bibitem{Sev02} Sevaliev, N.: Invariants for homology 3--spheres, 
{\em Springer, Ser. Low--Dim. Top} {\bf 140}, 2002.

\bibitem{S74} Stanley, R.P.: Combinatorial reciprocity theorem,
{\it Adv. in Math.} {\bf 14} (1974), 194--253.

\bibitem{Str} Sturmfels, B.: On vector partition functions,
 {\em J. Combin. Theory Ser. A} {\bf 72} no. 2 (1995), 302--309.

\bibitem{SZV} Szenes, A. and Vergne, M.: Residue formulae for
vector partitions and Euler-Maclaurin sums, {\it Adv. in Applied Math.} {\bf 30} (2003), 295--342.

\bibitem{wa} Wagreich, P.: The structure of quasihomogeneous singularities,
{\em Proc. of Symp. in Pure Path.}, {\bf 40} Part 2 (1983), 593--611.

\bibitem{Wall} Wall, C.T.C.: Singular points of plane curves, 
{\em Cambridge Univ. Press}, Cambridge, 2004.

\bibitem{Witten} Witten, E.: Monopoles and four--manifolds, 
{\em Mathematical Research Letters} {\bf 1} no. 6 (1994), 769--796.	 

\bibitem{Yau80} Yau, S.S.-T.: On maximally elliptic singularities, 
{\em Trans. of the AMS} {\bf 257} no. 2 (1980), 269--329.

\bibitem{Zieg} Ziegler, G.M.: Lectures on Polytopes, {\em Springer}, ed.7, (1995)



\end{thebibliography}
\end{document}